\documentclass[leqno,12pt]{report}
\usepackage[includehead,includefoot,margin=31mm]{geometry}
\usepackage[latin1]{inputenc}
\usepackage[cyr]{aeguill}
\usepackage{amsmath}
\usepackage{amsthm}
\usepackage{amsfonts}  
\usepackage{amssymb}
\usepackage{amsmath}
\usepackage{mathrsfs}
\usepackage[all]{xy}
\usepackage[french]{babel}   
\usepackage[T1]{fontenc}

\title{Sur les op\'erations de tores alg\'ebriques de complexit\'e un dans les vari\'et\'es affines}
\author{Kevin Langlois}
\date{}

\begin{document}

\maketitle

\begin{figure}[p]
  \centering
  \em \`A l'\^ile en forme de papillon. \rm
\end{figure}

\newpage
\strut 
\newpage

\clearpage

\newpage
\strut 
\newpage

\section*{\centering Remerciements}
Je souhaite en premier lieu adresser mes remerciements \`a mon directeur de th\`ese, Mikhail Zaidenberg. J'ai eu la chance pendant ces trois ann\'ees de th\`ese d'\^etre encadr\'e par une personne incroyable que ce soit du point de vue professionnel ou des \'echanges amicaux que nous avons eus. Vous m'avez beaucoup transmis avec patience, d\'elicatesse et bonne humeur. Vous m'avez laiss\'e la libert\'e de mener mes propres recherches. Vous avez contribu\'e de fa\c con significative \`a am\'eliorer mon propre potentiel. Vous m'avez form\'e au m\'etier d'enseignant chercheur. Je suis tr\`es fier d'\^etre votre \'el\`eve.    
\
\newline
\
\newline
Je remercie David A. Cox pour avoir accept\'e d'\^etre un des rapporteurs de cette th\`ese. Durant la deuxi\`eme ann\'ee de mast\`ere, j'ai appris la th\'eorie des vari\'et\'es toriques \`a partir du livre que vous avez \'ecrit avec John Litlle et Hal Schenck. Je vous dois de fa\c con indirecte la plupart de mes connaissances sur ce sujet, \`a nouveau, merci. 
\
\newline
\
\newline
Je remercie J\"urgen Hausen pour avoir accept\'e d'\^etre un des rapporteurs de cette th\`ese. Vous \^etes un des fondateurs avec Klaus Altmann de la th\'eorie des diviseurs poly\'edraux sur laquelle j'ai travaill\'e pendant trois ans. Sans vos travaux remarquables, cette th\`ese n'existerait pas, merci encore.
\
\newline
\
\newline
Je remercie Michel Brion et Hanspeter Kraft pour avoir accept\'e d'\^etre un des membres du jury. Merci encore de vous \^etre int\'eress\'es \`a mes travaux de recherche et d'avoir apport\'e des corrections ou des remarques pendant mes ann\'ees de th\`ese.  Je remercie \'egalement Alvaro Liendo pour le travail en collaboration du dernier chapitre de cette th\`ese et, Ivan Arzhantsev et Hubert Flenner pour m'avoir donn\'e quelques conseils.
\
\newline
\
\newline
{\em Je souhaiterais remercier toutes les personnes qui ont contribu\'e \`a ma formation. Donnons une liste non exhaustive.}
\
\newline
\
\newline
Un grand merci \`a Jos\'e Bertin, \`a Catherine Bouvier, \`a Jean-Pierre Demailly, \`a Zindine Djadli, \`a Odile Garotta, \`a Benoit Kloeckner et \`a Fr\'ed\'eric Mouton pour m'avoir form\'e au m\'etier d'enseignant. 
\
\newline
\
\newline
Lors de l'\'et\'e qui suit la premi\`ere ann\'ee de mast\`ere, j'ai eu la chance d'avoir un stage introductif \`a la th\'eorie des sch\'emas en groupes. Je remercie sinc\`erement Matthieu Romagny pour avoir supervis\'e ce stage et pour les \'echanges de courriers \'electroniques utiles que nous avons eus pendant la th\`ese.
\
\newline
\
\newline
Je dois beaucoup \`a Siegmund Kosarew qui a dirig\'e mon stage de premi\`ere ann\'ee de mast\`ere.  \`A cette \'epoque nos discussions hebdomadaires m'ont donn\'e certainement un premier aper\c cu de la recherche. Merci pour tout ce que vous m'avez apport\'e. 
\
\newline
\
\newline
Je tiens \`a saluer Emmanuel Auclair et Vincent Despiegiel. Vous m'avez transmis la passion des math\'ematiques d\`es les premi\`eres ann\'ees universitaires et la conviction de faire des math\'ematiques mon m\'etier. 
\
\newline
\
\newline
Je remercie le personnel administratif et d'entretien de l'institut Fourier, de l'U.F.R. de math\'ematiques de l'universit\'e Joseph Fourier, et de l'\'ecole doctorale EDMSTII pour leur travail remarquable.
\
\newline
\
\newline
{\em Je tiens \`a remercier les personnes qui valorisent la recherche math\'ematique au quotidien notamment dans mes centres d'int\'er\^ets.}
\
\newline
\
\newline
Je tiens \`a remercier les membres de l'A.N.R. Birpol (g\'eom\'etrie birationnelle et automorphismes polynomiaux), \`a savoir, J\'er\'emy Blanc, Serge Cantat, Julie D\'eserti, Adrien Dubouloz, Eric Edo, Jean-Philippe Furter, Julien Grivaux, St\'ephane Lamy, Fr\'ed\'eric Mangolte et Lucy Moser-Jauslin.
\
\newline
\
\newline
Encore merci \`a Bashar Alhajjar, Adrien Dubouloz, Lucy Moser-Jauslin, Shameek Paul, et Charlie Petitjean pour leur accueil chaleureux \`a Dijon. Je tiens aussi \`a saluer les amis b\^ alois, \`a savoir, Emilie Dufresne, Andriy Regeta, Maria Fernanda Robayo, Pierre-Marie Poloni, Immanuel Stampfli et Suzanna Zimmermann. Merci Pierre-Marie pour m'avoir invit\'e \`a faire un expos\'e \`a B\^ale.
\
\newline
\
\newline
Je tiens \`a saluer les personnes que j'ai rencontr\'ees lors des \'ev\'enements concernant la th\'eorie de Lie, \`a savoir, Pramod Ashar, C\'edric Bonnaf\'e, Olivier Dudas, St\'ephane Gaussent, Daniel Juteau, Thierry Levasseur, Pierre-Louis Montagard, Anne Moreau, Boris Pasquier, Nicolas Ressayre, Simon Riche, Alexis Tchoudjem. Encore merci \`a Pierre-Louis Montagard et Etienne Mann pour leur invitation \`a faire un expos\'e \`a Monpellier.
\
\newline
\
\newline
Je tiens \`a remercier les jeunes g\'eom\`etres non commutatifs clermontois. Merci \`a Colin Mrozinski et \`a Manon Thibault de Chanvalon pour votre soutien et pour vous int\'eresser \`a mes maths. La rencontre INTER'ACTIONS a \'et\'e un franc succ\`es! Merci de m'avoir invit\'e pour faire un expos\'e \`a cette magnifique conf\'erence.    
\
\newline
\
\newline
Je remercie sinc\`erement Bernard Teissier pour m'avoir invit\'e \`a faire un expos\'e au laboratoire Sophie Germain. Je tiens \`a saluer \'egalement 
Ana Bel\'en de Felipe, Monique Lejeune-Jalabert et Hussein Mourtada pour leur accueil chaleureux.

\
\newline
\
\newline
{\em \`A tous mes amis que j'ai rencontr\'e \`a l'universit\'e ou autre part.}
\
\newline
\
\newline
\`A Alban, Alexei, Binbin, Brahim, Charlotte, Chlo\'e, Claire, Cl\'elia, Eve-Marie, Guena, Gunnar, Jane, Jean, Jesus, Junyan, Karine, Laurent, Lorick, Manon, Marie, Marjo, Marta, Mickael, Monique, No\'emie, Fred, Olivier, Pierre-Alain, R\'emy, Roland, Ronan, Sandrine, Sasha, Teddy, Thibault, Vincent, etc.
\
\newline
\
\newline
{\em Aux fous rires de nos vies ridicules ...}
\
\newline
\
\newline
Je tiens \`a remercier tous mes amis d'enfances qui font de moi un homme heureux! Je m'adresse notamment \`a tous les membres de la DT38!
\
\newline
\
\newline
{\em Au petit R\'emy (photographe), \`a Fran\c cois (producteur), \`a Jacky (manager/photographe), \`a Mike (supporteur), \`a Mud (batterie), \`a Nico (chant) et \`a Timon (guitare).}
\
\newline
\
\newline
Je remercie mes parents pour leur soutien constant et pour avoir accept\'e ma passion. \`A mon cousin
Mathieu, que je consid\`ere comme mon grand fr\`ere, qui a eu 30 ans et qui vient d'avoir un enfant.   

{\footnotesize \tableofcontents}

\theoremstyle{plain}
\newtheorem{theorem}{Th\'eor\`eme}[section]
\newtheorem{lemme}[theorem]{Lemme}
\newtheorem{proposition}[theorem]{Proposition}
\newtheorem{corollaire}[theorem]{Corollaire}
\newtheorem*{theorem*}{Théorème}

\theoremstyle{definition}
\newtheorem{definition}[theorem]{D\'efinition}
\newtheorem{rappel}[theorem]{}
\newtheorem{conjecture}[theorem]{Conjecture}
\newtheorem{exemple}[theorem]{Exemple}
\newtheorem{notation}[theorem]{Notation}

\theoremstyle{remark}
\newtheorem{remarque}[theorem]{Remarque}
\newtheorem{note}[theorem]{Note}

\newtheorem{theo1}{\bf Theorem 4.2.1}

\chapter{Introduction g\'en\'erale}
Soit $X$ une vari\'et\'e alg\'ebrique affine normale munie d'une op\'eration fid\`ele d'un tore alg\'ebrique $\mathbb{T}$. 
Supposons que le corps de base $\mathbf{k}$ est alg\'ebriquement clos de caract\'eristique $0$. 
Alors $X$ peut \^etre d\'ecrite par des objets combinatoires de g\'eom\'etrie convexe. 
Plusieurs descriptions ont \'et\'e obtenues notamment par les travaux de Mumford, Dolgachev, 
Pinkham, Demazure, Timashev, Flenner, Zaidenberg, Altmann, Hausen, e.a. 
Dans cette th\`ese, nous \'etudions des probl\`emes nouveaux concernant les propri\'et\'es 
alg\'ebriques et g\'eom\'etriques de la vari\'et\'e $X$. 

Rappelons que la pr\'esentation d'Altmann-Hausen en termes de diviseurs poly\'edraux donne une description
explicite de l'alg\`ebre $M$-gradu\'ee $\mathbf{k}[X]$ des fonctions r\'eguli\`eres de la vari\'et\'e $X$, o\`u $M$ est le 
r\'eseau des caract\`eres de $\mathbb{T}$. Plus pr\'ecis\'ement, si $N$ est le r\'eseau dual \`a  $M$, et
$M_{\mathbb{Q}}$, $N_{\mathbb{Q}}$ sont les $\mathbb{Q}$-espaces vectoriels associ\'es \`a $M, N$ alors
cette pr\'esentation est donn\'ee par un triplet $(Y,\sigma, \mathfrak{D})$. La premi\`ere donn\'ee 
$Y$ est une vari\'et\'e semi-projective normale sur $\mathbf{k}$; c'est \`a dire une vari\'et\'e normale
qui est projective sur une vari\'et\'e affine. En fait, c'est un quotient 
d'un ouvert de $X$ par $\mathbb{T}$. En particulier, les orbites g\'en\'erales de l'op\'eration de $\mathbb{T}$
dans $X$ sont de codimension $\rm dim\,\it Y$. La deuxi\`eme donn\'ee $\sigma$
est un c\^one poly\'edral saillant de $N_{\mathbb{Q}}$. Le c\^one dual $\sigma^{\vee}\subset M_{\mathbb{Q}}$
est le c\^one des poids de $\mathbf{k}[X]$. La troisi\`eme donn\'ee est un diviseur de Weil $\mathfrak{D}$ 
sur la vari\'et\'e $Y$
dont les coefficients sont des $\sigma$-poly\`edres de $N_{\mathbb{Q}}$ tous \'egaux \`a $\sigma$ sauf pour un nombre fini.
Rappelons qu'un $\sigma$-poly\`edre dans $N_{\mathbb{Q}}$ est la somme de Minkowski d'un polytope de $N_{\mathbb{Q}}$ 
avec le c\^one $\sigma$. De plus, l'\'evaluation $\mathfrak{D}(m)$ de $\mathfrak{D}$ en le vecteur $m\in\sigma^{\vee}$ est un diviseur
de Weil rationnel $\mathbb{Q}$-Cartier, 
semi-ample et abondant lorsque $m$ est dans l'int\'erieur relatif de $\sigma^{\vee}$ (voir [AH]).
On a un isomorphisme d'alg\`ebres $M$-gradu\'ees, $\mathbf{k}[X]\simeq A[Y,\mathfrak{D}]$ o\`u
\begin{eqnarray*}
A[Y,\mathfrak{D}]:= \bigoplus_{m\in\sigma^{\vee}\cap M}H^{0}(Y,\mathcal{O}_{Y}(\lfloor \mathfrak{D}(m)\rfloor))\chi^{m}. 
\end{eqnarray*}
Les \'el\'ements $\chi^{m}$ v\'erifient les relations $\chi^{0} = 1$ et $\chi^{m}\cdot\chi^{m'} = \chi^{m + m'}$, pour tous $m,m'\in\sigma^{\vee}\cap M$. 
R\'eciproquement, en partant d'une pr\'esentation d'Altmann-Hausen $(Y,\sigma, \mathfrak{D})$, 
l'alg\`ebre $A[Y,\mathfrak{D}]$
d\'efinit une vari\'et\'e normale sur $\mathbf{k}$ munie d'une op\'eration fid\`ele de $\mathbb{T}$
avec quotient rationnel $Y$. Lorsque $Y$ est de dimension $0$ (cas de complexit\'e $0$), on a
\begin{eqnarray*}
H^{0}(Y, \mathcal{O}_{Y}(\lfloor \mathfrak{D}(m)\rfloor)) = \mathbf{k} 
\end{eqnarray*}
pour tout $m\in\sigma^{\vee}\cap M$, de sorte que
\begin{eqnarray*}
\mathbf{k}[X]\simeq A[Y,\mathfrak{D}] = \bigoplus_{m\in\sigma^{\vee}\cap M}\mathbf{k}\chi^{m}. 
\end{eqnarray*}
On retrouve ainsi la description bien connue des vari\'et\'es toriques affines en termes
de c\^ones poly\'edraux saillants. Une des motivations majeures des travaux pr\'esent\'es 
dans cette th\`ese est \em la g\'en\'eralisation de r\'esultats connus dans le cas 
torique (cas $\rm dim\,\it Y\rm = 0$) au contexte des diviseurs poly\'edraux
pr\'esent\'e ci-dessus (cas g\'en\'eral). \rm
Plus pr\'ecis\'ement, nous nous int\'eressons au cas o\`u $Y$ est une courbe alg\'ebrique r\'eguli\`ere $C$ (cas de complexit\'e $1$).
Dans ce contexte, on reconstruit ais\'ement $C$ \`a partir de $A[C,\mathfrak{D}]$.

Cette th\`ese est divis\'ee en trois parties.
Dans la premi\`ere partie (chapitre $3$), un r\'esultat donne une mani\`ere explicite de calculer
la normalisation d'une vari\'et\'e affine munie d'une op\'eration d'un tore alg\'ebrique de complexit\'e $1$
en termes de diviseurs poly\'edraux d'Altmann-Hausen, lorsque $\mathbf{k}$ est alg\'ebriquement clos de 
caract\'eristique $0$. Ce r\'esultat peut \^etre mis en parall\`ele avec les travaux de Hochster dans le cas torique
(voir [Ho]). Comme application, nous nous int\'eressons \`a la th\'eorie des id\'eaux int\'egralement clos
introduite par Zariski, Lejeune-Jalabert, Teissier, etc. (voir [Za, LeTe]). On donne une description combinatoire des id\'eaux homog\`enes
int\'egralement clos de l'alg\`ebre $A = A[C,\mathfrak{D}]$, g\'en\'eralisant l'approche de [KKMS]. 
Nous obtenons sous certaines conditions un crit\`ere de normalit\'e
pour les id\'eaux homog\`enes de $A$.  

Les calculs de la premi\`ere partie sugg\`erent une d\'emonstration de la validit\'e de 
la pr\'esentation d'Altmann-Hausen sur un corps quelconque dans le cas de la complexit\'e $1$. 
Ce qui est fait dans la deuxi\`eme partie (chapitre $4$). Sur un corps de base arbitraire, la descente galoisienne des vari\'et\'es affines 
normales munies d'une op\'eration d'un tore alg\'ebrique de complexit\'e $1$ 
est d\'ecrite par des objets combinatoires que nous appelons 
diviseurs poly\'edraux Galois stables. Nous \'etudions aussi d'autres types d'alg\`ebres
multigradu\'ees, celles qui sont d\'efinies par un diviseur poly\'edral sur 
un anneau de Dedekind.

Dans la troisi\`eme partie (chapitre $5$), nous \'etudions la description des racines de Demazure
en complexit\'e $1$ sur un corps quelconque dans la lign\'ee amorc\'ee par Liendo [Li].
Cette \'etude est aussi une g\'en\'eralisation en dimension plus grande de r\'esultats
d\^u \`a Flenner et \`a Zaidenberg concernant les op\'erations normalis\'ees du groupe
additif dans les surfaces affines complexes munies d'une action de $\mathbb{C}^{\star}$ (voir [FZ $2$]). 
\
\newline

\em Avertissement. \rm
Les trois prochains chapitres sont issus d'une publication (voir [La]) et de deux pr\'epublications (voir [La $2$, LL]). Les
deux prochains chapitres peuvent \^etre lus ind\'ependamment. 
Le chapitre $5$ utilise les conventions du chapitre $4$. Chaque chapitre est muni de sa propre 
introduction d\'etaill\'ee.   
 
\chapter{General introduction}
Let $X$ be a normal affine algebraic variety endowed with an effective action of an algebraic torus $\mathbb{T}$.
Assume that the ground field $\mathbf{k}$ is algebraically closed of characteristic $0$.
Then we can describe the variety $X$ by combinatorial objects arising in convex geometry.
There exist several descriptions due mainly to the works of Mumford, Dolgachev, 
Pinkham, Demazure, Timashev, Flenner, Zaidenberg, Altmann, Hausen, among others.
In this thesis, we study new problems concerning the algebraic and geometric properties of the
variety $X$. 
 
Let us recall that the Altmann-Hausen presentation in terms of polyhedral divisors provides a concrete
description of the $M$-graded algebra $\mathbf{k}[X]$ of regular functions on $X$, where $M$ is
the character lattice of $\mathbb{T}$. More precisely, if $N$ is the dual lattice of $M$, and 
$M_{\mathbb{Q}}$, $N_{\mathbb{Q}}$ are the $\mathbb{Q}$-vector spaces associated to $M, N$ 
then this presentation is given by a triple $(Y,\sigma, \mathfrak{D})$. The first data 
$Y$ is a normal semi-projective variety over the field $\mathbf{k}$, i.e., $Y$ is assumed to be 
projective over an affine variety. Actually, $Y$ is a quotient of an open subset
of $X$ by the torus $\mathbb{T}$. In particular, the general $\mathbb{T}$-orbits in $X$ have codimension $\rm dim\,\it Y$. 
The second data $\sigma$ is a strongly convex polyhedral cone of $N_{\mathbb{Q}}$. 
The dual cone $\sigma^{\vee}\subset M_{\mathbb{Q}}$ is the weight cone of the $M$-graded algebra
$\mathbf{k}[X]$. The third data is a Weil divisor $\mathfrak{D}$ over $Y$ whose coefficients
are $\sigma$-polyhedra of $N_{\mathbb{Q}}$ equal to $\sigma$ for all but finitely many prime divisors
of $Y$. Recall that a $\sigma$-polyhedron in $N_{\mathbb{Q}}$ is the Minkowski sum of a polytope
of $N_{\mathbb{Q}}$ with the cone $\sigma$. Furthermore, the evaluation $\mathfrak{D}(m)$ of $\mathfrak{D}$ 
in a vector $m\in\sigma^{\vee}$ 
is a rational $\mathbb{Q}$-Cartier divisor, 
semi-ample, and big for $m$ in the relative interior of the cone $\sigma^{\vee}$ (see [AH]).
We have an isomorphism of $M$-graded algebras, $\mathbf{k}[X]\simeq A[Y,\mathfrak{D}]$ where
\begin{eqnarray*}
A[Y,\mathfrak{D}]:= \bigoplus_{m\in\sigma^{\vee}\cap M}H^{0}(Y,\mathcal{O}_{Y}(\lfloor \mathfrak{D}(m)\rfloor))\chi^{m}. 
\end{eqnarray*}
The symbols $\chi^{m}$ satisfy the relations $\chi^{0} = 1$ and $\chi^{m}\cdot\chi^{m'} = \chi^{m + m'}$, for all 
$m,m'\in\sigma^{\vee}\cap M$. Conversely, given a presentation $(Y,\sigma, \mathfrak{D})$ the corresponding
algebra $A[Y,\mathfrak{D}]$ defines naturally a normal affine variety over $\mathbf{k}$
endowed with an effective $\mathbb{T}$-action with rational quotient $Y$. 
When the dimension of $Y$ is $0$ (complexity $0$ case), we have
\begin{eqnarray*}
H^{0}(Y, \mathcal{O}_{Y}(\lfloor \mathfrak{D}(m)\rfloor)) = \mathbf{k} 
\end{eqnarray*}
for any $m\in\sigma^{\vee}\cap M$ so that
\begin{eqnarray*}
\mathbf{k}[X]\simeq A[Y,\mathfrak{D}] = \bigoplus_{m\in\sigma^{\vee}\cap M}\mathbf{k}\chi^{m}, 
\end{eqnarray*}
and we recover the classical description of affine toric varieties in terms of strongly convex
polyhedral cones. One of the main motivations of the work presented in this thesis is
\em a generalization of well known results in the toric case (case of $\rm dim\,\it Y\rm = 0$) 
to the context of polyhedral divisors (general case). \rm
More precisely, we are interested in the case where $Y$ is a regular algebraic curve
(complexity $1$ case). In this context, we reconstruct easily the curve $C$ starting
from the algebra $A[C,\mathfrak{D}]$.

This thesis is divided into three parts. In the first part (Chapter $3$) we 
give an explicit description of the normalization of an affine variety
endowed with an algebraic torus action of complexity $1$ in terms of polyhedral
divisors, when the ground field $\mathbf{k}$ is algebraically closed of characteristic $0$. 
This result can be compared with the work of Hochster in the toric case
(see [Ho]). As an application we are interested in the theory of integrally closed ideals
introduced by Zariski, Lejeune-Jalabert, Teissier, etc. (see [Za, LeTe]). We provide a combinatorial
description of homogeneous integrally closed ideals of the algebra $A = A[C,\mathfrak{D}]$, 
extending the approach of [KKMS]. We obtain under some conditions a normality criterion for
homogeneous ideals of $A$.  

Following the computations of the first part, we can show next that the Altmann-Hausen 
presentation in complexity $1$ holds over an arbitrary field. This is done in the second
part (Chapter $4$). Still over an arbitrary field, the Galois descent of normal 
affine varieties endowed with an algebraic torus action of complexity $1$ is
described via a new combinatorial object that we call Galois invariant polyhedral
divisor. We study as well other types of multigraded algebras, namely, those defined
by a polyhedral divisor over a Dedekind domain. 

In the third part (Chapter $5$), we describe the Demazure roots in complexity $1$ for
more general fields following the work of Liendo [Li].
We obtain a generalization in higher dimensions of classical results concerning the quasihomogeneous
surfaces [FZ $2$] due in particular to Flenner and Zaidenberg.
\
\newline

\em Warning. \rm The next three chapters arise from a publication (see [La]) and two preprints (see [La $2$, LL]). The
next two chapters can be read independently.
Chapter $5$ uses the conventions of Chapter $4$. Each chapter is provided with its own detailed introduction.

\newpage
\strut 
\newpage
 
\chapter{Cl\^oture int\'egrale et op\'erations de tores alg\'ebriques de complexit\'e un
dans les vari\'et\'es affines}
\section{Introduction}

Dans ce chapitre, nous nous int\'eressons aux alg\`ebres multigradu\'ees de complexit\'e $1$
sur un corps alg\'ebriquement clos $\mathbf{k}$ de caract\'eristique z\'ero.

En utilisant la g\'eom\'etrie convexe d\'evelopp\'ee par Altmann et Hausen
nous obtenons des r\'esultats nouveaux sur des questions classiques d'alg\`ebre
commutative. Un de nos principaux th\'eor\`emes donne une description 
combinatoire des id\'eaux int\'egralement clos en termes de diviseurs poly\'edraux,
voir le th\'eor\`eme $3.6.6$. Un autre r\'esultat nous permet de calculer de fa\c con
explicite la normalisation d'une vari\'et\'e affine munie d'une op\'eration d'un
tore alg\'ebrique de complexit\'e un. Nous d\'ecrivons aussi la fermeture int\'egrale
des id\'eaux homog\`enes, voir les th\'eor\`emes $3.4.4$ et $3.6.2$,
et donnons de nouveaux exemples d'id\'eaux normaux homog\`enes, voir le th\'eor\`eme $3.7.3$. 

Donnons deux exemples illustrant notre probl\'ematique. Consid\'erons l'alg\`ebre des polyn\^omes
de Laurent en $n$ variables 
\begin{eqnarray*}
 L_{[n]} = L_{[n]}(\mathbf{k}):=\mathbf{k}\left[t_{1},t_{1}^{-1},t_{2},t_{2}^{-1},\ldots,t_{n},t_{n}^{-1}\right].
\end{eqnarray*}
Notons que $L_{[n]}$ est l'anneau des coordonn\'ees de la vari\'et\'e affine $(\mathbf{k}^{\star})^{n}$.
Soit $A$ une sous-alg\`ebre engendr\'ee par un nombre fini de mon\^omes. Soit $E\subset \mathbb{Z}^{n}$ 
le sous-ensemble des exposants correspondant aux mon\^omes apparaissant dans $A$. 
Sans perte de g\'en\'eralit\'e, nous pouvons supposer que $E$ engendre le r\'eseau $\mathbb{Z}^{n}$.
Il est connu [Ho] que la normalisation de $A$ est l'ensemble des combinaisons lin\'eaires de tous
les mon\^omes dont les exposants appartiennent au c\^one $\omega\subset \mathbb{Q}^{n}$
engendr\'e par $E$. Nous avons donc
\begin{eqnarray*}
\bar{A} = \bigoplus_{(m_{1},\,\ldots\,,\,m_{n})\in\,\omega\cap\mathbb{Z}^{n}}\mathbf{k}\, 
t_{1}^{m_{1}}\ldots\, t_{n}^{m_{n}},
\end{eqnarray*}
o\`u $\bar{A}$ est la normalisation de $A$. Par exemple, si $n = 1$ et si $A$
est la sous-alg\`ebre engendr\'ee par les mon\^omes $t_{1}^{2}$ et $t_{1}^{3}$
alors la normalisation de $A$ est l'anneau des polyn\^omes $\mathbf{k}[t_{1}]$.

Un probl\`eme analogue appara\^it pour les id\'eaux monomiaux.
Supposons que l'alg\`ebre $A$ soit normale. Soit $I$ un id\'eal de 
$A$ engendr\'e par des mon\^omes. L'enveloppe convexe dans $\mathbb{Q}^{n}$
de tous les exposants des mon\^omes de $I$ est un poly\`edre $P$ contenu dans
le c\^one $\omega$. Ce poly\`edre $P$ satisfait $P+ \omega \subset P$.
La fermeture int\'egrale de $I$ est \'egale \`a 
\begin{eqnarray*}
\bar{I} = \bigoplus_{(m_{1},\,\ldots\,,\,m_{n})\in\,P\cap\mathbb{Z}^{n}}\mathbf{k}\, t_{1}^{m_{1}}\ldots\, t_{n}^{m_{n}}.  
\end{eqnarray*}
Nous pouvons d\'eterminer $P$ par un ensemble fini de mon\^omes dans $A$ engendrant l'id\'eal $I$. 
Par exemple, si $n = 2$, $A = \mathbb{C}[t_{1},t_{2}]$ et si $I$ est l'id\'eal engendr\'e par les mon\^omes 
$t_{1}^{3}$ et $t_{2}^{3}$ alors 
\begin{eqnarray*}
\bar{I} = \left( t^{3}_{1},\,t_{1}^{2}t_{2},\,t_{1}t_{2}^{2},\,t_{2}^{3}\right ). 
\end{eqnarray*}
Voir [LeTe, HS, Va] pour plus de d\'etails concernant la fermeture int\'egrale des id\'eaux; nous 
rappelons les d\'efinitions essentielles ci-dessous. Notons que l'\'etude des id\'eaux int\'egralement
clos nous permet de d\'ecrire la normalisation d'un \'eclatement (voir [KKMS] pour le cas torique
et [Br] pour le cas sph\'erique). Cela est aussi utilis\'e dans le but de d\'ecrire les
modifications affines \'equivariantes par les op\'erations d'un tore alg\'ebrique (voir [KZ, Du]).

Par analogie avec le cas monomial pr\'esent\'e ci-dessus, nous consid\'erons le cas plus
g\'en\'eral des op\'erations de tores alg\'ebriques de complexit\'e $1$. Avant de
passer \`a la formulation de nos r\'esultats,
rappelons quelques notions.

Un tore alg\'ebrique $\mathbb{T}$ de dimension $n$ est un groupe alg\'ebrique isomorphe \`a  
$(\mathbf{k}^{\star})^{n}$. Soient $M$ le r\'eseau des caract\`eres de $\mathbb{T}$ 
et $A$ une alg\`ebre affine sur le corps $\mathbf{k}$. Alors se donner une op\'eration
alg\'ebrique de $\mathbb{T}$ dans $X = \rm Spec\,\it A$ est \'equivalent \`a se donner une
 $M$-graduation sur $A$. La complexit\'e de l'alg\`ebre $M$-gradu\'ee affine $A$ est la codimension
d'une orbite dans $X$ en position g\'en\'erale. Notons que le cas classique torique correspond
au cas de la complexit\'e $0$.

Soient $A$ un anneau int\`egre et $I$ un id\'eal de $A$.
Un \'el\'ement $a\in A$ est dit entier sur $I$ s'il existe 
$r\in\mathbb{Z}_{>0}$ et $c_{i}\in I^{i}$, $i = 1,\ldots, r$ tels que
$a^{r} + \sum_{i = 1}^{r}c_{i}a^{r-i} = 0$. La fermeture int\'egrale $\bar{I}$ de 
l'id\'eal $I$ est l'ensemble de tous les \'el\'ements de $A$ qui sont entiers sur $I$.
Il est connu que le sous-ensemble $\bar{I}$ est un id\'eal [HS, $1.3.1$]. Un id\'eal $I$ est int\'egralement
clos si $I = \bar{I}$. De plus, $I$ est dit
normal si pour tout entier strictement positif $e$, l'id\'eal $I^{e}$
est int\'egralement clos. Si $A$ est normal alors cette derni\`ere condition est \'equivalente
\`a la normalit\'e de l'alg\`ebre de Rees $A[It] = A\oplus\bigoplus_{i\geq 1}I^{i}t^{i}$.
Voir [Ri] pour plus de d\'etails.

Le but de ce chapitre  est de r\'epondre aux questions suivantes: \'etant donn\'es une alg\`ebre 
$M$-gradu\'ee affine $A$ de complexit\'e $1$ sur $\mathbf{k}$ et des \'el\'ements
homog\`enes
$a_{1},\ldots,a_{r}\in A$ tels que 
\begin{eqnarray*}
A = \mathbf{k}\left[a_{1},\ldots,\,a_{r}\right],
\end{eqnarray*}
peut-on d\'ecrire explicitement la normalisation de $A$ \`a 
partir des g\'en\'erateurs $a_{1},\ldots,a_{r}$? 
De plus, si $A$ est normale et si $I$ est un id\'eal homog\`ene dans $A$, 
peut on d\'ecrire explicitement la fermeture int\'egrale de $I$ \`a partir
d'un ensemble fini de g\'en\'erateurs homog\`enes? Il existe un lien entre
ces deux derni\`eres questions. En fait, la r\'eponse \`a la seconde question
peut \^etre d\'eduite de la premi\`ere en examinant la normalisation de l'alg\`ebre
de Rees $A[It]$ correspondant \`a l'id\'eal $I$. 

Pour r\'epondre \`a la premi\`ere question, il est utile d'attacher un objet combinatoire
appropri\'e \`a une alg\`ebre $M$-gradu\'ee normale $A$. Par exemple, dans le cas monomial
pr\'esent\'e ci-dessus,
si $A$ est normale alors le c\^one poly\'edral $\omega$ (engendr\'e par les poids de $A$) nous
permet de reconstruire l'alg\`ebre $A$. 

Rappelons qu'une $\mathbb{T}$-vari\'et\'e est une vari\'et\'e normale munie d'une op\'eration
fid\`ele d'un tore alg\'ebrique
$\mathbb{T}$. Il existe plusieurs descriptions combinatoires des $\mathbb{T}$-vari\'et\'es affines. Voir 
[De $2$, AH] pour la complexit\'e arbitraire, [KKMS, Ti] pour la complexit\'e $1$ et [FZ] pour le cas des surfaces.
Notons que la description de [AH] est g\'en\'eralis\'ee dans [AHS] pour des $\mathbb{T}$-vari\'et\'es.
Voir aussi l'article d'exposition [AOPSV] pour des applications concernant la th\'eorie des 
$\mathbb{T}$-vari\'et\'es. 

Dans ce chapitre, nous utilisons le point de vue de [AH] et de [Ti]. Pour simplifier
nous supposons que $M = \mathbb{Z}^{n}$ et que $\mathbb{T} = (\mathbf{k}^{\star})^{n}$. 
Une  
$\mathbb{T}$-vari\'et\'e affine $X = \rm Spec\,\it A$ de complexit\'e $1$ 
peut \^etre d\'ecrite par son c\^one des poids $\omega\subset\mathbb{Q}^{n}$ 
et un diviseur poly\'edral $\mathfrak{D}$ sur une courbe alg\'ebrique lisse $C$, dont les 
coefficients sont des poly\`edres dans $\mathbb{Q}^{n}$. Pour tout \'el\'ement  
$m = (m_{1},\ldots,m_{n})$ de $\omega\cap\mathbb{Z}^{n}$, nous avons une \'evaluation
$\mathfrak{D}(m)$ appartenant \`a l'espace vectoriel sur $\mathbb{Q}$ des diviseurs de Weil 
rationnels sur $C$. \'Etant donn\'ee un triplet $(C,\omega,\mathfrak{D})$
on peut construire une alg\`ebre $M$-gradu\'ee 
\begin{eqnarray*}
A[C,\mathfrak{D}] := \bigoplus_{m\in\omega\cap M}H^{0}(C,\mathcal{O}_{C}(\lfloor\mathfrak{D}(m)\rfloor ))\chi^{m},
\end{eqnarray*}
o\`u $\chi^{m}$ est le mon\^ome de Laurent $t_{1}^{m_{1}}\ldots t_{n}^{m_{n}}$. Voir [AH] pour des
d\'efinitions et des  \'enonc\'es pr\'ecis. Un des principaux r\'esultats de ce chapitre peut
\^etre \'enonc\'e de la fa\c con suivante (voir le th\'eor\`eme $3.4.4$).

\paragraph{Th\'eor\`eme.}{\em Soit $C$ une courbe alg\'ebrique lisse. Consid\'erons une sous-alg\`ebre
\begin{eqnarray*}
B = \mathbf{k}[C][f_{1}\chi^{s_{1}},\ldots,f_{r}\chi^{s_{r}}]\subset L_{[n]}(\mathbf{k}(C))
:=\mathbf{k}(C)\left[t_{1},t_{1}^{-1},\ldots,t_{n},t_{n}^{-1}\right] 
\end{eqnarray*}
telle que
\begin{eqnarray*}
\rm Frac\, \it B = \rm Frac\,\it L_{[n]}(\mathbf{k}(C)),
\end{eqnarray*}
o\`u $s_{i}\in\mathbb{Z}^{n}$, $\chi^{s_{i}}$ est le mon\^ome de Laurent correspondant, et
$f_{i}\in\mathbf{k}(C)^{\star}$. Alors la normalisation de
$B$ est l'alg\`ebre $A[C,\mathfrak{D}]$ dont le c\^one des poids $\omega$ est engendr\'e par $s_{1},\ldots, s_{r}$
et dont le coefficient du diviseur poly\'edral $\mathfrak{D}$ au point $z\in C$  est
\begin{eqnarray*}
\Delta_{z} = \left\{\,v\in\mathbb{Q}^{n},\,\left\langle s_{i}\,,\,v\right\rangle\geq -\rm ord_{\it z}\it\,f_{i},\,
i = 1,\ldots, r\,\right\},
\end{eqnarray*}
o\`u $\rm ord_{\it z}\it\, f_{i}$ est l'ordre d'annulation de $f_{i}$ en $z$.}
\paragraph{}
Ce th\'eor\`eme r\'epond \`a la premi\`ere question. Il g\'en\'eralise des 
r\'esultats bien connus dans le cas des surfaces munies d'une op\'eration alg\'ebrique de $\mathbf{k}^{\star}$,
voir [FZ, 3.9, 4.6].
Notons aussi que ce r\'esultat peut \^etre appliqu\'e 
\begin{enumerate}
 \item[(*)] pour trouver un diviseur poly\'edral propre repr\'esentant une sous-alg\`ebre (normale)
donn\'ee par un ensemble fini de g\'en\'erateurs homog\`enes;
\item[(**)] pour trouver un ensemble fini de g\'en\'erateurs d'une alg\`ebre repr\'esent\'ee par un diviseur poly\'edral
propre: l'id\'ee est de deviner un ensemble fini de g\'en\'erateurs en appliquant (*).  
\end{enumerate} 
La r\'eponse \`a la seconde question est donn\'ee par le th\'eor\`eme $3.6.2$.
Il est connu que l'ensemble des id\'eaux int\'egralement clos homog\`enes d'une
vari\'et\'e torique d\'ecrite par son c\^one des poids $\omega$ est en correspondance bijective
avec l'ensemble des $\omega$-poly\`edres entiers contenus dans $\omega$ (voir [KKMS] et la section $3.5$).
Nous donnons une correspondance analogue pour les id\'eaux int\'egralement clos homog\`enes
sur une $\mathbb{T}$-vari\'et\'e affine de complexit\'e $1$ (voir le th\'eor\`eme $3.6.6$)
qui est totalement combinatoire lorsque $C$ est affine (voir le corollaire $3.6.7$).

Tout id\'eal int\'egralement clos homog\`ene $I$ d'une  
$\mathbb{T}$-vari\'et\'e affine $X = \rm Spec\,\it A$ de complexit\'e $1$ 
peut \^etre d\'ecrit par le biais d'un couple 
$(P,\widetilde{\mathfrak{D}})$ o\`u $P$ est un poly\`edre entier dans $\mathbb{Q}^{n}$. 
Ce poly\`edre joue le m\^eme r\^ole que le poly\`edre de Newton pour le cas monomial. 
Le diviseur poly\'edral $\widetilde{\mathfrak{D}}$
correspond \`a la normalisation de l'alg\`ebre de Rees de $I$. Nous donnons une interpr\'etation
g\'eom\'etrique des coefficients de  $\widetilde{\mathfrak{D}}$. Supposons par exemple que
le c\^one des poids de $A$ est saillant, et soit $\widetilde{\Delta}_{z}$ le coefficient
de $\widetilde{\mathfrak{D}}$ au point $z\in C$. Alors le th\'eor\`eme $3.6.6$
donne des conditions sur les \'equations des faces de codimension $1$ de $\widetilde{\Delta}_{z}$ 
de sorte que $\widetilde{\mathfrak{D}}$ correspond \`a la normalisation de l'alg\`ebre de
Rees $A[It]$.

Nous donnons aussi des conditions suffisantes sur le couple $(P,\widetilde{\mathfrak{D}})$ 
pour que l'id\'eal correspondant $I$ soit normal (voir le th\'eor\`eme $3.7.3$). Pour le
cas des $\mathbf{k}^{\star}$-surfaces affines qui ne sont pas elliptiques,
nous obtenons une d\'emonstration combinatoire de la normalit\'e des id\'eaux int\'egralement
clos homog\`enes de ces surfaces. Comme autre application, nous obtenons un nouveau crit\`ere de
normalit\'e qui g\'en\'eralise le th\'eor\`eme de Reid-Roberts-Vitulli [RRV, $3.1$]
dans le cas de la complexit\'e $0$ (voir $3.7.5$).

\paragraph{Th\'eor\`eme.}{\em 
Soient $n\geq 1$ un entier et $\mathbf{k}^{[n+1]} = \mathbf{k}[x_{0},\ldots, x_{n}]$
l'alg\`ebre des polyn\^omes \`a $n+1$ variables sur $\mathbf{k}$.
Nous munissons $\mathbf{k}^{[n+1]}$ de la $\mathbb{Z}^{n}$-graduation donn\'ee par
l'\'egalit\'e
\begin{eqnarray*}
\mathbf{k}^{[n+1]} = \bigoplus_{m = (m_{1},\ldots, m_{n})\in\mathbb{N}^{n}}A_{m},
\,\,\,\it avec \,\,\,\it 
A_{m} = \mathbf{k}[x_{\rm \,0 \it }]\,x_{\rm 1\it }^{m_{\rm 1\it}}
\ldots x_{n}^{m_{n}}.
\end{eqnarray*}
Pour un id\'eal homog\`ene $I$ de $\mathbf{k}^{[n+1]}$, les assertions suivantes sont
\'equivalentes.
\begin{enumerate}
\item[\rm (i)] L'ideal $I$ est normal;
\item[\rm (ii)] Pour tout $e\in\{1,\ldots, n\}$, l'id\'eal $I^{e}$ est int\'egralement clos.
\end{enumerate}
} 

\paragraph{}
Donnons un court r\'esum\'e de chaque section de ce chapitre. 
Dans la section $3.3$, nous rappelons quelques faits sur les
op\'erations de tores alg\'ebriques de complexit\'e $1$ et sur
les diviseurs poly\'edraux. Dans la section $3.4$, nous traitons
le probl\`eme de la normalisation des alg\`ebres multigradu\'ees et
montrons le th\'eor\`eme $3.6.6$. Dans la section $3.5$, nous nous 
concentrons sur les id\'eaux monomiaux int\'egralement clos.
Dans la section $3.6$, nous \'etudions la description en termes de couples 
$(P,\widetilde{\mathfrak{D}})$ pour les id\'eaux int\'egralement clos
des $\mathbb{T}$-vari\'et\'es affines de complexit\'e $1$. Finalement,
dans la derni\`ere section, nous discutons de la normalit\'e des id\'eaux dans
un cas particulier. 
\paragraph{}
Tout au long de ce chapitre $\mathbf{k}$ est un corps alg\'ebriquement clos de caract\'eristique z\'ero.
Par une vari\'et\'e on entend un sch\'ema int\`egre s\'epar\'e de type fini sur $\mathbf{k}$.
\paragraph{}
\section{Introduction (english version)}
In this chapter, we are interested in multigraded affine algebras of complexity $1$
over an algebraically closed field $\mathbf{k}$ of characteristic zero.

Using the convex geometry developped by Altmann-Hausen we obtain some new 
results on classical questions of commutative algebra. 
One of our main theorems gives  
a description of integrally closed homogeneous ideals in terms 
of polyhedral divisors, see Theorem $3.6.6$. 
Another result allows us to compute effectively  the normalization of an affine variety 
with an algebraic torus
action of complexity one. We describe as well the integral closure of homogeneous ideals,
see Theorem $3.4.4$, Theorem $3.6.2$ and exhibit new examples 
of normal homogeneous ideals, see Theorem $3.7.3$.

The following two classical examples illustrate well the issues that arise. Consider
the algebra of Laurent polynomials in $n$ variables
\begin{eqnarray*}
 L_{[n]} = L_{[n]}(\mathbf{k}):=\mathbf{k}\left[t_{1},t_{1}^{-1},t_{2},t_{2}^{-1},\ldots,t_{n},t_{n}^{-1}\right].
\end{eqnarray*}
Notice that $L_{[n]}$ is the coordinate ring of the affine variety $(\mathbf{k}^{\star})^{n}$.
Let $A$ be a subalgebra generated by a finite number of monomials, 
and let $E\subset \mathbb{Z}^{n}$ be the subset of exponents corresponding to the monomials
of $A$. Without loss of generality, we may suppose that $E$ generates the lattice $\mathbb{Z}^{n}$.
It is known [Ho] that the normalization of $A$ is the set of linear combinations of
all monomials with their exponents belonging to the rational cone $\omega\subset \mathbb{Q}^{n}$
generated by $E$. We have
\begin{eqnarray*}
\bar{A} = \bigoplus_{(m_{1},\,\ldots\,,\,m_{n})\in\,\omega\cap\mathbb{Z}^{n}}\mathbf{k}\, 
t_{1}^{m_{1}}\ldots\, t_{n}^{m_{n}}
\end{eqnarray*}
where $\bar{A}$ is the normalization of $A$. For instance, if $n = 1$ and if $A$
is the subalgebra generated by the monomials $t_{1}^{2}$ and $t_{1}^{3}$
then the normalization of $A$ is the polynomial ring $\mathbf{k}[t_{1}]$.

A similar problem arises for monomial ideals.
Assume that the algebra $A$ is normal. Let $I$ be an ideal of
$A$ generated by monomials. The convex hull in $\mathbb{Q}^{n}$
of all exponents appearing in $I$ is a polyhedron $P$ contained in 
$\omega$. This polyhedron $P$ satisfies $P+ \omega \subset P$.
The integral closure of $I$ is equal to
\begin{eqnarray*}
\bar{I} = \bigoplus_{(m_{1},\,\ldots\,,\,m_{n})\in\,P\cap\mathbb{Z}^{n}}\mathbf{k}\, t_{1}^{m_{1}}\ldots\, t_{n}^{m_{n}}.  
\end{eqnarray*}
We can determine $P$ by means of a finite system of monomials generating the ideal $I$. 
For instance, if $n = 2$, $A = \mathbb{C}[t_{1},t_{2}]$ and if $I$ is the ideal generated by the monomial 
$t_{1}^{3}$ and $t_{2}^{3}$ then 
\begin{eqnarray*}
\bar{I} = \left( t^{3}_{1},\,t_{1}^{2}t_{2},\,t_{1}t_{2}^{2},\,t_{2}^{3}\right ). 
\end{eqnarray*}
See [LeTe, HS, Va] for more details concerning
the integral closure of ideals; we recall below the definition. Note that the study of integrally closed
ideals allows us to find the normalization of a  blowing-up (see [KKMS]
for toric varieties and [Br] for spherical varieties). This is used as well in order
to describe ($\mathbb{T}$-equivariant) affine modifications
(see [KZ, Du]).

By analogy with the monomial case presented above, we address more generally the actions of
algebraic tori of complexity $1$. Before
formulating our results, let us recall some notation.

An algebraic torus $\mathbb{T}$ of dimension $n$ is an algebraic group isomorphic 
to $(\mathbf{k}^{\star})^{n}$. Let $M$ be the character lattice of $\mathbb{T}$, and let
$A$ be an affine algebra over $\mathbf{k}$.  
Defining an algebraic action of $\mathbb{T}$ on $X = \rm Spec\,\it A$
 is the same as defining an $M$-grading of $A$. The complexity of the affine $M$-graded algebra
$A$ is the codimension of a general $\mathbb{T}$-orbit in $X$. Note that the classical toric case
corresponds to the complexity $0$ case.

Let $A$ be a domain and let $I$ be an ideal of $A$.
An element $a\in A$ is said to be integral over $I$ if there exist
$r\in\mathbb{Z}_{>0}$ and $c_{i}\in I^{i}$, $i = 1,\ldots, r$ such
that $a^{r} + \sum_{i = 1}^{r}c_{i}a^{r-i} = 0$. The integral closure $\bar{I}$ of 
the ideal $I$ is the set of all elements of $A$ that are integral over $I$.
It is known that $\bar{I}$ is an ideal [HS, $1.3.1$]. An ideal $I$ is integrally closed if
$I = \bar{I}$. Furthermore, $I$ is said to be
normal if for any positive integer $e$, the ideal $I^{e}$
is integrally closed. If $A$ is normal then this latter condition is equivalent 
to the normality of the Rees algebra $A[It] = A\oplus\bigoplus_{i\geq 1}I^{i}t^{i}$.
See [Ri] for more details.

The purpose of this paper is to answer the following questions : given an affine 
$M$-graded algebra $A$ of complexity $1$ over $\mathbf{k}$ and  homogeneous elements
$a_{1},\ldots,a_{r}\in A$ such that 
\begin{eqnarray*}
A = \mathbf{k}\left[a_{1},\ldots,\,a_{r}\right],
\end{eqnarray*} 
can one describe explicitly the normalization of $A$ in terms of the generators $a_{1},\ldots,a_{r}$? 
Furthermore if $A$ is normal and if $I$ is a homogenous ideal in $A$, can one describe effectively
 the integral closure of $I$ in terms of a given finite system of homogeneous generators of $I$? There
exists a connection between these two questions. Indeed the answer to the second can be
deduced from that to the first by examining the normalization of the Rees algebra $A[It]$
corresponding to $I$. 

To answer the first question, it is useful to attach an appropriate combinatorial object to a given normal 
$M$-graded algebra $A$. For instance, in the monomial case if $A$ is normal then
 the rational cone $\omega$ allows us to reconstruct $A$. 

Recall that a $\mathbb{T}$-variety is a normal variety endowed with an effective algebraic
$\mathbb{T}$-action. There exists
several combinatorial descriptions of affine $\mathbb{T}$-varieties. See [De $2$, AH] for 
arbitrary complexity, [KKMS, Ti] for complexity $1$ and [FZ] for the case of surfaces.
Note that the description of [AH] is generalized in [AHS] to non-affine 
$\mathbb{T}$-varieties.
See also the survey article [AOPSV] 
for applications of the theory of $\mathbb{T}$-varieties. 

In this chapter, we use the point of view of [AH] and [Ti]. To simplify things
we assume that $M = \mathbb{Z}^{n}$ and $\mathbb{T} = (\mathbf{k}^{\star})^{n}$. 
An affine  
$\mathbb{T}$-variety $X = \rm Spec\,\it A$ of complexity $1$ can be
described by its weight cone $\omega\subset\mathbb{Q}^{n}$ and by a polyhedral
divisor $\mathfrak{D}$ on a smooth algebraic curve $C$, whose coefficients
are polyhedra in $\mathbb{Q}^{n}$. For any element 
$m = (m_{1},\ldots,m_{n})$ of $\omega\cap\mathbb{Z}^{n}$, we have an evaluation
$\mathfrak{D}(m)$ belonging to the  $\mathbb{Q}$-linear space of
rational Weil divisors on $C$. Given a combinatorial data $(C,\omega,\mathfrak{D})$ 
one can construct an $M$-graded algebra
\begin{eqnarray*}
A[C,\mathfrak{D}] := \bigoplus_{m\in\omega\cap M}H^{0}(C,\mathcal{O}_{C}(\lfloor\mathfrak{D}(m)\rfloor ))\chi^{m},
\end{eqnarray*}
where $\chi^{m}$ is the Laurent monomial $t_{1}^{m_{1}}\ldots t_{n}^{m_{n}}$. See [AH] for definitions
and specific statements. One of the main results of this paper can be stated as follows (see
Theorem $3.4.4$).

\paragraph{Theorem.}{\em Let $C$ be a smooth algebraic curve. Consider a subalgebra
\begin{eqnarray*}
B = \mathbf{k}[C][f_{1}\chi^{s_{1}},\ldots,f_{r}\chi^{s_{r}}]\subset L_{[n]}(\mathbf{k}(C))
:=\mathbf{k}(C)\left[t_{1},t_{1}^{-1},\ldots,t_{n},t_{n}^{-1}\right] 
\end{eqnarray*}
such that
\begin{eqnarray*}
\rm Frac\, \it B = \rm Frac\,\it L_{[n]}(\mathbf{k}(C)),
\end{eqnarray*}
where $s_{i}\in\mathbb{Z}^{n}$, $\chi^{s_{i}}$ is the corresponding Laurent monomial, and
$f_{i}\in\mathbf{k}(C)^{\star}$. Then the  normalization of
$B$ is the algebra $A[C,\mathfrak{D}]$ with weight cone $\omega$ generated by $s_{1},\ldots, s_{r}$
and with the following coefficient of the polyhedral divisor $\mathfrak{D}$ at the point $z\in C$ $\rm :$
\begin{eqnarray*}
\Delta_{z} = \left\{\,v\in\mathbb{Q}^{n},\,\left\langle s_{i}\,,\,v\right\rangle\geq -\rm ord_{\it z}\it \,f_{i},\,
i = \rm 1,\it\ldots, r\,\right\},
\end{eqnarray*}
where $\rm ord_{\it z}\,\it f_{i}$ is the order of $f_{i}$ at $z$.}
\paragraph{}
This theorem answers the first question.
It generalizes well known results for the case of affine 
surfaces with $\mathbf{k}^{\star}$-actions [FZ, 3.9, 4.6].
Note also that it may be applied 
\begin{enumerate}
 \item[(*)] to find a proper polyhedral divisor representing a (normal) subalgebra given 
by generators ;
\item[(**)] to find generators of an algebra represented by a proper polyhedral divisor :
the idea is to guess some generating set and apply (*).  
\end{enumerate} 

The answer to the second question is given by Theorem $3.6.2$.
It is known that the set of integrally closed homogeneous ideals of the affine toric
variety with weight cone $\omega$ is in bijective correspondence with the 
set of integral $\omega$-polyhedra contained in $\omega$ (see [KKMS] and section $3.5$).
We provide a similar correspondence for integrally closed homogeneous
ideals on affine $\mathbb{T}$-variety of complexity $1$ (see Theorem $3.6.6$)
that is wholly combinatorial when $C$ is affine (see Corollary $3.6.7$).

Any integrally closed homogeneous ideal $I$ of an affine 
$\mathbb{T}$-variety $X = \rm Spec\,\it A$ of complexity $1$ can be described by means of a pair 
$(P,\widetilde{\mathfrak{D}})$ where $P$ is an integral polyhedron in $\mathbb{Q}^{n}$. 
This polyhedron plays the same role as the Newton
polyhedron does in the monomial case. The polyhedral divisor $\widetilde{\mathfrak{D}}$
corresponds to the normalization of the Rees algebra of $I$. We give a geometric 
interpretation of
the coefficients of $\widetilde{\mathfrak{D}}$. Assume for instance that the weight cone
of $A$ is strongly convex, and let $\widetilde{\Delta}_{z}$ be the coefficient
of $\widetilde{\mathfrak{D}}$ at a point $z\in C$. Then Theorem $3.6.6$
provides conditions on the equations of facets of $\widetilde{\Delta}_{z}$ so that
$\widetilde{\mathfrak{D}}$ corresponds to the normalization of the Rees algebra $A[It]$.

We provide as well sufficient conditions 
on $(P,\widetilde{\mathfrak{D}})$ in order for $I$ be normal 
(see Theorem $3.7.3$). For
the case of non-elliptic affine $\mathbf{k}^{\star}$-surfaces,
we obtain a combinatorial proof for the normality of any integrally closed
invariant ideals of such surfaces.
As another application, we obtain the following new criterion of normality
which generalizes Reid-Roberts-Vitulli's Theorem [RRV, $3.1$] in
the case of complexity $0$ (see $3.7.5$).

\paragraph{Theorem.}{\em 
Let $n\geq 1$ be an integer and let $\mathbf{k}^{[n+1]} = \mathbf{k}[x_{0},\ldots, x_{n}]$
be the algebra of polynomials in $n+1$ variables over $\mathbf{k}$.
We endow $\mathbf{k}^{[n+1]}$ with the $\mathbb{Z}^{n}$-grading given by
\begin{eqnarray*}
\mathbf{k}^{[n+1]} = \bigoplus_{m = (m_{1},\ldots, m_{n})\in\mathbb{N}^{n}}A_{m},
\,\,\,\it where\,\,\,\it 
A_{m} = \mathbf{k}[x_{\rm \,0 \it }]\,x_{\rm 1\it }^{m_{\rm 1\it}}
\ldots x_{n}^{m_{n}}.
\end{eqnarray*}
For a homogeneous ideal $I$ of $\mathbf{k}^{[n+1]}$ the following are equivalent.
\begin{enumerate}
\item[\rm (i)] The ideal $I$ is normal;
\item[\rm (ii)] For any $e\in\{1,\ldots, n\}$, the ideal $I^{e}$ is 
integrally closed.
\end{enumerate}
} 

\paragraph{}
Let us give a brief summary of the contents of each section. 
In Section $3.3$, we recall some notions on tori actions
of complexity $1$ and on polyhedral divisors of Altmann-Hausen.
In Section $3.4$, we treat the normalization problem
for multigraded algebras and show Theorem $3.4.4$. Section $3.5$ 
focuses on integrally closed monomial ideals.
In Section $3.6$, we study the description 
in terms of the pairs $(P,\widetilde{\mathfrak{D}})$ for integrally
closed homogeneous ideals of affine $\mathbb{T}$-varieties. Finally, in the last section,   
we discuss the problem of normality in a special class.
\paragraph{}
Throughout this paper $\mathbf{k}$ is an algebraically closed field of characteristic zero.
By a variety we mean an integral separated scheme of finite type over $\mathbf{k}$.
\paragraph{}
\section{$\mathbb{T}$-vari\'et\'es affines de complexit\'e un et g\'eom\'etrie convexe}
Nous rappelons ici les notions fondamentales sur les op\'erations de tores alg\'ebriques dont nous
aurons besoin par la suite. 
\begin{rappel}
Soit $N$ un r\'eseau de rang $n$ et $M = \rm Hom(\it N,\mathbb{Z})$
son r\'eseau dual. On note $N_{\mathbb{Q}} = \mathbb{Q}\otimes_{\mathbb{Z}}N$ et 
$M_{\mathbb{Q}} = \mathbb{Q}\otimes_{\mathbb{Z}}M$
les $\mathbb{Q}$-espaces vectoriels duaux associ\'es. Au r\'eseau $M$, on associe un tore 
alg\'ebrique $\mathbb{T}$ de dimension $n$ dont l'alg\`ebre
des fonctions r\'eguli\`eres est d\'efinie par 
\begin{eqnarray*}
\mathbf{k}[\mathbb{T}] = \bigoplus_{m\in M}\mathbf{k}\,\chi^{m}.
\end{eqnarray*}
La famille $(\chi^{m})_{m\in M}$ satisfait les relations $\chi^{m}\cdot\chi^{m'} = \chi^{m+m'}$, 
pour tous $m,m'\in M$. Le 
choix d'une base de $M$ donne un isomorphisme entre $\mathbf{k}[\mathbb{T}]$ et l'alg\`ebre 
des polyn\^omes de Laurent \`a $n$ variables.
Chaque fonction $\chi^{m}$ s'interpr\`ete comme un caract\`ere de $\mathbb{T}$
et tous les caract\`eres s'obtiennent ainsi.
\end{rappel} 
\begin{rappel}
Soit $X$ une vari\'et\'e affine et supposons que $\mathbb{T}$ op\`ere alg\'ebriquement dans $X$. 
Alors cela induit une op\'eration de $\mathbb{T}$ dans 
$A := \mathbf{k}[X]$ d\'efinie par $(t\cdot f)(x) = f(t\cdot x)$ avec $t\in \mathbb{T}$, 
$f\in \mathbf{k}[X]$ et $x\in X$, faisant du $\mathbf{k}$-espace vectoriel $A$ un 
$\mathbb{T}$-module rationnel. Le $\mathbb{T}$-module $A$ admet une d\'ecomposition 
$A = \bigoplus_{m\in M}A_{m}$ en somme directe de sous-espaces vectoriels o\`u
pour tout $m\in M$, 
\begin{eqnarray*}
A_{m} = \left\{f\in A\,|\,\forall t\in \mathbb{T},\,t\cdot f = \chi^{m}(t)f\right\}\,.
\end{eqnarray*}
L'alg\`ebre $A$ est ainsi munie d'une $M$-graduation. R\'eciproquement, 
toute $M$-graduation de la $\mathbf{k}$-alg\`ebre $A$ est obtenue
par une op\'eration alg\'ebrique de $\mathbb{T}$ dans $X = \rm Spec\,\it A$. La partie
\begin{eqnarray*}
S := \left\{m\in M\,|\, A_{m}\neq \{0\}\right\}
\end{eqnarray*}
de $M$ contenant $0$ et stable par l'addition est appel\'ee \em mono\"ide des poids \rm de $A$. 
Puisque $A$ est de type fini sur $\mathbf{k}$, 
l'ensemble $S$ engendre un c\^one poly\'edral $\omega\subset M_{\mathbb{Q}}$ 
dit \em c\^one des poids \rm de $A$. 

L'op\'eration de $\mathbb{T}$ dans $X$ est fid\`ele si et seulement si $S$ n'est pas 
contenu dans un sous-r\'eseau propre de $M$. 
Dans ce cas, le c\^one $\omega$ est de dimension $n$ et
il existe un unique c\^one poly\'edral saillant $\sigma\subset N_{\mathbb{Q}}$ tel que
\begin{eqnarray*}
\omega = \left\{m\in M_{\mathbb{Q}}\,|\,\forall v\in\sigma,\, m(v)\geq 0\right\}\,.
\end{eqnarray*}
On dit que $\omega$ est le \em c\^one dual \rm de $\sigma$. On le note $\sigma^{\vee}$. 
Une \em $\mathbb{T}$-vari\'et\'e \rm est une vari\'et\'e 
normale munie d'une op\'eration alg\'ebrique fid\`ele de $\mathbb{T}$.

\end{rappel}
\begin{rappel}
Consid\'erons le sous-corps $\mathbf{k}(X)^{\mathbb{T}}$ des fonctions rationnelles $\mathbb{T}$-invariantes sur $X$. 
Nous appelons \em complexit\'e \rm de l'op\'eration de $\mathbb{T}$ dans $X = \rm Spec\,\it A$, le degr\'e de 
transcendance de l'extension de corps $\mathbf{k}(X)^{\mathbb{T}}/\mathbf{k}$. La complexit\'e s'interpr\`ete 
g\'eom\'etriquement comme la codimension d'une orbite
en position g\'en\'erale. Lorsque l'op\'eration est fid\`ele, elle est \'egale \`a $\dim\, X - \dim\,\mathbb{T}$. 
Elle fut 
introduite dans [LV] pour les op\'erations de groupes alg\'ebriques r\'eductifs dans les espaces homog\`enes. 

Les $\mathbb{T}$-vari\'et\'es affines de complexit\'e $0$ sont les \em vari\'et\'es toriques affines. \rm 
Elles admettent une description bien connue en terme de c\^ones poly\'edraux saillants.
Pour plus d'informations, voir [CLS, Da, Fu, Od].
\paragraph{}
Etant donn\'es un sous-mono\"ide $E\subset M$ et un corps $K_{0}$,  
nous notons
\begin{eqnarray*}
K_{0}[E] = \bigoplus_{m\in E}K_{0}\,\chi^{m}.
\end{eqnarray*}
l'alg\`ebre du mono\"ide $E$ sur le corps $K_{0}$. Pour tout c\^one poly\'edral 
$\tau\subset M_{\mathbb{Q}}$, nous \'ecrivons $\tau_{M}:=\tau\cap M$. 
\end{rappel}
\paragraph{}
Nous rappelons une description combinatoire des $\mathbb{T}$-vari\'et\'es affines de complexit\'e $1$ 
obtenue dans [AH, Ti]. Voir [FZ] pour 
la pr\'esentation de Dolgachev-Pinkham-Demazure $(D.P.D.)$ des 
$\mathbb{C}^{\star}$-surfaces affines complexes. 
Notons que 
l'approche donn\'ee dans [Ti] provient d'une description 
ant\'erieure des op\'erations de groupes alg\'ebriques r\'eductifs 
de complexit\'e $1$ [Ti $2$]. Dans [Vol], on donne un lien entre
[AH] et la description classique de [KKMS].

\begin{rappel}
Soit $C$ une courbe alg\'ebrique lisse et $\sigma\subset N_{\mathbb{Q}}$ un c\^one poly\'edral saillant. Une partie
$\Delta\subset N_{\mathbb{Q}}$ est un \em $\sigma$-poly\`edre \rm si $\Delta$ est somme de Minkowski de $\sigma$ et 
d'un polytope $Q\subset N_{\mathbb{Q}}$. On note $\rm Pol_{\sigma}(\it N_{\mathbb{Q}})$ le mono\"ide 
ab\'elien des $\sigma$-poly\`edres de loi la somme de Minkowski et d'\'el\'ement neutre $\sigma$. 

Pour $\Delta\in\rm Pol_{\sigma}(\it N_{\mathbb{Q}})$, on d\'efinit la fonction $h_{\Delta}:\sigma^{\vee}\rightarrow \mathbb{Q}$
 par $h_{\Delta}(m)  = \min_{v\in\Delta}m(v)$, pour tout $m\in\sigma^{\vee}$. On dit que $h_{\Delta}$
est la \em fonction de support \rm de $\Delta$. Elle est identiquement nulle si et seulement si $\Delta = \sigma$. Pour tous $m,m'\in\sigma^{\vee}$, 
on a l'in\'egalit\'e de sous-additivit\'e 
\begin{eqnarray}
h_{\Delta}(m)+h_{\Delta}(m')\leq h_{\Delta}(m+m').
\end{eqnarray}
Un \em diviseur $\sigma$-poly\'edral \rm sur $C$ est une somme formelle
\begin{eqnarray*}
\mathfrak{D} = \sum_{z\in C}\Delta_{z}\,.\,z
\end{eqnarray*} 
 o\`u chaque $\Delta_{z}$ appartient \`a $\rm Pol_{\sigma}(\it N_{\mathbb{Q}})$ avec pour presque 
\footnote{ Cela signifie qu'il existe un sous-ensemble fini $F\subset C$ tel que pour tout $z\in C-F$, 
$\Delta_{z} = \sigma$.} tout $z\in C$, $\Delta_{z} = \sigma$. On note $\rm supp(\it\mathfrak{D}\rm)$, dit
\em support \rm de $\mathfrak{D}$, l'ensemble des points $z\in C$ tels que $\Delta_{z}\neq \sigma$.
On dit que $\mathfrak{D}$ est \em propre \rm lorsque 
pour tout $m\in\sigma^{\vee}_{M}$, l'\'evaluation 
\begin{eqnarray*}
\mathfrak{D}(m) := \sum_{z\in C}h_{\Delta_{z}}(m)\,.\, z
\end{eqnarray*}
est un diviseur de Cartier rationnel semi-ample et abondant (big) pour $m$ appartenant \`a l'int\'erieur relatif de 
$\sigma^{\vee}$.
\end{rappel}
\begin{rappel}
 La propret\'e de $\mathfrak{D}$ est d\'ecrite en distinguant les deux cas suivants [AH, $2.12$].
\begin{enumerate}
 \item[(i)]
Lorsque $C$ est une courbe projective lisse, le diviseur poly\'edral $\mathfrak{D}$ est propre si et seulement 
si pour tout vecteur $m\in\sigma^{\vee}_{M}$, on a $\rm deg(\it \mathfrak{D}(m)\rm ) \geq 0$ et si 
$\rm deg(\it \mathfrak{D}(m)\rm ) = 0$ alors $m$ appartient 
au bord de $\sigma^{\vee}$ et un multiple non nul de $\mathfrak{D}(m)$ 
est principal;
\item[(ii)]Lorsque $C$ est une courbe affine lisse, aucune condition n'est impos\'ee sur l'\'evaluation de $\mathfrak{D}$. 
\end{enumerate}
\end{rappel}
\begin{notation}
D'apr\`es l'in\'egalit\'e (3.1) dans $3.3.4$, si $\mathfrak{D}$ est un diviseur $\sigma$-poly\'edral sur $C$ alors pour
tous vecteurs $m,m'\in\sigma^{\vee}_{M}$, l'application 
\begin{eqnarray*}
H^{0}(C,\mathcal{O}_{C}(\lfloor\mathfrak{D}(m)\rfloor ))\times H^{0}(C,\mathcal{O}_{C}(\lfloor\mathfrak{D}(m')\rfloor ))\rightarrow
H^{0}(C,\mathcal{O}_{C}(\lfloor\mathfrak{D}(m+m')\rfloor )),\,\, (f,g)\mapsto f\cdot g  
\end{eqnarray*}
 est bien d\'efinie. Le $\mathbb{T}$-module 
\begin{eqnarray*}
A[C,\mathfrak{D}] := \bigoplus_{m\in\sigma^{\vee}_{M}}H^{0}(C,\mathcal{O}_{C}(\lfloor\mathfrak{D}(m)\rfloor ))\chi^{m},
\end{eqnarray*}
est donc une alg\`ebre $M$-gradu\'ee. On la notera $A[\mathfrak{D}]$ 
lorsque la mention de $C$ est claire. 
Pour tout \'el\'ement homog\`ene $f\chi^{m}$ de $\mathbf{k}(C)[M]$, 
nous sous-entendons que $f\in k(C)^{\star}$ et que $\chi^{m}$ 
est un caract\`ere du tore $\mathbb{T}$. 
\end{notation}
\paragraph{}
Le th\'eor\`eme suivant d\'ecrit les alg\`ebres affines 
normales $M$-gradu\'ees de complexit\'e $1$
et de c\^one des poids $\sigma^{\vee}$ (voir [AH, Ti]).
\begin{theorem}
\begin{enumerate}
 \item[\rm (i)] 
Si $\mathfrak{D}$ est un diviseur $\sigma$-poly\'edral propre sur une courbe alg\'ebrique lisse $C$ 
alors l'alg\`ebre multigradu\'ee $A = A[C,\mathfrak{D}]$
est normale, de type fini sur $\mathbf{k}$ et sa $M$-graduation fait de $X = \rm Spec\, \it A\rm$ 
une $\mathbb{T}$-vari\'et\'e de complexit\'e $1$. R\'eciproquement, l'alg\`ebre des fonctions
r\'eguli\`eres de toute $\mathbb{T}$-vari\'et\'e affine de complexit\'e $1$ est isomorphe comme alg\`ebre 
$M$-gradu\'ee \`a $A[C,\mathfrak{D}]$, o\`u $C$ est une courbe alg\'ebrique lisse et
$\mathfrak{D}$ est un diviseur poly\'edral propre sur $C$. 
\item[\rm (ii)]
Plus pr\'ecis\'ement, toute sous-alg\`ebre multigradu\'ee $A\subset\mathbf{k}(C)[M]$ 
normale, de type fini sur $\mathbf{k}$, satisfaisant $A_{0} = \mathbf{k}[C]$,
de m\^eme corps des fractions que $A$ et ayant $\sigma^{\vee}$ pour c\^one 
des poids, est \'egale \`a $A[C,\mathfrak{D}]$ pour un unique diviseur
$\sigma$-poly\'edral $\mathfrak{D}$ propre sur $C$. Compte tenu des hypoth\`eses ci-dessus, 
l'alg\`ebre $A$ d\'etermine de fa\c con canonique la courbe $C$. 
\end{enumerate}
\end{theorem}
\begin{notation}
Si $\mathfrak{D}$ est un diviseur poly\'edral
propre sur $C$ alors on note par $X[C,\mathfrak{D}]$ la $\mathbb{T}$-vari\'et\'e 
affine $\rm Spec \,\it A[C,\mathfrak{D}]\rm$ correspondante.
\end{notation}
\paragraph{}
L'assertion suivante est connue [AH, \S 8], [Li, Theorem 1.4(3)]. Elle permet de comparer deux donn\'ees combinatoires 
$(C, \sigma, \mathfrak{D})$ et $(C', \sigma', \mathfrak{D}')$ donnant des $\mathbf{k}$-alg\`ebres $M$-gradu\'ees isomorphes.
\begin{theorem}
Soient $C,C'$ des courbes alg\'ebriques lisses et $\mathfrak{D} , \mathfrak{D}'$ des diviseurs poly\'edraux propres
respectivement sur $C,C'$ selon des c\^ones poly\'edraux saillants $\sigma , \sigma'\subset N_{\mathbb{Q}}$. Alors 
$X[C,\mathfrak{D}]$ et $X[C',\mathfrak{D}']$ sont $\mathbb{T}$-isomorphes si et seulement si il existe un automorphisme  
de r\'eseau $F:N\rightarrow N$ tel que\footnote{Un automorphisme de r\'eseau $F:N\rightarrow N$ induit un 
automorphisme de l'espace vectoriel $N_{\mathbb{Q}}$ que l'on note $F_{\mathbb{Q}}$.} $F_{\mathbb{Q}}(\sigma) = \sigma'$,
 un isomorphisme $\phi : C\rightarrow C'$ de courbes alg\'ebriques,
$v_{1},\ldots,v_{r}\in N$ et $f_{1},\ldots,f_{r}\in k(C)^{\star}$ tels que pour tout $m\in\sigma'^{\,\vee}_{M}$,
\begin{eqnarray*}
\phi^{\star}(\mathfrak{D}')(m) = F_{\star}(\mathfrak{D})(m) + \sum_{i=1}^{r}m(v_{i})\cdot\rm div( \it f_{i}\rm )
\end{eqnarray*} 
avec 
\begin{eqnarray*}
\mathfrak{D} = \sum_{z\in C}\Delta_{z}\,.\, z,\,\,\mathfrak{D}' = \sum_{z'\in C'}\Delta'_{z'}\,.\, z',\,\,
\phi^{\star}(\mathfrak{D}') = \sum_{z'\in C'}\Delta'_{z'}\,.\,\phi^{-1}(z')\,\, \rm et \it\,\,
\end{eqnarray*}
\begin{eqnarray*}
F_{\star}(\mathfrak{D}) = \sum_{z\in C}(F_{\mathbb{Q}}(\rm \Delta_{\it z}) + \sigma )\,.\,\it z.
\end{eqnarray*}
\end{theorem}  
La proposition suivante montre que la fonction de support 
d'un $\sigma$-poly\`edre est lin\'eaire par morceaux.
La d\'emonstration de ce r\'esultat est ais\'ee et laiss\'ee aux lecteurs.
\begin{lemme}
Soit $\sigma\subset N_{\mathbb{Q}}$ un c\^one poly\'edral saillant, $\Delta\in\rm Pol_{\sigma}(\it N_{\mathbb{Q}})$ 
un poly\`edre et $S\subset \Delta$
une partie non vide. Notons $V(\Delta)$ l'ensemble des sommets de $\Delta$.
Alors $\Delta = \rm Conv(\it S\rm ) \it +\sigma$ si et seulement si 
pour tout $m\in\sigma^{\vee}$, $h_{\Delta}(m) = \min_{v\in S}m(v)$.
En particulier, pour tout $m\in\sigma^{\vee}$, on a $h_{\Delta}(m) = \min_{v\in V(\Delta )}m(v)$.
\end{lemme}
La terminologie suivante est classique pour les $\mathbb{C}^{\star}$-surfaces affines complexes [FZ]. Elle est 
introduite dans [Li] pour les $\mathbb{T}$-vari\'et\'es affines de complexit\'e $1$. 
\begin{definition}
Soit $X$ une vari\'et\'e affine. Une op\'eration alg\'ebrique de $\mathbb{T}$ de complexit\'e $1$ dans $X$ est dite 
\em elliptique \rm si 
l'alg\`ebre multigradu\'ee correspondante 
\begin{eqnarray*}
A = \mathbf{k}[X] = \bigoplus_{m\in M}A_{m}\,\,\,\,\rm avec \,\,\it A_{m} = 
\left\{f\in A\,|\,\forall t\in \mathbb{T},\,t\cdot f = \chi^{m}(t)f\right\}
\end{eqnarray*}
v\'erifie $A_{0} = \mathbf{k}$. Dans ce cas, on dit que l'alg\`ebre $M$-gradu\'ee $A$
est elliptique.
\end{definition} 
\begin{remarque} 
Consid\'erons $A = A[C,\mathfrak{D}]$. Alors $A$ est elliptique si
et seulement si $C$ est projective. L'ellipticit\'e de $A$ donne des 
contraintes g\'eom\'etriques sur la vari\'et\'e $X = \rm Spec\,\it A$. 
En effet, si $C$ est affine alors $X$ est toro\"idale [KKMS] et n'admet donc
que des singularit\'es toriques. 
Tandis que lorsque $C$ est projective de genre $g\geq 1$, la vari\'et\'e $X$ 
a au moins une singularit\'e qui n'est pas rationnelle [LS, Propositions 5.1, 5.5].     
\end{remarque}
La proposition suivante peut \^etre vue comme une g\'en\'eralisation de [FZ, Lemma 1.3(a)].
\begin{proposition}
Soit $X$ une $\mathbb{T}$-vari\'et\'e affine de complexit\'e $1$ et notons $\sigma\subset N_{\mathbb{Q}}$ 
le dual du c\^one des poids de 
l'alg\`ebre multigradu\'ee
\begin{eqnarray*}
A = \mathbf{k}[X] = \bigoplus_{m\in M}A_{m}\,\,\,\rm avec \,\,\,\it A_{m} = 
\left\{f\in A\,|\,\forall t\in T,\,t\cdot f = \chi^{m}(t)f\right\}
\end{eqnarray*}
obtenue par l'op\'eration de $\mathbb{T}$. Les assertions suivantes sont vraies.
\begin{enumerate}
\item[\rm (i)] Si l'op\'eration de $\mathbb{T}$ n'est pas elliptique alors 
le mono\"ide des poids de $\mathbf{k}[X]$ est $\sigma^{\vee}_{M}$ 
et pour tout $m\in\sigma^{\vee}_{M}$, le $A_{0}$-module $A_{m}$ est localement libre de rang $1$;
\item[\rm (ii)] Si l'op\'eration est elliptique alors $\sigma\neq\{0\}$; 
\item[\rm (iii)] L'intersection du sous-ensemble
\begin{eqnarray*}
L = \{m\in\sigma^{\vee}_{M}\,|\, A_{m} = \{0\}\}\subset M_{\mathbb{Q}}
\end{eqnarray*} 
avec toute droite vectorielle rencontrant l'int\'erieur relatif de $\sigma^{\vee}$ est finie.
\end{enumerate}
\end{proposition}
\begin{proof}
Si l'op\'eration de $\mathbb{T}$ n'est pas elliptique alors par le th\'eor\`eme $3.3.7$, on peut 
supposer qu'il existe une courbe alg\'ebrique affine lisse $C$ et un diviseur $\sigma$-poly\'edral propre $\mathfrak{D}$
sur $C$ tels que pour tout $m\in\sigma^{\vee}_{M}$, 
\begin{eqnarray}
A_{m} = H^{0}(C,\mathcal{O}_{C}(\lfloor\mathfrak{D}(m)\rfloor ))\chi^{m}\,\,\,\,\rm  et \it\,\,\,\, \mathbf{k}[X] = \bigoplus_{m\in\sigma^{\vee}_{M}}A_{m}.
\end{eqnarray}
Comme $C$ est affine, pour tout $m\in \sigma^{\vee}_{M}$, on a $A_{m}\neq \{0\}$. D'o\`u l'assertion $\rm (i)$. 

Si l'op\'eration est elliptique alors on
peut supposer que $\mathbf{k}[X]$ v\'erifie $(3.2)$ avec $C$ une courbe projective lisse de genre $g$. Si $\sigma = \{0\}$ alors 
$\sigma^{\vee} = M_{\mathbb{Q}}$. Pour tout $m\in \sigma^{\vee}_{M}$, on a par propret\'e: 
$\rm deg(\it \mathfrak{D}(m)\rm ) > 0$. L'\'egalit\'e $\mathfrak{D}(0) = 0$ donne une contradiction.
D'o\`u l'assertion $\rm (ii)$. 

Soit $m\in \sigma^{\vee}_{M}$ appartenant \`a l'int\'erieur relatif de $\sigma^{\vee}$.
Alors il existe $r_{0}\in\mathbb{Z}_{>0}$ tel que pour tout $z\in\rm supp(\it\mathfrak{D}\rm)$ et 
tout sommet $v$ de $\Delta_{z}$, $r_{0}v\in N$. Ainsi, par le lemme $3.3.10$,
\begin{eqnarray}
\mathfrak{D}(r_{0}m) = \lfloor\mathfrak{D}(r_{0}m)\rm )\rfloor.
\end{eqnarray}
Par la propret\'e de $\mathfrak{D}$, on peut supposer que 
\begin{eqnarray*}
\rm deg(\it \mathfrak{D}(r_{\rm 0\it }m)\rm )\it  = \rm deg(\it \lfloor\mathfrak{D}(r_{\rm 0 \it}m)\rfloor\rm ) >\it  d +g-\rm 1
\end{eqnarray*}
o\`u $d$ est le cardinal de $\rm supp(\it\mathfrak{D}\rm)$. Donc pour tout $r\in\mathbb{N}$,
\begin{eqnarray*}
\rm deg(\it \lfloor\mathfrak{D}((r_{\rm 0 \it}+r)m)\rfloor\rm )\it = 
\rm deg(\it \lfloor\mathfrak{D}(r_{\rm 0 \it}m)\rfloor\rm )\it + \rm deg(\it \lfloor\mathfrak{D}(rm)\rfloor\rm )
\end{eqnarray*}
\begin{eqnarray}
\rm\geq deg(\it \lfloor\mathfrak{D}(r_{\rm 0\it}m)\rfloor\rm )\it - d \rm >\it g-\rm 1.
\end{eqnarray}
Comme $\sigma\neq\{0\}$, la demi-droite $\mathbb{Q}_{\leq 0}\cdot m$ ne rencontre
$\sigma^{\vee}$ qu'en l'origine [CLS, Exercice 1.2.4].
Par le th\'eor\`eme de Riemann-Roch, $(3.4)$ donne l'inclusion 
\begin{eqnarray*}
 \mathbb{Q}\cdot m\,\cap\, L\subset \{0,m,2m,\ldots,(r_{0}-1)m\}\,.
\end{eqnarray*}
D'o\`u le r\'esultat.
\end{proof}
\begin{remarque}
Notons $\sigma = \mathbb{Q}_{\geq 0}^{2}$.
L'exemple du diviseur poly\'edral propre 
\begin{eqnarray*}
\mathfrak{D} = \Delta_{0}\cdot 0+ \Delta_{1}\cdot 1 + \Delta_{\infty}\cdot \infty,\,\,\,
\Delta_{0} = \left(\frac{1}{2},0\right) + \sigma,\,\,\, \Delta_{1} = \left(-\frac{1}{2},0\right) + \sigma,\,\,\,
\end{eqnarray*}
\begin{eqnarray*}
\Delta_{\infty} = [(1,0),(0,1)]+\sigma,
\end{eqnarray*} 
sur $\mathbb{P}^{1}$
montre qu'en g\'en\'eral, il existe des demi-droites vectorielles contenues dans le bord de $\sigma^{\vee}$ qui 
rencontrent $L$ en une infinit\'e de points. En effet, dans cet exemple, pour tout $r\in\mathbb{N}$, on a
\begin{eqnarray*}
H^{0}(\mathbb{P}^{1},\mathcal{O}_{\mathbb{P}^{1}}(\lfloor\mathfrak{D}(2r+1,0)\rfloor ))\chi^{(2r+1,0)} = \{0\}.
\end{eqnarray*}
\end{remarque}
\section{Normalisation des alg\`ebres affines multigradu\'ees de complexit\'e un}
Le but de cette section est de d\'ecrire explicitement la normalisation
des alg\`ebres affines multigradu\'ees de complexit\'e $1$.

Dans le cas de la complexit\'e $0$, toute alg\`ebre $M$-gradu\'ee affine
est r\'ealis\'ee comme sous-alg\`ebre $\mathbb{T}$-stable 
de $\mathbf{k}[M]$. La normalisation de $A$ est d\'etermin\'ee par son
c\^one des poids [Ho].
De fa\c con analogue, toute alg\`ebre $M$-gradu\'ee affine de complexit\'e $1$ est plong\'ee dans une 
alg\`ebre $\mathbf{k}(C)[M]$ o\`u $C$ est une courbe alg\'ebrique lisse. 
Nous rappelons ceci dans le paragraphe suivant. 
\begin{rappel}
Soit $A = \bigoplus_{m\in M}A_{m}$ 
une alg\`ebre $M$-gradu\'ee int\`egre de type fini sur $\mathbf{k}$.
Notons $K$ son corps des fractions. On suppose que $\rm tr.deg_{\it \mathbf{k}}\,\it K\rm^{\it \mathbb{T}} = 1$. 
Sans perte de g\'en\'eralit\'e, on peut supposer que la $M$-graduation de $A$ 
est fid\`ele. Alors pour tout vecteur $m\in M$,
\begin{eqnarray*}
K_{m} := \{f/g\,|\,\exists e,\, f\in A_{m+e},\,g\in A_{e},\,g\neq 0\}\subset K
\end{eqnarray*}
est un sous-espace vectoriel de dimension $1$ sur $K_{0} = K^{\mathbb{T}}$. Le choix d'une base de $M$ nous permet 
de construire une famille $(\chi^{m})_{m\in M}$ d'\'el\'ements de $K^{\star}$ v\'erifiant pour tous 
$m,m'\in M$,
\begin{eqnarray*} 
K_{m} = K_{0}\,\chi^{m}\,\,\,\,\rm et \,\,\,\,\it\chi^{m}\cdot\chi^{m'} = \chi^{m+m'}.
\end{eqnarray*}
Quitte \`a remplacer $A_{0}$ par $\bar{A_{0}}$, on peut 
supposer que $A_{0}$ est normale. Notons que l'alg\`ebre $A_{0}$ est de type fini sur $\mathbf{k}$. 
Soit $\widetilde{C}$ la courbe compl\`ete lisse sur $\mathbf{k}$ obtenue \`a partir de l'ensemble 
des valuations discr\`etes de $K_{0}/\mathbf{k}$, de sorte que
$K_{0} = \mathbf{k}(\widetilde{C})$. 

Si l'op\'eration de $\mathbb{T}$ dans $X = \rm Spec\,\it A$ n'est pas elliptique alors $K_{0}$ est
le corps des fractions de $A_{0}$. En effet, si $b\in K_{0}$ alors $b$ est un \'el\'ement alg\'ebrique sur
$\rm Frac\,\it A_{\rm 0}$. Il existe donc $a\in A_{0}$ non nul tel que $ab$ soit entier sur $A_{0}$. Par
normalit\'e de $A_{0}$, on a $ab \in \bar{A}\cap K_{0} = A_{0}$ et donc $K_{0} \subset \rm Frac\,\it A_{\rm 0}$.
L'inclusion r\'eciproque est imm\'ediate.
Dans tous les cas, il existe un unique ouvert non vide 
$C\subset \widetilde{C}$ tel que 
\begin{eqnarray*}
A_{0} = \mathbf{k}[C]\,\,\,\,\rm et \,\,\,\,\it A\subset \bigoplus_{m\in M}K_{m} = 
\mathbf{k}(C)[M]\,.
\end{eqnarray*}
Par le th\'eor\`eme $3.3.7$, 
l'\'egalit\'e $K = \rm Frac\,\it \mathbf{k}(C)[M]$ 
implique que $\bar{A} = A[C,\mathfrak{D}]$, pour un unique
diviseur poly\'edral propre $\mathfrak{D} = \sum_{z\in C}\Delta_{z}\cdot z$. 
\paragraph{}
Fixons un syst\`eme de g\'en\'erateurs homog\`enes 
\begin{eqnarray*}
A = \mathbf{k}[C][f_{1}\chi^{m_{1}},\ldots,f_{r}\chi^{m_{r}}]
\end{eqnarray*}
avec $f_{1},\ldots,f_{r}$ des fonctions rationnelles non nulles sur $C$ et $m_{1},\ldots,m_{r}$ des \'el\'ements de $M$.
Il s'agit de d\'eterminer les poly\`edres $\Delta_{z}$ en fonction de  
$((f_{1},m_{1}),\ldots,(f_{r},m_{r}))$.  Dans [FZ], on donne la pr\'esentation\footnote{
La pr\'esentation $D.P.D.$
est la description donn\'ee dans $3.3.7$ pour le cas des surfaces.} $D.P.D.$ de $\bar{A}$ 
pour le cas des surfaces affines complexes avec op\'eration parabolique ou hyperbolique de $\mathbb{C}^{\star}$. 
Nous rappelons ces r\'esultats dans le corollaire $3.4.7$.
\end{rappel}
Le lemme \'el\'ementaire suivant sera utilis\'e dans la preuve du th\'eor\`eme $3.4.4$.
La d\'emonstration de ce r\'esultat est laiss\'ee aux lecteurs.
\begin{lemme}
Soient $\sigma\subset N_{\mathbb{Q}}$ un c\^one poly\'edral saillant et $\Delta_{1},\Delta_{2}\in\rm Pol_{\sigma}(\it N_{\mathbb{Q}})$
des $\sigma$-poly\`edres. 
Alors $\Delta_{1} = \Delta_{2}$ si et seulement si pour tout $m\in\sigma^{\vee}_{M}$ appartenant \`a 
l'int\'erieur relatif de $\sigma^{\vee}$, $h_{\Delta_{1}}(m) = h_{\Delta_{2}}(m)$.
\end{lemme}

\begin{notation}
Soient $C$ une courbe alg\'ebrique lisse et $f = (f_{1}\chi^{m_{1}},\ldots,f_{r}\chi^{m_{r}})$ 
un $r$-uplet d'\'el\'ements homog\`enes de $\mathbf{k}(C)[M]$ 
tels que $\sum_{i = 1}^{r}\mathbb{Q}\,m_{i} = M_{\mathbb{Q}}$. Posons 
$\sigma \subset N_{\mathbb{Q}}$ le c\^one poly\'edral saillant dual de 
$\sum_{i = 1}^{r}\mathbb{Q}_{\geq 0}\,m_{i}$ et pour tout $z\in C$, consid\'erons 
$\Delta_{z}[f]\subset N_{\mathbb{Q}}$ le $\sigma$-poly\`edre d\'efini par les in\'egalit\'es 
\begin{eqnarray*}
m_{i}\geq -\rm ord_{\it z}\,\it f_{i},\,\,\,\,\,i= \rm 1,2,\ldots,\it r.
\end{eqnarray*}
On note $\mathfrak{D}[f]$ le diviseur $\sigma$-poly\'edral 
$\sum_{z\in C}\Delta_{z}[f]\,.\,z$.
\end{notation}
Dans le th\'eor\`eme suivant, nous d\'ecrivons la normalisation de 
l'alg\`ebre $A$ donn\'ee dans $3.4.1$.
\begin{theorem}
Soient $C$ une courbe alg\'ebrique lisse et
\begin{eqnarray*}
A = \mathbf{k}[C][f_{1}\chi^{m_{1}},\ldots,f_{r}\chi^{m_{r}}] \subset 
\mathbf{k}(C)[M]
\end{eqnarray*}
une sous-alg\`ebre $\mathbb{T}$-stable engendr\'ee par des \'el\'ements homog\`enes 
$f_{1}\chi^{m_{1}},\dots,f_{r}\chi^{m_{r}}$. 
Supposons que $A$ a le m\^eme corps des fractions que $\mathbf{k}(C)[M]$ 
et soit $\bar{A}$ la normalisation de $A$, 
vue comme sous-alg\`ebre de $\mathbf{k}(C)[M]$. Consid\'erons $\mathfrak{D}[f]$ d\'efini
comme dans $3.4.3$. Alors 
le dual $\sigma$ du c\^one poly\'edral 
$\sum_{i = 1}^{r}\mathbb{Q}_{\geq 0}m_{i}$ est saillant et $\mathfrak{D}[f]$ est 
l'unique diviseur $\sigma$-poly\'edral propre sur $C$ v\'erifiant $\bar{A} = A[C,\mathfrak{D}[f]]$.   
\end{theorem}
\begin{proof} 
D'apr\`es [HS, Theorem 2.3.2], la sous-alg\`ebre 
$\bar{A}\subset \mathbf{k}(C)[M]$ est 
$M$-gradu\'ee. Puisque $A$ a le m\^eme corps des fractions
que $\mathbf{k}(C)[M]$, le c\^one $\sigma$ est saillant. Par le th\'eor\`eme $3.3.7$, 
il existe un unique diviseur $\sigma$-poly\'edral propre $\mathfrak{D}$ sur $C$ tel que 
$\bar{A} = A[C,\mathfrak{D}]$.
Posons $B := A[C,\mathfrak{D}[f]]$. Alors pour tout $i\in \rm\{1\it ,\ldots,r\rm\}$ et tout $z\in C$,
on a l'in\'egalit\'e
\begin{eqnarray*}
h_{\Delta_{z}[f]}(m_{i})\geq -\rm ord_{\it z}\it\,f_{i},
\end{eqnarray*}
de sorte que pour tout $i\in \rm\{1\it ,\ldots,r\rm\}$,
\begin{eqnarray*}
\mathfrak{D}[f](m_{i}) = \sum_{z\in C}h_{\Delta_{z}[f]}(m_{i})\,.\,z\geq -\rm div(\it f_{i}\rm ).
\end{eqnarray*}
On obtient l'inclusion $A\subset B$. Comme $B$ est l'intersection d'anneaux de valuations discr\`etes
de $\rm Frac\,\it B/\mathbf{k}$ [De $2$, $\S 2.7$], l'alg\`ebre $B$ est normale (voir aussi l'argument
de d\'emonstration de $4.4.4\,\rm (i)$). D'o\`u $\bar{A} = A[\mathfrak{D}]\subset B$. 

Montrons l'\'egalit\'e $\mathfrak{D} = \mathfrak{D}[f]$. Supposons que $C$ est une 
courbe alg\'ebrique projective lisse de genre $g$. 
Pour tout $m'\in\sigma^{\vee}_{M}$, on a 
\begin{eqnarray}
H^{0}(C,\mathcal{O}_{C}(\lfloor \mathfrak{D}(m')\rfloor ))\subset H^{0}(C,\mathcal{O}_{C}(\lfloor \mathfrak{D}[f](m')\rfloor )).
\end{eqnarray}
Soit $m\in\sigma^{\vee}_{M}$ appartenant \`a l'int\'erieur relatif de $\sigma^{\vee}$. 
Prenons $s\in\mathbb{Z}_{>0}$ tels que $s\mathfrak{D}(m)$ et $s\mathfrak{D}[f](m)$
soient des diviseurs de Cartier \`a coefficients entiers. Alors par $(3.5)$, par le th\'eor\`eme 
de Riemann-Roch et puisque $\mathfrak{D}$ est propre, il existe
$r_{0}\in\mathbb{Z}_{>0}$ tel que pour tout entier $r\geq r_{0}$,
\begin{eqnarray*}
h^{0}(C,\mathcal{O}_{C}(rs\mathfrak{D}[f](m)))\geq r\rm deg(\it s\mathfrak{D}(m)\rm) + 1 - \it g \rm > 1.
\end{eqnarray*}
D'apr\`es [Ha, IV, Lemma 1.2], pour tout entier $r\geq r_{0}$, $\rm deg(\it rs\mathfrak{D}[f](m)\rm) > 0$ et donc $\mathfrak{D}[f](m)$ est semi-ample. Donc par 
[AH, Lemma 9.1], on a $\mathfrak{D}[f](m)\geq \mathfrak{D}(m)$. 
Cette in\'egalit\'e est aussi vraie lorsque $C$ est une courbe alg\'ebrique affine lisse. 
Revenons \`a l'hypoth\`ese o\`u $C$ est une courbe alg\'ebrique lisse quelconque. 
Les in\'egalit\'es 
\begin{eqnarray*}
\mathfrak{D}(m_{i}) + \rm div(\it f_{i}\rm )\geq 0,\,\,\,\it i\rm  = 1\it ,\ldots,r,
\end{eqnarray*}
entra\^inent que pour tout $z\in C$, $\Delta_{z}\subset \Delta_{z}[f]$. 
Donc pour tout $m\in\sigma^{\vee}$, $\mathfrak{D}(m)\geq\mathfrak{D}[f](m)$. 
Par le lemme $3.4.2$, on conclut que $\mathfrak{D} = \mathfrak{D}[f]$.
\end{proof}
L'exemple qui suit illustre comment on peut \'etablir la normalit\'e d'une alg\`ebre donn\'ee
\`a partir d'un syst\`eme de g\'en\'erateurs.
\begin{exemple}
Soit $z$ une variable et soit $\mathbb{T}$ le tore $(\mathbf{k}^{\star})^{2}$. Consid\'erons 
la sous-alg\`ebre $\mathbb{T}$-stable 
\begin{eqnarray*}
A = \mathbf{k}\left[\frac{z}{z-1}\chi^{(2,0)},\chi^{(0,1)},z\chi^{(2,2)}, 
\frac{z^{2}}{z-1}\chi^{(3,2)}\right]\subset \mathbf{k}(z)[\mathbb{Z}^{2}].
\end{eqnarray*}
Alors on a les \'egalit\'es $A^{\mathbb{T}} = \mathbf{k}$ et 
$(\rm Frac\it\,A)^{\mathbb{T}} = \mathbf{k}(z)$. Le c\^one des poids $\sigma^{\vee}\subset\mathbb{Q}^{2}$ 
de $A$ est le quadrant positif. Par le th\'eor\`eme $3.4.4$, la normalisation $\bar{A}$ de $A$ est \'egale \`a 
$A[\mathbb{P}^{1},\mathfrak{D}]$ o\`u $\mathfrak{D} = \Delta_{0}\cdot 0 + \Delta_{1}\cdot 1 + \Delta_{\infty}\cdot \infty$ est 
le diviseur $\sigma$-poly\'edral propre sur $\mathbb{P}^{1}$ d\'efini par 
\begin{eqnarray*}
\Delta_{0} = -\left(\frac{1}{2},0\right)+\sigma,\,\,
\Delta_{1} = \left(\frac{1}{2},0\right)+\sigma,\,\,
\Delta_{\infty} = \left[\left(\frac{1}{2},0),(0,\frac{1}{2}\right)\right]+\sigma.
\end{eqnarray*}
Par exemple, le poly\`edre $\Delta_{0}$ est donn\'e par les in\'egalit\'es
\begin{eqnarray*}
2x\geq -\rm ord_{0}\it\, \frac{z}{z-\rm 1\it }  =\rm -1,\,\,\, \it y\geq -\rm ord_{0}\,1 = 0,\,\,\,\it 
\rm 2\it x+ \rm 2\it y\geq -\rm ord_{0}\it\,z =\rm -1, 
\end{eqnarray*}
\begin{eqnarray*}
3x +2y \geq -\rm ord_{0}\,\it \frac{z^{\rm 2\it }}{z- \rm 1\it} \rm= -2\,.
\end{eqnarray*}
Un calcul direct montre qu'en fait $A = A[\mathbb{P}^{1},\mathfrak{D}]$. Si l'on pose 
\begin{eqnarray*}
t_{1}:= \frac{z}{z-1}\chi^{(2,0)},\,\,\,t_{2} := \chi^{(0,1)},\,\,\, t_{3} := z\chi^{(2,2)},\,\,\,
t_{4} := \frac{z^{2}}{z-1}\chi^{(3,2)},
\end{eqnarray*}
alors les fonctions $t_{1},t_{2},t_{3},t_{4}$ v\'erifient la relation irr\'eductible
$t_{4}^{2} - t_{1}^{2}t_{2}^{2}t_{3} - t_{1}t_{3}^{2} = 0$. On conclut que 
l'hypersurface $V(x_{4}^{2} - x_{1}^{2}x_{2}^{2}x_{3} - x_{1}x_{3}^{2})\subset \mathbb{A}^{4}$
identifi\'ee \`a $\rm Spec\,\it A$ est normale.
\end{exemple}
Rappelons que pour une vari\'et\'e alg\'ebrique affine $X$, on note $\rm SAut\,\it X$ le sous-groupe des automorphismes 
de $X$ engendr\'e par la r\'eunion des images des op\'erations alg\'ebriques de $\mathbb{G}_{a}$ dans $X$.
Ici une op\'eration de $\mathbb{G}_{a}$ dans $X$ est vue comme un morphisme de groupes de $\mathbb{G}_{a}$
vers $\rm Aut\,\it X$. 
Pour plus d'informations voir [AKZ, AFKKZ] o\`u la notion de vari\'et\'e flexible est introduite. 
L'exemple suivant est inspir\'e de 
[LS, Example 1.1]. Il correspond au cas de $d = 3$ et de $e = 2$.
\begin{exemple}
Soient $d,e\geq 2$ des entiers tels que $d + 1 = e^{2}$. Notons
\begin{eqnarray*}
\mathfrak{D}_{d,e} = \Delta_{0}^{d,e}\cdot 0  + \Delta_{1}^{d,e}\cdot 1 + \Delta_{\infty}^{d,e}\cdot \infty 
\end{eqnarray*}
le diviseur poly\'edral propre sur $\mathbb{P}^{1}$ avec 
\begin{eqnarray*}
\Delta_{0}^{d,e} = [(1,0),(1,1)] + \sigma_{de},\,\, \Delta_{1}^{d,e} = \left(-\frac{1}{e},0\right) + \sigma_{de},\,\,
\Delta_{\infty}^{d,e} = \left(\frac{e}{d} - 1,0\right) + \sigma_{de},
\end{eqnarray*}
o\`u $\sigma_{de}$ est le c\^one de $\mathbb{Q}^{2}$ engendr\'e par les vecteurs $(1,0)$ et $(1,de)$, 
et $X_{d,e} := X[\mathfrak{D}_{d,e}]$. Alors $X_{d,e}$ n'est pas torique. 
En effet, $X_{d,e}$ n'est pas $\mathbb{T}$-isomorphe \`a $X[\mathbb{P}^{1},\mathfrak{D}]$ 
avec $\mathfrak{D}$ un diviseur poly\'edral propre 
sur $\mathbb{P}^{1}$ port\'e par au plus $2$ points, i.e. $\rm Card\,Supp\,\it\mathfrak{D}\rm\leq 2$ 
(voir [AL, Corollary 1.4] et le th\'eor\`eme $3.3.9$). 

Montrons que le groupe $\rm SAut\,\it X_{d,e}$ 
op\`ere infiniment transitivement dans le lieu lisse de $X_{d,e}$. Pour cela consid\'erons la
sous-alg\`ebre
\begin{eqnarray*}
A_{d,e} :=  \mathbf{k}\left[\chi^{(0,1)},\,\frac{(1-z)^{d}}{z^{de - 1}}\chi^{(de,-1)},\,\frac{1-z}{z^{e}}\chi^{(e,0)},\, 
\frac{(1-z)^{e}}{z^{d}}\chi^{(d,0)}\right]\subset \mathbf{k}(z)[\mathbb{Z}^{2}] 
\end{eqnarray*}
o\`u $\mathbb{T}$ est le tore $(\mathbf{k}^{\star})^{2}$ et $z$ est une variable sur $\mathbf{k}(\mathbb{T})$. 
Posons 
\begin{eqnarray*}
 u := \chi^{(0,1)},\,\,\, v := \frac{(1-z)^{d}}{z^{de - 1}}\chi^{(de,-1)},\,\,\, x := \frac{1-z}{z^{e}}\chi^{(e,0)},\,\,\, 
y := \frac{(1-z)^{e}}{z^{d}}\chi^{(d,0)}.
\end{eqnarray*}
Alors les fonctions $u,v,x,y$ v\'erifient la relation irr\'eductible $uv + x^{d} - y^{e} = 0 \,\,\, (E)$,
identifiant $\rm Spec\,\it A_{d,e}$ avec l'hypersurface $H_{d,e}$ d'\'equation $(E)$ de $\mathbb{A}^{4}$. 
L'origine $O$ de $\mathbb{A}^{4}$ est l'unique point singulier de $H_{d,e}$. 

D'apr\`es le crit\`ere de normalit\'e de Serre [Vie, Proposition 2], $H_{d,e}$ est normale en $O$ 
et par [KZ, \S 5], [AKZ, Theorem 3.2], le groupe $\rm SAut\,\it H_{d,e}$ op\`ere
infiniment transitivement dans $H_{d,e}-\{O\}$ comme suspension\footnote{ 
Rappelons qu'une suspension d'une vari\'et\'e affine $V$ est une hypersurface de $\mathbb{A}^{2}\times V$
d'\'equation $uv - P(x) = 0$, o\`u $(u,v)$ est un syst\`eme de coordonn\'ees sur 
$\mathbb{A}^{2}$ et $P: V\rightarrow \mathbb{A}^{1}$ est une fonction r\'eguli\`ere.} du plan affine. Par ailleurs, 
on a $(\rm Frac\,\it A_{d,e})^{\mathbb{T}} = \mathbf{k}(z)$ et 
$\rm Frac\,\it A_{d,e} = \rm Frac\it\, \mathbf{k}(z)[\mathbb{Z}^{\rm 2}]$.
Un calcul facile montre que $\mathfrak{D}[u,v,x,y] = \mathfrak{D}_{d,e}$. On conclut par le th\'eor\`eme $3.4.4$.

Notons que le calcul de $\mathfrak{D}_{d,e}$ peut \^etre obtenue \`a partir de $H_{d,e}$ en utilisant [AH, \S 11].
\end{exemple}
Le th\'eor\`eme $3.4.4$ appliqu\'e \`a $\mathbb{T} = \mathbf{k}^{\star}$ donne le corollaire suivant. 
Les parties concernant les cas parabolique et hyperbolique
ont \'et\'e \'etablies dans [FZ, 3.9, 4.6]. Gr\^ace \`a ces r\'esultats, il est montr\'e dans [FZ]
que toute $\mathbb{C}^{\star}$-surface complexe normale affine parabolique ayant un bon quotient
isomorphe \`a $\mathbb{A}^{1}_{\mathbb{C}}$ est la normalisation d'une hypersurface de $\mathbb{A}^{3}_{\mathbb{C}}$
d'\'equation $x^{d} = yP(z)$, o\`u $P\in\mathbb{C}[z]$ est un polyn\^ome. Une description analogue est donn\'ee
dans le cas hyperbolique, c.f. [FZ, $4.8$]. 
\begin{corollaire}
Soit $C$ une courbe alg\'ebrique lisse, consid\'erons $\chi$ une variable sur $\mathbf{k}(C)$ et soit $A$ une sous-alg\`ebre
gradu\'ee de $\mathbf{k}(C)[\chi, \chi^{-1}]$ ayant le 
m\^eme corps des fractions
que $\mathbf{k}(C)[\chi, \chi^{-1}]$.
Alors les assertions suivantes sont vraies. 
\begin{enumerate}
\item[\rm (i)]\rm (cas elliptique et parabolique) \em
Si 
\begin{eqnarray*}
A = \mathbf{k}[C][f_{1}\chi^{m_{1}},\dots,f_{r}\chi^{m_{r}}]\subset \mathbf{k}(C)[\chi, \chi^{-1}],
\end{eqnarray*} 
avec $f_{i}\in \mathbf{k}(C)^{\star}$ et $m_{1},\ldots,m_{r}\in\mathbb{Z}_{>0}$, alors la pr\'esentation de
Dolgachev-Pinkham-Demazure (D.P.D.) de $\bar{A} = A_{C,D}$ est donn\'ee par le diviseur de Weil rationnel sur $C$
\begin{eqnarray*}
D = - \min_{1\leq i\leq r}\frac{\rm div(\it f_{i}\rm)}{\it m_{i}}\,;
\end{eqnarray*}
\item[\rm (ii)]\rm (cas hyperbolique) \em 
Si $C=\rm Spec\it\, A_{\rm 0}$ est une courbe affine lisse, 
\begin{eqnarray*}
A = \mathbf{k}[C][f_{1}\chi^{-m_{1}},\ldots,f_{r}\chi^{-m_{r}},g_{1}\chi^{n_{1}},\ldots,g_{s}\chi^{n_{s}} ]\subset 
\mathbf{k}(C)[\chi, \chi^{-1}],
\end{eqnarray*}
avec $n_{1},\ldots,n_{s}, m_{1},\ldots,m_{r}\in\mathbb{Z}_{>0}$, alors la pr\'esentation $D.P.D.$ 
de $\bar{A} = A_{0}[D_{-},D_{+}]$ est donn\'ee par les $\mathbb{Q}$-diviseurs 
\begin{eqnarray*}
D_{-} = - \min_{1\leq i\leq r}\frac{\rm div(\it f_{i}\rm)}{\it m_{i}}\,\,\,\rm et\,\,\,\it D_{+} = - \min_{\rm 1\it\leq i\rm\leq\it s}\frac{\rm div(\it g_{i}\rm)}{\it n_{i}}\,. 
\end{eqnarray*}
\end{enumerate}
\end{corollaire}
\begin{proof}
Nous donnons la d\'emonstration pour le cas de $\rm (i)$. L'assertion $\rm (ii)$ se traite de la
m\^eme fa\c con. D'apr\`es le th\'eor\`eme $3.4.4$, on peut \'ecrire $\bar{A} = A[C,\mathfrak{D}]$ avec
$\mathfrak{D} = \sum_{z\in C}\Delta_{z}\cdot z$
o\`u pour $z\in C$,
\begin{eqnarray*}
\Delta_{z} = \left\{v\in\mathbb{Q},\,m_{i}\cdot v\geq -\rm ord_{\it z}\it\, f_{i},\rm \,1\leq \it i\rm \leq \it r\right\}. 
\end{eqnarray*}
Puisque $D = \mathfrak{D}(1)$, nous obtenons le r\'esultat.
\end{proof}

\section{Id\'eaux monomiaux int\'egralement clos et poly\`edres entiers} 
Dans cette section, nous rappelons une description des id\'eaux monomiaux int\'egralement
clos en terme de poly\`edres entiers. Nous commen\c cons par quelques notions g\'en\'erales.
 \begin{rappel}
Puisque nous restons dans un contexte g\'eom\'etrique, la lettre $A$ d\'esigne une alg\`ebre int\`egre de 
type fini sur $\mathbf{k}$. Soit $I\subset A$ un id\'eal. Un \'el\'ement $a\in A$ est dit \em entier \rm (ou satisfaisant
une relation de d\'ependance int\'egrale) sur $I$ s'il existe $r\in\mathbb{Z}_{>0}$ et des \'el\'ements 
\begin{eqnarray*}
\lambda_{1}\in I,\,\lambda_{2}\in I^{2},\ldots,\,\lambda_{r}\in I^{r}\,\,\,\,\,\rm tels\,\,que\,\,\,\,\it a^{r}+\sum_{i \rm = 1}^{r}
\lambda_{i}a^{r\rm -\it i} \rm = 0\,.
\end{eqnarray*}
L'id\'eal $I$ est dit \em int\'egralement clos \rm (ou complet) si tout entier de $A$ sur $I$ appartient \`a $I$. 
On dit que $I$ est  \em normal \rm si pour tout
entier $i\geq 1$, l'id\'eal $I^{i}$ est int\'egralement clos. 

Le sous-ensemble $\bar{I}\subset A$, appel\'e \em cl\^oture int\'egrale \rm (ou fermeture int\'egrale)
 de $I$ dans $A$, est l'ensemble des \'el\'ements entiers sur $I$. 
C'est le plus petit id\'eal int\'egralement clos de $A$ contenant $I$ [HS, Corollary 1.3.1].
La d\'emonstration de ce fait utilise les r\'eductions d'id\'eaux (voir [NR]). 
Pour plus d'informations sur les 
id\'eaux int\'egralement clos,
voir par exemple [LeTe, HS, Va].
Consid\'erons l'alg\`ebre de Rees 
\begin{eqnarray*}
B = A[It] = A\oplus\bigoplus_{i\geq 1}I^{i}t^{i}\subset A[t]
\end{eqnarray*}
correspondante \`a un id\'eal $I$ de $A$. Alors la normalisation de $B$ est  
\begin{eqnarray}
 \bar{B} = \bar{A}\oplus\bigoplus_{i\geq 1}\overline{\bar{A}\,\, I^{i}}t^{i}\subset A[t]
\end{eqnarray}
o\`u $\bar{A}$ est la normalisation de $A$ (voir [Ri]). La normalisation de $B$ 
est \'egale \`a celle de $\bar{A}[\bar{A}It]$ [HS, Proposition 5.2.4]. 
Si $A$ est normale alors $A[It]$ est normale si et seulement si $I$ est normal. 

Supposons maintenant que $A$ est 
$M$-gradu\'ee normale. Un id\'eal $I$ de $A$ est dit \em homog\`ene \rm si $I$ est non nul et si $I$ est engendr\'e par des
\'el\'ements homog\`enes de $A$. Chaque id\'eal homog\`ene de $A$ est un sous-$\mathbb{T}$-module rationnel et
admet une d\'ecomposition en somme directe de sous-espaces propres. D'apr\`es [HS, Corollary 5.2.3], 
si $I$ est homog\`ene alors $\bar{I}$ est homog\`ene.
\end{rappel}
En l'absence de r\'ef\'erence, nous donnons la d\'emonstration du lemme suivant.
\begin{lemme}
Soit $A$ une alg\`ebre normale de type fini sur $\mathbf{k}$ et soit $I$ un id\'eal de $A$. Alors 
la fermeture int\'egrale de $A[It]$ dans son corps des fractions est
\'egale \`a celle de $A[\bar{I}t]$. En particulier, pour tout $i\in\mathbb{Z}_{>0}$, 
$\overline{I^{i}} = \overline{\bar{I}^{i}}$.  
\end{lemme}
\begin{proof} 
Pour tout $i\in\mathbb{Z}_{>0}$, on a  $I^{i}\subset \bar{I}^{i}$, soit $A[It]\subset A[\bar{I}t]$.
D'o\`u $\overline{A[It]}\subset \overline{A[\bar{I}t]}$. Par $(3.6)$ et puisque $A = \bar{A}$ est normale, 
$\bar{I}t$ est inclus dans $\overline{A[It]}$, donc 
$A[\bar{I}t]\subset\overline{A[It]}$ et finalement  $\overline{A[\bar{I}t]}\subset\overline{A[It]}$. La deuxi\`eme affirmation
est une cons\'equence de $(3.6)$ et de l'\'egalit\'e $\overline{A[It]} = \overline{A[\bar{I}t]}$.
\end{proof}  
\begin{rappel} 
Fixons un c\^one poly\'edral saillant $\sigma\subset N_{\mathbb{Q}}$. 
Consid\'erons une partie $F\subset M_{\mathbb{Q}}$ telle que $F\cap\sigma^{\vee}$ est non vide.
On note $\rm Pol_{\sigma^{\vee}}(\it F)$ l'ensemble
des poly\`edres de la forme $P = Q + \sigma^{\vee}$ avec $Q$ un polytope dont
les sommets appartiennent \`a $F$. Un \em $\sigma^{\vee}$-poly\`edre entier \rm
est un \'el\'ement de $\rm Pol_{\sigma^{\vee}}(\it M)$.

Si $P\in\rm Pol_{\sigma^{\vee}}(\it M)$ alors 
pour tout entier $e\geq 1$, on pose $eP := P+\ldots +P$ la somme de Minkowski de $e$ exemplaires de $P$. 
Le poly\`edre $eP$ est l'image de $P$ par l'homoth\'etie de centre $0$ et de rapport $e$. Si $e = 0$ 
alors on pose $eP = \sigma^{\vee}$. On dit que $P$ est \em normal \rm si pour tout entier $e\geq 1$,
\begin{eqnarray*}
(eP)\cap M = \{m_{1}+\ldots+m_{e}\,|\,m_{1},\ldots,m_{e}\in P\cap M\}\,.        
\end{eqnarray*}
Un sous-mono\"ide $E\subset M$ est dit \em satur\'e \rm si pour
$d\in\mathbb{Z}_{>0}$ et pour $m\in M$ tels que $dm\in E$, on a $m\in E$.
Cela est \'equivalent \`a ce que $E$ soit l'intersection de $\rm Cone(\it E)$ et du r\'eseau $M$.
\end{rappel}
L'assertion suivante est ais\'ee et nous laissons la d\'emonstration aux lecteurs.
\begin{lemme}
Pour un c\^one poly\'edral saillant $\sigma\subset N_{\mathbb{Q}}$ et un poly\`edre entier $P\in\rm Pol_{\sigma^{\vee}}(\it M)$, on note 
\begin{eqnarray*}
S = S_{P} := \{(m,e)\in M\times\mathbb{N}\,|\,m\in (eP)\cap M\}\,.
\end{eqnarray*}
Alors $S$ est un sous-mono\"ide satur\'e de $M\times\mathbb{Z}$. De plus, pour tout $e\geq 1$, l'enveloppe convexe de  
\begin{eqnarray*} 
E_{[e,P]} := \{m_{1}+\ldots +m_{e}\,|\,m_{1},\ldots,m_{e}\in P\cap M\}
\end{eqnarray*}
dans $M_{\mathbb{Q}}$ est \'egale \`a $eP$.
\end{lemme}  

Rappelons que tout id\'eal homog\`ene int\'egralement clos d'une vari\'et\'e torique affine est caract\'eris\'e 
par l'enveloppe convexe de l'ensemble de ses poids. Voir [Vit, $3.1$] pour le cas de l'alg\`ebre 
des polyn\^omes \`a plusieurs variables. Le r\'esultat suivant est connu (voir 
[CLS, Proposition $11.3.4$] pour le cas o\`u $\sigma^{\vee}$ est saillant, ainsi que [KKMS, Chapter I, \S$2$] pour
le cas g\'en\'eral). Par commodit\'e,
nous donnons une courte preuve. Cela nous sera utile pour la suite. 
Pour tout entier $e$, nous noterons $M_{e} = M\times\{e\}$.

\begin{theorem}
Soit $\sigma\subset N_{\mathbb{Q}}$ un c\^one poly\'edral saillant. Alors l'application 
\begin{eqnarray*}
P\mapsto I[P] = \bigoplus_{m\in P\cap M}\mathbf{k}\,\chi^{m}
\end{eqnarray*}
est une bijection entre $\rm Pol_{\sigma^{\vee}}(\it \sigma^{\vee}_{M})$ 
et l'ensemble des id\'eaux homog\`enes int\'egralement clos de $A:= \mathbf{k}[\sigma^{\vee}_{M}]$. 
Plus pr\'ecis\'ement, 
si $I = (\chi^{m_{1}},\ldots,\chi^{m_{r}})$ est un id\'eal homog\`ene de $A$ 
alors sa fermeture $\bar{I}$ est $I[P]$ o\`u 
\begin{eqnarray}
P = \rm Conv(\it m_{\rm 1},\ldots,m_{r}\rm)+\sigma^{\vee}.
\end{eqnarray}
Le c\^one des poids $\widetilde{\omega}$ de $A[It]$ est
l'unique c\^one v\'erifiant $\widetilde{\omega}\cap M_{e}
 = (eP)\cap M$, pour tout $e\in\mathbb{N}$.
De plus, si $P\in\rm Pol_{\sigma^{\vee}}(\it \sigma^{\vee}_{M})$
alors $P$ est normal si et seulement si $I[P]$ est normal.
\end{theorem}
\begin{proof}
La justification de la bijectivit\'e de $I\mapsto I[P]$ provient de
[KKMS, Chapter I, \S$2$] et nous l'omettons. 

Soit $I\subset A$
un id\'eal engendr\'e par $\chi^{m_{1}},\ldots,\chi^{m_{r}}$. 
Consid\'erons $P$ comme dans $(3.7)$. Montrons
l'\'egalit\'e $\bar{I} = I[P]$. 
Puisque $I\subset I[P]$ et comme $I[P]$ est int\'egralement clos, on
a $\bar{I}\subset I[P]$. Soit $\chi^{m}\in I[P]$. Alors $m\in P\cap M$.
Donc il existe $c_{1},\ldots, c_{r}\in \mathbb{Q}_{\geq 0}$ de somme
\'egale \`a $1$ et $c\in\sigma^{\vee}$ tels que 
\begin{eqnarray*}
 m = \sum_{i = 1}^{r}c_{i}m_{i} + c. 
\end{eqnarray*}
Soit $d'\in\mathbb{Z}_{>0}$ tel que $d'c_{i}\in M$ pour chaque $i$ et
tel que $dc\in\sigma^{\vee}_{M}$. Alors $\chi^{d'm}\in I^{d'}$ et donc
$\chi^{m}\in \bar{I}$. D'o\`u $\bar{I} = I[P]$ et la surjectivit\'e de $\varphi$.

D\'eterminons le c\^one des poids $\widetilde{\omega}$ de $A[It]$. D'apr\`es le lemme $3.5.2$,
on peut supposer que $I = I[P]$.
Soit $e\in\mathbb{Z}_{>0}$. Alors on a 
\begin{eqnarray*}
I[P]^{e} = \bigoplus_{m\in E[e,P]}\mathbf{k}\chi^{m}. 
\end{eqnarray*}
Notons $\overline{I[P]^{e}} = I[P_{e}]$, pour un poly\`edre 
$P_{e}\in\rm Pol_{\sigma^{\vee}}(\it \sigma^{\vee}_{M})$. Par le lemme $3.5.4$, 
on a d'une part,
\begin{eqnarray*}
eP = \rm Conv(\it E_{[e,P]}) \subset P_{e}. 
\end{eqnarray*}
D'autre part, $I[P]^{e}\subset I[eP]$ et donc $I[P_{e}]\subset I[eP]$.
D'o\`u $eP = P_{e}$. D'apr\`es l'\'egalit\'e $(3.6)$ de $3.5.1$, $\widetilde{\omega}$ est le c\^one des
poids de
\begin{eqnarray*}
\overline{A[It]} = A\oplus\bigoplus_{e\in\mathbb{Z}_{>0}}I[eP]t^{e}. 
\end{eqnarray*} 
D'o\`u $\widetilde{\omega}\cap M_{e} = (eP)\cap M$, pour tout entier $e\geq 0$.
Les autres assertions suivent ais\'ement.
\end{proof}

\paragraph{}
Le th\'eor\`eme suivant est une cons\'equence d'un r\'esultat
sur la normalit\'e des polytopes entiers [BGN, Theorem $1.1.3$]. 
\begin{theorem}
Soit $\sigma\subset N_{\mathbb{Q}}$ un c\^one poly\'edral saillant, notons 
$n = \rm dim\,\it N_{\mathbb{Q}}\rm $ et consid\'erons $P\in\rm Pol_{\sigma^{\vee}}(\it M)$. 
Alors pour tout entier $e\geq n-1$, le poly\`edre $eP$ est normal. 
En particulier, tout id\'eal stable int\'egralement clos d'une surface torique 
est normal.    
\end{theorem}
\begin{proof}
Soit $\mathscr{E}$ une subdivision de $\sigma^{\vee}$ par des c\^ones poly\'edraux saillants 
de dimension $n$. \'Ecrivons $P = Q + \sigma^{\vee}$ avec $Q$ un polytope entier. 
Soit $m\in (eP)\cap M$. Alors il existe $\tau\in\mathscr{E}$ tel que $m\in (eQ + \tau)\cap M$.
D'apr\`es [CLS, Proposition $7.1.9$], on a $m\in E_{[e,Q+\tau]}$ (voir $3.5.4$). Donc $m\in E_{[e,P]}$ et
$eP$ est normal. Le reste provient du th\'eor\`eme $3.5.5$.  
\end{proof}
 
\begin{remarque}
Notons que dans le cas o\`u $\mathbf{k}[\sigma^{\vee}_{M}] = 
\mathbf{k}[x_{1},\dots,x_{n}]$ est l'alg\`ebre des polyn\^omes, un id\'eal monomial 
$I\subset\mathbf{k}[x_{1},\ldots,x_{n}]$ est normal si et seulement si pour tout $i= 1,\ldots,n-1$, 
l'id\'eal $I^{i}$ est int\'egralement clos [RRV, $3.1$]. Nous verrons une g\'en\'eralisation de
ce r\'esultat dans la section $3.7$.

D'apr\`es [ZS, Appendix $5$], [HS, $\S 1.1.4$, $\S 14.4.4$], 
tout id\'eal int\'egralement clos d'une surface affine lisse est normal. Cependant cela ne s'applique 
pas aux vari\'et\'es toriques affines de dimension $3$. 
En effet, d'apr\`es [HS, Exercice 1.14], 
si $\sigma\subset\mathbb{Q}^{3}$ est l'octant positif et si
\begin{eqnarray*}
P := \rm Conv((2,0,0),\, (0,3,0),\,(0,0,7)) + \it\sigma^{\vee}  
\end{eqnarray*} 
alors $I[P]$ n'est pas normal.
\end{remarque}

\section{Id\'eaux homog\`enes int\'egralement clos et $\mathbb{T}$-vari\'et\'es affines de complexit\'e un} 
Dans cette section, $C$ d\'esigne une courbe alg\'ebrique lisse, le sous-ensemble 
$\sigma\subset N_{\mathbb{Q}}$ est un c\^one poly\'edral saillant et $\mathfrak{D} = \sum_{z\in C}\Delta_{z}\cdot z$
 est un diviseur $\sigma$-poly\'edral propre. Posons $A := A[C, \mathfrak{D}]$. 
Nous allons \'etudier
les id\'eaux homog\`enes int\'egralement clos
de $A$. Nous commen\c cons par d\'ecrire le c\^one des poids de l'alg\`ebre de 
Rees de $A$ associ\'ee \`a un id\'eal homog\`ene.
Rappelons que pour un entier $e$, $M_{e}$ d\'esigne $M\times\{e\}$.
\begin{lemme}
Soit $I\subset A$ un id\'eal engendr\'e 
par des \'el\'ements homog\`enes $f_{1}\chi^{m_{1}},\ldots, f_{r}\chi^{m_{r}}$. Notons
\begin{eqnarray*}
P = \rm Conv(\it m_{\rm 1\it},\ldots m_{r})+\sigma^{\vee}.
\end{eqnarray*}
Alors le c\^one des poids $\widetilde{\omega}$ 
de l'alg\`ebre $(M\times\mathbb{Z})$-gradu\'ee $A[It]$ v\'erifie $\widetilde{\omega}\cap M_{e} = 
(eP)\cap M$, pour tout $e\in\mathbb{N}$.
\end{lemme}
\begin{proof}
Soit $E$ l'ensemble des
poids de $I$. La partie
$J:=\bigoplus_{m\in E}\mathbf{k}\chi^{m}$ 
est un id\'eal de $A':=\mathbf{k}[S]$.
Par le lemme $3.5.2$, on a
\begin{eqnarray*}
 \overline{A'[Jt]} = \overline{B[\overline{JB}t]}\rm\,\,\, avec\,\,\, \it B:=\mathbf{k}[\sigma^{\vee}_{M}].
\end{eqnarray*}
Les \'el\'ements 
$\chi^{m_{1}},\ldots, \chi^{m_{r}}$ engendrent l'id\'eal $JB$. 
Puisque $\widetilde{\omega}$ est le c\^one des poids de $\overline{A'[Jt]}$,
le reste provient du th\'eor\`eme $3.5.5$. 
\end{proof}
L'assertion suivante permet de calculer explicitement la fermeture int\'egrale des
id\'eaux homog\`enes de l'alg\`ebre $A$.
\begin{theorem}
Soit $I$ un id\'eal de $A = A[C,\mathfrak{D}]$ engendr\'e par des \'el\'ements homog\`enes
$f_{1}\chi^{m_{1}},\ldots,f_{r}\chi^{m_{r}}$. 
Alors on a 
\begin{eqnarray*}
 A = \bigoplus_{m\in\sigma^{\vee}_{M}}
H^{\rm 0\it }(C,\mathcal{O}_{C}(\lfloor\widetilde{\mathfrak{D}}(m,\rm 0\it)\rfloor ))\chi^{m}
\end{eqnarray*}
et pour tout entier $e\geq 1$,
\begin{eqnarray*}
\overline{I^{e}} =  \bigoplus_{m\in (eP)\cap M}
H^{0}(C,\mathcal{O}_{C}(\lfloor\widetilde{\mathfrak{D}}(m,e)\rfloor ))\chi^{m}
\end{eqnarray*} 
o\`u $(P,\widetilde{\mathfrak{D}})$ v\'erifie les conditions suivantes.
\begin{enumerate}
 \item[\rm (i)] $P$ est le poly\`edre $\rm Conv(\it m_{\rm 1},\ldots,m_{r}\rm)+\it\sigma^{\vee}$;
\item[\rm (ii)] $\widetilde{\mathfrak{D}}$ est le diviseur poly\'edral sur $C$
dont ses coefficients sont d\'efinis par
\begin{eqnarray*}
\widetilde{\Delta}_{z} = \bigcap_{i = 1}^{r}\left\{\,(v,p)\in N_{\mathbb{Q}}\times\mathbb{Q},\,
m_{i}(v) + p \geq -\rm ord_{\it z}\it\, f_{i}\,\right\}\cap\left(\rm\Delta_{\it z}\it\times\mathbb{Q}\right), 
\end{eqnarray*}
pour tout $z\in C$.
\end{enumerate}
\end{theorem}
\begin{proof}
On applique le th\'eor\`eme  $3.4.4$ \`a l'alg\`ebre $A[It]$
en utilisant les \'el\'ements $f_{i}\chi^{m_{i}}$ et 
des g\'en\'erateurs homog\`enes de $A$. Le diviseur
poly\'edral correspondant $\widetilde{\mathfrak{D}}$ v\'erifie
l'assertion $\rm (ii)$ ci-dessus. On conclut par le lemme $3.6.1$. 
\end{proof}

\begin{exemple}
Reprenons l'exemple $3.4.5$. On consid\`ere l'id\'eal homog\`ene
\begin{eqnarray*}
I = (t_{2},t_{3},t_{4})\subset A = k[t_{1},t_{2},t_{3},t_{4}]\cong 
\frac{k[x_{1},x_{2},x_{3},x_{4}]}{(x_{4}^{2}-x_{1}^{2}x_{2}^{2}x_{3} - x_{1}x_{3}^{2})}\,. 
\end{eqnarray*}
Soit $\widetilde{\omega}\subset \mathbb{Q}^{3}$ le c\^one v\'erifiant 
$\widetilde{\omega}\cap \mathbb{Z}^{2}_{e} = (0,e) + \mathbb{Q}_{\geq 0}^{2}$,
pour tout entier $e\geq 0$. Soit $\widetilde{\mathfrak{D}}$ le diviseur $\widetilde{\omega}^{\vee}$-poly\'edral sur
$\mathbb{P}^{1}$ construit par les g\'en\'erateurs $t_{2},t_{3},t_{4}$. 
Les coefficients non triviaux de $\widetilde{\mathfrak{D}}$ sont
\begin{eqnarray*}
 \widetilde{\Delta}_{0} = -\left(\frac{1}{2},0,0\right) + \widetilde{\omega}^{\vee},\,\, 
\widetilde{\Delta}_{1} = \left(\frac{1}{2},0,0\right) + \widetilde{\omega}^{\vee},\,\,
\end{eqnarray*}
\begin{eqnarray*}
\widetilde{\Delta}_{\infty} = \rm Conv\left( (0,1,-1),
\left(\frac{1}{2},0,0\right),\left(0,\frac{1}{2},0\right)\right) + \it \widetilde{\omega}^{\vee}.
\end{eqnarray*}
\end{exemple}
Fixons  un poly\`edre  $P\in\rm Pol_{\it\sigma^{\vee}}(\it \sigma^{\vee}_{M})$.
Consid\'erons le c\^one $\widetilde{\omega}\subset M_{\mathbb{Q}}\times\mathbb{Q}$ 
satisfaisant la relation
\begin{eqnarray*}
\widetilde{\omega}\cap M_{e} = (eP)\cap M, 
\end{eqnarray*}
pour tout $e\in\mathbb{N}$. Dans les deux prochains lemmes, nous
\'etudions des diviseurs $\widetilde{\omega}^{\vee}$-poly\'edraux dont les
pi\`eces gradu\'ees correspondant \`a $M_{e}$ forment un id\'eal 
de $A$.

\begin{lemme}
Soit $\widetilde{\mathfrak{D}} = \sum_{z\in C}\widetilde{\Delta}_{z}\cdot z$ un 
diviseur $\widetilde{\omega}^{\vee}$-poly\'edral propre.
Alors les conditions suivantes sont \'equivalentes.
\begin{enumerate}
 \item[\rm (i)] 
Pour tout 
$z\in C$, le poly\`edre $\Delta_{z}$ est la projection orthogonale de 
$\widetilde{\Delta}_{z}$ parall\`element \`a $\mathbb{Q}\,(0_{M},1)$
et si $(v,p)$ est un sommet de $\widetilde{\Delta}_{z}$ alors $p\leq 0$;
\item[\rm (ii)] Pour tout entier $e\geq 0$, 
\begin{eqnarray*}
 I_{[e]} := \bigoplus_{m\in (eP)\cap M}
H^{ 0}(C,\mathcal{O}_{C}(\lfloor \widetilde{\mathfrak{D}}(m, e )\rfloor ))\chi^{m}
\end{eqnarray*}
est un id\'eal de $A$ et $I_{[0]} = A$.
\end{enumerate}
\end{lemme}
\begin{proof}
Soit $z\in C$. Notons $\Delta'_{z}$ la projection orthogonale de $\widetilde{\Delta}_{z}$
parall\`element \`a l'axe $\mathbb{Q}\,(0_{M},1)$. Montrons l'implication $(\rm i)\Rightarrow (\rm ii)$.
Puisque pour $m\in\sigma^{\vee}$,
$h_{\widetilde{\Delta}_{z}}(m,0) = h_{\Delta'_{z}}(m)$,
on a l'\'egalit\'e $I_{[\rm 0\it ]} = A$. Le lemme $3.3.10$ implique
que pour tout $e\in\mathbb{N}$, $I_{[e]}\subset A$. D'o\`u
l'assertion $(\rm ii)$.

R\'eciproquement, la propret\'e de $\mathfrak{D}$ et l'\'egalit\'e $A = I_{[0]}$ montre que
$\Delta_{z} = \Delta'_{z}$.
Soit $(v,p)$ un sommet de $\widetilde{\Delta}_{z}$. Il reste \`a montrer que $p\leq 0$.
La partie   
\begin{eqnarray*}
\lambda := \left\{\,(m,e)\in M_{\mathbb{Q}}\times\mathbb{Q},\,h_{\widetilde{\Delta}_{z}}(m,e) = m(v)+ep\,\right\} 
\end{eqnarray*}
est un c\^one d'int\'erieur non vide [AH, $\S 1$]. 
Donc l'ensemble $\lambda_{M\times\mathbb{Z}} = \lambda\cap (M\times\mathbb{Z})$ contient
un \'el\'ement $(m',e')$ dans l'int\'erieur relatif de $\lambda$ v\'erifiant $e'\geq 1$. Soit 
$v'\in \Delta_{z}$ tels que 
\begin{eqnarray*}
 m'(v') = h_{\Delta_{z}}(m').
\end{eqnarray*}
Par propret\'e de $\widetilde{\mathfrak{D}}$ et [AH, Lemma $9.1$], il s'ensuit que 
\begin{eqnarray*}
 m'(v) + e'p = h_{\widetilde{\Delta}_{z}}(m',e')\leq h_{\Delta_{z}}(m') = m'(v'). 
\end{eqnarray*}
Comme $v$ appartient \`a $\Delta_{z}$, on a $m'(v')\leq m'(v)$ et donc $p\leq 0$. D'o\`u le r\'esultat.
\end{proof}
\begin{lemme}
Soit $\widetilde{\mathfrak{D}}$ un diviseur $\widetilde{\omega}^{\vee}$-poly\'edral sur $C$.
Supposons que $\widetilde{\mathfrak{D}}$ v\'erifie la condition
$\rm (ii)$ du lemme $3.6.4$. Alors pour tout 
$e\in\mathbb{N}$, $I_{[e]}$ est un id\'eal homog\`ene int\'egralement clos.
\end{lemme}
\begin{proof}
Soit $e\in\mathbb{N}$. 
Il suffit de montrer que tout \'el\'ement homog\`ene 
de $\bar{I}_{[e]}$ appartient \`a $I_{[e]}$. 
Soit $a\in\bar{I}_{[e]}$ un \'el\'ement homog\`ene. 
L'\'el\'ement $a\chi^{(0,e)}$ appartient \`a la normalisation de
$A[C,\widetilde{\mathfrak{D}}]$. Puisque $A[C,\widetilde{\mathfrak{D}}]$ est normal [De $2$, $\S 2.7$], 
nous avons $a\in I_{[e]}$.  
\end{proof}
Le th\'eor\`eme suivant d\'ecrit les id\'eaux homog\`enes int\'egralement
clos en complexit\'e $1$. Soit $S$ le mono\"ide des poids de $A$.
Dans ce th\'eor\`eme, nous consid\'erons
des couples $(P,\widetilde{\mathfrak{D}})$ v\'erifiant les conditions suivantes.
\begin{enumerate}
\item[\rm (i)] $P$ est un poly\`edre de $\rm Pol_{\sigma^{\vee}}(\it S)$;
\item[\rm (ii)] $\widetilde{\mathfrak{D}} = \sum_{z\in C}\widetilde{\Delta}_{z}\cdot z$ 
est un diviseur $\widetilde{\omega}^{\vee}$-poly\'edral
o\`u  $\widetilde{\omega}\subset M_{\mathbb{Q}}\times\mathbb{Q}$ v\'erifie 
\begin{eqnarray*}
\widetilde{\omega}\cap M_{e} = (eP)\cap M, 
\end{eqnarray*}
pour tout entier $e\geq 0$;
\item[\rm (iii)] Pour tout 
$z\in C$, $\Delta_{z}$ est la projection orthogonale de $\widetilde{\Delta}_{z}$ parall\`element \`a 
$\mathbb{Q}\,(0_{M},1)$
et tout sommet $(v,p)\in\widetilde{\Delta}_{z}$ satisfait $p\leq 0$; 
\item[\rm (iv)]
Pour tout $z\in C$, $\widetilde{\Delta}_{z}$ est 
l'intersection dans $\Delta_{z}\times\mathbb{Q}$
d'un nombre fini de demi-espaces 
\begin{eqnarray*}
\Delta_{m,e} :=\left\{\, (v,p)\in N_{\mathbb{Q}}\times\mathbb{Q},\,\,m(v)+p\geq e\,\right\}  
\end{eqnarray*}
avec $e\in\mathbb{Z}$ et $m\in P\cap M$ tel que les sections globales du faisceau 
$\mathcal{O}_{C}(\lfloor\widetilde{\mathfrak{D}}(m,1)\rfloor)$ engendrent le germe
$\mathcal{O}_{C}(\lfloor\widetilde{\mathfrak{D}}(m,1)\rfloor)_{z}$.
\end{enumerate}

\begin{theorem}
Soit $\mathfrak{D} = \sum_{z\in C}\Delta_{z}\cdot z$ un diviseur $\sigma$-poly\'edral
propre et soit $A = A[C,\mathfrak{D}]$.
Alors il existe une bijection entre l'ensemble des id\'eaux homog\`enes int\'egralement clos de $A$
et l'ensemble des couples $(P,\widetilde{\mathfrak{D}})$ v\'erifiant $\rm (i) - \rm (iv)$.
Cette correspondance est donn\'ee par
\begin{eqnarray}
(P,\widetilde{\mathfrak{D}})\mapsto I = \bigoplus_{m\in P\cap M}
H^{ 0}(C,\mathcal{O}_{C}(\lfloor \widetilde{\mathfrak{D}}(m, 1 )\rfloor ))\chi^{m}. 
\end{eqnarray}
De plus, sous cette correspondance, l'alg\`ebre $A[C,\widetilde{\mathfrak{D}}]$ s'identifie 
\`a la normalisation de l'alg\`ebre de Rees $A[It]$.
\end{theorem}
\begin{proof}
Soit $(P,\widetilde{\mathfrak{D}})$ un couple v\'erifiant les conditions $\rm (i)-\rm (iv)$.
Montrons que $\widetilde{\mathfrak{D}}$ est propre. Soient
$z_{1},\ldots, z_{s}$ les points du support de $\widetilde{\mathfrak{D}}$.
D'apr\`es $\rm (iv)$, on a pour $j = 1,\ldots,s$,
\begin{eqnarray*}
\widetilde{\Delta}_{z_{j}} = (\Delta_{z}\times\mathbb{Q})\cap\bigcap_{i = 1}^{r_{j}}\Delta_{m_{ij},e_{ij}} 
\end{eqnarray*}
avec $e_{ij}\in\mathbb{Z}$ et $m_{ij}\in P\cap M$ tel qu'il existe $f_{ij}\chi^{(m_{ij},1)}\in 
A[C,\widetilde{\mathfrak{D}}]$ homog\`ene satisfaisant
 $\rm ord_{\it z_{j}}\it\, f_{ij} = h_{\widetilde{\rm \Delta\it }_{\it z_{j}}}(\it m_{ij},\rm 1)$.
Consid\'erons l'alg\`ebre $B$ engendr\'ee par $A$ et par les \'el\'ements
\begin{eqnarray}
 f_{ij}\chi^{(m_{ij},1)},\,\,1\leq j\leq s,\, 1\leq i\leq r_{j}. 
\end{eqnarray}
Par le th\'eor\`eme $3.6.2$, l'anneau $A[C,\widetilde{\mathfrak{D}}]$ est la normalisation de $B$.
Ce qui donne la propret\'e de $\widetilde{\mathfrak{D}}$.

Les lemmes $3.6.4$ et $3.6.5$
montrent que l'application $(3.8)$ est bien d\'efinie.
Ainsi, $A[C,\widetilde{\mathfrak{D}}] = \overline{A[It]}$ 
o\`u $t = \chi^{(0,1)}$. Par l'\'egalit\'e $(3.6)$ de $3.5.1$, 
on d\'eduit l'injectivit\'e.

Montrons la surjectivit\'e. Soit $I$ un id\'eal int\'egralement clos de $A$
engendr\'e par des \'el\'ements homog\`enes $f_{1}\chi^{m_{1}},\ldots, f_{r}\chi^{m_{r}}$.
Consid\'erons le couple $(P,\widetilde{\mathfrak{D}})$ obtenu \`a partir du th\'eor\`eme $3.6.2$. 
Des lemmes $3.6.1$, $3.6.4$ et $3.6.5$, 
on obtient $\rm (i)$, $\rm (ii)$, $\rm (iii)$. Montrons $\rm (iv)$. Pour un point $z\in C$,
\'ecrivons
\begin{eqnarray*}
\widetilde{\Delta}_{z} = (\Delta_{z}\times\mathbb{Q})\cap\bigcap_{i = 1}^{r}\Delta_{m_{i},e_{i}} 
\end{eqnarray*}
avec pour tout $i$, $e_{i}:=-\rm ord_{\it z}\it\, f_{i}$. Soit $E$ un sous-ensemble minimal de $\{1,\ldots, r\}$
tel que 
\begin{eqnarray*}
\widetilde{\Delta}_{z} = (\Delta_{z}\times\mathbb{Q})\cap \bigcap_{i\in E}\Delta_{m_{i},e_{i}}. 
\end{eqnarray*}
Fixons un \'el\'ement $i'\in E$. Alors on a
\begin{eqnarray}
\{(v,p)\in N_{\mathbb{Q}}\times\mathbb{Q}, m_{i'}(v)+ p = e_{i'}\}\cap\widetilde{\Delta}_{z}\neq\emptyset. 
\end{eqnarray}
Donc $h_{\widetilde{\Delta}_{z}}(m_{i'},1) = e_{i'} = -\rm ord_{\it z}\it\, f_{i'}$. D'o\`u 
la condition $\rm (iv)$.
\end{proof}
\paragraph{}
Lorsque $C$ est affine, tout fibr\'e en droite au-dessus de $C$ est globalement engendr\'e.
Ainsi, comme cons\'equence imm\'ediate, on a le r\'esultat suivant. 

\begin{corollaire}
Supposons que $C$ est affine.
Alors
il existe une bijection entre l'ensemble des id\'eaux homog\`enes int\'egralement clos de $A$
et l'ensemble des couples $(P,\widetilde{\mathfrak{D}})$ v\'erifiant $\rm (ii), (iii)$ du th\'eor\`eme
$3.6.6$ et les conditions  suivantes.
\begin{enumerate}
\item[\rm (i)'] $P$ est un poly\`edre de $\rm Pol_{\sigma^{\vee}}(\it \sigma^{\vee}_{M})$;
\item[\rm (iv)']
Pour tout $z\in C$, le poly\`edre $\widetilde{\Delta}_{z}$ est 
l'intersection dans $\Delta_{z}\times\mathbb{Q}$
d'un nombre fini de demi-espaces 
\begin{eqnarray*}
\Delta_{m,e} :=\left\{\, (v,p)\in N_{\mathbb{Q}}\times\mathbb{Q},\,\,m(v)+p\geq e\,\right\}  
\end{eqnarray*}
avec $e\in\mathbb{Z}$ et $m\in P\cap M$.
\end{enumerate}
La correspondance est donn\'ee par l'application $(3.8)$ de $3.6.6$.
\end{corollaire}
\begin{remarque}
Soient $g_{2},g_{3}\in\mathbf{k}$ tels que 
$g_{2}^{3}-27g_{3}^{2}\neq 0$.
Consid\'erons la courbe elliptique $C\subset \mathbb{P}^{2}$
d'\'equation de Weierstrass $y^{2} = 4x^{3}-g_{2}x-g_{3}$,
$O$ son point \`a l'infini et $A$
l'alg\`ebre gradu\'ee
\begin{eqnarray*}
\bigoplus_{m\in \mathbb{N}}H^{0}\left(C,\mathcal{O}_{C}\left(m\cdot O\right)\right)\chi^{m}. 
\end{eqnarray*}
Soit $\widetilde{\omega}\subset\mathbb{Q}^{2}$ le c\^one  
\begin{eqnarray*}
\mathbb{Q}_{\geq 0}(1,0)+\mathbb{Q}_{\geq 0}(1,1), 
\end{eqnarray*}
$P = \mathbb{Q}_{\geq 1}$
et $\widetilde{\mathfrak{D}} := \widetilde{\Delta}\cdot O$
le diviseur poly\'edral sur $C$ avec
$\widetilde{\Delta} = (1,0) + \widetilde{\omega}^{\vee}$. Alors
$(P,\widetilde{\mathfrak{D}})$ v\'erifie les conditions $\rm (i)$,
$\rm (ii)$, $\rm (iii)$ et $\rm (iv)'$ de $3.6.6$ et $3.6.7$. Cependant
la condition $\rm (iv)$ n'est pas satisfaite. L'alg\`ebre $A[C,\widetilde{\mathfrak{D}}]$
est donc distincte de la normalisation de $A[It]$ o\`u $I$ est l'id\'eal provenant
de $(3.8)$ et $t:=\chi^{(0,1)}$. Cela montre que l'enonc\'e du corollaire 3.6.7 ne se g\'en\'eralise pas
au cas o\`u $C$ est une courbe alg\'ebrique lisse, \'eventuellement projective.   

\end{remarque}

\section{Exemples d'id\'eaux homog\`enes normaux}
Dans cette section, nous consid\'erons l'alg\`ebre $A = A[C,\mathfrak{D}]$
comme dans la section $3.6$. Supposons que $C$ est affine et fixons 
un id\'eal homog\`ene int\'egralement clos $I\subset A$. 
Nous allons donner des conditions sur le couple 
$(P,\widetilde{\mathfrak{D}})$
associ\'e (voir $3.6.6$) pour que $I$ soit normal. Comme
d'habitude, nous notons $\widetilde{\Delta}_{z}$ le coefficient 
de $\widetilde{\mathfrak{D}}$ au point $z\in C$. 
\begin{notation} 
Pour tout $z\in C$, nous consid\'erons le sous-ensemble
\begin{eqnarray*}
\widetilde{P}_{z}: = \rm Conv\left(\it\left\{ (m,i)\in (P\cap M)\times \mathbb{Z}, \,\,
 h_{\widetilde{\rm \Delta\it}_{z}}
(m,\rm 1\it) \geq -i\right\}\right).
\end{eqnarray*}
On v\'erifie ais\'ement que $\widetilde{P}_{z}$ est un poly\`edre entier de $M_{\mathbb{Q}}\times\mathbb{Q}$.
Par ailleurs, pour un ouvert non vide $U\subset C$, 
$\widetilde{\mathfrak{D}}|U := \sum_{z\in U}\widetilde{\Delta}_{z}\cdot z$ est 
le diviseur poly\'edral sur $U$ obtenu par restriction de $\widetilde{\mathfrak{D}}$.
\end{notation}
L'assertion suivante est inspir\'ee de la description de $X[C,\widetilde{\mathfrak{D}}]$
comme vari\'et\'e toro\"idale (voir [LS,$\S 2.6$], [KKMS, Chapter $2$, $4$]). 

\begin{lemme}
Soit $z\in C$. Supposons que $C$ est affine et que $\mathfrak{D}$, 
$\widetilde{\mathfrak{D}}$ ont au plus le point $z$ dans leurs supports. 
Si le poly\`edre $\widetilde{P}_{z}$ est normal alors l'id\'eal
$I$ est normal. 
\end{lemme}
\begin{proof}
Fixons $e\in \mathbb{Z}_{>0}$. D\'eterminons un sous-ensemble de g\'en\'erateurs de 
$\overline{I^{e}}$. Pour tout vecteur $(m,i)\in M\times\mathbb{Z}$,
posons 
\begin{eqnarray*}
B_{(m,i)} := H^{0}(C,\mathcal{O}_{C}(-i\cdot z))\chi^{m}. 
\end{eqnarray*}
Si $m\in(eP)\cap M$ alors on a l'\'egalit\'e
\begin{eqnarray*}
H^{0}(C,\mathcal{O}_{C}(\lfloor \widetilde{\mathfrak{D}}(m,e)\rfloor))\chi^{m} = 
\bigcup_{i\in\mathbb{Z},\, i\geq - h_{z,e}(m)}
B_{(m,i)} 
\end{eqnarray*}
 o\`u $h_{z,e}(m) : = h_{\widetilde{\Delta}_{z}}(m,e)$. Donc
l'id\'eal $\overline{I^{e}}$ est engendr\'e par
\begin{eqnarray*}
\bigcup_{(m,i)\in (e\widetilde{P}_{z})\cap(M\times\mathbb{Z})} B_{(m,i)}.
\end{eqnarray*}
Montrons que pour tout 
$(m,i)\in (e\widetilde{P}_{z})\cap(M\times\mathbb{Z})$, la partie
$B_{(m,i)}$ est incluse dans $I^{e}$.
Fixons un tel couple $(m,i)$. Par normalit\'e de $\widetilde{P}_{z}$,
il existe
\begin{eqnarray*}
(m_{1},i_{1}),\ldots,(m_{e},i_{e})\in \widetilde{P}_{z}\cap (M\times\mathbb{Z})\,\,\,\,
\rm tels\,\,que\,\,\,\,\it\sum_{j = \rm 1 \it}^{e}(m_{j},i_{j}) = (m,i). 
\end{eqnarray*}
Pour chaque $j = 1,\ldots, e$, $B_{(m_{j},i_{j})}$ est contenu
dans $I$. Puisque la multiplication 
\begin{eqnarray*}
B_{(m_{1},i_{1})}\otimes\ldots \otimes B_{(m_{e},i_{e})}\rightarrow B_{(m,i)}
\end{eqnarray*}
est surjective, on a $B_{(m,i)}\subset I^{e}$. D'o\`u le r\'esultat.
\end{proof}
Le th\'eor\`eme suivant se d\'eduit du lemme pr\'ec\'edent par localisation.
Il peut \^etre vu comme un analogue de l'assertion $3.5.6$.
\begin{theorem}
Supposons que $C$ est affine. Soit $I\subset A = A[C,\mathfrak{D}]$
un id\'eal homog\`ene int\'egralement clos et $(P,\widetilde{\mathfrak{D}})$
le couple correspondant.
Si pour tout point $z$ appartenant au support
de $\widetilde{\mathfrak{D}}$, 
le poly\`edre $\widetilde{P}_{z}$ est
normal alors l'id\'eal $I$ est normal. Pour tout 
entier $e\geq n:=\dim\,N_{\mathbb{Q}}$, l'id\'eal $\overline{I^{e}}$ est normal. 
En particulier, tout id\'eal stable int\'egralement clos d'une 
$\mathbf{k}^{\star}$-surface affine non elliptique est normal.
\end{theorem}
\begin{proof}
Notons $z_{1},\ldots, z_{r}$ les points distincts du support
de $\widetilde{\mathfrak{D}}$. Par le th\'eor\`eme des restes
chinois, l'application 
\begin{eqnarray*}
\pi :\mathbf{k}[C]\rightarrow \mathbf{k}^{r},\,\,\,f\mapsto (f(z_{1}),\ldots, f(z_{r})) 
\end{eqnarray*}
est surjective. Soit $(e_{1},\ldots,e_{r})$ la base canonique de 
$\mathbf{k}^{r}$ et pour $i = 1,\ldots, r$, prenons
$f_{i}$ un \'el\'ement de $\pi^{-1}(\{e_{i}\})$.
Pour une fonction r\'eguli\`ere $f\in\mathbf{k}[C]$, nous
notons $C_{f} = C-V(f)$ son lieu de non-annulation. 
Soient $f_{r+1},f_{r+2},\ldots,f_{s}\in\mathbf{k}[C]$ tels que 
\begin{eqnarray*}
U:= C-\{z_{1},\ldots, z_{r}\} = \bigcup_{i = r+1}^{s}C_{f_{j}}.
\end{eqnarray*}
Alors on a l'\'egalit\'e 
\begin{eqnarray*}
C = \bigcup_{i = 1}^{s}C_{f_{i}}. 
\end{eqnarray*}
Consid\'erons $z_{r+1},z_{r+2},\ldots ,z_{s}\in U$ des points
tels que $f_{i}(z_{i})\neq 0$ pour $i = r+1,\ldots, s$.
Fixons un entier $i\in\mathbb{N}$ tel
que $1\leq i\leq s$. Puisque $\widetilde{P}_{z_{i}}$ est un
poly\`edre entier normal, par le lemme $3.7.2$, l'id\'eal
\begin{eqnarray*}
I_{f_{i}}: =  \bigoplus_{m\in P\cap M}
H^{0}(C_{f_{i}},\mathcal{O}_{C}(\lfloor\widetilde{\mathfrak{D}}(m,1)\rfloor ))\chi^{m} = 
I\otimes_{\mathbf{k}[C]}\mathbf{k}[C_{f_{i}}]\subset A\otimes_{\mathbf{k}[C]}\mathbf{k}[C_{f_{i}}]
\end{eqnarray*}
correspondant au couple $(P,\widetilde{\mathfrak{D}}|C_{f_{i}})$ est normal.
Donc nous avons 
\begin{eqnarray}
A[C_{f_{i}},\widetilde{\mathfrak{D}}|C_{f_{i}}] = 
\overline{(A\otimes_{\mathbf{k}[C]}\mathbf{k}[C_{f_{i}}])
[I_{f_{i}}t]} 
\end{eqnarray}
\begin{eqnarray*}
= (A\otimes_{\mathbf{k}[C]}\mathbf{k}[C_{f_{i}}])
[I_{f_{i}}t] = A[It]\otimes_{\mathbf{k}[C]}
\mathbf{k}[C_{f_{i}}].
\end{eqnarray*}
\'Ecrivons 
\begin{eqnarray*}
A[It] = \bigoplus_{(m,e)\in\widetilde{\sigma}^{\vee}_{M\times\mathbb{Z}}}
I_{(m,e)}\chi^{m}t^{e} 
\end{eqnarray*}
avec $I_{(m,e)}\subset\mathbf{k}(C)$, pour tout $(m,e)$.
Alors par $(3.11)$, le mono\"ide des poids
de $A[It]$ est $\widetilde{\sigma}^{\vee}_{M\times\mathbb{Z}}$.
Cela implique que chaque $I_{(m,e)}$ est un id\'eal fractionnaire
de $\mathbf{k}[C]$. Donc pour tout vecteur $(m,e)\in 
\widetilde{\sigma}^{\vee}_{M\times\mathbb{Z}}$, il existe un
diviseur de Cartier entier $D_{(m,e)}$ sur $C$ tel que
\begin{eqnarray*}
 I_{(m,e)} = H^{0}(C,\mathcal{O}_{C}(D_{(m,e)})).
\end{eqnarray*} 
Comme 
\begin{eqnarray*}
A[It]\otimes_{\mathbf{k}[C]}\mathbf{k}[C_{f_{i}}] =  
\bigoplus_{(m,e)\in\widetilde{\sigma}^{\vee}_{M\times\mathbb{Z}}}
H^{0}(C_{f_{i}},\mathcal{O}_{C}(D_{(m,e)}))\chi^{m}t^{e},
\end{eqnarray*}
on a d'une part,
\begin{eqnarray*}
\bigcap_{i = 1}^{s}A[It]\otimes_{\mathbf{k}[C]}\mathbf{k}[C_{f_{i}}]
=A[It] 
\end{eqnarray*}
et d'autre part,
\begin{eqnarray*}
\bigcap_{i = 1}^{s}A[C_{f_{i}},\widetilde{\mathfrak{D}}|C_{f_{i}}]
= A[C,\widetilde{\mathfrak{D}}], 
\end{eqnarray*}
on conclut par $(3.11)$ que $A[It] = A[C,\widetilde{\mathfrak{D}}]$. 
Cela donne la normalit\'e de $I$. Le reste de la preuve est une 
cons\'equence directe du th\'eor\`eme $3.5.6$. 
\end{proof}
La prochaine assertion est une traduction combinatoire de [RRV, Proposition $3.1$]
via la correspondance du th\'eor\`eme $3.5.5$.
\begin{lemme}
Soit $n\in\mathbb{N}$ un entier, notons $\sigma^{\vee} = \mathbb{Q}^{n+1}_{\geq 0}$ et
soit $P$ un $\sigma^{\vee}$-poly\`edre entier contenu dans $\sigma^{\vee}$.
Alors les assertions suivantes sont \'equivalentes.
\begin{enumerate}
\item[\rm (i)] Le poly\`edre est normal;
\item[\rm (ii)] Pour tout $s\in\{1,\ldots,n\}$, 
on a l'\'egalit\'e $(sP)\cap\mathbb{Z}^{n+1} = E_{[s,P]}$ (voir $3.5.2$).
\end{enumerate}
\end{lemme}
\paragraph{}
Comme application du th\'eor\`eme $3.7.3$, nous obtenons une
caract\'erisation de la normalit\'e pour une classe d'id\'eaux 
de l'alg\`ebre des polyn\^omes.
\begin{corollaire}
Soit $n\in\mathbb{Z}_{>0}$.
Consid\'erons l'alg\`ebre des polyn\^omes 
\begin{eqnarray*}
\mathbf{k}^{[n+1]} = \mathbf{k}[x_{0},x_{1},\ldots,x_{n}]
\end{eqnarray*}
\`a $n+1$ variables munie de la $\mathbb{Z}^{n}$-graduation 
\begin{eqnarray*}
\mathbf{k}^{[n+1]} = \bigoplus_{(m_{1},\ldots,m_{r})\in\mathbb{N}^{n}}\mathbf{k}[x_{0}]x_{1}^{m_{1}}\ldots\, x_{n}^{m_{n}}
\end{eqnarray*}
et soit $I$ un id\'eal homog\`ene de $A$. Alors les assertions suivantes sont \'equivalentes.
\begin{enumerate}
\item[\rm (i)] L'id\'eal $I$ est normal;
\item[\rm (ii)] Pour tout $e\in\{1,\ldots, n\}$, l'id\'eal $I^{e}$ est int\'egralement clos.
\end{enumerate}
\end{corollaire}
\begin{proof}
Posons $C:=\mathbb{A}^{1} = \rm Spec\,\mathbf{k}[\it x_{\rm \,0}]$, $\sigma := \mathbb{Q}^{n}_{\geq 0}$ et
$M := \mathbb{Z}^{n}$.
Consid\'erons le diviseur $\sigma$-poly\'edral $\mathfrak{D}$ sur la courbe $C$ dont l'\'evaluation est
identiquement nulle. Alors on a  $\mathbf{k}^{[n+1]} = A = A[C,\mathfrak{D}]$. 

Montrons l'implication $\rm (ii)\Rightarrow (i)$.
Soit $(P,\widetilde{\mathfrak{D}})$ le couple correspondant \`a l'id\'eal $I$.
Notons $z_{1},\ldots, z_{r}$ les points distincts du support de $\widetilde{\mathfrak{D}}$.
Pour $i = 1,\ldots,r$, consid\'erons le polyn\^ome
\begin{eqnarray*}
f_{i}(x_{0}) = \prod_{j\neq i}(x_{0} - z_{j}).
\end{eqnarray*}
Alors pour tout $i$, nous avons 
\begin{eqnarray*}
I_{f_{i}}:=I\otimes_{\mathbf{k}[C]}\mathbf{k}[C_{f_{i}}] = 
\bigoplus_{(m,e)\in \widetilde{P}_{z_{i}}\cap (M\times\mathbb{Z})}
\mathbf{k}\left[\frac{1}{f_{i}}\right](x_{0}-z_{i})^{e}\chi^{m}
\end{eqnarray*} 
o\`u pour $m=(m_{1},\ldots,m_{r})$, $\chi^{m} := x_{1}^{m_{1}}\ldots x_{n}^{m_{n}}$.
Fixons $s\in\{1,\ldots, n\}$. Alors on a d'une part,
\begin{eqnarray*}
I_{f_{i}}^{s} = \bigoplus_{(m,e)\in E_{[s,\widetilde{P}_{z_{i}}]}}
\mathbf{k}\left[\frac{1}{f_{i}}\right](x_{0}-z_{i})^{e}\chi^{m}
\end{eqnarray*}
et d'autre part,
\begin{eqnarray*}
\overline{I_{f_{i}}^{s}} = \bigoplus_{m\in (sP)\cap M}H^{0}(C_{f_{i}},
\mathcal{O}(\lfloor \widetilde{\mathfrak{D}}(m,e)\rfloor))\chi^{m} =
\bigoplus_{(m,e)\in (s\widetilde{P}_{z_{i})}\cap (M\times\mathbb{Z})}
\mathbf{k}\left[\frac{1}{f_{i}}\right](x_{0}-z_{i})^{e}\chi^{m}, 
\end{eqnarray*}
comparer avec [HS, Proposition $1.1.4$].
Puisque $I^{s}$ est int\'egralement clos, l'id\'eal $I_{f_{i}}^{s}\subset A_{f_{i}}$ l'est encore.
Ce qui donne par les \'egalit\'es pr\'ec\'edentes,
\begin{eqnarray*}
(s\widetilde{P}_{z_{i}})\cap (M\times\mathbb{Z}) 
= E_{[s,\widetilde{P}_{z_{i}}]}.
\end{eqnarray*}
En appliquant le lemme $3.7.4$, on d\'eduit que $\widetilde{P}_{z_{i}}$ est normal. 
Par le th\'eor\`eme $3.7.3$, on obtient que $I$ est normal. D'o\`u l'implication 
$\rm (ii)\Rightarrow (i)$. La r\'eciproque est ais\'ee.
\end{proof}
\newpage
\strut 
\newpage
\chapter{La pr\'esentation d'Altmann-Hausen en complexit\'e un
sur un corps quelconque}
\section{Introduction}
Dans ce chapitre, nous nous int\'eressons \`a une description combinatoire des alg\`ebres normales multigradu\'ees
affines de complexit\'e $1$ sur un corps quelconque. D'un point de vue g\'eom\'etrique, ces alg\`ebres sont
reli\'ees \`a la classification des op\'erations de tores alg\'ebriques d\'eploy\'es de complexit\'e
$1$ dans les vari\'et\'es affines. Soit $\mathbf{k}$ un corps et consid\'erons un tore alg\'ebrique
d\'eploy\'e $\mathbb{T}$ sur $\mathbf{k}$. 
Rappelons que dans ce contexte une $\mathbb{T}$-vari\'et\'e est une vari\'et\'e normale sur $\mathbf{k}$ 
munie d'une op\'eration fid\`ele de $\mathbb{T}$. La plupart des travaux classiques sur 
les $\mathbb{T}$-vari\'et\'es (voir [KKMS, Do, Pi, De $2$, Ti$2$, FZ, AH, Ti, AHS, AOPSV, etc.]) demandent 
que le corps de base $\mathbf{k}$ soit alg\'ebriquement clos de caract\'eristique z\'ero. 
Mentionnons tout de m\^eme que la description des $\mathbb{G}_{m}$-vari\'et\'es affines 
[De $2$] due \`a Demazure est vraie sur un corps arbitraire.

Donnons une liste de quelques r\'esultats du pr\'esent chapitre.

- La pr\'esentation d'Altmann-Hausen des $\mathbb{T}$-vari\'et\'es affines de complexit\'e $1$ 
en termes de diviseurs poly\'edraux est vraie sur un corps quelconque, voir le th\'eor\`eme $4.6.3$. 

- Cette description est vraie aussi pour une classe d'alg\`ebres multigradu\'ees sur un anneau de Dedekind, pour plus
de d\'etails voir le th\'eor\`eme $4.4.4$.

- Nous \'etudions comment change l'alg\`ebre associ\'ee \`a un diviseur poly\'edral lorsque qu'on
\'etend les scalaires, voir $4.4.12$ et $4.5.9$.

Comme autre application, nous donnons une description combinatoire des $\mathbf{G}$-vari\'et\'es 
affines de complexit\'e $1$, o\`u $\mathbf{G}$ est un tore alg\'ebrique possiblement non d\'eploy\'e sur un corps
$\mathbf{k}$, en utilisant quelques faits \'el\'ementaires de descente galoisienne. 
Ces $\mathbf{G}$-vari\'et\'es affines sont d\'ecrites par un nouvel objet combinatoire, que l'on appelle
{\em diviseur poly\'edral stable par Galois}, voir le th\'eor\`eme $4.7.10$.

Avant de passer \`a la formulation de nos r\'esultats, rappelons quelques notions. Nous commen\c cons par le cas le plus simple d'une 
op\'eration d'un tore alg\'ebrique d\'eploy\'e. Rappelons qu'un tore alg\'ebrique d\'eploy\'e 
$\mathbb{T}$ de dimension $n$ sur le corps $\mathbf{k}$
est un groupe alg\'ebrique isomorphe \`a $\mathbb{G}_{m}^{n}$, o\`u $\mathbb{G}_{m}$ 
est le groupe multiplicatif du corps $\mathbf{k}$.
Soit $M = \rm Hom(\it\mathbb{T},\mathbb{G}_{\it m})$ le r\'eseau des caract\`eres du tore $\mathbb{T}$. 
Alors d\'efinir une op\'eration de $\mathbb{T}$ dans une vari\'et\'e affine $X$ est \'equivalent \`a fixer 
une $M$-graduation sur l'alg\`ebre $A = \mathbf{k}[X]$, o\`u $\mathbf{k}[X]$ est l'anneau des coordonn\'ees
de $X$. Suivant la classification des $\mathbb{G}_{m}$-surfaces affines [FiKa]
nous disons comme dans [Li, $1.1$] ou dans $3.3.11$ qu'une alg\`ebre $M$-gradu\'ee $A$ est \em elliptique \rm si la pi\`ece gradu\'ee $A_{0}$ 
est r\'eduite au corps $\mathbf{k}$. Consid\'erons le corps $\mathbf{k}(X)$ des fonctions rationnelles  
sur la vari\'et\'e $X$ et soit $K_{0}$ son sous-corps des fonctions invariantes sous l'op\'eration de $\mathbb{T}$. 
La complexit\'e de l'op\'eration de $\mathbb{T}$ dans $X$ est le degr\'e de transcendance de $K_{0}$ sur le corps $\mathbf{k}$.

Dans le but de d\'ecrire des classes particuli\`eres d'alg\`ebres multigradu\'ees de complexit\'e $1$, 
nous aurons \`a consid\'erer des objets combinatoires de la g\'eom\'etrie convexe et 
de la g\'eom\'etrie des courbes alg\'ebriques. 
Soit $C$ une courbe r\'eguli\`ere sur le corps $\mathbf{k}$. Un point de $C$ est suppos\'e \^etre ferm\'e, et 
en particulier, \'eventuellement non rationnel. De plus, il est connu que le corps r\'esiduel en tout point
de $C$ est une extension de corps de degr\'e fini sur $\mathbf{k}$.

Pour reformuler notre premier r\'esultat, nous avons besoin de rappeler quelques notations disponibles dans 
[AH, Section $1$].
D\'esignons par $N = \rm Hom(\it\mathbb{G}_{\it m},
\mathbb{T})$ le r\'eseau des sous-groupes \`a $1$ param\`etre du tore alg\'ebrique d\'eploy\'e $\mathbb{T}$ 
qui est vu comme le dual du r\'eseau $M$. 
Comme d'habitude, on note $M_{\mathbb{Q}} = \mathbb{Q}\otimes_{\mathbb{Z}}M$, 
$N_{\mathbb{Q}} = \mathbb{Q}\otimes_{\mathbb{Z}}N$ les espaces vectoriels sur $\mathbb{Q}$ associ\'es \`a $M,N$ et
on consid\`ere $\sigma\subset N_{\mathbb{Q}}$ un c\^one poly\'edral saillant. Nous pouvons d\'efinir comme
dans [AH] un diviseur de Weil $\mathfrak{D} = \sum_{z\in C}\Delta_{z}\cdot z$ avec des coefficients $\sigma$-poly\'edraux 
dans $N_{\mathbb{Q}}$, qui est appel\'e diviseur poly\'edral d'Altmann-Hausen. Plus pr\'ecis\'ement,
chaque $\Delta_{z}\subset N_{\mathbb{Q}}$ est un poly\`edre dont le c\^one de r\'ecession est $\sigma$ et 
$\Delta_{z} = \sigma$ pour presque tout $z\in C$. En d\'esignant par $\kappa_{z}$ le corps r\'esiduel 
du point $z\in C$ et par $[\kappa_{z}:\mathbf{k}]\cdot \Delta_{z}$ l'image de $\Delta_{z}$ par l'homoth\'etie 
de rapport $\rm deg(\it z) = [\kappa_{z}:\mathbf{k}]$,
la somme
\begin{eqnarray*}
\rm deg\,\it\mathfrak{D} = \sum_{z\in C}[\kappa_{z}:\mathbf{k}]\cdot\rm\Delta_{\it z} 
\end{eqnarray*}
est un poly\`edre de $N_{\mathbb{Q}}$. Cette somme peut \^etre vue comme la somme finie
de Minkowski de tous les poly\`edres  
$[\kappa_{z}:\mathbf{k}]\cdot\Delta_{z}$ 
diff\'erents de $\sigma$. 
En consid\'erant le c\^one dual $\sigma^{\vee}\subset M_{\mathbb{Q}}$ de $\sigma$,
nous d\'efinissons la fonction \'evaluation 
\begin{eqnarray*}
\sigma^{\vee}\rightarrow \rm Div_{\it\mathbb{Q}}(\it C),\,\,\,
m\mapsto \mathfrak{D}(m) = \sum_{z\in C}\min_{l\in\rm\Delta_{\it z}\it}\,\langle m, l\it\rangle\cdot z 
\end{eqnarray*}
\`a valeurs dans l'espace vectoriel $\rm Div_{\it\mathbb{Q}}(\it C)$ sur $\mathbb{Q}$ des 
diviseurs de Cartier rationnels sur $C$. 
Comme dans le cas classique [AH, $2.12$], nous introduisons la condition de propret\'e pour le diviseur
poly\'edral $\mathfrak{D}$ (voir $4.4.1$, $4.5.4$, $4.6.2$) que nous rappelons ci-apr\`es.

\begin{definition}
Un diviseur $\sigma$-poly\'edral $\mathfrak{D} = \sum_{z\in C}\Delta_{z}\cdot z$ 
est dit \em propre \rm s'il satisfait une des conditions suivantes.
\begin{enumerate}
\item[\rm (i)]
$C$ est affine. 
\item[\rm (ii)]
$C$ est projective et $\rm deg\it\, \mathfrak{D}$
est strictement contenu dans $\sigma$. De plus, 
si $\rm deg\it\, \mathfrak{D}(m) \rm = 0$ alors $m$ appartient 
au bord de $\sigma^{\vee}$ et un multiple entier non nul de $\mathfrak{D}(m)$ est
un diviseur principal. 
\end{enumerate}
\end{definition}
Par exemple, si $C = \mathbb{P}^{1}_{\mathbf{k}}$ est la droite projective
alors le diviseur poly\'edral $\mathfrak{D}$ est propre si et seulement si $\rm deg\,\it \mathfrak{D}$ 
est strictement inclus dans $\sigma$. Un des r\'esultats principaux de ce chapitre 
peut \^etre \'enonc\'e comme suit. 
\begin{theorem}
Soit $\mathbf{k}$ un corps.
\begin{enumerate}
 \item[\rm (i)] Pour tout diviseur $\sigma$-poly\'edral propre $\mathfrak{D}$ sur une courbe r\'eguli\`ere $C$ sur $\mathbf{k}$ 
on peut associer une alg\`ebre normale $M$-gradu\'ee de type fini sur $\mathbf{k}$ et de complexit\'e $1$,
donn\'ee par 
\begin{eqnarray*}
A[C,\mathfrak{D}] = \bigoplus_{m\in\sigma^{\vee}\cap M}A_{m},\,\,\, avec\,\,\, \it 
A_{m} = H^{\rm 0\it}(C,\mathcal{O}_{C}(\lfloor \mathfrak{D}(m)\rfloor)).  
\end{eqnarray*}
\item[\rm (ii)] R\'eciproquement, toute alg\`ebre normale effectivement\footnote{Une alg\`ebre $M$-gradu\'ee
est dite effectivement $M$-gradu\'ee si l'ensemble de ses poids engendre le r\'eseau $M$.} $M$-gradu\'ee de type fini sur $\mathbf{k}$
et de complexit\'e $1$ est isomorphe \`a $A[C,\mathfrak{D}]$, pour une courbe r\'eguli\`ere $C$
sur $\mathbf{k}$ et un diviseur poly\'edral propre $\mathfrak{D}$ sur $C$.
\end{enumerate}
\end{theorem}
Dans la d\'emonstration de l'assertion $\rm (ii)$, nous utilisant un calcul explicite donn\'e dans 
le chapitre pr\'ec\'edent
(voir aussi [La]). Nous divisons la d\'emonstration en deux parties. Dans le \em cas non elliptique \rm 
nous montrons que l'assertion est vraie dans le contexte plus g\'en\'eral des anneaux de Dedekind. 
Plus pr\'ecis\'ement, nous donnons un dictionnaire parfait analogue \`a $4.1.2 \rm (i),(ii)$ pour les 
alg\`ebres $M$-gradu\'ees d\'efinies par un diviseur poly\'edral sur un anneau de Dedekind (voir $4.4.1, 4.4.2$ 
et le th\'eor\`eme $4.4.4$). Nous donnons dans le paragraphe $4.4.6$ un exemple  
avec un diviseur poly\'edral sur $\mathbb{Z}[\sqrt{-5}]$. Dans le \em cas elliptique, \rm 
nous consid\'erons une alg\`ebre $M$-gradu\'ee elliptique $A$ sur $\mathbf{k}$ satisfaisant les hypoth\`eses 
de $4.1.2\,\rm (ii)$. Par un r\'esultat bien connu (voir par exemple [EGA II, $7.4$]), 
nous construisons une courbe projective r\'eguli\`ere \`a partir du corps des fonctions
alg\'ebriques $K_{0} = (\rm Frac\,\it A)^{\mathbb{T}}$. Dans cette construction, les points de $C$ sont alors
identifi\'es avec les places de $K_{0}$. 
Ensuite, nous montrons que cette alg\`ebre $M$-gradu\'ee est d\'ecrite par un diviseur poly\'edral sur $C$ 
(voir le th\'eor\`eme $4.5.6$).

Passons plus g\'en\'eralement au cas des vari\'et\'es affines avec une op\'eration d'un tore alg\'ebrique
possiblement non d\'eploy\'e. Le lecteur peut consulter [Bry, CTHS, Vos, ELST] pour la th\'eorie des
vari\'et\'es toriques arithm\'etiques et [Hu] pour le cas plus g\'en\'eral des vari\'et\'es sph\'eriques.
Soit $\mathbf{G}$ un tore sur $\mathbf{k}$; alors $\mathbf{G}$ se d\'eploie 
dans une extension galoisienne finie $E/\mathbf{k}$. 
Soit $\rm Var_{\it \mathbf{G}, E}(\it\mathbf{k})$ la cat\'egorie des $\mathbf{G}$-vari\'et\'es affines
de complexit\'e $1$ (se d\'eployant dans $E/\mathbf{k}$), pour une d\'efinition pr\'ecise voir $4.7.4$.
Pour un objet $X\in\rm Var_{\it \mathbf{G}, E}(\it\mathbf{k})$ nous notons
$[X]$ sa classe d'isomorphisme et 
\begin{eqnarray*}
X_{E} = X\times_{\rm Spec\,\it\mathbf{k}}\rm Spec\,\it E
\end{eqnarray*}
l'extension de $X$ sur le corps $E$. Fixons $X\in\rm Var_{\it \mathbf{G}, E}(\it\mathbf{k})$. 
Comme application des r\'esultats pr\'ec\'edents, nous \'etudions l'ensemble point\'e
\begin{eqnarray*}
\left(\left\{[Y]\,|\, Y\in \rm Var_{\it \mathbf{G}, E}(\it\mathbf{k})\rm \,\,\, et\it\,\,\, X_{E}
\simeq_{\rm Var_{\it \mathbf{G}, E}(\it E)}\it
Y_{E}\right\},[X]\right) 
\end{eqnarray*}
des classes d'isomorphismes des $E/\mathbf{k}$-formes de la vari\'et\'e $X$ qui est en bijection avec le premier ensemble point\'e $H^{1}(E/\mathbf{k}, \rm Aut_{\it\mathbf{G}_{E}}(\it X_{E}\rm ))$ 
de cohomologie galoisienne.
Par des arguments \'el\'ementaires (voir $4.7.7$) ces derniers ensembles point\'es sont en bijection avec l'ensemble point\'e
des classes de conjugaison des op\'erations semi-lin\'eaires homog\`enes de $\mathfrak{S}_{E/\mathbf{k}}$
dans l'alg\`ebre multigradu\'ee $E[X_{E}]$, o\`u $\mathfrak{S}_{E/\mathbf{k}}$ est le groupe de Galois de 
l'extension $E/\mathbf{k}$.
En traduisant ceci dans le langage des diviseurs poly\'edraux, nous obtenons une description
combinatoire des $E/\mathbf{k}$-formes de $X$, voir le th\'eor\`eme $4.7.10$. 
Ce th\'eor\`eme peut \^etre vu comme une premi\`ere \'etape de l'\'etude des $E/\mathbf{k}$-formes des 
$\mathbf{G}$-vari\'et\'es de complexit\'e $1$.

Donnons un court r\'esum\'e de chaque section de ce chapitre. Dans la section $4.3$, 
nous rappelons comment \'etendre la pr\'esentation D.P.D.
des alg\`ebres gradu\'ees paraboliques au contexte des anneaux de Dedekind.
Ce fait a \'et\'e mentionn\'e dans l'introduction de [FZ] et trait\'e en premier lieu
par Nagat Karroum dans un m\'emoire de mast\`ere dirig\'e par Hubert Flenner [Ka]. 
Dans les sections $4.4$ et $4.5$, nous \'etudions respectivement une
classe d'alg\`ebres multigradu\'ees sur un anneau de Dedekind et une classe 
d'alg\`ebres multigradu\'ees elliptiques sur un corps. Dans la section $4.6$, nous classifions
les $\mathbb{T}$-vari\'et\'es affines de complexit\'e $1$. Enfin, dans la section $4.7$, nous 
traitons le cas des op\'erations de tores alg\'ebriques possiblement non d\'eploy\'e.   
\paragraph{}
Soit $\mathbf{k}$ un corps. 
Par une \em vari\'et\'e \rm $X$ sur $\mathbf{k}$ on entend un sch\'ema int\`egre s\'epar\'e de
type fini sur le corps $\mathbf{k}$; on suppose en outre que le corps $\mathbf{k}$ est
alg\'ebriquement clos dans le corps des fonctions rationnelles $\mathbf{k}(X)$. En particulier, cela implique 
que $X$ est g\'eom\'etriquement irr\'eductible [Liu, \S 3.2.2, Corollary 2.14].
\paragraph{}
\section{Introduction (english version)}
In this chapter, we are interested in a combinatorial description of multigraded normal
affine algebras of complexity $1$. From a geometric viewpoint, these algebras are related 
to the classification of algebraic torus actions of complexity $1$ on affine varieties.
Let $\mathbf{k}$ be a field and consider a split algebraic torus $\mathbb{T}$ over $\mathbf{k}$. 
Recall that a $\mathbb{T}$-variety is a normal variety over $\mathbf{k}$ endowed with an effective 
$\mathbb{T}$-action. Most classical works on $\mathbb{T}$-varieties
(see [KKMS, Do, Pi, De $2$, Ti$2$, FZ, AH, Ti, AHS, AOPSV, etc.]) require the ground field $\mathbf{k}$ to be 
algebraically closed of characteristic zero. It is worthwile mentioning that 
the description of affine $\mathbb{G}_{m}$-varieties [De $2$] due to Demazure holds over any field.

Let us list the most important results of the chapter.

- The Altmann-Hausen presentation of affine $\mathbb{T}$-varieties
of complexity $1$ in terms of polyhedral divisor holds over an arbitrary field,
see Theorem $4.6.3$. 

- This description holds as well for an important class of 
multigraded algebras over Dedekind domains, see Theorem $4.4.4$.

- We study how the algebra associated to a polyhedral divisor changes when we extend the scalars, 
see $4.4.12$ and $4.5.9$.

- As another application, we provide a combinatorial description of 
affine $\mathbf{G}$-varieties of complexity $1$, where $\mathbf{G}$ is a (not necessarily split) 
torus over $\mathbf{k}$, by using elementary 
facts on Galois descent. This class of affine $\mathbf{G}$-varieties are classified via a new combinatorial
object, which we call a (Galois) invariant polyhedral divisor, see Theorem $4.7.10$. 

Let us now discuss these results in more detail. We start with a simple case of varieties with an 
action of a split torus.
Recall that a split algebraic torus $\mathbb{T}$ of dimension $n$ over the field $\mathbf{k}$
is an algebraic group isomorphic to $\mathbb{G}_{m}^{n}$, where $\mathbb{G}_{m}$ is the multiplicative  group of the field $\mathbf{k}$.
Let $M = \rm Hom(\it\mathbb{T},\mathbb{G}_{\it m})$ be the
character lattice of the torus $\mathbb{T}$. Then defining a $\mathbb{T}$-action on an affine variety
$X$ is equivalent to fixing an $M$-grading on the algebra $A = \mathbf{k}[X]$, where $\mathbf{k}[X]$
is the coordinate ring of $X$. Following the classification of affine $\mathbb{G}_{m}$-surfaces [FiKa]
we say as in [Li, $1.1$] or in $3.3.11$ that the $M$-graded algebra $A$ is \em elliptic \rm if the graded piece $A_{0}$ is reduced to $\mathbf{k}$. 
Multigraded affine algebras are classified via a numerical invariant called
complexity. Consider the field $\mathbf{k}(X)$ of rational functions on $X$ and 
its subfield $K_{0}$ of $\mathbb{T}$-invariant functions. The complexity of the $\mathbb{T}$-action on $X$ is the 
transcendence degree of $K_{0}$ over the field $\mathbf{k}$.  

In order to describe particular classes of multigraded algebras of complexity $1$, 
we have to consider combinatorial 
objects coming from convex geometry and from the geometry of algebraic curves. Let $C$ be a regular curve 
over $\mathbf{k}$. A point of $C$ is assumed to be a closed 
point, and in particular, not necessarily rational. Thus, the residue field extension of $\mathbf{k}$ 
at any point of $C$
has finite degree. 

To reformulate our first result, we need some combinatorial notions of convex geometry, see [AH, Section $1$].
Denote by $N = \rm Hom(\it\mathbb{G}_{\it m},
\mathbb{T})$ the lattice of one parameter subgroups of the torus $\mathbb{T}$ which is the dual of the lattice $M$. 
 Let $M_{\mathbb{Q}} = \mathbb{Q}\otimes_{\mathbb{Z}}M$, 
$N_{\mathbb{Q}} = \mathbb{Q}\otimes_{\mathbb{Z}}N$ be the associated dual $\mathbb{Q}$-vector spaces
of $M,N$ and let $\sigma\subset N_{\mathbb{Q}}$ be a strongly convex polyhedral cone. We can define 
as in [AH] a Weil divisor $\mathfrak{D} = \sum_{z\in C}\Delta_{z}\cdot z$ with $\sigma$-polyhedral 
coefficients in $N_{\mathbb{Q}}$, called a polyhedral divisor of Altmann-Hausen. More precisely,
each $\Delta_{z}\subset N_{\mathbb{Q}}$ is a polyhedron with a tail cone $\sigma$ and 
$\Delta_{z} = \sigma$ for all but finitely many points $z\in C$. Denoting by $\kappa_{z}$ the residue 
field of the point $z\in C$ and by $[\kappa_{z}:\mathbf{k}]\cdot \Delta_{z}$ the
image of $\Delta_{z}$ under the homothety of ratio $\rm deg(\it z) = [\kappa_{z}:\mathbf{k}]$,
the sum
\begin{eqnarray*}
\rm deg\,\it\mathfrak{D} = \sum_{z\in C}[\kappa_{z}:\mathbf{k}]\cdot\rm\Delta_{\it z} 
\end{eqnarray*}
is a polyhedron in $N_{\mathbb{Q}}$. This sum may be seen as the finite Minkowski sum of all polyhedra 
$[\kappa_{z}:\mathbf{k}]\cdot\Delta_{z}$ 
different from $\sigma$. Considering the dual cone $\sigma^{\vee}\subset M_{\mathbb{Q}}$ of $\sigma$,
we define an evaluation function  
\begin{eqnarray*}
\sigma^{\vee}\rightarrow \rm Div_{\it\mathbb{Q}}(\it C),\,\,\,
m\mapsto \mathfrak{D}(m) = \sum_{z\in C}\min_{l\in\rm\Delta_{\it z}\it}\,\langle m, l\it\rangle\cdot z 
\end{eqnarray*}
with value in the vector space $\rm Div_{\it\mathbb{Q}}(\it C)$ of Weil
$\mathbb{Q}$-divisors over $C$. As in the classical case [AH, $2.12$] we introduce the technical
condition of properness for the polyhedral divisor $\mathfrak{D}$ (see $4.4.1$, $4.5.4$, $4.6.2$) that we recall
thereafter. 

\paragraph{Definition 4.2.1.}
A $\sigma$-polyhedral divisor $\mathfrak{D} = \sum_{z\in C}\Delta_{z}\cdot z$ 
is called \em proper \rm if it
satisfies one of the following conditions.
\begin{enumerate}
\item[\rm (i)]
$C$ is affine. 
\item[\rm (ii)]
$C$ is projective and $\rm deg\it\, \mathfrak{D}$
is strictly contained in the cone $\sigma$. Furthermore, 
if $\rm deg\it\, \mathfrak{D}(m) \rm = 0$ then $m$ belongs 
to the boundary of $\sigma^{\vee}$ and
some non-zero integral multiple of $\mathfrak{D}(m)$ is principal. 
\end{enumerate}

For instance, if $C = \mathbb{P}^{1}_{\mathbf{k}}$ is the projective line then the polyhedral divisor 
$\mathfrak{D}$ is proper if and only if $\rm deg\,\it \mathfrak{D}$ is strictly included in $\sigma$. 
One of the main results of this paper can be stated as follows.
\paragraph{Theorem 4.2.2.}{\em
Let $\mathbf{k}$ be field.
\begin{enumerate}
 \item[\rm (i)] To any proper $\sigma$-polyhedral divisor $\mathfrak{D}$ on a regular curve $C$
over $\mathbf{k}$ one can associate a normal finitely generated effectively\footnote{An $M$-graded
algebra is said to be effectively $M$-graded if its set of weight generates the lattice $M$.} $M$-graded domain of complexity $1$
over $\mathbf{k}$, given by
\begin{eqnarray*}
A[C,\mathfrak{D}] = \bigoplus_{m\in\sigma^{\vee}\cap M}A_{m},\,\,\, where\,\,\, \it 
A_{m} = H^{\rm 0\it}(C,\mathcal{O}_{C}(\lfloor \mathfrak{D}(m)\rfloor)).  
\end{eqnarray*}
\item[\rm (ii)] Conversely, any normal finitely generated effectively $M$-graded domain of
complexity $1$ over $\mathbf{k}$ is isomorphic to $A[C,\mathfrak{D}]$ for some regular curve $C$
over $\mathbf{k}$ and some proper polyhedral divisor $\mathfrak{D}$ over $C$.
\end{enumerate}
}

In the proof of assertion $\rm (ii)$, we use an effective calculation from [La]. 
We divide the proof into two cases. In the \em non-elliptic case \rm we show that the assertion
holds more generally in the context of Dedekind domains. More precisely, we give a perfect dictionary
similar to $4.1.2 \rm (i),(ii)$ for $M$-graded algebras defined by a polyhedral divisor over
a Dedekind ring (see $4.4.1, 4.4.2$ and Theorem $4.4.4$). We deal in $4.4.6$ with an example 
of a polyhedral divisor over $\mathbb{Z}[\sqrt{-5}]$. In the \em elliptic case, \rm 
we consider an elliptic $M$-graded algebra $A$ over $\mathbf{k}$ satisfying the 
assumptions of $4.1.2\,\rm (ii)$. By a well known result (see [EGA II, $7.4$]), 
we can construct a regular projective curve arising from the 
algebraic function field $K_{0} = (\rm Frac\,\it A)^{\mathbb{T}}$. In this construction, 
the points of $C$ are identified with the places of $K_{0}$. 
Then we show that the $M$-graded algebra is described by a polyhedral divisor 
over $C$ (see Theorem $4.5.6$).

Let us pass further to the general case of varieties with an action of a not necessarily split
torus. The reader may consult [Bry, CTHS, Vos, ELST]
for the theory of non-split toric varieties and [Hu] for the spherical embeddings.
Let $\mathbf{G}$ be a torus over $\mathbf{k}$; then $\mathbf{G}$ splits in a finite Galois
extension $E/\mathbf{k}$. Let $\rm Var_{\it \mathbf{G}, E}(\it\mathbf{k})$ be the category
of affine $\mathbf{G}$-varieties of complexity one splitting in $E/\mathbf{k}$ (see $4.7.4$).
For an object $X\in\rm Var_{\it \mathbf{G}, E}(\it\mathbf{k})$ we let
$[X]$ be its isomorphism class and $X_{E} = X\times_{\rm Spec\,\it\mathbf{k}}\rm Spec\,\it E$
be the extension of $X$ over the field extension. Fixing $X\in\rm Var_{\it \mathbf{G}, E}(\it\mathbf{k})$, 
as an application of our previous results, we study the pointed set
\begin{eqnarray*}
\left(\left\{[Y]\,|\, Y\in \rm Var_{\it \mathbf{G}, E}(\it\mathbf{k})\rm \,\,\, and\,\,\, \it X_{E}
\simeq_{\rm Var_{\it \mathbf{G}, E}(\it E)}\it
Y_{E}\right\},[X]\right) 
\end{eqnarray*}
of isomorphism classes of $E/\mathbf{k}$-forms of $X$ that is in bijection with the first pointed 
set $H^{1}(E/\mathbf{k}, \rm Aut_{\it\mathbf{G}_{E}}(\it X_{E}\rm ))$ of non-abelian Galois cohomology.
By elementary arguments (see $4.7.7$) the latter pointed sets are described by all possible
homogeneous semi-linear $\mathfrak{S}_{E/\mathbf{k}}$-actions on the multigraded algebra $E[X_{E}]$, 
where $\mathfrak{S}_{E/\mathbf{k}}$ is the Galois group of $E/\mathbf{k}$.
Translating this to the language of polyhedral divisors, we obtain a combinatorial description 
of $E/\mathbf{k}$-forms of $X$, see Theorem $4.7.10$. This theorem can be viewed as a first step 
towards the study of the forms of 
$\mathbf{G}$-varieties of complexity $1$.

Let us give a brief summary of the contents of each section. In Section $4.3$, we recall how to extend the D.P.D.
presentation of a parabolic graded algebra to the context of a Dedekind domain.
 This fact has been mentioned in [FZ] and 
firstly treated by Nagat Karroum in a master dissertation [Ka]. 
In Sections $4.4$ and $4.5$, 
we study respectively a class of multigraded algebras over Dedekind domains and a class of elliptic
multigraded algebras over a field. In Section $4.6$, we classify 
split affine $\mathbb{T}$-varieties of complexity $1$. 
The last section is devoted to the non-split case.

\paragraph{}
Let $\mathbf{k}$ be a field.
By a \em variety \rm $X$ over $\mathbf{k}$ we mean an integral
separated scheme of finite type over $\mathbf{k}$ ; one assumes in addition that $\mathbf{k}$ is 
algebraically closed in the field of rational functions $\mathbf{k}(X)$. In particular, $X$ is geometrically
irreducible [Liu, \S 3.2.2, Corollary 2.14].
\paragraph{}
\section{Alg\`ebres gradu\'ees normales sur un anneau de Dedekind et pr\'esentation D.P.D.}
Dans cette section, nous g\'en\'eralisons la pr\'esentation D.P.D. introduite dans [FZ, Section $3$]
au contexte des anneaux de Dedekind (voir [Ka]). Commen\c cons par une d\'efinition
bien connue.
\begin{rappel}
Un anneau int\`egre $A_{0}$ est dit de \em Dedekind \rm
s'il n'est pas un corps et s'il satisfait les conditions
suivantes.
\begin{enumerate}
\item[(i)] L'anneau $A_{0}$ est noeth\'erien.

\item[(ii)] L'anneau $A_{0}$ est int\'egralement clos dans son corps des fractions.

\item[(iii)] Tout id\'eal premier non nul de $A_{0}$ est un id\'eal maximal.
\end{enumerate}
\end{rappel}
Donnons quelques exemples classiques d'anneaux de Dedekind.
\begin{exemple}
Soit $K$ un corps de nombres.
Si $\mathbb{Z}_{K}$ d\'esigne l'anneau des entiers de $K$ 
alors $\mathbb{Z}_{K}$ est un anneau de Dedekind. 

Soit $A$ une alg\`ebre normale de type fini et de dimension $1$
sur un corps $\mathbf{k}$. D'un point de vue g\'eom\'etrique,
le sch\'ema $C = \rm Spec\,\it A$ est une courbe affine r\'eguli\`ere sur $\mathbf{k}$.
L'anneau des coordonn\'ees $A = \mathbf{k}[C]$ est de Dedekind.

L'alg\`ebre des s\'eries formelles $\mathbf{k}[[t]]$ \`a une variable sur le corps $\mathbf{k}$ 
est un anneau de Dedekind. Plus g\'en\'eralement, tout anneau principal (et donc tout
anneau de valuation discr\`ete) qui n'est pas un corps est un anneau de Dedekind.
\end{exemple}  
\begin{rappel}
Soit $A_{0}$ un anneau int\`egre de corps des fractions $K_{0}$. 
Rappelons qu'un \em id\'eal fractionnaire \rm $\mathfrak{b}$ de $A_{0}$ est
un sous-module de $K_{0}$ non nul de type fini sur $A_{0}$.
En fait, tout id\'eal fractionnaire de $A_{0}$ est de la forme 
$\frac{1}{f}\cdot\mathfrak{a}$, 
o\`u $f\in A_{0}$ est un \'el\'ement non nul et 
$\mathfrak{a}$ est un id\'eal non nul de $A_{0}$. Si $\mathfrak{b}$ est \'egal
\`a $u\cdot A_{0}$, pour $u\in K_{0}$ non nul, alors on dit
que  $\mathfrak{b}$ est un \em id\'eal fractionnaire principal. \rm 
\end{rappel}
Le r\'esultat suivant donne une description des id\'eaux fractionnaires
de $A_{0}$ en termes de diviseurs de Weil sur le sch\'ema $Y = \rm Spec\,\it A_{\rm 0}$,
lorsque $A_{0}$ est un anneau de Dedekind. Ce r\'esultat est bien connu. Nous incluons 
ici une courte d\'emonstration. 
\begin{theorem}
Soit $A_{0}$ un anneau de Dedekind de corps des fractions $K_{0}$. Posons
$Y =\rm Spec\,\it A_{\rm 0}$. Alors l'application 
\begin{eqnarray*}
\rm Div_{\it \mathbb{Z}}(\it Y\rm )\rightarrow\rm Id(\it A_{\rm 0}\rm )\it , \,\,
D\mapsto H^{\rm 0}(Y,\mathcal{O}_{Y}( D))
\end{eqnarray*}
est une bijection entre l'ensemble $\rm Div_{\it \mathbb{Z}}(\it Y\rm )$ des diviseurs de Weil entiers sur $Y$
et l'ensemble $\rm Id(\it A_{\rm 0}\rm )$ des id\'eaux fractionnaires de $A_{0}$.
Tout id\'eal fractionnaire est localement libre de rang $1$ comme module
sur $A_{0}$ et l'application naturelle de multiplication 
\begin{eqnarray*}
H^{\rm 0}(Y,\mathcal{O}_{Y}( D))\otimes H^{\rm 0}(Y,\mathcal{O}_{Y}( D'))\rightarrow 
H^{\rm 0}(Y,\mathcal{O}_{Y}( D+D'))
\end{eqnarray*}
est surjective.
Un diviseur de Weil $D$ sur le sch\'ema $Y$ est principal (resp. effectif) si
et seulement si l'id\'eal correspondant est principal (resp. contient $A_{0}$).
\end{theorem}
\begin{proof}
Par [Liu, \S 4.1.1, Proposition 1.12] le localis\'e de $A_{0}$ en tout id\'eal premier
est un anneau principal. Donc par [Ha, II.$6.11$] le groupe des diviseurs de Weil entiers sur $Y$ 
s'identifie au groupe des diviseurs de Cartier. En particulier, tout module $H^{0}(Y,\mathcal{O}_{Y}(D))$
sur $A_{0}$ est de type fini [Ha, II.$5.5$], localement libre de rang $1$ et non nul. 
Par cons\'equent, l'application 
\begin{eqnarray*}
 \rm Div_{\it \mathbb{Z}}(\it Y\rm )
\rightarrow\rm Id(\it A_{\rm 0}\rm )
\end{eqnarray*}
est bien d\'efinie. 

Soient $D, D'$ des diviseurs de
$\rm Div_{\it \mathbb{Z}}(\it Y\rm )$.
Alors par les observations pr\'ec\'edentes, 
les faisceaux 
$\mathcal{O}_{Y}(D)\otimes\mathcal{O}_{Y}(D')$
et $\mathcal{O}_{Y}(D+D')$ de $\mathcal{O}_{Y}$-modules
sont isomorphes. Cela induit un isomorphisme au niveau
des sections globales.

Tout id\'eal premier non nul de $A_{0}$ est le module des
sections globales d'un faisceau inversible sur $\mathcal{O}_{Y}$.
Ainsi par la d\'ecomposition en produit d'id\'eaux premiers 
au sein d'un anneau de Dedekind, l'application 
$\rm Div_{\it \mathbb{Z}}(\it Y\rm )
\rightarrow\rm Id(\it A_{\rm 0}\rm )$ est surjective.

Supposons que l'on ait l'\'egalit\'e
\begin{eqnarray*}
H^{0}(Y,\mathcal{O}_{Y}(D)) = 
H^{0}(Y,\mathcal{O}_{Y}(D')),
\end{eqnarray*}
pour des \'el\'ements $D,D'\in \rm Div_{\it \mathbb{Z}}(\it Y\rm )$. 
Alors nous pouvons \'ecrire $D = D_{+}-D_{-}$ et $D = D'_{+}-D'_{-}$, o\`u
$D_{+},D'_{+},D_{-},D'_{+}$ sont des diviseurs de Weil effectifs entiers. En appliquant des produits
tensoriels,
nous obtenons la relation
\begin{eqnarray*}
H^{0}(Y,\mathcal{O}_{Y}(-D_{-}-D'_{+})) = 
H^{0}(Y,\mathcal{O}_{Y}(-D'_{-}-D_{+})) 
\end{eqnarray*}
entre id\'eaux de $A_{0}$. \`A nouveau en utilisant la d\'ecomposition en produit d'id\'eaux
premiers, nous avons $-D_{-}-D'_{+} = -D'_{-}-D_{+}$ de sorte que $D = D'$. 
On conclut que l'application est injective. 

Supposons que $H^{0}(Y,\mathcal{O}_{Y}(D))$
contient $A_{0}$. \'Ecrivons $D = D_{+}-D_{-}$ avec
$D_{+},D_{-}$ des diviseurs de Weil effectifs entiers ayant des supports
disjoints. Alors par hypoth\`ese, nous avons
\begin{eqnarray*}
H^{0}(Y,\mathcal{O}_{Y}(0)) = A_{0} = A_{0}\cap 
H^{0}(Y,\mathcal{O}_{Y}(D)) = 
H^{0}(Y,\mathcal{O}_{Y}(-D_{-})). 
\end{eqnarray*}
Cela donne $D_{-} = 0$ et donc le diviseur $D$ est effectif.
Le reste de la d\'emonstration est ais\'e.
\end{proof}
\begin{notation}
Soit $A_{0}$ un anneau de Dedekind. Pour un diviseur de Weil rationnel $D$ 
sur $Y = \rm Spec\,\it A_{\rm 0}$, nous d\'esignons par $A_{0}[D]$ 
l'anneau gradu\'e\footnote{Rappelons que $\lfloor iD\rfloor$ est le diviseur 
de Weil entier obtenu \`a partir de $iD$ en prenant la partie enti\`ere sur 
chaque coefficient.}
\begin{eqnarray*}
\bigoplus_{i\in\mathbb{N}}
H^{0}(Y,\mathcal{O}_{Y}(\lfloor iD\rfloor ))\,t^{i},  
\end{eqnarray*}
o\`u $t$ est une variable sur le corps 
$K_{0}$. Notons que l'anneau int\`egre
$A_{0}[D]$ est normal comme intersection d'anneaux de valuation discr\`ete de
corps des fractions $K_{0}(t)$; le lecteur peut consulter les arguments de
d\'emonstration de [De $2$, $2.7$].
\end{notation}
Le prochain lemme donne une pr\'esentation D.P.D. pour une classe naturelle
de sous-anneaux gradu\'es de $K_{0}[t]$. Ce r\'esultat nous sera utile
pour la prochaine section. Nous donnons ici une d\'emonstration \'el\'ementaire
en utilisant la description de $4.3.4$ concernant les id\'eaux fractionnaires
de $A_{0}$.
\begin{lemme}
Soit $A_{0}$ un anneau de Dedekind avec corps des fractions $K_{0}$.
Soit
\begin{eqnarray*}
A = \bigoplus_{i\in\mathbb{N}}A_{i}\,t^{i}\subset K_{0}[t]
\end{eqnarray*}
une sous-alg\`ebre normale de type fini sur $A_{0}$ o\`u pour tout $i\in\mathbb{N}$, $A_{i}\subset K_{0}$. 
Supposons que le corps des fractions de $A$ est exactement $K_{0}(t)$. 
Alors il existe un et un seul diviseur de Weil rationnel $D$ sur le sch\'ema affine 
$Y = \rm Spec\it\,A_{\rm 0}$ tel que $A = A_{0}[D]$.
De plus, nous avons $Y = \rm Proj\,\it A$.
\end{lemme} 
\begin{proof}
Par le th\'eor\`eme $4.3.4$ et le lemme $2.2$ de [GY], pour tout module non nul $A_{i}$,
il existe un diviseur $D_{i}\in \rm Div_{\it \mathbb{Z}}(\it Y\rm )$ tel que 
\begin{eqnarray*}
A_{i} = H^{0}(Y,\mathcal{O}_{Y}(D_{i})).
\end{eqnarray*}
Par [Bou, III.$3$, Proposition $3$], il existe $d\in\mathbb{Z}_{>0}$ tel
que la sous-alg\`ebre 
\begin{eqnarray*}
A^{(d)} := \bigoplus_{i\geq 0}A_{di}\,t^{di}
\end{eqnarray*} 
est engendr\'ee par la partie $A_{d}\,t^{d}$. 
En proc\'edant par r\'ecurrence, pour tout $i\in\mathbb{N}$, on a $D_{di} = iD_{d}$. Posons $D = D_{d}/d$.
Alors en utilisant l'hypoth\`ese de normalit\'e des anneaux $A$ et $A_{0}[D]$, nous obtenons 
que pour tout \'el\'ement homog\`ene $f\in K_{0}[t]$,
\begin{eqnarray*}
f\in A_{0}[D]\Leftrightarrow f^{d}\in A_{0}[D]  
\Leftrightarrow f^{d}\in A \Leftrightarrow f\in A.
\end{eqnarray*}
Cela donne l'\'egalit\'e $A = A_{0}[D]$.

Soit $D'$ un autre diviseur de Weil rationnel sur $Y$ tel que
$A = A_{0}[D']$. En comparant les pi\`eces gradu\'ees de
$A_{0}[D]$ et de $A_{0}[D']$, il s'ensuit que  
$\lfloor iD\rfloor = \lfloor iD'\rfloor$, pour chaque entier $i\in\mathbb{N}$.
D'o\`u $D = D'$ et ainsi la d\'ecomposition est unique.

Il reste \`a montrer l'\'egalit\'e $Y = \rm Proj\,\it A$.
Posons d'abord $V =  \rm Proj\,\it A$. Par l'exercice $5.13$ de [Ha, II]
et la proposition $3$ de [Bou, III.$1$],
nous pouvons supposer que $A = A_{0}[D]$ est engendr\'ee comme alg\`ebre sur $A_{0}$
par la partie
$A_{1}t$. Puisque que le faisceau $\mathcal{O}_{Y}(D)$ est
localement libre de rang $1$ sur $\mathcal{O}_{Y}$, il existe
$g_{1},\ldots,g_{s}\in A_{0}$ tels que
\begin{eqnarray*}
Y = \bigcup_{j = 1}^{s}Y_{g_{j}}\,\,\,\rm  avec\,\,\, 
\it Y_{g_{j}} = \rm Spec\,\it (A_{\rm 0\it})_{g_{j}}
\end{eqnarray*}
et tels que pour $e = 1,\ldots ,s$, on a 
\begin{eqnarray*}
A_{1}\otimes_{A_{0}} (A_{0})_{g_{e}} =
 \mathcal{O}_{Y}(D)(Y_{g_{e}}) = h_{e}\cdot A_{0}, 
\end{eqnarray*}
pour $h_{e}\in K_{0}^{\star}$. Soit $\pi:V \rightarrow Y$ le
morphisme naturel induit par l'inclusion $A_{0}\subset A$.
L'image inverse de l'ouvert $Y_{g_{e}}$ sous l'application 
$\pi$ est
\begin{eqnarray*}
\rm Proj \it\, A\otimes_{A_{\rm 0\it}} (A_{\rm 0\it})_{g_{e}} =
\rm Proj \it\, (A_{\rm 0\it })_{g_{e}}
[A_{\rm 1 \it}\otimes_{A_{\rm 0\it}} (A_{\rm 0\it})_{g_{e}}t] =
\rm Proj \it\, (A_{\rm 0\it})_{g_{e}}[h_{e}t] = Y_{g_{e}},
\end{eqnarray*}
et les recollements sont les m\^emes. Ainsi, l'application $\pi$ permet d'identifier $Y$ avec $V$, comme demand\'e.  
\end{proof}
Comme cons\'equence des arguments de d\'emonstration de [FZ, $3.9$],
nous obtenons le corollaire suivant.
\begin{corollaire}
Soient $A_{0}$ un anneau de Dedekind avec corps des fractions 
$K_{0}$ et $t$ une variable sur $K_{0}$. Consid\'erons
la sous-alg\`ebre
\begin{eqnarray*}
A = A_{0}[f_{1}t^{m_{1}},\ldots, f_{r}t^{m_{r}}]\subset K_{0}[t] 
\end{eqnarray*}
o\`u $m_{1},\ldots ,m_{r}$ sont des entiers strictement positifs
et les \'el\'ements
$f_{1},\ldots , f_{r}\in K_{0}^{\star}$ sont pris de sorte que le corps des fractions de $A$ est 
le corps $K_{0}(t)$. Alors la normalisation de l'anneau int\`egre $A$ est $A_{0}[D]$, o\`u 
$D$ est le diviseur de Weil rationnel 
\begin{eqnarray*}
 D = -\min_{1\leq i\leq r}\frac{\rm div\,\it f_{i}}{m_{i}}\,.
\end{eqnarray*}
\end{corollaire}

\section{Alg\`ebres multigradu\'ees normales sur un anneau de Dedekind et diviseurs poly\'edraux}

Soient $A_{0}$ un anneau de Dedekind et $K_{0}$ son corps des fractions. 
\'Etant donn\'e un r\'eseau $M$, le but de cette section est d'\'etudier 
les sous-alg\`ebres $M$-gradu\'ees normales de $K_{0}[M]$ de type fini
sur $A_{0}$. Remarquons que le fait de demander que ces alg\`ebres
aient le m\^eme corps des fractions que celui de $K_{0}[M]$ n'est pas une hypoth\`ese 
restrictive. Nous montrons ci-apr\`es que ces alg\`ebres admettent une
description combinatoire faisant intervenir des diviseurs poly\'edraux.  

Dans la suite, nous fixons conform\'ement aux notations du premier chapitre,
des r\'eseaux duaux $M, N$ et un c\^one poly\'edral saillant 
$\sigma\subset N_{\mathbb{Q}}$. La d\'efinition suivante introduit la notion 
de diviseurs poly\'edraux sur
un anneau de Dedekind.
\begin{definition}
Soit $A_{0}$ un anneau de Dedekind.
Consid\'erons le sous-ensemble $Z$ des points ferm\'es du sch\'ema affine 
$Y=\rm Spec\,\it A_{\rm 0\it}$. Un diviseur \em $\sigma$-poly\'edral $\mathfrak{D}$
\rm sur $A_{0}$ est une somme formelle 
\begin{eqnarray*}
\mathfrak{D} = \sum_{z\in Z}\Delta_{z}\cdot z, 
\end{eqnarray*}
o\`u $\Delta_{z}$ appartient \`a  
$\rm Pol_{\it\sigma\rm}(\it N_{\mathbb{Q}}\rm)$
et pour tout $z\in Z$, en dehors d'un ensemble fini, on a $\Delta_{z} = \sigma$. 

Pour des \'el\'ements $z_{1},\ldots,z_{r}$ de $Z$ tels que pour tout $z\in Z$ et 
tout $i = 1,\ldots, r$,
$z\neq z_{i}$ implique $\Delta_{z} = \sigma$, si la mention de $A_{0}$ est claire,
alors nous notons
\begin{eqnarray*}
\mathfrak{D} = \sum_{i = 1}^{r}\Delta_{z_{i}}\cdot z_{i}.
\end{eqnarray*}
\end{definition}
En partant d'un diviseur $\sigma$-poly\'edral 
$\mathfrak{D}$, nous construisons une alg\`ebre $M$-gradu\'ee
sur $A_{0}$ de la m\^eme mani\`ere que dans [AH, Section $3$]. 

\begin{rappel}
Soit $m\in\sigma^{\vee}$. \em L'\'evaluation \rm de $\mathfrak{D}$ en un vecteur $m\in\sigma^{\vee}$
est le diviseur de Weil rationnel
\begin{eqnarray*}
\mathfrak{D}(m) = \sum_{z\in Z}h_{\Delta_{z}}(m)\cdot z\,\,\,\,\rm avec\,\,\,\it h_{\rm\Delta_{\it z}\it}(m) = 
\min_{v\in \rm\Delta_{\it z}\it}\langle m,v\rangle.
\end{eqnarray*}
Par analogie avec les notations de [FZ] pour les anneaux gradu\'es, nous d\'esignons par $A_{0}[\mathfrak{D}]$ 
le sous-anneau $M$-gradu\'e 
\begin{eqnarray*}
\bigoplus_{m\in\sigma^{\vee}_{M}}A_{m}\chi^{m} \subset K_{0}[M]
\,\,\,\rm avec\,\,\,\it
A_{m} = H^{\rm 0 \it}\left(Y,\mathcal{O}_{Y}\rm 
\left(\it \lfloor \mathfrak{D}(m)\rfloor \right) \right). 
\end{eqnarray*}
\end{rappel} 
\begin{notation}
Soit 
\begin{eqnarray*}
 f = (f_{1}\chi^{m_{1}},\ldots, f_{r}\chi^{m_{r}})
\end{eqnarray*}
un $r$-uplet d'\'el\'ements homog\`enes de $K_{0}[M]$. Ici on sous-entend que
chaque $f_{i}$ appartient \`a $K_{0}^{\star}$. 
Supposons que les vecteurs $m_{1},\ldots, m_{r}$ engendrent le 
c\^one $\sigma^{\vee}$. 
Nous d\'esignons par $\mathfrak{D}[f]$ le diviseur $\sigma$-poly\'edral
\begin{eqnarray*}
\sum_{z\in Z}\Delta_{z}[f]\cdot z
\rm \,\,\,\, avec\,\,\,\,\it
\rm \Delta_{\it z}\it[f] = \left\{\, v\in N_{\mathbb{Q}}\,|\,
\left\langle m_{i},v \right\rangle\geq 
-\rm ord_{\it z}\it\,f_{i},\,\, i = \rm 1,2,\it\ldots, r\,\right\}.
\end{eqnarray*}
Notons que dans la section $4.5$, nous utilisons une notation analogue 
pour les diviseurs poly\'edraux sur une courbe projective r\'eguli\`ere ;
nous rempla\c cons donc l'ensemble $Z$ par la courbe projective r\'eguli\`ere $C$.
\end{notation}
Le r\'esultat principal de cette section est le th\'eor\`eme suivant.
Pour une d\'emonstration de la partie $\rm (iii)$, nous r\'ef\'erons le lecteur aux arguments
de d\'emonstration du th\'eor\`eme $3.4.4$ (voir aussi [La, $2.4$]).
\begin{theorem}
Soient $A_{0}$ un anneau de Dedekind avec corps des fractions $K_{0}$ et 
$\sigma\subset N_{\mathbb{Q}}$ un c\^one poly\'edral saillant. 
Alors les assertions suivantes sont vraies.
\begin{enumerate}
\item[\rm (i)]
Si $\mathfrak{D}$ est un diviseur $\sigma$-poly\'edral sur $A_{0}$ 
alors l'alg\`ebre $A_{0}[\mathfrak{D}]$
est normale, noeth\'erienne, et a son corps des fractions \'egal \`a celui de $K_{0}[M]$.
\item[\rm (ii)] 
R\'eciproquement, soit 
\begin{eqnarray*}
A = \bigoplus_{m\in\sigma^{\vee}_{M}}A_{m}\chi^{m}\subset K_{0}[M]
\end{eqnarray*}
une sous-alg\`ebre $M$-gradu\'ee normale noeth\'erienne sur $A_{0}$ avec c\^one des poids 
$\sigma^{\vee}$ et $A_{m}\subset K_{0}$, pour tout $m\in\sigma^{\vee}_{M}$. 
Supposons que les anneaux int\`egres $A$ et $K_{0}[M]$ 
ont le m\^eme corps des fractions.
Alors il existe un unique diviseur $\sigma$-poly\'edral $\mathfrak{D}$ sur $A_{0}$ 
tel que $A = A_{0}[\mathfrak{D}]$.
\item[\rm (iii)]
De fa\c con plus explicite, soit 
\begin{eqnarray*}
 f = (f_{1}\chi^{m_{1}},\ldots,f_{r}\chi^{m_{r}})
\end{eqnarray*}
un $r$-uplet d'\'el\'ements homog\`enes de $K_{0}[M]$
avec $m_{1},\ldots, m_{r}$ des vecteurs non nuls engendrant 
le r\'eseau $M$. Alors la normalisation de l'anneau 
\begin{eqnarray*}
A = A_{0}[f_{1}\chi^{m_{1}},\ldots ,f_{r}\chi^{m_{r}}] 
\end{eqnarray*}
est exactement $A_{0}[\mathfrak{D}[f]]$ (voir $4.4.3$).
\end{enumerate} 
\end{theorem}
Commen\c cons par un exemple \'el\'ementaire.
\begin{exemple}
Posons ici $A_{0} = \mathbb{Z}$.
Soient $x,y$ des variables ind\'ependantes sur $\mathbb{Q}$.
Consid\'erons le sous-anneau $\mathbb{Z}^{2}$-gradu\'e
\begin{eqnarray*}
A = \mathbb{Z}\left[\,\frac{2}{3}\,xy^{2},\,\frac{1}{9}\,x,\,
\frac{4}{3}\,x^{2}y\,\right]
\subset \mathbb{Q}(x,y). 
\end{eqnarray*}
Nous allons calculer la normalisation de $A$.
Notons que $N_{\mathbb{Q}}$ est 
identifi\'e avec le plan rationnel $\mathbb{Q}^{2}$.
Par le th\'eor\`eme $4.4.4$, nous avons $\bar{A} = 
A_{0}[\mathfrak{D}]$ o\`u 
$\mathfrak{D} = \Delta_{2}\cdot(2) + \Delta_{3}\cdot(3)$.
Les coefficients $\Delta_{2}$ et $\Delta_{3}$ sont donn\'es
par les \'egalit\'es suivantes.
\begin{eqnarray*}
\Delta_{2} = \left\{(v_{1},v_{2})\in\mathbb{Q}^{2}\,|
\,v_{1}+2v_{2}\geq -1,\,v_{1}\geq 0,\,
2v_{1}+v_{2}\geq -2\right\} 
\end{eqnarray*}
et
\begin{eqnarray*}
\Delta_{3} = \left\{(v_{1},v_{2})\in\mathbb{Q}^{2}\,|
\,v_{1}+2v_{2}\geq 1,\,v_{1}\geq 2,\,
2v_{1}+v_{2}\geq 1\right\}.
\end{eqnarray*}
Plus pr\'ecis\'ement, le c\^one des poids de $A$ est $\omega = \mathbb{Q}_{\geq 0}(1,2) + \mathbb{Q}_{\geq 0}(1,0)$.
Pour tout $(m_{1},m_{2})\in\omega_{\mathbb{Z}^{2}}$, on a 
\begin{eqnarray*}
\mathfrak{D}(m_{1},m_{2}) = 
-\frac{m_{2}}{2}\cdot(2)+
\left(2m_{1} -\frac{1}{2}m_{2}\right)\cdot(3).
 \end{eqnarray*}
Les pi\`eces gradu\'ees sont donn\'ees par
\begin{eqnarray*}
A_{0}[\mathfrak{D}] 
= \bigoplus_{(m_{1},m_{2})\in\omega_{\mathbb{Z}^{2}}}
H^{0}(Y,\mathcal{O}_{Y}(
\lfloor\mathfrak{D}(m_{1},m_{2})\rfloor))
\,x^{m_{1}}y^{m_{2}} 
\end{eqnarray*}
o\`u $Y = \rm Spec\,\it\mathbb{Z}$. En fait, 
\begin{eqnarray}
A_{0}[\mathfrak{D}] = \mathbb{Z}
\left[\,\frac{1}{9}\,x,\,
\frac{2}{3}\,xy,\,\frac{2}{3}\,xy^{2}\,\right]. 
\end{eqnarray}
En effet, soit $(m_{1},m_{2})\in \omega_{\mathbb{Z}^{2}}$
et supposons que $m_{2} = 2r$ est pair. Alors l'entier 
$m_{1}-r$ est positif. La pi\`ece gradu\'ee $A_{(m_{1},m_{2})}$ de $A_{0}[\mathfrak{D}]$
correspondante au couple $(m_{1},m_{2})$ est 
\begin{eqnarray*}
A_{(m_{1},m_{2})} = 
\mathbb{Z}\,\frac{2^{r}}{3^{2m_{1}-r}}\,x^{m_{1}}y^{m_{2}} = 
\mathbb{Z}\, \left(\frac{1}{9}\,x\right)^{m_{1}-r}\cdot
\left(\frac{2}{3}\,xy^{2}\right)^{r}. 
\end{eqnarray*}
Supposons que $m_{2} = 2r+1$ est impair. Alors $m_{1}-(r+1)\geq 0$
et 
\begin{eqnarray*}
 A_{(m_{1},m_{2})} =
\mathbb{Z}\,\frac{2^{r+1}}{3^{2m_{1}-(r+1)}}
\,x^{m_{1}}y^{m_{2}}
 = \mathbb{Z}\,\frac{2}{3}\,xy\cdot 
\left(\frac{1}{9}\,x\right)^{m_{1}-(r+1)}\cdot
\left(\frac{2}{3}\,xy^{2}\right)^{r}.
\end{eqnarray*}
Ainsi, toutes les pi\`eces gradu\'ees de $A_{0}[\mathfrak{D}]$ sont engendr\'ees par les
\'el\'ements $\frac{1}{9}\,x,
\frac{2}{3}\,xy$, $\frac{2}{3}\,xy^{2}$. On conclut que l'\'egalit\'e $(4.1)$ est vraie.
\end{exemple}
Dans l'exemple suivant, l'anneau de Dedekind $A_{0}$ n'est pas principal.
\begin{exemple}
Pour un corps de nombres $K$, le groupe des classes $\rm Cl\,\it K$  
est le quotient du groupe des id\'eaux fractionnaires de l'anneau des entiers de $K$ par
le sous-groupe des id\'eaux fractionnaires principaux. 
En d'autres termes,
$\rm Cl\,\it K \rm = Pic\,\it Y$ o\`u $Y = \rm Spec\,\it \mathbb{Z}_{K}$
est le sch\'ema affine associ\'e \`a l'anneau des entiers de $K$. 
Il est connu que le groupe $\rm Cl\,\it K$ est fini. 
De plus, l'anneau $\mathbb{Z}_{K}$ est principal si et seulement 
si $\rm Cl\,\it K$ est trivial.

Donnons un exemple o\`u $\mathbb{Z}_{K}$ n'est pas principal.
Posons $K = \mathbb{Q}(\sqrt{-5})$.
Alors $\mathbb{Z}_{K} = \mathbb{Z}[\sqrt{-5}]$ et le groupe $\rm Cl\,\it K$
est isomorphe \`a $\mathbb{Z}/2\mathbb{Z}$. 
Un ensemble de repr\'esentants dans $\rm Cl\,\it K$ 
est donn\'e par les id\'eaux fractionnaires
$\mathfrak{a}=(2,\,1+\sqrt{-5})$ et $\mathbb{Z}_{K}$.
\'Etant donn\'ees $x,y$ deux variables ind\'ependantes sur $K$,
consid\'erons l'anneau $\mathbb{Z}^{2}$-gradu\'e
\begin{eqnarray*}
A = \mathbb{Z}_{K}\left[\,3\,x^{2}y,\,2\,y,\,
6\,x\,\right].  
\end{eqnarray*}
D\'ecrivons la normalisation de $A$. 
En d\'esignant respectivement par $\mathfrak{b}$, 
$\mathfrak{c}$ les id\'eaux premiers $(3,\,1+\sqrt{-5})$ et 
$(3,\,1-\sqrt{-5})$, nous avons les d\'ecompositions
\begin{eqnarray*}
(2) = \mathfrak{a}^{2},\,\,\,
(3) = \mathfrak{b}\cdot\mathfrak{c}. 
\end{eqnarray*}
Observons que les id\'eaux 
$\mathfrak{a}$, $\mathfrak{b}$, 
$\mathfrak{c}$ sont distincts.
Ainsi,
\begin{eqnarray*}
\rm div\, 2 = 2\cdot\it\mathfrak{a}\,\,\,\,\rm et
\it\,\,\,\,\rm div\,3 = \it\mathfrak{b} + \mathfrak{c},  
\end{eqnarray*}
o\`u $\mathfrak{a},\mathfrak{b},\mathfrak{c}$ sont vus comme des points
ferm\'es de $Y=\rm Spec\,\it\mathbb{Z}_{K}$.
Soit $\mathfrak{D}$ le diviseur poly\'edral sur $\mathbb{Z}_{K}$ 
donn\'e par $\Delta_{\mathfrak{a}}\cdot\mathfrak{a}+
\Delta_{\mathfrak{b}}\cdot\mathfrak{b}+
\Delta_{\mathfrak{c}}\cdot\mathfrak{c}$ avec coefficients poly\'edraux
\begin{eqnarray*}
 \Delta_{\mathfrak{a}} = \left\{(v_{1},v_{2})\in\mathbb{Q}^{2}\,|
\,2v_{1}+v_{2}\geq 0,\,v_{2}\geq -2,\,
v_{1}\geq -2\right\}\,\,\,\rm et
\end{eqnarray*}
\begin{eqnarray*}
  \Delta_{\mathfrak{b}} = \Delta_{\mathfrak{c}} = 
\left\{(v_{1},v_{2})\in\mathbb{Q}^{2}\,|
\,2v_{1}+v_{2}\geq -1,\,v_{2}\geq 0,\,
v_{1}\geq -1\right\}.
\end{eqnarray*}
Par le th\'eor\`eme $4.4.4$, nous obtenons $\bar{A} = A_{0}[\mathfrak{D}]$
o\`u $A_{0}=\mathbb{Z}_{K}$. Le c\^one des poids de $A$
est le premier quadrant $\omega = (\mathbb{Q}_{\geq 0})^{2}$. 
Un calcul ais\'e montre que pour tous $m_{1},m_{2}\in\mathbb{N}$,
\begin{eqnarray*}
\mathfrak{D}(m_{1},m_{2}) = 
\min\left( m_{1}-2m_{2},\,-2m_{1}+4m_{2}\right)
\cdot\mathfrak{a} +
\min\left(-\frac{m_{1}}{2},\,-m_{1} + m_{2}\right)
\cdot(\mathfrak{b} +\mathfrak{c}).
\end{eqnarray*}
En posant
\begin{eqnarray*}
\omega_{1}=\mathbb{Q}_{\geq 0}(0,1) + 
\mathbb{Q}_{\geq 0}(2,1)\,\,\,
\rm et\it\,\,\,
\omega_{\rm 2\it}=\mathbb{Q}_{\geq \rm 0} \rm(2,1) + 
\it \mathbb{Q}_{\geq \rm 0}\rm(1,0), 
\end{eqnarray*}
sur le c\^one $\omega_{1}$, nous avons
\begin{eqnarray*}
\mathfrak{D}(m_{1},m_{2}) = (m_{1}-2m_{2})\cdot\mathfrak{a}
-\frac{m_{1}}{2}\cdot(\mathfrak{b}+\mathfrak{c}),
\end{eqnarray*}
et sur $\omega_{2}$, il vient
\begin{eqnarray*}
 \mathfrak{D}(m_{1},m_{2}) = (-2m_{1}+4m_{2})\cdot\mathfrak{a}+
(-m_{1}+m_{2})\cdot(\mathfrak{b}+\mathfrak{c}).
\end{eqnarray*}
Pour $i=1,2$, nous posons aussi
\begin{eqnarray*}
 A_{\omega_{i}}=\bigoplus_{(m_{1},m_{2})\in
\omega_{i}\cap\mathbb{Z}^{2}}A_{(m_{1},m_{2})}
\end{eqnarray*}
comme \'etant la somme des pi\`eces gradu\'ees de $A_{0}[\mathfrak{D}]$  
correspondantes au mono\"ide $\omega_{i}\cap \mathbb{Z}^{2}$.
Alors $A_{\omega_{2}}$ est engendr\'ee comme alg\`ebre sur $\mathbb{Z}_{K}$
par les \'el\'ements $6x$ et $3x^{2}y$. 
Fixons un couple   
$(m_{1},m_{2})\in\omega_{1}\cap\mathbb{Z}^{2}$.
Si $m_{1} = 2r$ est pair alors $r-m_{1}\leq 0$. 
Il s'ensuit que
\begin{eqnarray*}
A_{(m_{1},m_{2})} = \mathbb{Z}_{K}\,
\left(3\,xy^{2}\right)^{r}\cdot
\left(2\,y\right)^{m_{2}-r}.  
\end{eqnarray*}
D'un autre c\^ot\'e, si $m_{1} = 2r+1$ est impair alors $m_{2}-r-1\geq 0$
et $A_{(m_{1},m_{2})}$ est l'id\'eal de 
$\mathbb{Z}_{K}$ engendr\'e par les \'el\'ements
\begin{eqnarray*}
 \left(3\,xy^{2}\right)^{r}\cdot
\left(2\,y\right)^{m_{2}-r-1}\cdot 
(3(1+\sqrt{-5})\,xy),\,\,\,
\left(3\,xy^{2}\right)^{r}\cdot
\left(2\,y\right)^{m_{2}-r-1}\cdot 6\,xy.
\end{eqnarray*}
On conclut que
\begin{eqnarray*}
\bar{A} = A_{0}[\mathfrak{D}] = 
\mathbb{Z}_{K}\left[\,
2\,y,\,\,6\,xy,\,\, 3(1+\sqrt{-5})\,xy,\,\, 
3\,x^{2}y,\,\,6x\,\right]. 
\end{eqnarray*}
\end{exemple}
Pour la d\'emonstration du th\'eor\`eme $4.4.4$, nous avons besoin de
quelques r\'esultats pr\'eliminaires. Nous commen\c cons par 
rappeler un fait bien connu [GY, $1.1$] donnant une 
\'equivalence entre les propri\'et\'es noeth\'erienne 
et de finitude d'une alg\`ebre multigradu\'ee. Notons que le prochain r\'esultat
ne se g\'en\'eralise pas pour la classe des alg\`ebres gradu\'ees par un
groupe ab\'elien $G$ arbitraire; un contre-exemple est construit dans [GY, $3.1]$.
\begin{theorem}
Soient $G$ groupe ab\'elien de type fini dont la loi est not\'ee additivement 
et $A$ un anneau $G$-gradu\'e. 
Alors les assertions suivantes sont \'equivalentes.
\begin{enumerate}
 \item[\rm (i)]
L'anneau $A$ est noeth\'erien.
 \item[\rm (ii)]
La pi\`ece gradu\'ee $A_{0}$ correspondante \`a l'\'el\'ement neutre de $G$ 
est un anneau noeth\'erien et l'alg\`ebre $A$ est de type fini
sur $A_{0}$. 
\end{enumerate}
\end{theorem}
Le prochain lemme nous permet de montrer que l'anneau $A_{0}[\mathfrak{D}]$, 
provenant d'un diviseur poly\'edral $\mathfrak{D}$ sur un anneau de Dedekind $A_{0}$, 
est noeth\'erien.

\begin{lemme}
Soient $D_{1},\ldots, D_{r}$ des diviseurs de Weil rationnels sur le sch\'ema affine  
$Y = \rm Spec\,\it A_{\rm 0\it}$. Alors l'alg\`ebre $\mathbb{Z}^{r}$-gradu\'ee
\begin{eqnarray*}
B = \bigoplus_{(m_{1},\ldots,m_{r})\in\mathbb{N}^{r}}
H^{0}\left(Y,\mathcal{O}_{Y}
\left(\left\lfloor
\sum_{i = 1}^{r}m_{i}D_{i}\right\rfloor\right)\right) 
\end{eqnarray*}
est de type fini sur $A_{0}$. 
\end{lemme}
\begin{proof}
Soit $d\in\mathbb{Z}_{>0}$ tel que pour chaque $i = 1,\ldots,r$,
le diviseur de Weil
$dD_{i}$ est entier. 
Consid\'erons le polytope entier
\begin{eqnarray*}
Q= \left\{\,(m_{1},\ldots,m_{r})\in\mathbb{Q}^{r}\,|\,\,
0\leq m_{i}\leq d,\,i = 1,\ldots,r\,\right\}. 
\end{eqnarray*}
Le sous-ensemble $Q\cap\mathbb{N}^{r}$ \'etant de cardinal fini, 
le module 
\begin{eqnarray*}
E: = \bigoplus_{(m_{1},\ldots,m_{r})\in\mathbb{N}^{r}\cap Q}
H^{0}\left(Y,\mathcal{O}_{Y}
\left(\left\lfloor
\sum_{i = 1}^{r}m_{i}D_{i}\right\rfloor\right)\right)
\end{eqnarray*}
est de type fini sur $A_{0}$ (voir l'assertion $4.3.5$). 
Soit $(m_{1},\ldots,m_{r})\in\mathbb{N}^{r}$.
\'Ecrivons $m_{i} = dq_{i}+r_{i}$ avec $q_{i},r_{i}\in\mathbb{N}$
tel que $0\leq r_{i}<d$. L'\'egalit\'e
\begin{eqnarray*}
\left\lfloor \sum_{i = 1}^{r}m_{i}D_{i}\right\rfloor =  
\sum_{i =1}^{r}q_{i}\left\lfloor dD_{i}
\right\rfloor  + \left\lfloor 
\sum_{i = 1}^{r}r_{i}D_{i}
\right\rfloor
\end{eqnarray*}
implique que tout \'el\'ement homog\`ene de $B$ peut 
\^etre exprim\'e comme un polyn\^ome en les \'el\'ements
de $E$. Si $f_{1},\ldots, f_{s}$ engendrent
le module $E$ sur $A_{0}$ alors nous avons $A = A_{0}[f_{1},\ldots,f_{s}]$.
Cela montre notre assertion.
\end{proof}
Ensuite, nous donnons une d\'emonstration de la premi\`ere partie du th\'eor\`eme $4.4.4$.
\begin{proof}
Posons $A = A_{0}[\mathfrak{D}]$. Par le th\'eor\`eme $4.3.4$, toute pi\`ece gradu\'ee de $A$
correspondant au vecteur $m\in\sigma^{\vee}_{M}$ a au moins un \'el\'ement non nul.
Puisque que le c\^one des poids $\sigma^{\vee}$ est d'int\'erieur non vide les alg\`ebres
$A$ et $K_{0}[M]$ ont donc le m\^eme corps des fractions.

Montrons que $A$ est un anneau normal. Pour cela nous adaptons l'argument habituel
au cas multigradu\'e.
\'Etant donn\'es un point ferm\'e $z\in Z$ et un \'el\'ement $v\in \Delta_{z}$, consid\'erons 
l'application 
\begin{eqnarray*}
\nu_{z,v}:K_{0}[M]-\{0\}\rightarrow\mathbb{Z} 
\end{eqnarray*}
d\'efinie comme suit. Soit $\alpha \in K_{0}[M]$ un \'el\'ement non nul
ayant pour d\'ecomposition en homog\`enes
\begin{eqnarray*}
\alpha = \sum_{i = 1}^{r}f_{i}\chi^{m_{i}}\,\,\,\rm avec\,\,\,\it f_{i}\in K_{\rm 0\it}^{\star}. 
\end{eqnarray*}
Alors nous posons  
\begin{eqnarray*}
\nu_{z,v}(\alpha) = \min_{1\leq i\leq r}\left\{\rm ord_{\it z}\it\,f_{i} 
\rm + \it\left\langle m_{i},v \right\rangle\right\}. 
\end{eqnarray*}
L'application $\nu_{z,v}$ d\'efinit une valuation discr\`ete 
sur le corps $\rm Frac\,\it A$. 
D\'esignons par $\mathcal{O}_{v,z}$ l'anneau local associ\'e. 
Par d\'efinition de l'alg\`ebre $A_{0}[\mathfrak{D}]$, on a
 
\begin{eqnarray*}
A = K_{0}[M]\cap\bigcap_{z\in Z}\bigcap_{v\in\Delta_{z}}
\mathcal{O}_{v,z}\,
\end{eqnarray*}
et $A$ est normal comme intersection d'anneaux normaux avec corps des
fractions $\rm Frac\,\it A$. 

Il reste \`a montrer que $A$ est noeth\'erien.
Par le th\'eor\`eme de la base de Hilbert, il suffit
de montrer que $A$ est de type fini sur $A_{0}$.
Soient $\lambda_{1},\ldots,\lambda_{e}\subset \sigma^{\vee}$ des c\^ones poly\'edraux r\'eguliers 
d'int\'erieur non vide formant une subdivision de $\sigma^{\vee}$ et tels que 
pour chaque $i = 1,\ldots,e$, l'application \'evaluation
\begin{eqnarray*}
\sigma^{\vee}\rightarrow \rm Div_{\it\mathbb{Q}}(\it Y),\,\,\,
m\mapsto\mathfrak{D}(m)
\end{eqnarray*}
est lin\'eaire sur $\lambda_{i}$. Fixons $i\in\mathbb{N}$ tel que 
$1\leq i\leq e$. Consid\'erons les \'el\'ements distincts $v_{1},\ldots, v_{n}$
de la partie g\'en\'eratrice minimale (base de Hilbert)\footnote{On peut montrer qu'il existe une et seule
partie g\'en\'eratrice minimale pour l'inclusion de tout mono\"ide provenant d'un c\^one poly\'edral saillant
[CLS, 1.2].
Cette partie g\'en\'eratrice est un ensemble fini en vertu du lemme de Gordan.}
du mono\"ide $\lambda_{i}\cap M$. D\'esignons par $A_{\lambda_{i}}$ l'alg\`ebre
\begin{eqnarray*}
\bigoplus_{m\in\lambda_{i}\cap M}
H^{0}(Y, \mathcal{O}_{Y}(\lfloor \mathfrak{D}(m)\rfloor )\chi^{m}.
\end{eqnarray*}
Alors les vecteurs $v_{1},\ldots,v_{n}$ forment une base
du r\'eseau $M$ et donc nous avons un isomorphisme
d'alg\`ebres
\begin{eqnarray*}
A_{\lambda_{i}} \simeq 
\bigoplus_{(m_{1},\ldots,m_{n})\in\mathbb{N}^{n}}
H^{0}\left(Y,\mathcal{O}_{Y}\left(\left\lfloor
\sum_{i = 1}^{n}m_{i}\,\mathfrak{D}(v_{i})\right\rfloor\right)\right). 
\end{eqnarray*}
Par le lemme $4.4.8$, l'alg\`ebre $A_{\lambda_{i}}$ est de type fini sur $A_{0}$.
L'application surjective
\begin{eqnarray*}
A_{\lambda_{1}}\otimes\ldots\otimes A_{\lambda_{e}}\rightarrow A 
\end{eqnarray*}
montre que $A$ est \'egalement de type fini.
\end{proof}
Pour la deuxi\`eme partie du th\'eor\`eme $4.4.4$, nous avons besoin du lemme suivant. 
\begin{lemme}
Supposons que $A$ v\'erifie les hypoth\`eses de l'\'enonc\'e $4.4.4\,\rm (ii)$.
Alors les assertions suivantes sont vraies.
\begin{enumerate}
\item[\rm (i)] Pour chaque $m\in\sigma^{\vee}_{M}$, nous avons $A_{m}\neq \{0\}$.
En d'autres termes, le mono\"ide des poids de l'alg\`ebre $M$-gradu\'ee $A$ 
est exactement $\sigma^{\vee}_{M}$.
\item[\rm (ii)] Si $L\subset M_{\mathbb{Q}}$ est un c\^one poly\'edral 
contenu dans $\sigma^{\vee}$ alors l'anneau
\begin{eqnarray*}
A_{L}:= \bigoplus_{m\in L\cap M}A_ {m}\chi^{m} 
\end{eqnarray*}
est normal et noeth\'erien. 
\end{enumerate}
\end{lemme}
\begin{proof}
Soit
\begin{eqnarray*}
 S=\left\{m\in\sigma^{\vee}_{M},\,A_{m}\neq \{0\}\right\}
\end{eqnarray*}
le mono\"ide des poids de $A$. Supposons 
que $S\neq \sigma^{\vee}_{M}$. Alors il existe
$e\in\mathbb{Z}_{>0}$ et $m\in M$ tels que $m\not\in S$ et $e\cdot m\in S$. 
Puisque $A$ est un anneau noeth\'erien, par [GY, Lemma $2.2$] le module $A_{em}$
sur l'anneau $A_{0}$ est un id\'eal fractionnaire de $A_{0}$. Par le th\'eor\`eme $4.3.4$, 
nous obtenons l'\'egalit\'e
\begin{eqnarray*}
A_{em} = H^{0}(Y,\mathcal{O}_{Y}(D_{em})), 
\end{eqnarray*}
pour un diviseur de Weil entier
$D_{em}\in \rm Div_{\it \mathbb{Z}\rm}(\it Y\rm )$.
Soit $f$ une section non nulle de 
\begin{eqnarray*}
H^{0}\left(Y,\mathcal{O}_{Y}\left(
\left\lfloor \frac{D_{em}}{e}\right\rfloor\right)\right).
\end{eqnarray*}
Cet \'el\'ement existe en vertu du th\'eor\`eme $4.3.4$. Nous avons les
in\'egalit\'es 
\begin{eqnarray*}
\rm div\,\it f^{e}\rm\geq \it 
-e\left\lfloor \frac{D_{em}}{e}\right\rfloor
\rm\geq\it -D_{em}.   
\end{eqnarray*}
La normalit\'e de $A$ implique $f\in A_{m}$. Cela contredit notre hypoth\`ese
et donne $\rm (i)$. 

Pour la deuxi\`eme assertion, nous notons que par $4.4.7$ et par l'argument
de d\'emonstration de [AH, Lemma $4.1$], l'alg\`ebre $A_{L}$ 
est noeth\'erienne. 

Il reste \`a montrer que $A_{L}$ est normal.
Soit $\alpha\in\rm Frac\,\it A_{L}$ un \'el\'ement entier sur $A_{L}$. 
Par normalit\'e de $A$ et de $K_{0}[\chi^{m}]$,
nous obtenons $\alpha\in A\cap K_{0}[\chi^{m}] = A_{L}$ et donc
$A_{L}$ est normal. 
\end{proof}
Dans la suite, nous introduisons quelques notations utiles de g\'eom\'etrie 
convexe.
\begin{notation}
Soient 
\begin{eqnarray*}
 (m_{i},e_{i}),\,\,  i = 1,\ldots,r,
\end{eqnarray*}
des \'el\'ements de $M\times \mathbb{Z}$
tels que les vecteurs $m_{1},\ldots,m_{r}$ sont non nuls et engendrent
le r\'eseau $M$. Alors le c\^one $\omega = \rm Cone(\it m_{\rm 1\it},\ldots, m_{r})$
est de pleine dimension dans $M_{\mathbb{Q}}$.
Consid\'erons le $\omega^{\vee}$-poly\`edre 
\begin{eqnarray*}
\Delta =  \left\{\, v\in N_{\mathbb{Q}},\,
\left\langle m_{i},v \right\rangle\geq 
-e_{i},\,\, i = 1,2,\ldots, r\,\right\}.
\end{eqnarray*}
Soit 
$L = \mathbb{Q}_{\geq 0}\cdot m$ un c\^one de dimension
$1$ contenu dans $\omega$ avec vecteur primitif $m$.
En d'autres termes, l'\'el\'ement $m$ engendre le mono\"ide 
$L\cap M$. D\'esignons par $\mathscr{H}_{L}$ la base de Hilbert dans
le r\'eseau $\mathbb{Z}^{r}$ du c\^one poly\'edral saillant non nul
\begin{eqnarray*}
p^{-1}(L)\cap (\mathbb{Q}_{\geq 0})^{r},\,\,\,\,\rm avec 
\,\,\,\,\it p:\mathbb{Q}^{r}\rightarrow M_{\mathbb{Q}}.
\end{eqnarray*}
L'application lin\'eaire $p$ envoie la base canonique sur la famille de vecteurs
$(m_{1},\ldots,m_{r})$. 
Nous posons
\begin{eqnarray*}
 \mathscr{H}_{L}^{\star} =
\left\{\,(s_{1},\ldots,s_{r})\in\mathscr{H}_{L}\,,\,\,\,
\sum_{i = 1}^{r}s_{i}\cdot m_{i}\neq 0\,\right\}. 
\end{eqnarray*}
Pour chaque vecteur $(s_{1},\ldots,s_{r})\in 
\mathscr{H}_{L}^{\star}$, il existe un unique entier $\lambda(s_{1},\ldots,s_{r})\in\mathbb{Z}_{>0}$
tel que  
\begin{eqnarray*}
\sum_{i = 1}^{r}s_{i}\cdot m_{i} = \lambda(s_{1},\ldots,s_{r})\cdot m.
\end{eqnarray*}
\end{notation}
Les arguments de la d\'emonstration du lemme suivant utilisent seulement des
faits \'el\'ementaires d'alg\`ebre commutative et de g\'eom\'etrie 
convexe. Ils permettent d'obtenir la pr\'esentation d'Altmann-Hausen \'enonc\'ee dans 
le th\'eor\`eme $4.4.4\, \rm (ii) $.
\begin{lemme} 
Consid\'erons les m\^emes notations que dans $4.4.10$.
Posons $h_{\Delta}(m) = 
\min_{v\in\Delta}\langle m,v\rangle$.
Nous avons l'\'egalit\'e
\begin{eqnarray*}
h_{\Delta}(m)
 = -\min_{(s_{1},\ldots,s_{r})\,\in\,\mathscr{H}_{L}^{\star}}
\,\frac{\sum_{i = 1}^{r}s_{i}\cdot e_{i}}{
\lambda(s_{1},\ldots, s_{r})}\,.
\end{eqnarray*}
\end{lemme}
\begin{proof}
Soit $t$ une variable sur le corps $\mathbb{C}$. 
Consid\'erons la sous-alg\`ebre $M$-gradu\'ee 
\begin{eqnarray*}
A = \mathbb{C}[t][t^{e_{1}}\chi^{m_{1}},\ldots,t^{e_{r}}\chi^{m_{r}}]\subset 
\mathbb{C}(t)[M]. 
\end{eqnarray*}
Le corps des fractions de $A$ est le m\^eme que celui de
$\mathbb{C}(t)[M]$. En utilisant les r\'esultats de [Ho], la normalisation de 
l'alg\`ebre $A$ est 
\begin{eqnarray*}
\bar{A} = \mathbb{C}[\omega_{0}\cap(M\times\mathbb{Z})],\,\,\,\rm  
avec\,\,\,\it \omega_{\rm 0\it }\subset M_{\mathbb{Q}}\times\mathbb{Q}
\end{eqnarray*}
le c\^one rationnel engendr\'e par les vecteurs $(0,1),(m_{1},e_{1}),\ldots,(m_{r},e_{r})$. 
En fait, un calcul direct montre que  
\begin{eqnarray*}
\omega_{0} = \{(w,-\min\left\langle w,\Delta\right\rangle + e)\,|\, w\in\omega, \, e\in\mathbb{Q}_{\geq 0}\}
\end{eqnarray*}
et donc 
\begin{eqnarray*}
 \bar{A} = \bigoplus_{m\in\omega\cap M}
H^{0}(\mathbb{A}^{1}_{\mathbb{C}},
\mathcal{O}_{\mathbb{A}^{1}_{\mathbb{C}}}(
\lfloor h_{\Delta}(m)\rfloor \cdot 0))\chi^{m}
\end{eqnarray*}
avec $\mathbb{A}^{1}_{\mathbb{C}}=\rm Spec\,\it 
\mathbb{C}[t]$.

Le sous-r\'eseau $G\subset M$ engendr\'e par la partie
$p(\mathscr{H}_{L}^{\star})$ est un sous-groupe de
 $\mathbb{Z}\cdot m$. Par cons\'equent, il existe un unique entier 
$d\in\mathbb{Z}_{>0}$ tel que $G = d\,\mathbb{Z}\cdot m$.
Pour un \'el\'ement $m'\in\omega\cap M$, nous d\'esignons par $A_{m'}$ 
(resp. $\bar{A}_{m'}$) la pi\`ece gradu\'ee de $A$ (resp. $\bar{A}$) correspondante \`a $m'$. 
Alors la normalisation $\bar{A}^{(d)}_{L}$ de l'alg\`ebre 
\begin{eqnarray*}
A^{(d)}_{L} : = 
\bigoplus_{s\geq 0}A_{sdm}\chi^{sdm}
\rm \,\,\,est\,\,\, \it 
B_{L} = \bigoplus_{s\geq \rm 0\it}
\bar{A}_{sdm}\chi^{sdm}.
\end{eqnarray*} 
De plus, 
\begin{eqnarray*}
A_{L} =\bigoplus_{s\geq 0}A_{sm}\chi^{sm} 
\end{eqnarray*}
est engendr\'ee comme alg\`ebre sur $\mathbb{C}[t]$ par les \'el\'ements 
\begin{eqnarray*}
f_{(s_{1},\ldots,s_{r})} :=
\prod_{i = 1}^{r}(t^{e_{i}}\chi^{m_{i}})^{s_{i}} = 
t^{\sum_{i = 1}^{r}s_{i}e_{i}}\chi^{\lambda(s_{1},\ldots, s_{r})m},
\end{eqnarray*}
o\`u $(s_{1},\ldots,s_{r})$ parcourt $\mathscr{H}_{L}^{\star}$.
Par le choix de l'entier $d$, nous avons donc 
$A^{(d)}_{L} = A_{L}$. En consid\'erant
la $G$-graduation de $A^{(d)}_{L}$,
pour tout $(s_{1},\ldots,s_{r})\in
\mathscr{H}_{L}^{\star}$, 
l'\'el\'ement
$f_{(s_{1},\ldots,s_{r})}$ de l'anneau gradu\'e $A^{(d)}_{L}$
a pour degr\'e
\begin{eqnarray*}
\rm deg\,\it f_{
(s_{\rm 1\it},\ldots, s_{r})} := 
\frac{\lambda(
s_{\rm 1\it},\dots, s_{r})}{d}\,\,.
\end{eqnarray*}
En posant 
\begin{eqnarray*}
 D = -\min_{(s_{1},\ldots,s_{r})
\in\mathscr{H}_{L}^{\star}}
\frac{\rm div\it\,
f_{(s_{\rm 1\it},\ldots,s_{r})}}
{\rm deg\,\it f_{
(s_{\rm 1\it},\ldots, s_{r})}}
=
-\min_{(s_{1},\ldots,s_{r})\in\mathscr{H}_{L}^{\star}}
d\cdot \frac{\sum_{i = 1}^{r}s_{i}e_{i}}{
\lambda(s_{1},\ldots, s_{r})}\cdot 0,
\end{eqnarray*} 
par le corollaire $4.3.7$, nous obtenons 
\begin{eqnarray*}
\bar{A}^{(d)}_{L} = 
\bigoplus_{s\geq 0}
H^{0}
(\mathbb{A}^{1}_{\mathbb{C}},
\mathcal{O}_{\mathbb{A}^{1}_{\mathbb{C}}}
(\lfloor sD\rfloor))\chi^{sdm}.
\end{eqnarray*}
L'\'egalit\'e $\bar{A}^{(d)}_{L} = B_{L}$
implique que pour tout entier $s\geq 0$,
\begin{eqnarray*}
 H^{0}(\mathbb{A}^{1}_{\mathbb{C}},
\mathcal{O}_{\mathbb{A}^{1}_{\mathbb{C}}}(
\lfloor h_{\Delta}(sd\cdot m)
\rfloor \cdot 0)) =  H^{0}
(\mathbb{A}^{1}_{\mathbb{C}},
\mathcal{O}_{\mathbb{A}^{1}_{\mathbb{C}}}(\lfloor sD\rfloor)).
\end{eqnarray*}
D'o\`u par le lemme $4.3.6$, il vient
\begin{eqnarray*}
D = 
h_{\Delta}( d\cdot m)
\cdot 0. 
\end{eqnarray*}
En divisant par $d$, nous obtenons la formule d\'esir\'ee.
\end{proof}
Soit $A$ une alg\`ebre $M$-gradu\'ee satisfaisant les hypoth\`eses
de $4.4.4\,\rm (ii)$. En utilisant la pr\'esentation D.P.D. sur chaque
c\^one de dimension $1$ contenu dans le c\^one des poids $\sigma^{\vee}$ 
(voir le lemme $4.3.6$), 
nous pouvons construire une application
\begin{eqnarray*}
\sigma^{\vee}\rightarrow \rm Div_{\mathbb{Q}}(\it Y\rm),\,\,\,
\it m\mapsto D_{m} 
\end{eqnarray*}
\`a valeurs dans l'espace vectoriel des diviseurs de Weil rationnels de 
$Y = \rm Spec\,\it A_{\rm 0}$.
Elle est convexe, positivement homog\`ene et v\'erifie
pour tout  $m\in\sigma^{\vee}_{M}$,
\begin{eqnarray*}
A_{m} = H^{0}(C,\mathcal{O}_{C}(\lfloor D_{m}\rfloor)). 
\end{eqnarray*}
Par le lemme $4.4.11$, cette application est lin\'eaire par morceaux 
(voir [AH, $2.11$]) ou de fa\c con \'equivalente 
$m\mapsto D_{m}$ est l'application \'evaluation d'un diviseur poly\'edral.
La d\'emonstration suivante pr\'ecise cette id\'ee.
\begin{proof}[D\'emonstration de $4.4.4\,\rm (ii)$]
Par l'assertion $4.4.7$, nous pouvons consid\'erer
\begin{eqnarray*}
f = (f_{1}\chi^{m_{1}},\ldots,f_{r}\chi^{m_{r}}) 
\end{eqnarray*}
un syst\`eme de g\'en\'erateurs homog\`enes de $A$ avec des vecteurs non nuls $m_{1},\ldots,m_{r}\in M$. 
Nous utilisons ici la m\^eme notation que dans $4.4.3$. D\'esignons alors par
$\mathfrak{D}$ le diviseur $\sigma$-poly\'edral $\mathfrak{D}[f]$.
Montrons que $A = A_{0}[\mathfrak{D}]$.
Soient $L = \mathbb{Q}_{\geq 0}\cdot m$ un c\^one de dimension $1$ 
contenu dans $\omega = \sigma^{\vee}$ et $m$ le vecteur primitif de $L$
pour le r\'eseau $M$. 
Par le lemme $4.4.9$, la sous-alg\`ebre gradu\'ee
\begin{eqnarray*}
A_{L}:=\bigoplus_{m'\in L\cap M}A_{m'}\chi^{m'}\subset K_{0}[\chi^{m}]  
\end{eqnarray*}
est normale, noeth\'erienne et a le m\^eme corps des fractions 
que celui de $K_{0}[\chi^{m}]$. 
De plus, avec les m\^emes notations que dans le paragraphe 
$4.4.10$, l'alg\`ebre $A_{L}$ est engendr\'ee par l'ensemble
\begin{eqnarray*}
\left\{\,
\prod_{i = 1}^{r}(f_{i}\chi^{m_{i}})^{s_{i}},\,\,
(s_{1},\ldots,s_{r})\in\mathscr{H}_{L}^{\star}\,
\right\}. 
\end{eqnarray*}
Par le corollaire $4.3.7$, si on a 
\begin{eqnarray*}
D_{m} : = -\min_{(s_{1},\ldots,s_{r})
\,\in\,\mathscr{H}_{L}^{\star}}
\,\frac{\sum_{i=1}^{r}s_{i}\,\rm div\,\it f_{i}}
{\lambda(s_{\rm 1\it},\ldots, s_{r})} 
\end{eqnarray*}
alors par rapport \`a la variable $\chi^{m}$, on a $A_{L} = A_{0}[D_{m}]$. 
Par le lemme $4.4.11$, pour tout point ferm\'e $z\in Z$,
\begin{eqnarray*}
h_{\Delta_{z}[f]}(m) =  
-\min_{(s_{1},\ldots,s_{r})
\,\in\,\mathscr{H}_{L}^{\star}}
\,\frac{\sum_{i=1}^{r}s_{i}\,\rm ord_{\it z}\it\,f_{i}}
{\it \lambda(s_{\rm 1}\it,\ldots, s_{r})}\,.
\end{eqnarray*}
Ce qui entra\^ine $\mathfrak{D}(m) = D_{m}$.
Puisque l'\'egalit\'e est vraie pour tout vecteur primitif d'un
c\^one de dimension $1$ contenu dans $\sigma^{\vee}$,  
on conclut que $A = A_{0}[\mathfrak{D}]$.
L'unicit\'e de $\mathfrak{D}$ est ais\'ee.
\end{proof}
En utilisant des faits connus sur les anneaux de Dedekind nous obtenons l'assertion
suivante.
\begin{proposition}
Soient $A_{0}$ un anneau de Dedekind et $B_{0}$ la fermeture int\'egrale de 
$A_{0}$ dans une extension s\'eparable finie $L_{0}/K_{0}$, o\`u $K_{0} = \rm Frac\,\it A_{\rm 0}$.
Soit $\mathfrak{D} = \sum_{z\in Z}\Delta_{z}\cdot z$ un diviseur poly\'edral sur $A_{0}$
o\`u $Z\subset Y = \rm Spec\,\it A_{\rm 0}$ est l'ensemble des points ferm\'es. En posant
$Y' = \rm Spec\,\it B_{\rm 0}$ et en consid\'erant la projection naturelle $p:Y'\rightarrow Y$,
l'anneau $B_{0}$ est de Dedekind et nous avons la formule
\begin{eqnarray*}
A_{0}[\mathfrak{D}]\otimes_{A_{0}}B_{0} = B_{0}[p^{\star}\mathfrak{D}]\rm\,\,\, avec
\,\,\,\it p^{\star}\mathfrak{D} = \sum_{z\in Z}\rm \Delta_{\it z}\it\cdot p^{\star}(z). 
\end{eqnarray*}
\end{proposition}
Donnons un exemple pour illustrer l'assertion pr\'ec\'edente.
\begin{exemple}
Consid\'erons le diviseur poly\'edral  
\begin{eqnarray*}
 \mathfrak{D} = \Delta_{(t)}\cdot (t) + \Delta_{(t^{2}+1)}\cdot (t^{2}+1)
\end{eqnarray*}
sur l'anneau de Dedekind $A_{0} = \mathbb{R}[t]$ avec coefficients poly\'edraux    
\begin{eqnarray*}
\Delta_{(t)} = (-1,0) + \sigma,\,\,\, \Delta_{(t^{2}+1)} = [(0,0), (0,1)] + \sigma,
\end{eqnarray*}
et $\sigma\subset \mathbb{Q}^{2}$ est le c\^one poly\'edral engendr\'e par les vecteurs $(1,0)$ et $(1,1)$.
Un calcul ais\'e montre que 
\begin{eqnarray*}
A_{0}[\mathfrak{D}] = \mathbb{R}\left[t, -t\chi^{(1,0)}, \chi^{(0,1)}, t(t^{2}+1)\chi^{(1,-1)}\right]\simeq 
\frac{\mathbb{R}[x_{1},x_{2},x_{3},x_{4}]}{((1 + x^{2}_{1})x_{2} + x_{3}x_{4})},
\end{eqnarray*}
o\`u $x_{1},x_{2},x_{3},x_{4}$ sont des variables ind\'ependantes sur $\mathbb{R}$. 
Posons $\zeta = \sqrt{-1}$. En consid\'erant la projection naturelle 
$p:\mathbb{A}^{1}_{\mathbb{C}}\rightarrow \mathbb{A}^{1}_{\mathbb{R}}$,
nous obtenons
\begin{eqnarray*}
p^{\star}\mathfrak{D} = \Delta_{0}\cdot 0 + \Delta_{(t^{2}+1)}\cdot \zeta + \Delta_{(t^{2}+1)}\cdot (-\zeta).
\end{eqnarray*}
En posant $B_{0} = \mathbb{C}[t]$, on conclut que 
$A_{0}[\mathfrak{D}]\otimes_{\mathbb{R}}\mathbb{C} = B_{0}[p^{\star}\mathfrak{D}]$
est le complexifi\'e de $A_{0}[\mathfrak{D}]$.
\end{exemple}

\section{Alg\`ebres multigradu\'ees elliptiques et corps de fonctions alg\'ebriques d'une variable}
Dans cette section, nous \'etudions un autre type d'alg\`ebres multigradu\'ees.
Ces alg\`ebres sont d\'ecrites par un diviseur poly\'edral propre sur un corps de 
fonctions alg\'ebriques d'une variable. Fixons un corps quelconque $\mathbf{k}$.
Rappelons une d\'efinition classique. 
\begin{rappel}
Un \em corps de fonctions alg\'ebriques \rm (sous-entendu d'une variable)
sur le corps $\mathbf{k}$ est une extension de corps $K_{0}/\mathbf{k}$ v\'erifiant 
les conditions suivantes.
\begin{enumerate}
 \item[(i)]
Le degr\'e de transcendance de $K_{0}$ sur $\mathbf{k}$ est \'egal \`a $1$.
\item[(ii)]
Tout \'el\'ement de $K_{0}$ qui est alg\'ebrique sur le sous-corps $\mathbf{k}$
appartient \`a $\mathbf{k}$. 
\item[(iii)] Il existe $\alpha_{1},\ldots, \alpha_{r}\in K_{0}$ tels que $K_{0} = \mathbf{k}(\alpha_{1},\ldots,\alpha_{r})$.
\end{enumerate}
\end{rappel}
En vertu de notre convention sur les vari\'et\'es alg\'ebriques, une courbe projective r\'eguli\`ere $C$ sur 
$\mathbf{k}$ donne naturellement un corps de fonctions alg\'ebriques $K_{0}/\mathbf{k}$
en posant $K_{0} = \mathbf{k}(C)$. Comme application directe du crit\`ere valuatif de propret\'e 
(voir [EGA II, Section $7.4$]), tout corps de fonctions alg\'ebriques $K_{0}/\mathbf{k}$
est le corps des fonctions rationnelles (unique \`a isomorphisme pr\`es) d'une courbe projective r\'eguli\`ere $C$ sur $\mathbf{k}$. 
Dans le prochain paragraphe, nous rappelons comment on construit la courbe $C$ en partant
d'un corps de fonctions alg\'ebriques $K_{0}$.
\begin{rappel}
Rappelons qu'un \em anneau de valuation \rm dans $K_{0}$ est un sous-anneau
propre de $\mathcal{O} \subset K_{0}$
contenant strictement $\mathbf{k}$ et tel que pour tout \'el\'ement
non nul $f\in K_{0}$, soit on a $f\in \mathcal{O}$ soit on a $\frac{1}{f}\in\mathcal{O}$. 
Par [St, $1.1.6$] tout anneau de valuation dans $K_{0}$ est l'anneau associ\'e \`a une 
valuation discr\`ete de 
$K_{0}/\mathbf{k}$. Un sous-ensemble $P\subset K_{0}$
est appel\'e une \em place \rm de $K_{0}$
s'il existe un anneau de valuation $\mathcal{O}$
de $K_{0}$ tel que $P$ est l'id\'eal maximal de $\mathcal{O}$. 
Nous d\'esignons par $\mathscr{R}_{\mathbf{k}}\, K_{0}$ l'ensemble
de toutes les places de $K_{0}$. Cette derni\`ere est souvent appel\'ee
la \em surface de Riemann \rm de $K_{0}$.
Par [EGA II, $7.4.18$] l'ensemble 
$\mathscr{R}_{\mathbf{k}}\, K_{0}$ peut \^etre muni d'une structure de 
courbe projective $C$ sur le corps $\mathbf{k}$
tel que $K_{0} = \mathbf{k}(C)$.
\end{rappel}
D\'esormais, nous consid\'erons l'ensemble 
$C =\mathscr{R}_{\mathbf{k}}\, K_{0}$ avec sa structure de sch\'ema. Par convention
un \'el\'ement $z$ appartenant \`a $C$ est un point ferm\'e; qui correspond exactement
\`a une place de $K_{0}$.
Nous notons par $P_{z}$ la place associ\'ee \`a $z$. 
\begin{rappel}
Pour une place $P$ de $K_{0}$, nous posons $\kappa(P) = \mathcal{O}/P$, o\`u  
$\mathcal{O}$ est l'anneau de valuation de la place $P$. 
Le corps $\kappa(P)$ est une  extension finie de $\mathbf{k}$ 
[St, $1.1.15$] et nous l'appelons le \em corps r\'esiduel \rm de $P$.
Nous posons \'egalement $\kappa_{z} = \kappa(P_{z})$.
Pour une fonction rationnelle
$f\in\mathbf{k}(C)^{\star}$, son \em diviseur principal \rm associ\'e est
\begin{eqnarray*}
\rm div\,\it f = \sum_{z\in C}\rm ord_{\it z}\,\it f \cdot z,
\end{eqnarray*}
o\`u $\rm ord_{\it z}\,\it f$ est la valeur en $f$ de la valuation correspondante \`a $P_{z}$. Rappelons que le \em degr\'e \rm d'un 
diviseur de Weil rationnel $D = \sum_{z\in C}a_{z}\cdot z$ est le nombre rationnel
\begin{eqnarray*}
\rm deg\,\it D = \sum_{z\in C}[\kappa_{z} : \mathbf{k}]\, a_{z}.  
\end{eqnarray*}
Par le th\'eor\`eme $1.4.11$ dans [St], nous avons $\rm deg\,div\,\it f\rm = 0$.
\end{rappel}
\begin{rappel}
Soient $M,N$ des r\'eseaux duaux et $\sigma\subset N_{\mathbb{Q}}$ un c\^one
poly\'edral saillant. 
Un \em diviseur $\sigma$-poly\'edral \rm sur $K_{0}$ (ou sur $C$) est une
somme formelle
$\mathfrak{D} = \sum_{z\in C}\Delta_{z}\cdot z$
avec $\Delta_{z} \in\rm Pol_{\it \sigma}(\it N_{\mathbb{Q}})$ 
et $\Delta_{z} = \sigma$ pour tout $z\in C$ en dehors d'un sous-ensemble fini.
Nous consid\'erons aussi la \em fonction \'evaluation \rm
\begin{eqnarray*}
\mathfrak{D}(m) = \sum_{z\in C}h_{\Delta_{z}}(m)\cdot z \,; 
\end{eqnarray*}
qui est un diviseur de Weil rationnel sur la courbe $C$. Le \em degr\'e \rm de $\mathfrak{D}$ 
est la somme de Minkowski 
\begin{eqnarray*}
\rm deg\it\, \mathfrak{D} = 
\sum_{z\in C}[\kappa_{z}:\mathbf{k}]\cdot
\rm\Delta_{\it z}. 
\end{eqnarray*}
\'Etant donn\'e $m\in\sigma^{\vee}$, nous avons naturellement la relation
 $h_{\it \rm deg\it\, \mathfrak{D}} (m) = \rm deg\it\, \mathfrak{D}(m)$. 
\end{rappel}
Nous pouvons maintenant introduire la notion de propret\'e des diviseurs
poly\'edraux. Voir [AH, $2.7$, $2.12$] pour d'autres cas particuliers.
\begin{definition}
Un diviseur $\sigma$-poly\'edral $\mathfrak{D} = 
\sum_{z\in C}\Delta_{z}\cdot z$ est dit \em propre \rm s'il 
satisfait les assertions suivantes
\begin{enumerate}
 \item[(i)] Le poly\`edre $\rm deg\it\, \mathfrak{D}$ est strictement 
inclus dans le c\^one $\sigma$.
\item[(ii)] Si on a $\rm deg\it\, \mathfrak{D}(m) \rm = 0$
alors le vecteur $m$ appartient au bord de $\sigma^{\vee}$ et un
multiple entier non nul de $\mathfrak{D}(m)$ est un diviseur principal. 
\end{enumerate}
\end{definition}
Notre r\'esultat principal donne une description analogue
\`a $4.4.4$, mais celui-ci concerne les corps de fonctions alg\'ebriques.
Pour l'assertion $4.5.6\,\rm (iii)$, nous renvoyons le lecteur aux arguments
de d\'emonstration de $3.4.4$.
\begin{theorem}
Soient $\mathbf{k}$ un corps et 
$C = \mathscr{R}_{\mathbf{k}}\, K_{0}$ la surface de Riemann
d'un corps de fonctions alg\'ebriques d'une variable $K_{0}/\mathbf{k}$. 
Alors les assertions suivantes sont vraies.
\begin{enumerate}
 \item[\rm (i)]Soit 
\begin{eqnarray*}
A = \bigoplus_{m\in\sigma^{\vee}_{M}}
A_{m}\chi^{m} \subset K_{0}[M]
\end{eqnarray*}
une sous-alg\`ebre $M$-gradu\'ee normale noeth\'erienne sur $\mathbf{k}$ de c\^one des poids $\sigma^{\vee}$. 
On suppose que pour tout $m\in\sigma^{\vee}_{M}$, $A_{m}\subset K_{0}$ et que $A_{0} = \mathbf{k}$.
Si $A$ et $K_{0}[M]$ ont le m\^eme corps des fractions alors il existe un unique diviseur
$\sigma$-poly\'edral propre $\mathfrak{D}$ sur la courbe $C$ 
tel que $A = A[C,\mathfrak{D}]$, o\`u
\begin{eqnarray*}
A[C,\mathfrak{D}] = \bigoplus_{m\in\sigma^{\vee}_{M}}
H^{0}(C,\mathcal{O}_{C}(\lfloor \mathfrak{D}(m)\rfloor))\chi^{m}. 
\end{eqnarray*}
\item[\rm (ii)]
Soit $\mathfrak{D}$ un diviseur $\sigma$-poly\'edral propre sur $C$.
Alors l'alg\`ebre $A = A[C,\mathfrak{D}]$ est $M$-gradu\'ee, normale,
de type fini sur $\mathbf{k}$ et a le m\^eme corps des fractions que  $K_{0}[M]$. Le c\^one $\sigma^{\vee}$
est le c\^one des poids de $A$.
\item[\rm (iii)]  
De fa\c con plus explicite, soit
\begin{eqnarray*}
A = \mathbf{k}[f_{1}\chi^{m_{1}},\ldots, f_{r}\chi^{m_{r}}] 
\end{eqnarray*}
une sous-alg\`ebre $M$-gradu\'ee de $K_{0}[M]$ avec $f_{1},\ldots, f_{r}\in \mathbf{k}(C)^{\star}$
et $m_{1},\ldots, m_{r}\in M$ non nuls engendrant un c\^one rationnel $\sigma^{\vee}$. 
Posons $f = (f_{1}\chi^{m_{1}},\ldots, f_{r}\chi^{m_{r}} )$.
Supposons que $A$ et $K_{0}[M]$ aient le m\^eme corps des fractions. 
Alors le c\^one $\sigma\subset N_{\mathbb{Q}}$ est poly\'edral saillant,
$\mathfrak{D}[f]$ est un diviseur $\sigma$-poly\'edral propre sur $C$ et $\bar{A} = A[C,\mathfrak{D}[f]]$ (voir la notation $4.4.3$).
\end{enumerate}
\end{theorem}
Pour la d\'emonstration de $4.5.6$, nous avons besoin de quelques pr\'eliminaires. 
Nous commen\c cons par donner quelques r\'esultats sur les sous-alg\`ebres $A_{L}$
de l'alg\`ebre $M$-gradu\'ee $A$ satisfaisant les hypoth\`eses de l'\'enonc\'e $4.5.6\,\rm (i)$.

\begin{lemme}
Soit $A$ une alg\`ebre $M$-gradu\'ee satisfaisant 
les hypoth\`eses de $4.5.6\,\rm (i)$. 
\'Etant donn\'e un c\^one $L = \mathbb{Q}_{\geq 0}\cdot m\subset \sigma^{\vee}$
de dimension $1$ avec
vecteur primitif $m$, consid\'erons la sous-alg\`ebre
\begin{eqnarray*}
A_{L}  = \bigoplus_{m'\in L\cap M}A_{m'}\chi^{m'}. 
\end{eqnarray*}
Posons aussi
\begin{eqnarray*}
Q(A_{L})_{0}  = \left\{\,\frac{a}{b}\,|\,\,\, 
a\in A_{sm},\, b\in A_{sm},\, b\neq 0,
\,s\geq 0 \,\right\}. 
\end{eqnarray*}
Alors les assertions suivantes sont vraies.
\begin{enumerate}
\item[\rm (i)] L'alg\`ebre $A_{L}$ est normale et de type fini sur $\mathbf{k}$.
\item[\rm (ii)] Soit on a $Q(A_{L})_{0} = \mathbf{k}$, soit on a $Q(A_{L})_{0}  = K_{0}$.
\item[\rm (iii)] Si $Q(A_{L})_{0} = \mathbf{k}$ alors il existe $\beta\in K_{0}^{\star}$
et $d\in\mathbb{Z}_{>0}$ tels que $A_{L} = \mathbf{k}[\beta \chi^{dm}]$.   
\end{enumerate}
\end{lemme}
\begin{proof}
La d\'emonstration de l'assertion $(i)$ est analogue \`a celle de $4.4.9\,\rm (ii)$ et donc
nous l'omettons.

Le corps $Q(A_{L})_{0}$ est une extension de
$\mathbf{k}$ contenue dans $K_{0}$. Si le degr\'e de transcendance de
$Q(A_{L})_{0}$ sur $\mathbf{k}$ est $0$ alors par normalit\'e de $A_{L}$,
$Q(A_{L})_{0} = \mathbf{k}$. 
Nous supposons que nous sommes dans le cas contraire, \`a savoir, 
l'extension $K_{0}/Q(A_{L})_{0}$ est alg\'ebrique. Soit $\alpha\in K_{0}$.
Alors il existe des \'el\'ements $a_{1},\ldots, a_{d}\in Q(A_{L})_{0}$
avec $a_{d}\neq 0$ tels que 
\begin{eqnarray*}
\alpha^{d} = \sum_{j = 1}^{d}a_{j}\alpha^{d-j}. 
\end{eqnarray*}
Posons
\begin{eqnarray*}
I = \{i\in\{1,\ldots,d\},\,\,a_{i}\neq 0\}. 
\end{eqnarray*}
Pour tout $i\in I$, nous \'ecrivons $a_{i} = \frac{p_{i}}{q_{i}}$,
avec $p_{i}, q_{i}\in A_{L}$ des \'el\'ements homog\`enes de
m\^eme degr\'e. En consid\'erant $q =\prod_{i\in I}q_{i}$, nous obtenons
\begin{eqnarray*}
(\alpha q)^{d} = \sum_{j = 1}^{d}a_{j}q^{j}(\alpha q)^{d-j}. 
\end{eqnarray*}
La normalit\'e de $A_{L}$ donne $\alpha q\in A_{L}$. Cela montre
que $\alpha = \alpha q/q\in Q(A_{L})_{0}$. 

Pour montrer l'assertion $\rm (iii)$, nous consid\'erons $S\subset\mathbb{Z}\cdot m$ le c\^one des
poids de l'alg\`ebre gradu\'ee $A_{L}$. Puisque $L$ est contenu dans
le c\^one des poids $\sigma^{\vee}$, le mono\"ide $S$ est non trivial. 
Par cons\'equent, si $G$ est le sous-groupe
engendr\'e par $S$ alors il existe un unique entier $d\in\mathbb{Z}_{>0}$ 
tel que $G = \mathbb{Z}\,d\cdot m$. 
En posant $u = \chi^{dm}$, nous pouvons \'ecrire
\begin{eqnarray*}
A_{L} = \bigoplus_{s\geq 0} A_{sdm}u^{s}. 
\end{eqnarray*}
Ainsi, pour des \'el\'ements homog\`enes de m\^eme degr\'e $a_{1}u^{l}, a_{2}u^{l}\in A_{L}$,
nous avons $\frac{a_{1}}{a_{2}}\in Q(A_{L})^{\star}_{0} = \mathbf{k}^{\star}$
de sorte que
\begin{eqnarray*}
A_{L} = \bigoplus_{s\in S'}kf_{s}\, u^{s}, 
\end{eqnarray*}
o\`u $S':=\frac{1}{d}\,S$ et $f_{s}\in \mathbf{k}(C)^{\star}$.
Fixons des \'el\'ements homog\`enes $f_{s_{1}}u^{s_{1}},\ldots ,f_{s_{r}}u^{s_{r}}$ formant
un sous-ensemble g\'en\'erateur de l'alg\`ebre $G$-gradu\'ee $A_{L}$. Consid\'erons $d' := \rm p.g.c.d.\it (s_{\rm 1\it},\ldots, s_{r})$. 
Si $d'>1$ alors l'inclusion $S\subset dd'\,\mathbb{Z}\cdot m$ donne une contradiction. Donc $d' = 1$ 
et il existe des entiers $l_{1},\ldots,l_{r}$ tels que $1 = \sum_{i = 1}^{r}l_{i}s_{i}$. L'\'el\'ement
\begin{eqnarray*}
 \beta u =\prod_{i = 1}^{r}(f_{s_{i}}u^{s_{i}})^{l_{i}}
\end{eqnarray*}
v\'erifie alors
\begin{eqnarray*}
 \frac{(\beta u)^{s_{1}}}{f_{s_{1}}u^{s_{1}}}\in Q(A_{L})^{\star}_{0} 
= \mathbf{k}^{\star}.
\end{eqnarray*}
Par normalit\'e de $A_{L}$, on a $\beta u\in A_{L}$ et donc $A_{L} = 
\mathbf{k}[\beta u] = \mathbf{k}[\beta \chi^{dm}]$. Cela \'etablit l'assertion
$\rm (iii)$.  
\end{proof}
Le lemme suivant est bien connu; voir par exemple [De $2$, Section $3$], [Liu, \S 7.4.1, Proposition 4.4] et 
l'argument de d\'emonstration de [AH, $9.1$].
\begin{lemme}
Soient $D_{1},D_{2}, D$ des diviseurs de Weil rationnels sur la courbe $C$. Alors
les assertions suivantes sont vraies.
\begin{enumerate}
 \item[\rm (i)] Si le degr\'e de $D$ est strictement positif alors il existe $d\in\mathbb{Z}_{>0}$ 
tel que le faisceau inversible $\mathcal{O}_{C}(\lfloor dD\rfloor)$ de $\mathcal{O}_{C}$-modules 
est tr\`es ample. De plus, l'alg\`ebre gradu\'ee
\begin{eqnarray*}
B = \bigoplus_{l\geq 0}H^{0}(C,\mathcal{O}_{C}(\lfloor lD\rfloor))t^{l}, 
\end{eqnarray*}
o\`u $t$ est une variable sur $K_{0} = \mathbf{k}(C)$, est de type fini sur $\mathbf{k}$. Le corps
des fractions de $B$ est $\mathbf{k}(C)(t)$.
\item[\rm (ii)] Supposons que
pour $i = 1,2$, nous avons soit $\rm deg\,\it D_{i}
\rm >0$ ou soit $rD_{i}$ est un diviseur principal, pour $r\in\mathbb{Z}_{>0}$.
Si pour tout $s\in\mathbb{N}$, on a l'inclusion
\begin{eqnarray*}
H^{0}(C,
\mathcal{O}_{C}(\lfloor sD_{1}\rfloor))\subset  H^{0}(C,
\mathcal{O}_{C}(\lfloor sD_{2}\rfloor))
\end{eqnarray*}
alors on a l'in\'egalit\'e $D_{1}\leq D_{2}$.
\end{enumerate}
\end{lemme}

Dans le prochain corollaire, nous gardons les m\^emes notations que dans le lemme $4.5.7$.
En utilisant le th\'eor\`eme de Demazure (voir [De $2$, th\'eor\`eme 3.5]) pour les alg\`ebres gradu\'ees normales, 
nous montrons que chaque $A_{L}$ admet une pr\'esentation D.P.D. sur une m\^eme courbe
projective r\'eguli\`ere. 
\begin{corollaire}
Il existe un unique diviseur de Weil rationnel $D$ sur $C$ tel que
\begin{eqnarray*}
A_{L} = \bigoplus_{s\geq 0}H^{0}(C,
\mathcal{O}_{C}(\lfloor sD\rfloor))\chi^{sm}. 
\end{eqnarray*}
De plus, les assertions suivantes sont vraies.
\begin{enumerate}
 \item[\rm (i)] Si $Q(A_{L})_{0} = \mathbf{k}$ alors il existe $f\in K_{0}^{\star}$
et $d\in\mathbb{Z}_{>0}$ tels que
$D = \frac{\rm div\,\it f}{\it d}$. 
\item[\rm (ii)] Si $Q(A_{L})_{0} = K_{0}$ alors $\rm deg\,\it D\rm >0$.
\item[\rm (iii)] Si $f_{1}\chi^{s_{1}m},\ldots, f_{r}\chi^{s_{r}m}$ sont des \'el\'ements homog\`enes de
degr\'e $s_{1},\ldots, s_{r}>0$ de l'alg\`ebre gradu\'ee $A_{L}$ (par rapport \`a
la variable $\chi^{m}$) alors 
\begin{eqnarray*}
D = -\min_{1\leq i\leq r}\frac{\rm div\,\it f_{i}}{s_{i}}\,. 
\end{eqnarray*}
\end{enumerate}
\end{corollaire}
\begin{proof}
$\rm (i)$ Supposons que $Q(A_{L})_{0} = \mathbf{k}$. Par le lemme $4.5.7$,
on a 
$A_{L} = \mathbf{k}[\beta\,\chi^{dm}]$,
pour $\beta\in K_{0}^{\star}$ et $d\in\mathbb{Z}_{>0}$.
Ainsi, nous pouvons prendre $D = \frac{\rm div\,\beta^{-1}}{\it d}$. 
L'unicit\'e dans ce cas est facile. Cela donne l'assertion $\rm (i)$.

$\rm (ii)$ Le corps des fonctions rationnelles de la vari\'et\'e 
normale $\rm Proj\,\it A_{L}$ est
$K_{0} = Q(A_{L})_{0}$. Puisque $\rm Proj\,\it A_{L}$ est une courbe projective r\'eguli\`ere (car normale)
sur $A_{0} = \mathbf{k}$, nous pouvons l'identifier avec la surface de Riemann 
des places de $K_{0}$. Par cons\'equent, l'existence et l'unicit\'e de $D$ 
suivent du th\'eor\`eme de Demazure (voir [De $2$, th\'eor\`eme 3.5]).
De plus, $Q(A_{L})_{0}\neq \mathbf{k}$ implique que
$\rm dim_{\it \mathbf{k}}\it A_{\it sm}\rm\geq 2,$ pour un entier $s\in\mathbb{Z}_{>0}$. 
D'o\`u par [St, $1.4.12$], nous obtenons $\rm deg\, \it D\rm >0$.
  
La d\'emonstration de l'assertion $\rm (iii)$ suit de $4.5.8$ et de [FZ, $3.9$]. 
\end{proof}
Comme cons\'equence du corollaire $4.5.9$, nous pouvons \`a nouveau appliquer la formule de g\'eom\'etrie
convexe de $4.4.11$ pour obtenir l'existence du diviseur poly\'edral $\mathfrak{D}$ dans l'\'enonc\'e $4.5.6\,\rm (i)$.
 
\begin{proof}[D\'emonstration de $4.5.6\,\rm (i)$]
Gardons les m\^emes notations que dans $4.4.3$ et
$\,4.4.10$.
Soit 
\begin{eqnarray*}
f = (f_{1}\chi^{m_{1}},\ldots, f_{r}\chi^{m_{r}})
\end{eqnarray*}
un syst\`eme de g\'en\'erateurs homog\`enes de $A$ avec $f_{i} \in K_{0}^{\star}$. 
Consid\'erons un c\^one 
\begin{eqnarray*}
L = \mathbb{Q}_{\geq 0}\cdot m\subset \sigma^{\vee}
\end{eqnarray*}
de dimension $1$ 
avec vecteur primitif $m\in M$. Par le corollaire $4.5.9$, nous avons
\begin{eqnarray*}
A_{L} = \bigoplus_{s\geq 0}H^{0}(C,
\mathcal{O}_{C}(\lfloor sD_{m}\rfloor))\chi^{sm}, 
\end{eqnarray*}
 pour un unique diviseur de Weil rationnel $D_{m}$ sur $C$. Par l'argument de
d\'emonstration de [AH, Lemma $4.1$], l'alg\`ebre $A_{L}$
est engendr\'ee par 
\begin{eqnarray*}
\prod_{i = 1}^{r}(f_{i}\chi^{m_{i}})^{s_{i}}\,\,\,\rm avec\,\,\,\it  (s_{\rm 1\it},\ldots s_{r})
\in\mathscr{H}_{L}^{\star}.
\end{eqnarray*}
Par le corollaire $4.5.9\,\rm (iii)$ et le lemme $4.4.11$, nous avons $\mathfrak{D}[f](m) = D_{m}$ et donc 
$A = A[C,\mathfrak{D}[f]]$. 

Il reste \`a montrer que $\mathfrak{D} = \mathfrak{D}[f]$
est propre; l'unicit\'e de $\mathfrak{D}$ sera donn\'ee par le 
lemme $4.5.8\,\rm (ii)$. D\'esignons par $S\subset C$ l'union des supports
des diviseurs $\rm div\,\it f_{i}$, pour $i = 1,\ldots, r$.
Soit $v\in \rm deg\,\it \mathfrak{D}$. Nous pouvons \'ecrire
\begin{eqnarray*}
v = \sum_{z\in S}[\kappa(P_{z}):\mathbf{k}]\cdot v_{z}, 
\end{eqnarray*}
pour $v_{z}\in\Delta_{z}[f]$.
Par cons\'equent, pour chaque $i = 1,\ldots,r$, on a 
\begin{eqnarray*}
\langle m_{i}, \sum_{z\in S}[\kappa(P_{z}):
\mathbf{k}]\cdot v_{z}\rangle \geq -\sum_{s\in S}[\kappa(P_{z}):\mathbf{k}]\cdot
\rm ord_{\it z}\,\it f_{i} =-\rm deg\, div\,\it f_{i}\rm  = 0 
\end{eqnarray*}
et donc $\rm deg\,\it \mathfrak{D}\subset \sigma.$
Si $\rm deg\,\it \mathfrak{D} = \sigma$ alors on conclut que $\rm Frac\,\it A\neq 
\rm Frac\,\it K_{\rm 0\it}[M]$,
contredisant notre hypoth\`ese.
D'o\`u $\rm deg\,\it \mathfrak{D}\neq \sigma$.
Soit $m\in\sigma^{\vee}_{M}$ tel que
$\rm deg\,\it\mathfrak{D}(m)\rm = 0$.
Alors $m$ appartient au bord de 
$\sigma^{\vee}$.
Consid\'erons le c\^one $L$ de dimension $1$ engendr\'e par $m$.
En appliquant le corollaire $4.5.9\,\rm (i)$ pour l'alg\`ebre $A_{L}$,
nous d\'eduisons qu'un multiple entier non nul de 
$\mathfrak{D}(m)$ est un diviseur principal. Cela montre que $\mathfrak{D}$ est propre. 
\end{proof}
\begin{proof}[D\'emonstration de $4.5.6 \,\rm (ii)$]
Montrons que
$A = A[C,\mathfrak{D}]$ et $K_{0}[M]$ ont le m\^eme corps des fractions. Soit 
$L = \mathbb{Q}_{\geq 0}\cdot m$ un c\^one de dimension $1$ intersectant $\sigma^{\vee}$
avec son int\'erieur relatif et ayant $m$ comme vecteur primitif. Puisque
$\rm deg\,\it \mathfrak{D}(m)\,\rm >0$, par le lemme $4.5.8\,\rm (i)$, 
$\rm Frac\,\it A_{L} = K_{\rm 0}(\it\chi^{m})$. 
Comme cons\'equence, $\sigma^{\vee}$ est le c\^one des poids de l'alg\`ebre $M$-gradu\'ee $A$. 
La d\'emonstration de la normalit\'e de $A$ est analogue \`a celle de $4.4.4\,\rm (i)$.

Montrons que $A$ est de type fini sur $\mathbf{k}$. D\'esignons par $\rm relint\,\it\sigma^{\vee}$
l'int\'erieur relatif du c\^one $\sigma^{\vee}$. Tout d'abord, nous pouvons
consid\'erer une subdivision de $\sigma^{\vee}$ par des c\^ones poly\'edraux r\'eguliers 
$\omega_{1},\ldots, \omega_{s}$ tels que pour tout $i = 1,\ldots, s$, on a 
$\omega_{i}\cap \rm relint\,\it\sigma^{\vee} \neq \emptyset$, $\omega_{i}$ est de pleine dimension,
et $\mathfrak{D}$ est lin\'eaire sur $\omega_{i}$. Fixons $1\leq i\leq s$ et un entier $k\in\mathbb{Z}_{>0}$. 
Soit $(e_{1},\ldots, e_{n})$ une base du r\'eseau $M$ engendrant le c\^one $\omega_{i}$ et telle que
$e_{1}\in \rm relint\,\it\sigma^{\vee}$.
Par propret\'e de $\mathfrak{D}$, il existe $d\in\mathbb{Z}_{>0}$ tel que chaque $\mathfrak{D}(de_{j})$ est
un diviseur de Weil entier donnant un faisceau inversible globalement engendr\'e. En posant
 \begin{eqnarray*}
A_{\omega_{i},k} = \bigoplus_{(a_{1},\ldots, a_{n})\in\mathbb{Z}^{n}}H^{0}\left(C,
\mathcal{O}\left(\left\lfloor \sum_{i = 1}^{n}a_{i}ke_{i}\right\rfloor\right)\right)\chi^{\sum_{i}a_{i}ke_{i}}, 
\end{eqnarray*}
nous consid\'erons $f_{1}\chi^{m_{1}},\ldots, f_{r}\chi^{m_{r}}\in A_{\omega_{i},d}$
obtenus par un ensemble fini de g\'en\'erateurs de l'espace vectoriel des sections globales 
de chaque $\mathcal{O}(\mathfrak{D}(de_{j}))$,
et un ensemble fini de g\'en\'erateurs homog\`enes de l'alg\`ebre gradu\'ee
\begin{eqnarray*}
B = \bigoplus_{l\geq 0}H^{0}(C,\mathcal{O}_{C}(\mathfrak{D}(dle_{1})))\chi^{lde_{1}}, 
\end{eqnarray*}
voir le lemme $4.5.8\,\rm (i)$. En utilisant $4.5.6\,\rm (iii)$, la normalisation de
$\mathbf{k}[f_{1}\chi^{m_{1}},\ldots, f_{r}\chi^{m_{r}}]$ est $A_{\omega_{i},d}$. 
Donc par le th\'eor\`eme $2$ dans [Bou, V$3.2$],
l'alg\`ebre $A_{\omega_{i}} = A_{\omega_{i},1}$ est de type fini sur $\mathbf{k}$. On conclut
en utilisant la surjection 
$A_{\omega_{1}}\otimes\ldots\otimes A_{\omega_{s}}\rightarrow A$.
\end{proof}
Dans la prochaine assertion, nous \'etudions
comment l'alg\`ebre associ\'ee \`a un diviseur poly\'edral sur une courbe projective r\'eguli\`ere
change lorsque nous \'etendons les scalaires \`a une cl\^oture alg\'ebrique du
corps de base.
Les assertions $\rm (i), (ii)$ proviennent de faits classiques de la th\'eorie
des corps de fonctions alg\'ebriques d'une variable. Les d\'emonstrations sont
donc omises. Rappelons qu'un corps est dit \em parfait \rm si toutes ses
extensions alg\'ebriques sont s\'eparables.
\begin{proposition}
Supposons que le corps $\mathbf{k}$ est parfait et soit $\bar{\mathbf{k}}$ une cl\^oture
alg\'ebrique de $\mathbf{k}$. D\'esignons par $\mathfrak{S}_{\bar{\mathbf{k}}/\mathbf{k}}$
le groupe de Galois absolu de $\mathbf{k}$. Pour une courbe projective r\'eguli\`ere 
$C$ sur $\mathbf{k}$ associ\'ee \`a un corps de fonctions alg\'ebriques $K_{0}/\mathbf{k}$,
les assertions suivantes sont vraies.
\begin{enumerate}
\item[\rm (i)] $\bar{K_{0}} = \bar{\mathbf{k}}\otimes_{\mathbf{k}} K_{0}$ est un corps de fonctions alg\'ebriques
sur $\bar{\mathbf{k}}$ dont sa surface de Riemann s'identifie \`a la courbe $C_{\bar{\mathbf{k}}} = C\times_{\rm Spec\,\it \mathbf{k}}\rm Spec\,\it \bar{\mathbf{k}}$.
\item[\rm (ii)] Le groupe $\mathfrak{S}_{\bar{\mathbf{k}}/\mathbf{k}}$ op\`ere naturellement dans
$C_{\bar{\mathbf{k}}}$ et dans l'ensemble des places de $C_{\bar{\mathbf{k}}}$.
Toute orbite sous l'op\'eration de $\mathfrak{S}_{\bar{\mathbf{k}}/\mathbf{k}}$ 
dans l'ensemble des places de $C_{\bar{\mathbf{k}}}$ est un ensemble fini qui est une fibre de l'application
surjective
\begin{eqnarray*}
S:\mathscr{R}_{\bar{\mathbf{k}}}\bar{K}_{0}\rightarrow  
C = \mathscr{R}_{\mathbf{k}} K_{0},\,\,\,P\rightarrow P\cap K_{0},
\end{eqnarray*}
et toute fibre de $S$ est d\'ecrite ainsi.
\item[\rm (iii)] Si $\mathfrak{D} = \sum_{z\in C}\Delta_{z}\cdot z$ 
est un diviseur $\sigma$-poly\'edral propre sur $C$ alors
\begin{eqnarray*}
A[C,\mathfrak{D}]\otimes_{\mathbf{k}}\bar{\mathbf{k}} = A[C_{\bar{\mathbf{k}}},
\mathfrak{D}_{\bar{\mathbf{k}}}], 
\end{eqnarray*}
o\`u $\mathfrak{D}_{\bar{\mathbf{k}}}$ est le diviseur $\sigma$-poly\'edral propre sur
$C_{\bar{\mathbf{k}}}$ d\'efini par
\begin{eqnarray*}
\mathfrak{D}_{\bar{\mathbf{k}}} = \sum_{z\in C}\Delta_{z}\cdot S^{\star}(z)\,\,\,\rm avec\,\,\,\it 
S^{\star}(z) = \sum_{z'\in S^{-\rm 1\it}(z)}z'. 
\end{eqnarray*}
\end{enumerate}
\end{proposition}
\begin{proof}
$\rm (iii)$ \'Etant donn\'e un diviseur de Weil
rationnel $D$ sur $C$, par [St, Theorem $3.6.3$,], il vient
\begin{eqnarray*}
H^{0}(C,\mathcal{O}_{C}(\lfloor D\rfloor))\otimes_{\mathbf{k}}\bar{\mathbf{k}} = 
H^{0}(C_{\bar{\mathbf{k}}},\mathcal{O}_{C_{\bar{\mathbf{k}}}}(\lfloor S^{\star}\,D\rfloor)).
\end{eqnarray*}
La d\'emonstration de $\rm (iii)$ suit du calcul de $A[C,\mathfrak{D}]\otimes_{\mathbf{k}}\bar{\mathbf{k}}$.
La propret\'e de $\mathfrak{D}_{\bar{\mathbf{k}}}$ est donn\'ee par $4.5.6\, \rm (i)$.
\end{proof}
L'\'enonc\'e ci-dessus ne se g\'en\'eralise pas dans le cas imparfait, comme
expliqu\'e dans la remarque suivante.
\begin{remarque}
Il est bien connu que toute extension de corps de type fini sur un corps parfait
est s\'eparable. Cependant, dans le cas imparfait, nous pouvons consid\'erer
le corps de fonctions alg\'ebriques
\begin{eqnarray*}
K_{0} = \rm Frac\,\it \frac{\mathbf{k}[X,Y]}{(tX^{\rm 2\it}+s+Y^{\rm 2\it})}\,,   
\end{eqnarray*}
qui est ins\'eparable sur  $\mathbf{k} = \mathbb{F}_{2}(s,t)$, o\`u $s,t$ sont deux variables ind\'ependantes sur $\mathbb{F}_{2}$.
Si on fixe une cl\^oture alg\'ebrique $\bar{\mathbf{k}}$ de $\mathbf{k}$, alors pour 
tout diviseur poly\'edral propre $\mathfrak{D}$ sur $C = \mathscr{R}_{\mathbf{k}}\, K_{0}$,
l'anneau $B = A[C,\mathfrak{D}]\otimes_{\mathbf{k}}\bar{\mathbf{k}}$ contient au moins un \'el\'ement nilpotent non nul.    
\end{remarque}

\section{Description des $\mathbb{T}$-vari\'et\'es affines de complexit\'e un dans le cas o\`u $\mathbb{T}$
est d\'eploy\'e}
Comme application des r\'esultats de la section pr\'ec\'edente,
nous pouvons donner maintenant une description des $\mathbb{T}$-vari\'et\'es affines de complexit\'e $1$ 
sur un corps $\mathbf{k}$, o\`u $\mathbb{T}$ est un tore alg\'ebrique d\'eploy\'e. 
Dans le paragraphe
suivant, nous rappelons quelques faits \'el\'ementaires sur les op\'erations
de tores alg\'ebriques. 
\begin{rappel}
Soit $\mathbb{T}$ un tore alg\'ebrique d\'eploy\'e sur $\mathbf{k}$.
D\'esignons par $M$ et
$N$ ses r\'eseaux duaux des caract\`eres et
des sous-groupes \`a $1$ param\`etre. Soit $X = \rm Spec\,\it A$ une vari\'et\'e affine sur $\mathbf{k}$. 
Supposons que $\mathbb{T}$ op\`ere dans $X$. Alors le comorphisme $A\rightarrow 
A\otimes_{\mathbf{k}}\mathbf{k}[\mathbb{T}]$ donne une $M$-graduation sur $A$.  
Nous disons que $X$ est une \em $\mathbb{T}$-vari\'et\'e \rm si $X$ est normale 
et si l'op\'eration de $\mathbb{T}$ dans $X$ est fid\`ele\footnote{Voyant $\mathbb{T}$
comme un foncteur en groupes repr\'esentable, cela veut dire que la transformation 
naturelle de foncteurs en groupes $\mathbb{T}\rightarrow\rm Aut\,\it X$ 
a son noyau trivial.}. Cela est \'equivalent \`a dire que l'alg\`ebre 
$A$ est $M$-gradu\'ee normale et que l'ensemble de ses poids engendre le r\'eseau $M$. 
 \end{rappel}

\begin{definition}
Soient $C$ une courbe alg\'ebrique r\'eguli\`ere sur $\mathbf{k}$ et $\sigma\subset N_{\mathbb{Q}}$ un
c\^one poly\'edral saillant. Un diviseur $\sigma$-poly\'edral $\mathfrak{D} = \sum_{z\in C}\Delta_{z}\cdot z$ 
est dit \em propre \rm s'il satisfait au moins un des \'enonc\'es suivants.
\begin{enumerate}
\item[\rm (i)]
$C$ est affine. En particulier, $\mathfrak{D}$ est un diviseur poly\'edral sur
l'anneau de Dedekind $A_{0} = \mathbf{k}[C]$. 
\item[\rm (ii)]
$C$ est projective et $\mathfrak{D}$ est un diviseur poly\'edral propre au sens de $4.5.5$. 
\end{enumerate}
Nous d\'esignons par $A[C,\mathfrak{D}]$ l'alg\`ebre $M$-gradu\'ee associ\'ee.
\end{definition}
En combinant les r\'esultats $4.4.4$ et $4.5.6$, on peut d\'ecrire une 
$\mathbb{T}$-vari\'et\'e affine d\'eploy\'ee de complexit\'e $1$ par un diviseur poly\'edral propre.
\begin{theorem}
Soit $\mathbb{T}$ un tore alg\'ebrique d\'eploy\'e sur $\mathbf{k}$ avec r\'eseau des caract\`eres $M$.
\begin{enumerate}
 \item[\rm (i)] Pour chaque $\mathbb{T}$-vari\'et\'e $X = \rm Spec\,\it A$ sur $\mathbf{k}$ 
de complexit\'e un, il existe un diviseur $\sigma$-poly\'edral propre $\mathfrak{D}$ sur une courbe r\'eguli\`ere $C$
sur $\mathbf{k}$ tel que $A\simeq A[C,\mathfrak{D}]$ en tant qu'alg\`ebres $M$-gradu\'ees.
\item[\rm (ii)] R\'eciproquement, si $\mathfrak{D}$ est un diviseur $\sigma$-poly\'edral propre
sur une courbe r\'eguli\`ere $C$ alors $X = \rm Spec\,\it A$, avec $A = A[C,\mathfrak{D}]$, d\'efinit une 
$\mathbb{T}$-vari\'et\'e affine de complexit\'e un. 
\end{enumerate}
\end{theorem}
\begin{proof}
$\rm (i)$ 
Consid\'erons le sous-corps $Q(A)_{0}\subset \mathbf{k}(X)$
engendr\'e par les quotients $a/b$, o\`u $a,b\in A$ sont des \'el\'ements homog\`enes 
de m\^eme degr\'e. 
Soit $\sigma\subset N_{\mathbb{Q}}$ le dual du c\^one des poids de $A$.
Remarquons que nous pouvons choisir des vecteurs de poids $\chi^{m}\in\rm Frac\,\it A$
tels que pour tous $m,m'\in M$, $\chi^{m}\cdot\chi^{m'} = \chi^{m+m'}$ et $\chi^{0} = 1$, et donnant lieu \`a une inclusion
\begin{eqnarray*}
A\subset \bigoplus_{m\in M}Q(A)_{0}\,\chi^{m} = Q(A)_{0}[M],
\end{eqnarray*}
o\`u $A$ est un sous-anneau $M$-gradu\'e. Les anneaux $A$ et $Q(A)_{0}[M]$ ont m\^eme 
corps des fractions.
Supposons que $A_{0}\neq \mathbf{k}$. Posons $K_{0}= \rm Frac\,\it A_{\rm 0\it}$.
Alors par normalit\'e de $A$, $K_{0} = Q(A)_{0}$. 
De plus, $A_{0}$ est un anneau de Dedekind, 
par le th\'eor\`eme $4.4.4\,\rm (ii)$, nous obtenons $A = A[C,\mathfrak{D}]$, pour un
diviseur $\sigma$-poly\'edral $\mathfrak{D}$ sur $A_{0}$.
Si $A_{0} = \mathbf{k}$ alors on conclut par le th\'eor\`eme $4.5.6\,\rm (i)$. 
L'assertion $\rm (ii)$ suit imm\'ediatement de 
$4.4.4\,\rm (i)$ et $4.5.6\, \rm (ii)$.
\end{proof}

\section{Op\'erations de tores alg\'ebriques de complexit\'e un et diviseurs poly\'edraux Galois stables}
En vue des r\'esultats de la section pr\'ec\'edente, nous donnons une description
combinatoire des vari\'et\'es normales affines munies d'une op\'eration d'un
tore alg\'ebrique (possiblement non d\'eploy\'e) de complexit\'e $1$.
Cela peut \^etre mis en parall\`ele avec des descriptions bien connues
concernant les vari\'et\'es toriques et les plongements sph\'eriques
d'espaces homog\`enes; voir [Bry, CTHS, Vos, ELST, Hu]. 
\begin{rappel}
Pour une extension de corps $F/\mathbf{k}$ et un sch\'ema $X$ sur $\mathbf{k}$,
nous posons
\begin{eqnarray*}
X_{F} = X\times_{\rm Spec\,\it\mathbf{k}}\rm Spec\,\it F.
\end{eqnarray*}
C'est un sch\'ema sur $F$. 
Un \em tore alg\'ebrique \rm de dimension $n$
est un groupe alg\'ebrique $\mathbf{G}$ sur $\mathbf{k}$ (i.e. un sch\'ema 
en groupes lisse et de type fini) tel qu'il existe une extension 
galoisienne finie $E/\mathbf{k}$ donnant un isomorphisme 
de groupes alg\'ebriques $\mathbf{G}_{E}\simeq \mathbb{G}_{m,E}^{n}$ o\`u
$\mathbb{G}_{m}$ est le sch\'ema en groupes multiplicatif sur $\mathbf{k}$.
Nous disons que le tore $\mathbf{G}$ \em se d\'eploie (ou se d\'ecompose) dans l'extension $E/\mathbf{k}$. \rm
Pour plus de d\'etails sur les groupes alg\'ebriques r\'eductifs
non d\'eploy\'es voir [BoTi, Sp].
\end{rappel}
Dans toute cette section, $\mathbf{G}$ est un tore sur $\mathbf{k}$ qui se d\'eploie 
dans une extension galoisienne finie $E/\mathbf{k}$. D\'esignons par $\mathfrak{S}_{E/\mathbf{k}}$ le groupe
de Galois de l'extension $E/\mathbf{k}$. Consid\'erons \`a nouveau les r\'eseaux duaux
$M$ et $N$, des caract\`eres et des sous-groupes \`a $1$ param\`etre, du tore
d\'eploy\'e $\mathbf{G}_{E}$. Notons que dans la suite la plupart des vari\'et\'es que nous
\'etudions sont sur le corps $E$. Nous commen\c cons par pr\'eciser une notion
classique. 
\begin{definition}
\
\begin{enumerate}
\item[\rm (i)]
Une op\'eration de $\mathfrak{S}_{E/\mathbf{k}}$ dans une vari\'et\'e $V$ sur $E$ 
est dite \em semi-lin\'eaire \rm si $\mathfrak{S}_{E/\mathbf{k}}$ op\`ere par 
automorphismes de sch\'emas et si pour tout 
$g\in \mathfrak{S}_{E/\mathbf{k}}$, le diagramme
\begin{eqnarray*}
  \xymatrix{
    V \ar[r]^g \ar[d] & V \ar[d] \\ 
    \rm Spec\,\it E \ar[r]_g & \rm Spec\,\it E}
\end{eqnarray*}
est commutatif.
\item[\rm (ii)] Soit $B$ une alg\`ebre sur $E$. Une op\'eration \em semi-lin\'eaire \rm de 
$\mathfrak{S}_{E/\mathbf{k}}$ dans $B$ 
est une op\'eration par automorphismes 
d'anneaux telle que pour tous $a\in B$, $\lambda\in E$,
et $g\in \mathfrak{S}_{E/\mathbf{k}}$,
\begin{eqnarray*}
g\cdot (\lambda a) = g(\lambda)g\cdot a.
\end{eqnarray*}
En particulier, $g\cdot (\lambda a) = \lambda g\cdot a$ si $\lambda\in\mathbf{k}$. 
\end{enumerate}
Si $V$ est affine alors avoir une op\'eration semi-lin\'eaire de 
$\mathfrak{S}_{E/\mathbf{k}}$ dans $V$ est \'equivalent \`a avoir une 
op\'eration semi-lin\'eaire du m\^eme groupe dans l'alg\`ebre $E[V]$.
\end{definition}
Ensuite, nous rappelons une description bien connue des tores alg\'ebriques par le
biais des op\'erations de groupes finis dans les r\'eseaux.
\begin{rappel}
Le groupe de Galois $\mathfrak{S}_{E/\mathbf{k}}$ op\`ere naturellement dans le tore
\begin{eqnarray*}
\mathbf{G}_{E} = \mathbf{G}\times_{\rm Spec\,\it \mathbf{k}}\rm Spec \,\it E
\end{eqnarray*}
par le second facteur. L'op\'eration correspondante dans $E[M]$
est d\'etermin\'ee par une op\'eration lin\'eaire de 
$\mathfrak{S}_{E/\mathbf{k}}$ dans $M$ (voir [ELST, Proposition $2.5$], [Vos, Section $1$]).

R\'eciproquement, \'etant donn\'ee une op\'eration $\mathbb{Z}$-lin\'eaire de
$\mathfrak{S}_{E/\mathbf{k}}$ dans le r\'eseau $M$, nous avons une op\'eration
semi-lin\'eaire dans $E[M]$ d\'efinie par
\begin{eqnarray*}
g\cdot (\lambda\chi^{m}) = g(\lambda)\chi^{g\cdot m},
\end{eqnarray*}
o\`u $g\in \mathfrak{S}_{E/\mathbf{k}}$, $\lambda\in E$ et $m\in M$. Cette op\'eration respecte 
la structure d'alg\`ebre de Hopf de $E[M]$. Comme cons\'equence du lemme de Speiser, nous obtenons un
tore $\mathbf{G}$ sur $\mathbf{k}$ qui se d\'eploie dans $E/\mathbf{k}$. De plus, 
l'op\'eration semi-lin\'eaire construite dans $\mathbf{G}_{E} = \mathbf{G}\times_{\rm Spec\,\it \mathbf{k}}
\rm Spec\,\it E$ 
est exactement l'op\'eration naturelle initiale sur le second facteur.
\end{rappel} 
La d\'efinition suivante introduit la cat\'egorie des $\mathbf{G}$-vari\'et\'es.
\begin{definition}  
Une \em $\mathbf{G}$-vari\'et\'e de complexit\'e \rm $d$
est une vari\'et\'e normale sur $\mathbf{k}$ avec une op\'eration de $\mathbf{G}$
telle que $X_{E}$ est une $\mathbf{G}_{E}$-vari\'et\'e de complexit\'e $d$ 
(au sens de la section pr\'ec\'edente).
Un \em $\mathbf{G}$-morphisme \rm entre deux $\mathbf{G}$-vari\'et\'es $X$ et $Y$ sur $\mathbf{k}$ 
est un morphisme $f:X\rightarrow Y$ de vari\'et\'es sur $\mathbf{k}$ tel que le diagramme
\begin{eqnarray*}
	\xymatrix{
    \mathbf{G}\times X \ar[r]^{\rm id\it\times f} \ar[d] & \mathbf{G}\times Y \ar[d] \\ 
    X \ar[r]_f & Y}
\end{eqnarray*}
est commutatif.
\end{definition}
Une classe importante d'op\'erations semi-lin\'eaires sont celles qui respectent 
l'op\'eration d'un tore alg\'ebrique d\'eploy\'e.
Nous fixons une op\'eration de $\mathfrak{S}_{E/\mathbf{k}}$ dans
$\mathbf{G}_{E}$ donn\'ee comme dans le paragraphe $4.7.3$.
\begin{definition}
\
\begin{enumerate}
\item[\rm (i)] Soit $B$ une alg\`ebre $M$-gradu\'ee sur $E$. Une op\'eration semi-lin\'eaire de
 $\mathfrak{S}_{E/\mathbf{k}}$ dans $B$ est dite \em homog\`ene \rm si elle envoie
les \'el\'ements homog\`enes sur les \'el\'ements homog\`enes. 
\item[\rm (ii)] Une op\'eration semi-lin\'eaire de $\mathfrak{S}_{E/\mathbf{k}}$ dans une
$\mathbf{G}_{E}$-vari\'et\'e $V$ \em respecte l'op\'eration de $\mathbf{G}_{E}$ \rm 
si pour tout $g\in\mathfrak{S}_{E/\mathbf{k}}$, le diagramme
\begin{eqnarray*}
		\xymatrix{
    \mathbf{G}_{E}\times V \ar[r]^{g\times g} \ar[d] & \mathbf{G}_{E}\times V \ar[d] \\ 
    V \ar[r]_g & V}
\end{eqnarray*}
est commutatif. 
\end{enumerate}

Sous l'hypoth\`ese que $V$ est affine, une op\'eration semi-lin\'eaire de $\mathfrak{S}_{E/\mathbf{k}}$,
respectant l'op\'eration du tore $\mathbf{G}_{E}$, correspond exactement \`a une op\'eration
semi-lin\'eaire homog\`ene de $\mathfrak{S}_{E/\mathbf{k}}$ dans l'alg\`ebre $E[V]$.
\end{definition}
Le r\'esultat suivant est classiquement \'enonc\'e pour la cat\'egorie des vari\'et\'es
quasi-projectives, voir par exemple la d\'emonstration de [Hu $2$, $1.10$].
Dans le contexte des $\mathbf{G}$-vari\'et\'es affines, nous donnons un
court argument.
\begin{lemme}
Soit $V$ une $\mathbf{G}_{E}$-vari\'et\'e de complexit\'e $d$ sur le corps $E$ 
munie d'une op\'eration semi-lin\'eaire de $\mathfrak{S}_{E/\mathbf{k}}$ compatible
avec l'op\'eration de $\mathbf{G}_{E}$. 
Alors le quotient $X = V/\mathfrak{S}_{E/\mathbf{k}}$ est une $\mathbf{G}$-
vari\'et\'e affine de complexit\'e $d$. En faisant agir $\mathfrak{S}_{E/\mathbf{k}}$
dans $X_{E} = X\times_{\rm Spec\,\it\mathbf{k}}\rm Spec\,\it E$ par le second facteur, nous avons un isomorphisme de 
$\mathbf{G}_{E}$-vari\'et\'es $X_{E}\simeq V$ respectant les op\'erations 
de $\mathfrak{S}_{E/\mathbf{k}}$.
\end{lemme}
\begin{proof}
Il est connu que l'alg\`ebre $R = B^{\mathfrak{S}_{E/\mathbf{k}}}$ est de type fini sur $\mathbf{k}$.
Montrons que $R$ est normale. En consid\'erant $L$ le corps des fractions de $R$ et $f\in L$ 
un \'el\'ement entier sur $R$, par normalit\'e de $B$, nous avons $f\in B\cap L = R$. Cela montre
la normalit\'e de $R$.
En utilisant la d\'efinition pr\'ec\'edente, la vari\'et\'e $X$ est munie d'une op\'eration de $\mathbf{G}$. 
Le reste de la d\'emonstration suit du lemme de Speiser.
\end{proof}
Fixant une $\mathbf{G}$-vari\'et\'e affine $X$ de complexit\'e $d$ sur $E$, une \em $E/\mathbf{k}$-forme \rm de $X$
est une $\mathbf{G}$-vari\'et\'e affine $Y$ sur $\mathbf{k}$ telle que nous avons un $\mathbf{G}_{E}$-isomorphisme
$X_{E}\simeq Y_{E}$. Notre but est de donner une description combinatoire de l'ensemble des $E/\mathbf{k}$-formes
d'une $\mathbf{G}$-vari\'et\'e affine donn\'ee $X$.
Rappelons dans ce contexte quelque notion de cohomologie galoisienne non ab\'elienne (voir par exemple [Ser, III Section 1]
pour la descente galoisienne des vari\'et\'es alg\'ebriques). 
\begin{rappel}
Soient $Y, Y'$ des $E/\mathbf{k}$-formes d'une $\mathbf{G}$-vari\'et\'e affine fix\'ee $X$. 
Le groupe de Galois 
$\mathfrak{S}_{E/\mathbf{k}}$ op\`ere dans l'ensemble des $\mathbf{G}_{E}$-isomorphismes
entre $Y_{E}$ et $Y'_{E}$. Par cons\'equent, il op\`ere aussi par automorphismes de groupes
dans le groupe des $\mathbf{G}_{E}$-automorphismes $\rm Aut_{\it\mathbf{G}_{E}}(\it X_{E}\rm )$ de $X_{E}$.
Plus pr\'ecis\'ement, rappelons que pour tout $g\in\mathfrak{S}_{E/\mathbf{k}}$ et 
tout $\mathbf{G}_{E}$-isomorphisme $\varphi: Y_{E}\rightarrow Y'_{E}$,
on d\'efinit $g(\varphi)$ par le diagramme commutatif suivant
\begin{eqnarray*}
   \xymatrix{
    Y_{E} \ar[r]^{g(\varphi)} \ar[d]_{g} & Y'_{E} \ar[d]^{g} \\ 
    Y_{E} \ar[r]_{\varphi} & Y'_{E}}.
\end{eqnarray*}
Notons que cette op\'eration de $\mathfrak{S}_{E/\mathbf{k}}$ d\'epend
de la donn\'ee  des $E/\mathbf{k}$-formes $Y,Y'$. Maintenant, \'etant donn\'e un $\mathbf{G}_{E}$-isomorphisme 
$\psi: X_{E}\rightarrow Y_{E}$ l'application 
\begin{eqnarray*}
a:\mathfrak{S}_{E/\mathbf{k}}\rightarrow \rm Aut_{\it\mathbf{G}_{E}}(\it X_{E}\rm ),\it
\,\,\, g\mapsto a_{g} = \psi^{\rm -1\it}\circ g(\psi)
\end{eqnarray*}
est un \em $1$-cocycle. \rm Cela signifie que pour tous $g,g'\in \mathfrak{S}_{E/\mathbf{k}}$, nous avons
\begin{eqnarray*}
a_{g}\circ g(a_{g'}) = \psi^{-1}\circ g(\psi)\circ g\left(\psi^{-1}\circ g'(\psi)\right) = a_{gg'}.
\end{eqnarray*}
Soit $\phi :Y\rightarrow Y'$ un $\mathbf{G}$-isomorphisme et prenons un $\mathbf{G}_{E}$-isomorphisme
$\varphi : X_{E}\rightarrow Y'_{E}$ donnant un $1$-cocycle $b$ comme avant. Le diagramme
\begin{eqnarray*}
   \xymatrix{
    X_{E} \ar[r]^{\psi} \ar[d]_{\alpha} & Y_{E} \ar[d]^{\phi' = \phi\times\rm id} \\ 
    X_{E} \ar[r]_{\varphi} & Y'_{E}}
\end{eqnarray*}
est commutatif, o\`u $\alpha\in \rm Aut_{\it\mathbf{G}_{E}}(\it X_{E}\rm )$ et $\phi'$ est 
l'extension naturelle de $\phi$. 
Puisque pour tout $g\in \mathfrak{S}_{E/\mathbf{k}}$,
$g(\phi') = \phi'$, il s'ensuit que
\begin{eqnarray*}
b_{g} =  \alpha\circ a_{g}\circ g\left(\alpha^{-1}\right).
\end{eqnarray*} 
Dans ce cas, nous disons que les cocycles $a$ et $b$ sont \em cohomologues. \rm
Nous obtenons ainsi une application $\Phi$ entre l'ensemble point\'e des classes 
d'isomorphisme de $E/\mathbf{k}$-formes de $X$ et l'ensemble point\'e 
\begin{eqnarray*}
H^{1}(E/\mathbf{k}, \rm Aut_{\it\mathbf{G}_{E}}(\it X_{E}\rm ))
\end{eqnarray*}
des classes de cohomologie de $1$-cocycles $a : \mathfrak{S}_{E/\mathbf{k}}
\rightarrow \rm Aut_{\it\mathbf{G}_{E}}(\it X_{E}\rm )$.

R\'eciproquement, partant d'un cocycle $a$, l'application
\begin{eqnarray*}
\mathfrak{S}_{E/\mathbf{k}} \rightarrow \rm Aut_{\it\mathbf{G}_{E}}(\it X_{E}\rm ),\,\,\,
\it g\mapsto a_{g}\circ g
\end{eqnarray*}
est une op\'eration semi-lin\'eaire dans $X_{E}$ respectant l'op\'eration de $\mathbf{G}_{E}$. 
D'apr\`es le lemme $5.6$, on peut associer une 
$E/\mathbf{k}$-forme $W$ de $X$ en prenant le quotient
$X_{E}/\mathfrak{S}_{E/\mathbf{k}}$. En rempla\c cant $a$ par un autre cocycle qui lui est cohomologue,
on obtient par ce dernier proc\'ed\'e une $E/\mathbf{k}$-forme de $X$ isomorphe \`a $W$. Ainsi,
nous d\'eduisons que l'application $\Phi$ est bijective. 

En outre, soit $\gamma$ une op\'eration semi-lin\'eaire de
$\mathfrak{S}_{E/\mathbf{k}}$ dans $X_{E}$. Remarquons que pour tous $g,g'\in \mathfrak{S}_{E/\mathbf{k}}$,
le diagramme
\begin{eqnarray*}
   \xymatrix{
    X_{E} \ar[r]^{\gamma(g')} \ar[d]_{g^{-1}} & X_{E} \ar[d]^{g^{-1}} \\ 
    X_{E} \ar[r]_{g(\gamma(g'))} & X_{E}}
\end{eqnarray*}
est commutatif. D'o\`u l'\'egalit\'e 
$a_{g} = \gamma(g)\circ g^{-1}$
d\'efinit un $1$-cocycle $a$. Une v\'erification directe montre que  
$H^{1}(E/\mathbf{k}, \rm Aut_{\it\mathbf{G}_{E}}(\it X_{E}\rm ))$ est aussi en bijection
avec l'ensemble point\'e des classes de conjugaison des op\'erations semi-lin\'eaires
de $\mathfrak{S}_{E/\mathbf{k}}$
dans $X_{E}$ compatibles avec l'op\'eration de $\mathbf{G}_{E}$.
\end{rappel}
Comme expliqu\'e dans le paragraphe pr\'ec\'edent, d\'eterminer l'ensemble point\'e des $E/\mathbf{k}$-formes de $X$
est \'equivalent \`a d\'ecrire toutes les op\'erations semi-lin\'eaires compatibles possibles de 
$\mathfrak{S}_{E/\mathbf{k}}$ dans $X_{E}$. Ainsi, g\'en\'eralisant la notion de diviseurs poly\'edraux, 
nous consid\'erons la contre-partie combinatoire de cette classification.
\begin{definition}
Soient $C$ une courbe r\'eguli\`ere sur $E$ et $\sigma\subset N_{\mathbb{Q}}$ un c\^one poly\'edral saillant.
\begin{enumerate}
 \item [\rm (a)] Un \em diviseur $\sigma$-poly\'edral principal \em $\mathfrak{F}$ sur $C$ est un couple 
$(\varphi, \mathfrak{D})$ o\`u  
$\varphi : \sigma^{\vee}_{M}\rightarrow \mathbf{k}(C)^{\star}$ est un morphisme de mono\"ides et
$\mathfrak{D}$ est un diviseur $\sigma$-poly\'edral sur $C$ tel que pour tout $m\in\sigma^{\vee}_{M}$,
\begin{eqnarray*}
\mathfrak{D}(m) = \rm div_{\it C}\,\it \varphi(m). 
\end{eqnarray*}
Habituellement nous \'ecrivons $\mathfrak{F}$
et $\mathfrak{D}$ par la m\^eme lettre. 

\item[\rm (b)] Un \em diviseur $\sigma$-poly\'edral $\mathfrak{S}_{E/\mathbf{k}}$-stable \rm sur $C$ est la donn\'ee de 
$(\mathfrak{D},\mathfrak{F},\star,\cdot)$ v\'erifiant les conditions suivantes.
\begin{enumerate}
\item[\rm (i)] La courbe $C$ est munie d'une op\'eration semi-lin\'eaire de $\mathfrak{S}_{E/\mathbf{k}}$,
\begin{eqnarray*}
\mathfrak{S}_{E/\mathbf{k}}\times C\rightarrow C, \,\,\, (g,z)\mapsto g\star z.
\end{eqnarray*}
Cela donne naturellement une op\'eration lin\'eaire dans l'espace vectoriel
des diviseurs de Weil rationnels sur $C$. Plus pr\'ecis\'ement,
\'etant donn\'es $g\in\mathfrak{S}_{E/\mathbf{k}}$ et un diviseur de Weil rationnel $D$
sur $C$, nous posons 
\begin{eqnarray*}
g\star D = \sum_{z\in C}a_{g^{-1}\star z}\cdot z,\,\,\,\rm avec\,\,\,\it D = \sum_{z\in C}a_{z}\cdot z.
\end{eqnarray*}
\item[(ii)] Le r\'eseau $M$ est munie d'une op\'eration $\mathbb{Z}$-lin\'eaire de $\mathfrak{S}_{E/\mathbf{k}}$,
\begin{eqnarray*}
\mathfrak{S}_{E/\mathbf{k}}\times M\rightarrow M,\,\,\, (g,m)\mapsto g\cdot m
\end{eqnarray*}
pr\'eservant le sous-ensemble $\sigma^{\vee}_{M}$.
\item[\rm (iii)] $\mathfrak{D}$  est un diviseur $\sigma$-poly\'edral propre 
 sur la courbe $C$. De plus, on a une application $$g\mapsto f_{g}, \,\,\mathfrak{S}_{E/\mathbf{k}}\rightarrow {\rm Hom}(M, E(C)^{\star})$$ qui v\'erifie la condition $f_{gh}(m) = g(f_{h}(m))f_{g}(h\cdot m)$ pour tous $g,h\in \mathfrak{S}_{E/\mathbf{k}}$ et $m\in M$. Posons $\mathfrak{F}_{g}(m) :=  \rm div_{\it C}\,\it f_{g}(m)$. Nous verrons $\mathfrak{F}_{g}$ comme un diviseur poly\'edral principal et $\mathfrak{F}$ d\'esignera l'application $g\mapsto \mathfrak{F}_{g}$.
\end{enumerate}
La liste $(\mathfrak{D}, \mathfrak{F}, \star, \cdot)$ satisfait l'\'egalit\'e
\begin{eqnarray*}
g\star(\mathfrak{D}(m)) = \mathfrak{D}(g\cdot m) + \mathfrak{F}_{g}(m),
\end{eqnarray*}
o\`u $m\in \sigma^{\vee}_{M}$ et $g\in\mathfrak{S}_{E/\mathbf{k}}$.
\end{enumerate}
\end{definition}
Le r\'esultat suivant permet de simplifier la description des diviseurs poly\'edraux $\mathfrak{S}_{E/\mathbf{k}}$-stables dans un cas particulier. 
Nous incluons ici un courte d\'emonstration.
\begin{lemme}
Soit $E_{0}/K_{0}$ une extension galoisienne finie de groupe de Galois $\mathfrak{S}_{E_{0}/K_{0}}$.
Supposons que $\mathfrak{S}_{E_{0}/K_{0}}$ op\`ere lin\'eairement dans le r\'eseau $M$.
Pour tout $g\in \mathfrak{S}_{E_{0}/K_{0}}$, consid\'erons un morphisme de groupes 
$f_{g}: M\rightarrow E_{0}^{\star}$ satisfaisant les \'egalit\'es
\begin{eqnarray*}
f_{gh}(m) = g\left(f_{h}(m)\right)f_{g}(h\cdot m),
\end{eqnarray*}
o\`u $g,h\in\mathfrak{S}_{E_{0}/K_{0}}$ et $m\in M$. Supposons 
que pour  $T = \rm Hom(\it M,E_{\rm 0}^{\star})$ on ait $H^{1}(E_{0}/K_{0}, T) =1$.
Alors il existe un morphisme de groupes $b : M\rightarrow E_{0}^{\star}$
tel que pour tous $g\in\mathfrak{S}_{E_{0}/K_{0}}$, $m\in M$,
\begin{eqnarray*}
f_{g}(m) = b(g\cdot m)g(b(m))^{-1}\,.
\end{eqnarray*}
\end{lemme}
\begin{proof}
Le groupe oppos\'e de $\mathfrak{S}_{E_{0}/K_{0}}$ est le groupe $H$ 
avec sa structure d'ensemble $\mathfrak{S}_{E_{0}/K_{0}}$ et dont
la loi de
composition interne est d\'efinie par $g\star h = hg$,
o\`u $g,h\in H$. Pour $g\in H$, nous d\'esignons par $a_{g} : M\rightarrow E_{0}^{\star}$ 
le morphisme de groupes o\`u pour tout $m\in M$, on a 
\begin{eqnarray*}
a_{g}(m) = g^{-1}(f_{g}(m)).
\end{eqnarray*} 
Nous pouvons aussi d\'efinir une op\'eration de $H$ par
automorphisme de groupes dans le groupe ab\'elien
\begin{eqnarray*}
T = \rm Hom(\it M,E_{\rm 0}^{\star})
\end{eqnarray*}
sur $E_{0}$ en posant $(g\cdot\alpha)(m) = g^{-1}(\alpha(g\cdot m))$,
o\`u $\alpha\in T$, $g\in H$, et $m\in M$. En consid\'erant 
$g,h\in H$, nous obtenons
\begin{eqnarray*}
a_{h\star g}(m) = (gh)^{-1}(f_{gh}(m)) = (gh)^{-1}(g(f_{h}(m))f_{g}(h\cdot m))
 = a_{h}(m)(h\cdot a_{g})(m)
\end{eqnarray*}
de sorte que l'application $g\mapsto a_{g}$ est un $1$-cocycle.
On a un isomorphisme de groupes ab\'eliens 
\begin{eqnarray*}
H^{1}(H,T) \simeq H^{1}(E_{0}/K_{0}, T)  = 1.
\end{eqnarray*}
Donc il existe $b\in T$ tel que pour tout $g\in H$, 
$a_{g} = b \cdot (g\cdot b^{-1})$. 
Ces derni\`eres \'egalit\'es donnent le r\'esultat.
\end{proof}
Le prochain th\'eor\`eme donne une classification des $\mathbf{G}$-vari\'et\'es affines 
de complexit\'e $1$ en termes de diviseurs poly\'edraux stables par Galois. 
\begin{theorem}
Soit $\mathbf{G}$ un tore sur $\mathbf{k}$ se d\'eployant dans une extension
galoisienne finie $E/\mathbf{k}$.
D\'esignons par $\mathfrak{S}_{E/\mathbf{k}}$ le groupe de Galois de l'extension $E/\mathbf{k}$.
Alors les assertions suivantes sont vraies.
\begin{enumerate}
\item[\rm (i)] Toute $\mathbf{G}$-vari\'et\'e affine de complexit\'e $1$ 
se d\'eployant dans $E/\mathbf{k}$ est d\'ecrite par
un diviseur poly\'edral $\mathfrak{S}_{E/\mathbf{k}}$-stable sur une courbe
r\'eguli\`ere au dessus du corps $E$. 
\item[\rm (ii)] R\'eciproquement, soit $C$ une courbe r\'eguli\`ere sur $E$. Pour
un diviseur poly\'edral $\mathfrak{S}_{E/\mathbf{k}}$-stable $(\mathfrak{D},\mathfrak{F},\star,\cdot)$ 
sur $C$ on peut munir l'alg\`ebre $A[C,\mathfrak{D}]$ d'une op\'eration semi-lin\'eaire homog\`ene
de $\mathfrak{S}_{E/\mathbf{k}}$ et associer une $\mathbf{G}$-vari\'et\'e de complexit\'e $1$ sur $\mathbf{k}$ 
se d\'eployant dans $E/\mathbf{k}$
en prenant $X = \rm Spec\,\it A$, o\`u
\begin{eqnarray*}
A = A[C,\mathfrak{D}]^{\mathfrak{S}_{E/\mathbf{k}}}.
\end{eqnarray*}
\end{enumerate}
\end{theorem}
\begin{proof}
$\rm (i)$ Soit $X$ une $\mathbf{G}$-vari\'et\'e affine de complexit\'e $1$ sur $\mathbf{k}$.
Par le th\'eor\`eme $4.6.3$, nous pouvons supposer que $B = A[C,\mathfrak{D}]$ est l'anneau des coordonn\'ees
de $X_{E}$ pour un diviseur $\sigma$-poly\'edral propre $\mathfrak{D}$ sur une courbe r\'eguli\`ere $C$. 
L'alg\`ebre $B$ est munie d'une op\'eration semi-lin\'eaire homog\`ene de $\mathfrak{S}_{E/\mathbf{k}}$. 
Posons $E_{0} = E(C)$. En \'etendant l'op\'eration sur $E_{0}[M]$, nous remarquons que $E_{0}$ et $E[C]$ sont
des parties stables. Nous obtenons ainsi une op\'eration semi-lin\'eaire de $\mathfrak{S}_{E/\mathbf{k}}$
dans $C$. Pr\'ecisons que si $C$ est projective alors on peut d\'efinir l'op\'eration de $\mathfrak{S}_{E/\mathbf{k}}$
dans $C$ de la mani\`ere suivante: \'etant donn\'ee une place $P\subset E_{0}$, nous posons 
\begin{eqnarray*}
g\star P = \{g\star f\,|\, f\in P\}.
\end{eqnarray*}
Dans le cas o\`u $C$ est arbitraire, le lemme de Speiser nous donne 
l'\'egalit\'e
\begin{eqnarray*}
E_{0} = E\cdot K_{0}\,\,\, \rm avec\,\,\,\it  K_{\rm 0} = \it E_{\rm 0\it }^{\it \mathfrak{S}_{E/\mathbf{k}}}.
\end{eqnarray*}
L'extension $E_{0}/K_{0}$ est galoisienne finie. Nous avons une identification naturelle 
$\mathfrak{S}_{E/\mathbf{k}}\simeq\mathfrak{S}_{E_{0}/K_{0}}$
avec le groupe de Galois de $E_{0}/K_{0}$. Pour tous $m\in M$, $g\in\mathfrak{S}_{E/\mathbf{k}}$,
on a
\begin{eqnarray}
g\cdot \left(f\chi^{m}\right) = g(f)f_{g}(m)\chi^{\Gamma(g,m)},
\end{eqnarray}
pour un \'el\'ement $f_{g}$ du groupe ab\'elien $T = \rm Hom(\it M,E_{\rm 0}^{\star})$ et
un vecteur $\Gamma(g,m)\in M$.
On observe que $\Gamma$ d\'efinit une op\'eration lin\'eaire de $\mathfrak{S}_{E/\mathbf{k}}$ dans $M$. 
D\'esignons alors par $g\cdot m$ le vecteur $\Gamma(g,m)$.
Pour tous $g,h\in\mathfrak{S}_{E/\mathbf{k}}$, il vient
\begin{eqnarray*}
f_{gh}(m)\chi^{m} = gh\cdot\chi^{m} = g\cdot(h\cdot\chi^{m}) = g(f_{h}(m))f_{g}(h\cdot m)\chi^{gh\cdot m}.
\end{eqnarray*}
Posons $\mathfrak{F}_{g}(m) :=  \rm div_{\it C}\,\it f_{g}(m)$. Pour l'assertion $\rm (i)$, il reste \`a montrer l'\'egalit\'e
\begin{eqnarray}
g\star\left(\mathfrak{D}(m)\right) = \mathfrak{D}(g\cdot m) + \mathfrak{F}_{g}(m),
\,\,\,\forall m\in\sigma^{\vee}_{M}, \forall g\in\mathfrak{S}_{E/\mathbf{k}}.
\end{eqnarray}
Tout d'abord, nous remarquons que si $f\in E_{0}^{\star}$ et $g\in\mathfrak{S}_{E/\mathbf{k}}$
alors $g\star\rm div\,\it f = \rm div\,\it g(f)$. Soit $f\chi^{m}\in B$ un \'el\'ement homog\`ene de degr\'e $m$. 
La transformation de $f\chi^{m}$ par $g$ est un \'el\'ement de $B$ de degr\'e $g\cdot m$ et donc
\begin{eqnarray*}
\rm div\it\, g(f)f_{g}(m) + \mathfrak{D}(g\cdot m)\rm \geq 0.
\end{eqnarray*}
Cela implique que
\begin{eqnarray*}
g\star\left(-\rm div\,\it f \right)\rm \leq \it
\mathfrak{F}_{g}(m) + \mathfrak{D}(g\cdot m).
\end{eqnarray*}
Par les corollaires $4.3.7$ et $4.5.9\,\rm (iii)$, nous obtenons
\begin{eqnarray*}
 g\star \left(\mathfrak{D}(m)\right)\leq \mathfrak{D}(g\cdot m) + \mathfrak{F}_{g}(m).
\end{eqnarray*}
L'\'egalit\'e oppos\'ee s'obtient par un argument analogue. On conclut que $(\mathfrak{D},\mathfrak{F}, \star, \cdot)$
est un diviseur poly\'edral stable.

$\rm (ii)$ On d\'efinit une op\'eration semi-lin\'eaire homog\`ene de $\mathfrak{S}_{E/\mathbf{k}}$
dans $A[C,\mathfrak{D}]$ par l'\'egalit\'e $(4.2)$ o\`u ici $\mathfrak{F}_{g}(m) =  \rm div_{\it C}\,\it f_{g}(m)$. 
\end{proof}
Donnons un exemple \'el\'ementaire.
\begin{exemple}
Consid\'erons le diviseur $\sigma$-poly\'edral $\mathfrak{D}$ sur la droite affine
complexe $\mathbb{A}^{1}_{\mathbb{C}} = \rm Spec\,\it\mathbb{C}[t]$
d\'efini par
\begin{eqnarray*}
((1,0) + \sigma)\cdot \zeta + ((0,1) + \sigma)\cdot (-\zeta) + ((1,-1)+\sigma)\cdot 0,
\end{eqnarray*}
o\`u $\sigma$ est le premier quadrant $\mathbb{Q}_{\geq 0}^{2}$ et $\zeta = \sqrt{-1}$.
Nous munissons $\mathfrak{D}$ d'une structure de diviseur poly\'edral 
$\mathfrak{S}_{\mathbb{C}/\mathbb{R}}$-stable en consid\'erant d'abord $\mathfrak{F}$
induit par le morphisme $(m_{1},m_{2})\mapsto t^{m_{2}-m_{1}}$. Nous avons une op\'eration lin\'eaire de 
$\mathfrak{S}_{\mathbb{C}/\mathbb{R}}$,
\begin{eqnarray*}
\mathfrak{S}_{\mathbb{C}/\mathbb{R}}\rightarrow \rm GL_{2}(\it\mathbb{Z}),\,\,\, g\mapsto \begin{pmatrix}
   0 & 1 \\
   1 & 0 
\end{pmatrix}
\end{eqnarray*}
dans le r\'eseau $\mathbb{Z}^{2}$, 
o\`u $g$ repr\'esente le g\'en\'erateur du groupe $\mathfrak{S}_{\mathbb{C}/\mathbb{R}}\simeq \mathbb{Z}/2\mathbb{Z}$.
L'alg\`ebre des polyn\^omes \`a une variable $\mathbb{C}[t]$ poss\`ede la conjugaison
des nombres complexes $\star$. Un calcul direct donne
\begin{eqnarray*}
A = \mathbb{C}\left[t,\frac{1}{t(t-\zeta)}\chi^{(1,0)},\frac{t}{t+\zeta}\chi^{(0,1)}\right]
\end{eqnarray*}
et donc $X = \rm Spec \,\it A$ est isomorphe \`a l'espace affine complexe $\mathbb{A}^{3}_{\mathbb{C}}$.
Plus concr\`etement, l'op\'eration de $\mathfrak{S}_{\mathbb{C}/\mathbb{R}}$
dans l'alg\`ebre $A$ est obtenue par
\begin{eqnarray*}
g\cdot (f(t)\chi^{(m_{1},m_{2})}) = \bar{f(t)}t^{2(m_{1}-m_{2})}\chi^{(m_{2},m_{1})}. 
\end{eqnarray*}
En posant $x = t^{-1}(1-\zeta)^{-1}\chi^{(1,0)}$ et $y = t(1+\zeta)^{-1}\chi^{(0,1)}$, nous observons
que $A^{\mathfrak{S}_{\mathbb{C}/\mathbb{R}}} = \mathbb{R}[t, x+y, \zeta(x-y)]$.
D'o\`u: $X/\mathfrak{S}_{\mathbb{C}/\mathbb{R}}\simeq \mathbb{A}^{3}_{\mathbb{R}}$.
\end{exemple}
Dans la suite,
nous d\'ecrivons l'ensemble point\'e des classes d'isomorphisme
des $E/\mathbf{k}$-formes d'une 
$\mathbf{G}$-vari\'et\'e affine de complexit\'e $1$ en termes de diviseurs
poly\'edraux stables. 
\begin{definition}
Les diviseurs $\sigma$-poly\'edraux stables
$(\mathfrak{D},\mathfrak{F},\star,\cdot)$ et 
$(\mathfrak{D},\mathfrak{F}',\star',\cdot')$ sur $C$ sont \em conjugu\'es \rm 
s'ils v\'erifient les conditions suivantes. 
Il existe $\varphi\in \rm Aut(\it C)$, un diviseur $\sigma$-poly\'edral principal
$\mathfrak{E}$ sur $C$, et un automorphisme lin\'eaire $F\in\rm Aut(\it M)$
donnant un automorphisme de la $E$-alg\`ebre $A[C,\mathfrak{D}]$
tels que pour tout $g\in\mathfrak{S}_{E/\mathbf{k}}$, les diagrammes 
\begin{eqnarray*}
  \xymatrix{
    C \ar[r]^{g\star} \ar[d]_{\varphi} & C \ar[d]^{\varphi} \\ 
     C \ar[r]_{g\star'} & C} 
\rm\,\,\, et\,\,\,\it
  \xymatrix{
    M \ar[r]^{g\cdot} \ar[d]_{F} & M \ar[d]^{F} \\ 
     M \ar[r]_{g\cdot'} & M}
\end{eqnarray*}
sont commutatifs, et pour tous $m\in M$ et $g\in\mathfrak{S}_{E/\mathbf{k}}$, 
\begin{eqnarray*}
\mathfrak{E}(g\star m)\cdot \varphi^{\star}(f_{g}(m)) =  g(\mathfrak{E}(m))\cdot f_{g}'(F(m)).
\end{eqnarray*}
Consid\'erons $X$ une $\mathbf{G}$-vari\'et\'e affine de complexit\'e $1$
d\'ecrite par un diviseur poly\'edral stable $(\mathfrak{D},\mathfrak{F},\star,\cdot)$. 
Nous d\'esignons par $\mathscr{E}_{X}(E/\mathbf{k})$ l'ensemble des classes de conjugaison
de diviseurs poly\'edraux stables par $\mathfrak{S}_{E/\mathbf{k}}$ sur $C$ de
la forme $(\mathfrak{D},\mathfrak{F}',\star',\cdot')$.
\end{definition}
Comme cons\'equence directe de la discussion de $4.7.7$ et du fait que $3.3.9$
se g\'en\'eralise sur un corps quelconque, nous obtenons le r\'esultat suivant. 
\begin{corollaire}
Soit $C$ une courbe r\'eguli\`ere sur le corps $E$. 
\'Etant donn\'ee une $\mathbf{G}$-vari\'et\'e affine $X$ de complexit\'e $1$
associ\'ee \`a un diviseur poly\'edral $\mathfrak{S}_{E/\mathbf{k}}$-stable
$(\mathfrak{D},\mathfrak{F},\star,\cdot)$ sur $C$, 
nous avons la bijection d'ensembles point\'es
\begin{eqnarray*}
\mathscr{E}_{X}(E/\mathbf{k})\simeq H^{1}(E/\mathbf{k}, \rm Aut_{\it\mathbf{G}_{E}}(\it X_{E}\rm )).
\end{eqnarray*}
\end{corollaire}  
\newpage
\strut 
\newpage

\chapter{Racines des $\mathbb{T}$-vari\'et\'es affines
de complexit\'e un sur un corps parfait}
\section{Introduction}
Ce chapitre est tir\'e d'un travail en collaboration avec Alvaro Liendo. Dans ce travail,
nous nous int\'eressons aux op\'erations normalis\'ees du groupe additif
dans les $\mathbb{T}$-vari\'et\'es affines de complexit\'e $1$.

Afin de donner une illustration g\'eom\'etrique de la th\'eorie des racines
des $\mathbb{T}$-vari\'et\'es, supposons d'abord que le corps
de base $\mathbf{k}$ est alg\'ebriquement clos. 
Fixons un tore alg\'ebrique $\mathbb{T}\simeq (\mathbf{k}^{\star})^{n}$ de dimension
$n$. Rappelons qu'un caract\`ere du tore $\mathbb{T}$ est un morphisme de groupes 
alg\'ebriques de $\mathbb{T}$ vers le groupe multiplicatif $\mathbb{G}_{m}\simeq \mathbf{k}^{\star}$.
Si $X$ est une $\mathbb{T}$-vari\'et\'e alors une op\'eration du groupe additif 
dans $X$, 
\begin{eqnarray*}
\mathbb{G}_{a}\times X\rightarrow X,\,\,\, (\lambda, x)\mapsto \lambda\star x, 
\end{eqnarray*}
est dite normalis\'ee par l'op\'eration de $\mathbb{T}$ dans $X$,
\begin{eqnarray*}
\mathbb{T}\times X\rightarrow X,\,\,\, (t,x)\mapsto t\cdot x, 
\end{eqnarray*}
s'il existe un caract\`ere $\alpha$ du tore $\mathbb{T}$ tel que pour tout $t\in\mathbb{T}$,
pour tout $\lambda \in\mathbb{G}_{a} = \mathbf{k}_{+}$ et pour tout $x\in X$, nous
avons
\begin{eqnarray*}
t\cdot (\lambda \star (t^{-1}\cdot x)) = (\alpha(t)\lambda)\star x. 
\end{eqnarray*}
Notons que cette d\'efinition peut \^etre mise en parall\`ele avec la th\'eorie des syst\`emes de racines
des groupes de Lie. Pour simplifier, consid\'erons le groupe lin\'eaire $\rm GL_{\it n}$ 
sur le corps $\mathbf{k}$. Consid\'erons un sous-groupe unipotent \`a $1$ param\`etre
de la forme 
\begin{eqnarray*}
U_{i,j} = \left\{u_{i,j}(\lambda) =  I_{n} + \lambda\cdot E_{ij}|\,\lambda\in\mathbf{k}\right\}, 
\end{eqnarray*}
o\`u $1\leq i, j\leq n$ sont distincts, 
$I_{n}$ est la matrice unit\'e d'ordre $n$ et $E_{ij}$ est la matrice de format $n\times n$
ayant pour coefficient $1$ en position $(i,j)$
et $0$ en dehors. Soit
\begin{eqnarray*}
T = \{t = \rm diag\it(t_{\rm 1\it},\ldots, t_{n})|\,t_{\rm 1\it},t_{\rm 2\it},\ldots, t_{n}\in\mathbf{k}^{\star}\}
\subset \rm GL_{\it n} 
\end{eqnarray*}
le sous-groupe des matrices diagonales. Alors pour tout $t\in T$, pour tout $\lambda\in \mathbf{k}$ 
et pour tous $1\leq i,j\leq n$ distincts, on a la relation
\begin{eqnarray*}
t\cdot u_{ij}(\lambda)\cdot t^{-1} = u_{ij}(t_{i}t_{j}^{-1}\lambda). 
\end{eqnarray*}
Le caract\`ere $\alpha = \alpha_{ij} :T\rightarrow \mathbb{G}_{m}$, $t\mapsto t_{i}t_{j}^{-1}$
est appel\'e une racine de $\rm GL_{\it n}$ relativement au tore $T$.
Par analogie, on peut comparer $\rm GL_{\it n}$ avec le groupe $\rm Aut\it (X)$
des automorphismes polynomiaux d'une $\mathbb{T}$-vari\'et\'e affine $X$. En effet, un sous-groupe unipotent
\`a $1$ param\`etre de $\rm Aut\it (X)$ est une op\'eration fid\`ele de $\mathbb{G}_{a}$ dans $X$. L'image 
du morphisme naturel $\mathbb{T}\rightarrow \rm Aut\it (X)$ joue le m\^eme
r\^ole que le tore $T$ dans le cas de $\rm GL_{\it n}$. Compte tenu de la 
d\'efinition pr\'ec\'edente, une racine de $\rm Aut\it (X)$ correspond \`a une
op\'eration fid\`ele du groupe additif dans $X$ normalis\'ee par l'op\'eration de $\mathbb{T}$.

Cette analogie est aussi confirm\'ee par les travaux de Demazure donnant une description explicite
de la composante neutre du groupe des automorphismes d'une vari\'et\'e torique compl\`ete lisse, voir [De]. Voir
aussi [AHHL] pour une g\'en\'eralisation aux cas des $\mathbb{T}$-vari\'et\'es compl\`etes rationnelles
de complexit\'e $1$ via la th\'eorie des anneaux totaux de coordonn\'ees
et [Li3] pour les racines du groupe $\rm Aut\it (\mathbb{A}^{n})$. 
En ce qui concerne l'\'etude des $\mathbb{T}$-vari\'et\'es affines de complexit\'e $1$, une classification 
par des objets combinatoires des 
op\'erations normalis\'ees du groupe additif  
a \'et\'e donn\'ee par Liendo [Li] dans le cas o\`u
le corps de base $\mathbf{k}$ est alg\'ebriquement clos de caract\'eristique z\'ero.
Cette derni\`ere approche est inspir\'ee du cas des $\mathbb{C}^{\star}$-surfaces affines (voir [FZ 2, FZ 3])
et de la th\'eorie des racines de Demazure (voir [De, $4.5$], [Li, $2$])
pour les vari\'et\'es toriques affines. Comme application, le lecteur 
peut consulter [AL] pour une description de certaines op\'erations
de $\rm SL_{2}$ dans les $\mathbb{T}$-vari\'et\'es affines.

Dans ce chapitre, nous donnons 
une g\'en\'eralisation de [Li] dans le cas o\`u $\mathbf{k}$
est un corps parfait. Donnons une liste de quelques r\'esultats de ce chapitre. 

- Nous g\'en\'eralisons la correspondance classique sur un corps quelconque entre op\'erations normalis\'ees du groupe
additif dans les vari\'et\'es toriques affines et racines
de Demazure, voir le th\'eor\`eme $5.4.5$.

- Nous d\'ecrivons de fa\c con combinatoire les op\'erations normalis\'ees
du groupe additif de type horizontal dans
certaines vari\'et\'es toriques affines sur un corps parfait, ayant une op\'eration de
tores alg\'ebriques de complexit\'e $1$, voir le th\'eor\`eme $5.6.8$.

- Nous classifions en termes de diviseurs poly\'edraux
les op\'erations normalis\'ees du groupe additif dans les $\mathbb{T}$-vari\'et\'es
affines de complexit\'e $1$, sur un corps quelconque dans le cas vertical, et
sur un corps parfait dans le cas horizontal, voir les th\'eor\`emes $5.5.4$ et $5.6.12$.
En particulier, si $X = \rm Spec\it\, A$ alors on montre que
$A^{\mathbb{G}_{a}}$ est de type fini sur $\mathbf{k}$ (voir plus g\'en\'eralement [Ku]
lorsque la caract\'eristique de $\mathbf{k}$ est nulle).

- Nous montrons qu'une $\mathbb{G}_{m}$-surface affine non hyperbolique sur un corps 
parfait admettant une op\'eration normalis\'ee du groupe additif de type
horizontal est torique, voir le corollaire $5.6.6$.  

Pour formuler nos r\'esultats, nous rappelons quelques notions sur les $\mathbb{T}$-vari\'et\'es
affines munies d'une op\'eration du groupe additif. \`A partir de maintenant, toutes les vari\'et\'es
sont d\'efinies sur un corps arbitraire $\mathbf{k}$. Se donner une op\'eration du groupe additif dans une
vari\'et\'e affine $X$ est \'equivalent \`a fixer un syst\`eme it\'er\'e de d\'erivations d'ordre 
sup\'erieur localement fini $\partial = \{\partial^{(i)}\}_{i\in\mathbb{N}}$
sur $\mathbf{k}[X]$, appel\'e aussi LFIHD pour abr\'eger. Cette correspondance 
est donn\'ee via le morphisme naturel
\begin{eqnarray*}
\mathbf{k}[X]\rightarrow \mathbf{k}[X]\otimes_{\mathbf{k}} \mathbf{k}[\mathbb{G}_{a}],\,\,\,
f\mapsto \sum_{i = 0}^{\infty}\partial^{(i)}(f)\otimes \lambda^{i}. 
\end{eqnarray*}
Lorsque $\mathbf{k}$ est de caract\'eristique $0$, la LFIHD $\partial$ est uniquement
d\'etermin\'ee par la d\'erivation localement nilpotente $\partial^{(1)}$; plus
pr\'ecis\'ement, on a pour tout $i\in\mathbb{Z}_{>0}$, $\partial^{(i)} = (\partial^{(1)})^{\circ i}/i!$.
Pour plus de renseignements sur la th\'eorie des LFIHD, voir [Mi, Cr]. 

Soit $\mathbb{T}$ un tore alg\'ebrique 
d\'eploy\'e sur $\mathbf{k}$ et notons $M$ le r\'eseau des caract\`eres de $\mathbb{T}$. 
Une op\'eration du groupe additif dans une $\mathbb{T}$-vari\'et\'e affine $X$ est dite normalis\'ee 
si la LFIHD associ\'ee est homog\`ene 
pour la $M$-graduation naturelle de $\mathbf{k}[X]$, voir $5.3.7$, $5.3.9$.
Cela g\'en\'eralise pour un corps de base arbitraire la d\'efinition \'enonc\'ee
pr\'ec\'edemment. 

On distingue deux types d'op\'erations normalis\'ees du groupe additif dans $X$; 
si on a $\mathbf{k}(X)^{\mathbb{T}}\subset \mathbf{k}(X)^{\mathbb{G}_{a}}$ alors 
on dit que l'op\'eration de $\mathbb{G}_{a}$ est de type vertical. Dans le cas contraire, l'op\'eration de $\mathbb{G}_{a}$
 est de type horizontal. On donne une m\^eme qualification, pour les LFIHD homog\`enes.
D'un point de vue g\'eom\'etrique (i.e. lorsque $\mathbf{k}$
est alg\'ebriquement clos), une op\'eration normalis\'ee de $\mathbb{G}_{a}$ 
dans $X$ est de type vertical si un quotient rationnel $X\dashrightarrow Y$ pour l'op\'eration de $\mathbb{T}$
est $\mathbb{G}_{a}$-invariant; de sorte que toute orbite g\'en\'erale 
de $\mathbb{G}_{a}$ est contenue dans l'adh\'erence d'une orbite de $\mathbb{T}$.

Soit $X$ une $\mathbb{T}$-vari\'et\'e affine de complexit\'e $1$ sur un corps $\mathbf{k}$.
D'apr\`es les r\'esultats du chapitre pr\'ec\'edent, $X$ est d\'ecrite par un triplet
$(C, \sigma, \mathfrak{D})$, o\`u $\sigma$ est un c\^one poly\'edral saillant dont le dual
est le c\^one des poids de $\mathbf{k}[X]$, $C$ est une courbe r\'eguli\`ere sur $\mathbf{k}$ et
$\mathfrak{D}$ est un diviseur $\sigma$-poly\'edral propre sur $C$ (voir $4.6$).
On a un diagramme commutatif d'applications rationnelles:
\begin{eqnarray*}
 \xymatrix{ X \ar@{-->}[rr]^f \ar@{-->}[rd]^{\pi} && C\times X_{\sigma} \ar[ld]^{\rm proj} \\ & C }, 
\end{eqnarray*}
o\`u $\pi$ est un quotient rationnel, $X_{\sigma}$ est la vari\'et\'e torique associ\'ee au c\^one $\sigma$
et $f$ est birationnelle $\mathbb{T}$-\'equivariante. Voir [Ti] pour une description 
de l'application $f$ en fonction de $\mathfrak{D}$ dans le cas o\`u $\mathbf{k} = \mathbb{C}$.
Suivant les arguments de d\'emonstration de [FZ $2$, $3.12$], [Li, $3.1$], [Li $2$], 
pour d\'ecrire les LFIHD homog\`enes de type vertical sur $\mathbf{k}[X]$, il est utile de classifier 
les LFIHD homog\`enes sur $\mathbf{k}[X_{\sigma}]$. 

Une LFIHD $\partial$
homog\`ene non triviale sur l'alg\`ebre $\mathbf{k}[X_{\sigma}]$ est d\'etermin\'ee par son degr\'e $e$ \`a la multiplication
par un scalaire non nul pr\`es. Le c\^one des poids de la sous-alg\`ebre $M$-gradu\'ee $\rm ker\,\it \partial$ est une
face duale d'une ar\^ete $\rho$ de $\sigma$. Le vecteur $e$ est une racine de Demazure de rayon distingu\'e $\rho$
du c\^one $\sigma$ (voir $5.4.5$). R\'eciproquement, une racine de Demazure $e$ du c\^one $\sigma$ donne lieu \`a
une LFIHD sur $\mathbf{k}[X_{\sigma}]$ (voir $5.4.2$). Pour les LFIHD de type vertical sur $\mathbf{k}[X]$,
on remarque que dans ce cas l'application $f$ est $\mathbb{T}\ltimes\mathbb{G}_{a}$-\'equivariante. Ainsi, on obtient
une description analogue au cas de $X_{\sigma}$ avec des conditions 
suppl\'ementaires attach\'ees \`a $\mathfrak{D}$. Voir $5.5.5$ et le th\'eor\`eme $5.5.4$ pour 
une classification compl\`ete.

Passons maintenant au cas d'une op\'eration normalis\'ee du groupe additif de type
horizontal dans une $\mathbb{T}$-vari\'et\'e affine $X$ de complexit\'e $1$. 
Dans cette situation, l'op\'eration naturelle de $\mathbb{T}\ltimes\mathbb{G}_{a}$
dans $X$ est de complexit\'e $0$. Supposons que le corps $\mathbf{k}$ est parfait.
Plus pr\'ecis\'ement, on montre que $C\simeq \mathbb{A}^{1}_{\mathbf{k}}$ ou
$C\simeq \mathbb{P}^{1}_{\mathbf{k}}$ (voir $5.6.2$) et
on a un diagramme commutatif
\begin{eqnarray*}
 \xymatrix{\mathbb{A}^{1}_{\mathbf{k}}\times\mathbb{T} \ar[rr]^{\phi} 
\ar[rd] && X_{\widetilde{\sigma}} \ar[ld]^{\iota} \\ & X }. 
\end{eqnarray*}
La vari\'et\'e $X_{\widetilde{\sigma}}$ est torique pour un tore contenant $\mathbb{T}$.
L'application
$\iota$ est une immersion ouverte $\mathbb{T}$-\'equivariante faisant
de $X_{\widetilde{\sigma}}$ un ouvert principal de Zariski. 
Le morphisme $\phi$ est un rev\^etement cyclique $\mathbb{T}$-\'equivariant.
En fait, l'op\'eration de $\mathbb{G}_{a}$ dans $X$ provient d'une op\'eration de $\mathbb{G}_{a}$
dans $\mathbb{A}^{1}_{\mathbf{k}}\times \mathbb{T}$ sur le premier facteur, de sorte que le morphisme $\iota\circ \phi$
respecte les op\'erations de $\mathbb{T}\ltimes\mathbb{G}_{a}$. 
Par cons\'equent, nous proc\'edons en deux \'etapes. Tout d'abord, on d\'ecrit de fa\c con combinatoire
les op\'erations normalis\'ees de $\mathbb{G}_{a}$ de type horizontal dans la $\mathbb{T}$-vari\'et\'e
$X_{\widetilde{\sigma}}$, voir $5.6.8$. Ensuite, on donne des conditions sur $\mathfrak{D}$
pour que l'application $\iota$ soit $\mathbb{T}\ltimes\mathbb{G}_{a}$-\'equivariante, voir $5.6.11$.
On obtient ainsi une classification compl\`ete, voir $5.6.12$.  

Donnons un court r\'esum\'e de chaque section de ce chapitre. Dans la section $5.3$, on donne
quelques propri\'et\'es g\'en\'erales sur les op\'erations normalis\'ees du groupe additif dans
les $\mathbb{T}$-vari\'et\'es affines. Dans la section $5.4$, on donne une description des
LFIHD homog\`enes pour les alg\`ebres de vari\'et\'es toriques affines. Enfin dans les 
sections $5.5$ et $5.6$, on traite respectivement les cas vertical et horizontal.

\section{Introduction (english version)}
This chapter is a joint work with Alvaro Liendo.
In this work, we are interested in normalized $\mathbb{G}_{a}$-actions on 
affine $\mathbb{T}$-varieties of complexity $1$.

In order to give a geometric illustration of the roots theory of $\mathbb{T}$-varieties,
let us assume first that the ground field $\mathbf{k}$ is algebraically closed.
Fix an $n$-dimensional algebraic torus $\mathbb{T}\simeq (\mathbf{k}^{\star})^{n}$. 
Recall that a character of the torus $\mathbb{T}$ is a morphism of algebraic groups from
$\mathbb{T}$ to the multiplicative group $\mathbb{G}_{m}\simeq \mathbf{k}^{\star}$.
If $X$ is a $\mathbb{T}$-variety then an additive group action on $X$, 
\begin{eqnarray*}
\mathbb{G}_{a}\times X\rightarrow X,\,\,\, (\lambda, x)\mapsto \lambda\star x, 
\end{eqnarray*}
is called normalized by the $\mathbb{T}$-action on $X$,
\begin{eqnarray*}
\mathbb{T}\times X\rightarrow X,\,\,\, (t,x)\mapsto t\cdot x, 
\end{eqnarray*}
if there exists a character $\alpha$ of $\mathbb{T}$ such that for all $t\in\mathbb{T}$, 
$\lambda \in\mathbb{G}_{a} = \mathbf{k}_{+}$, and $x\in X$ we have 
\begin{eqnarray*}
t\cdot (\lambda \star (t^{-1}\cdot x)) = (\alpha(t)\lambda)\star x. 
\end{eqnarray*}
This definition can be compared with the notion of roots in the classical Lie theory.
For simplicity, let us consider the general linear group $\rm GL_{\it n}$ over the field $\mathbf{k}$. 
Consider a $1$-parameter unipotent subgroup of $\rm GL_{\it n}$ of the form 
\begin{eqnarray*}
U_{i,j} = \left\{u_{i,j}(\lambda) =  I_{n} + \lambda\cdot E_{ij}|\,\lambda\in\mathbf{k}\right\}, 
\end{eqnarray*}
where $1\leq i, j\leq n$ are distinct indices, 
$I_{n}$ is the neutral element in $\rm GL_{\it n}$ and $E_{ij}$ is the matrix
with entry $1$ in position $(i,j)$ and $0$'s elsewhere. Letting 
\begin{eqnarray*}
T = \{t = \rm diag\it(t_{\rm 1\it},\ldots, t_{n})|\,t_{\rm 1\it},t_{\rm 2\it},\ldots, t_{n}\in\mathbf{k}^{\star}\}
\subset \rm GL_{\it n} 
\end{eqnarray*}
be the subgroup of diagonal matrices we have for all $t\in T$, $\lambda\in \mathbf{k}$, and for all distinct 
$1\leq i,j\leq n$ the relation
\begin{eqnarray*}
t\cdot u_{ij}(\lambda)\cdot t^{-1} = u_{ij}(t_{i}t_{j}^{-1}\lambda). 
\end{eqnarray*}
The character $\alpha = \alpha_{ij} :T\rightarrow \mathbb{G}_{m}$, $t\mapsto t_{i}t_{j}^{-1}$
is called a root of $\rm GL_{\it n}$ for the torus $T$. 

Similarly, one can compare $\rm GL_{\it n}$ with the group $\rm Aut\it (X)$ which
consists of polynomial automorphisms of the affine $\mathbb{T}$-variety $X$. Indeed, a $1$-parameter unipotent
subgroup in $\rm \rm Aut\it (X)$ is a faithful $\mathbb{G}_{a}$-action on the variety $X$. 
The image of the natural morphism $\mathbb{T}\rightarrow \rm Aut\it (X)$ plays the same role
as the torus $T$ does in the case of $\rm GL_{\it n}$. Finally, by our definitions,
a root of $\rm Aut\it (X)$ corresponds to a faithful normalized additive group action on the $\mathbb{T}$-variety 
$X$.

This analogy is also confirmed by the work of Demazure giving an explicit method
to describe the neutral component of the automorphism group of a smooth complete toric variety,
see [De]. See also [AHHL] for generalization to the case of rational complete $\mathbb{T}$-varieties 
via the total coordinate rings and [Li3] for the roots of $\rm Aut\it (\mathbb{A}^{n})$.

Concerning the study of affine $\mathbb{T}$-varieties of complexity $1$, a complete classification 
of normalized additive group actions in combinatorial terms is given in [Li] when the base field $\mathbf{k}$
is algebraically closed of characteristic $0$. This latter approach were inspired by the case of
complex affine $\mathbb{C}^{\star}$-surfaces (see [FZ 2, FZ 3]) and the theory of Demazure roots for 
affine toric varieties (see [De, $4.5$], [Li, $2$]). As an application, the reader may consult [AL] 
for a description of particular $\rm SL_{2}$-actions on affine $\mathbb{T}$-varieties.
In this chapter, we provide a generalization of [Li] to the case where $\mathbf{k}$ is 
a perfect field. Let us list some significant results.

-We generalize the classical correspondence for an arbitrary field between normalized additive
group actions on affine toric varieties and Demazure roots, see Theorem $5.4.5$.

-We describe normalized additive group actions of horizontal type on
a class of affine toric varieties having a torus action of complexity $1$,
see Theorem $5.6.8$.

-We classify in terms of polyhedral divisors the normalized additive group
actions on affine $\mathbb{T}$-varieties of complexity $1$, over an arbitrary
field for the vertical case, and over a perfect field for the horizontal case,
see Theorems $5.5.4$ and $5.6.12$. In particular, if $X = \rm Spec\it\, A$ then we
show that the subalgebra
$A^{\mathbb{G}_{a}}$ is finitely generated over $\mathbf{k}$ (more generally, see [Ku]
for the characteristic $0$ setting).

-Every non-hyperbolic affine $\mathbb{G}_{m}$-surface over a perfect field which has
a normalized additive group action of horizontal type is toric, see Corollary $5.6.6$.

In order to formulate our results let us recall some notions concerning affine $\mathbb{T}$-varieties
with an additive group action. Below, the varieties are defined over an arbitrary field $\mathbf{k}$. 
Giving an additive group action on an affine variety $X$ is equivalent to fixing a locally finite iterative
higher derivation $\partial = \{\partial^{(i)}\}_{i\in\mathbb{N}}$
on the algebra $\mathbf{k}[X]$, called LFIHD for short. This correspondence is given via the natural
morphism
\begin{eqnarray*}
\mathbf{k}[X]\rightarrow \mathbf{k}[X]\otimes_{\mathbf{k}} \mathbf{k}[\mathbb{G}_{a}],\,\,\,
f\mapsto \sum_{i = 0}^{\infty}\partial^{(i)}(f)\otimes \lambda^{i}. 
\end{eqnarray*}
If $\mathbf{k}$ is of characteristic $0$ then the LFIHD $\partial$ is uniquely determined by
the locally nilpotent derivation $\partial^{(1)}$. More precisely, for any $i\in\mathbb{Z}_{>0}$
we have $\partial^{(i)} = (\partial^{(1)})^{\circ i}/i!$. See [Mi, Cr] for more details about LFIHDs. 

Let $\mathbb{T}$ be a split algebraic torus over $\mathbf{k}$ and denote by $M$ 
the character lattice of $\mathbb{T}$. An additive group action on a $\mathbb{T}$-variety $X$ is called
normalized if the associated LFIHD is homogeneous for the natural $M$-grading on $\mathbf{k}[X]$, 
see $5.3.7$, $5.3.9$. This generalizes to an arbitrary field the previous geometric definition. 
To classify normalized $\mathbb{G}_{a}$-actions on $X$ it is convenient to distinguish two types of them.  
If $\mathbf{k}(X)^{\mathbb{T}}\subset \mathbf{k}(X)^{\mathbb{G}_{a}}$ then 
we say that the $\mathbb{G}_{a}$-action is of vertical type. Otherwise, the $\mathbb{G}_{a}$-action
is horizontal. We give the same names for homogeneous LFIHD.
From a geometric viewpoint (i.e. when $\mathbf{k}$ is algebraically closed), a normalized $\mathbb{G}_{a}$-action on 
$X$ is of vertical type if a rational quotient $X\dashrightarrow Y$ for the $\mathbb{T}$-action
is $\mathbb{G}_{a}$-invariant, so that a $\mathbb{G}_{a}$-orbit in general position is contained in the closure
of a $\mathbb{T}$-orbit.

Let $X$ be an affine $\mathbb{T}$-variety of complexity $1$ over a field $\mathbf{k}$.
According to the results of the previous chapter, $X$ can be described by a triplet
$(C,\sigma,\mathfrak{D})$, where $\sigma$ is a strongly convex polyhedral cone
equal to the dual cone of the weight cone of $\mathbf{k}[X]$, $C$ is a regular curve over $\mathbf{k}$, and
$\mathfrak{D}$ is a proper $\sigma$-polyhedral divisor on $C$ (see $4.6$).
We have a commutative diagram of rational maps
\begin{eqnarray*}
 \xymatrix{ X \ar@{-->}[rr]^f \ar@{-->}[rd]^{\pi} && C\times X_{\sigma} \ar[ld]^{\rm proj} \\ & C }, 
\end{eqnarray*}
where $\pi$ is a rational quotient, $X_{\sigma}$ is the toric variety associated to $\sigma$,
and $f$ is a $\mathbb{T}$-equivariant birational map. See [Ti] for a description of the map $f$ in terms 
of $\mathfrak{D}$ 
and in the case where $\mathbf{k} = \mathbb{C}$.
Following the argument in [FZ $2$, $3.12$], [Li, $3.1$], [Li $2$], to describe homogeneous
LFIHDs of vertical type on $\mathbf{k}[X]$ it is useful to classify homogeneous 
LFIHD on $\mathbf{k}[X_{\sigma}]$. 

Actually, a non-trivial homogeneous LFIHD $\partial$ on $\mathbf{k}[X_{\sigma}]$ 
is uniquely determined (up to the multiplication by a nonzero scalar) by its degree $e$. 
The weight cone of the $M$-graded subalgebra $\rm ker\,\it \partial$ is the dual face of a
ray $\rho\subset\sigma$. The lattice vector $e$ is a Demazure root with distinguished ray $\rho$
of the cone $\sigma$ (see $5.4.5$). Conversely, a Demazure root $e$ of the cone $\sigma$
gives rise to a homogeneous LFIHD on $\mathbf{k}[X_{\sigma}]$ (see $5.4.2$). 
For the LFIHDs of vertical type on $\mathbf{k}[X]$, we remark that in this case the
map $f$ is $\mathbb{T}\ltimes\mathbb{G}_{a}$-equivariant. Thus, one obtains a similar 
description as for the variety $X_{\sigma}$ with additional conditions provided 
by the polyhedral divisor $\mathfrak{D}$. See $5.5.5$ and Theorem $5.5.4$ for
a complete classification.

Let us pass to the case of normalized additive group action of horizontal type
on an affine $\mathbb{T}$-variety $X$ of complexity $1$.
In this situation, the natural $\mathbb{T}\ltimes\mathbb{G}_{a}$-action on
$X$ is of complexity $0$. Assume that the field $\mathbf{k}$
is perfect. More precisely, one can show that $C\simeq \mathbb{A}^{1}_{\mathbf{k}}$ or
$C\simeq \mathbb{P}^{1}_{\mathbf{k}}$ (see $5.6.2$) and we have a commutative
diagram
\begin{eqnarray*}
 \xymatrix{\mathbb{A}^{1}_{\mathbf{k}}\times\mathbb{T} \ar[rr]^{\phi} 
\ar[rd] && X_{\widetilde{\sigma}} \ar[ld]^{\iota} \\ & X }. 
\end{eqnarray*}
The variety $X_{\widetilde{\sigma}}$ is toric for a torus containing $\mathbb{T}$.
The map $\iota$ is an open $\mathbb{T}$-equivariant immersion making $X_{\widetilde{\sigma}}$
a Zariski principal open subset. 
The morphism $\phi$ is a $\mathbb{T}$-equivariant cyclic covering.
Actually, the $\mathbb{G}_{a}$-action on $X$ can be obtained from a $\mathbb{G}_{a}$-action
on $\mathbb{A}^{1}_{\mathbf{k}}\times \mathbb{T}$ on the first factor, so that the morphism $\iota\circ \phi$
respects the $\mathbb{T}\ltimes\mathbb{G}_{a}$-actions. 
Therefore, we proceed in two steps. First of all, one describes the normalized $\mathbb{G}_{a}$-actions
of horizontal type on the $\mathbb{T}$-variety
$X_{\widetilde{\sigma}}$, see $5.6.8$. Then, we provide some conditions on $\mathfrak{D}$
for the map $\iota$ to be $\mathbb{T}\ltimes\mathbb{G}_{a}$-equivariant, see $5.6.11$.
We obtain as well a complete classification, see $5.6.12$.  

Let us give a brief summary of the contents of each section. In Section $5.3$, we
give some general properties of normalized additive group actions on $\mathbb{T}$-varieties. 
In Section $5.4$, we provide a description of homogeneous LFIHD for algebras
of affine toric varieties. Finally, Sections $5.5$ and $5.6$ treat respectively
the vertical and horizontal cases.

\section{Op\'erations du groupe additif et L.F.I.H.D. homog\`enes}
Dans cette section, nous nous int\'eressons \`a une relation
entre les op\'erations de $\mathbb{G}_{a}$, 
qui sont normalis\'ees par une op\'eration d'un tore alg\'ebrique d\'eploy\'e,
et aux syst\`emes it\'er\'es de d\'erivations d'ordre sup\'erieur 
localement finis (LFIHD). Soit $\mathbf{k}$ un corps et d\'esignons par
$X = \rm Spec\,\it A$ une vari\'et\'e affine sur $\mathbf{k}$.
\begin{rappel}
Tout d'abord, consid\'erons une op\'eration 
\begin{eqnarray*}
\phi:\mathbb{G}_{a}\times X\rightarrow X 
\end{eqnarray*}
du groupe additif $\mathbb{G}_{a}$ sur le corps $\mathbf{k}$.
Alors le morphisme d'alg\`ebres $\phi^{\star}$ donne une suite d'op\'erateurs lin\'eaires 
$\partial = \{\partial^{(i)}\}_{i\in\mathbb{N}}$ de l'espace vectoriel $A$
sur $\mathbf{k}$ d\'efinie de la mani\`ere suivante. Pour chaque \'el\'ement $f\in A$, 
nous \'ecrivons
\begin{eqnarray*}
\phi^{\star}(f) = \sum_{i = 0}^{\infty}\partial^{(i)}(f)\cdot x^{i}\in 
A\otimes_{\mathbf{k}}\mathbf{k}[\mathbb{G}_{a}], 
\end{eqnarray*}
o\`u l'alg\`ebre des fonctions r\'eguli\`eres de $\mathbf{k}[\mathbb{G}_{a}] = \mathbf{k}[x]$ 
est celle des polyn\^omes \`a une variable $x$. 
Un calcul facile sur les alg\`ebres de Hopf montre que $\partial$
v\'erifie les conditions suivantes (cf. [Mi]). 
\end{rappel}
\begin{enumerate}
 \item[(i)] L'op\'erateur $\partial^{(0)}$
est l'application identit\'e.
\item[(ii)] Pour tout entier $i\in\mathbb{N}$ et tous $f_{1},f_{2}\in A$, nous avons
la \em r\`egle de Leibniz \rm
\begin{eqnarray*}
\partial^{(i)}(f_{1}\cdot f_{2}) = 
\sum_{j = 0}^{i}\partial^{(j)}(f_{1})\cdot
\partial^{(i-j)}(f_{2}). 
\end{eqnarray*}
\item[(iii)] La suite $\partial$ est localement finie. Cela veut dire
que pour tout $f\in A$, il existe $r\in\mathbb{Z}_{>0}$ tel que
pour tout entier $i\geq r$, $\partial^{(i)}(f) = 0$.
\item[(iv)] Pour tous $i,j\in\mathbb{N}$ et toute fonction r\'eguli\`ere $f\in A$,
on a
\begin{eqnarray*}
\left(\partial^{(i)}\circ\partial^{(j)}\right)
(f) = \binom{i+j}{i}\,\partial^{(i+j)}(f).  
\end{eqnarray*}
\end{enumerate}
Une suite d'op\'erateurs lin\'eaires $\partial$ sur l'alg\`ebre $A$ satisfaisant
les conditions $\rm (i),(ii),(iii)$, $\rm (iv)$ est appel\'ee 
\em un syst\`eme it\'er\'e de d\'erivations d'ordre sup\'erieur localement fini. \rm
Suivant la convention anglaise, nous disons plut\^ot que $\partial$
est une LFIHD. Pr\'ecisons que les conditions $\rm (i)$, $\rm (ii)$ correspondent aux lettres
H et D (d\'erivations d'ordre sup\'erieur), la condition $\rm (iii)$ aux lettres L et F (localement fini) et
$\rm (iv)$ \`a la lettre $I$ (it\'er\'e). Par exemple si $\partial$ v\'erifie seulement $\rm (i), (ii), (iv)$ 
alors on dit que $\partial$ est un \em syst\`eme it\'er\'e de d\'erivations d'ordre sup\'erieur. \rm 

Plus g\'en\'eralement, le lecteur peut consulter [HaSc] 
pour la notion de d\'erivations de Hasse-Schmit et [Voj] pour des applications
\`a la g\'eom\'etrie alg\'ebrique.
R\'eciproquement, \'etant donn\'ee une LFIHD $\partial$ sur l'alg\`ebre $A$,
son \em application exponentielle \rm
\begin{eqnarray*}
e^{x\partial} := \sum_{i = 0}^{\infty}\partial^{(i)}\,x^{i}  
\end{eqnarray*}
est le morphisme provenant d'une op\'eration du groupe additif $\mathbb{G}_{a}$ dans la vari\'et\'e 
$X = \rm Spec\,\it A$. De cette mani\`ere, on a une correspondance bijective entre l'ensemble des op\'erations
du groupe additif dans $X$ et l'ensemble des LFIHD sur $\mathbf{k}[X]$.
\begin{remarque}
Consid\'erons une LFIHD $\partial$ sur $A$. Pour $i\in\mathbb{Z}_{>0}$, 
nous consid\'erons 
\begin{eqnarray*}
\left(\partial^{(1)}\right )^{\circ \,i} = \partial^{(1)}\circ\ldots\circ\partial^{(1)}  
\end{eqnarray*}
la composition de $i$ exemplaires de $\partial^{(1)}$. En d\'esignant par $p$ 
la caract\'eristique du corps $\mathbf{k}$, on a l'\'egalit\'e  
\begin{eqnarray*}
\partial^{(i)} = \frac{\left(\partial^{(1)}\right)^{\circ\,i_{0}}\circ
\left(\partial^{(p)}\right)^{\circ\,i_{1}}\circ\ldots \circ
\left(\partial^{(p^{r})}\right)^{\circ\,i_{r}}}
{(i_{0})!(i_{1})!\ldots (i_{r})!},\, 
\end{eqnarray*}
o\`u $i = \sum_{j = 1}^{r}i_{j}\cdot p^{j}$ est le d\'eveloppement $p$-adique\footnote{ 
Si $p = 0$ alors on convient que le d\'eveloppement $p$-adique est $i = i_{0}$.} de l'entier $i$. 
Lorsque $p = 0$ l'op\'eration du groupe additif $\mathbb{G}_{a}$ est par cons\'equent
uniquement d\'etermin\'ee par la d\'erivation localement nilpotente $\partial^{(1)}$.
\end{remarque}
En caract\'eristique z\'ero, l'alg\`ebre des invariants d'une op\'eration du groupe additif dans la vari\'et\'e 
$X = \rm Spec\,\it A$ est le noyau de la d\'erivation localement nilpotente
associ\'ee sur $A$. La d\'efinition suivante d\'ecrit cette situation
dans le cas de la caract\'eristique arbitraire.
\begin{definition}
Pour une LFIHD $\partial$ sur l'alg\`ebre $A$, son \em noyau \rm 
est le sous-ensemble 
\begin{eqnarray*}
\rm ker\,\it \partial := \left\{\, f\in A\,|\,\,\rm pour\,\,tout\,\it 
i\in\mathbb{Z}_{\rm >0},\,\,\,\it\partial^{(i)}(f) = \rm 0 
\right\}.  
\end{eqnarray*}
C'est l'alg\`ebre des invariants $A^{\mathbb{G}_{a}}\subset A$ pour l'op\'eration
de $\mathbb{G}_{a}$ correspondante \`a $\partial$. La LFIHD  
$\partial$ est dite \em triviale \rm si $\rm ker\,\it \partial = A$. 
Un sous-espace vectoriel $V\subset A$ est appel\'e \em $\partial$-stable \rm
si pour chaque entier $i\in\mathbb{N}$, nous avons l'inclusion 
$\partial^{(i)}(V)\subset V$. En particulier, le sous-espace vectoriel $\rm ker\,\it\partial$
est $\partial$-stable.
\end{definition}
Le prochain r\'esultat donne des propri\'et\'es utiles sur
les op\'erations du groupe additif (voir [CM, Cr]). Notons
qu'il existe des analogues bien connus aux assertions suivantes
dans le contexte des d\'erivations localement nilpotentes.
Pour plus de d\'etails voir [Fre].
\begin{proposition}
Pour toute $\rm LFIHD$ non triviale $\partial$ sur l'alg\`ebre $A$,
les assertions suivantes sont vraies.
\begin{enumerate}
 \item[\rm (a)] Le sous-anneau $\rm ker\,\it\partial\subset A$ est factoriellement clos , c'est \`a dire,
pour tous $f_{1},f_{2}\in A$, 
\begin{eqnarray*}
f_{1}\cdot f_{2}\in \rm ker\,\it\partial\rm - \{0\}\Rightarrow 
\it f_{\rm 1}, \it f_{\rm 2} \in\rm ker\,\it\partial.
\end{eqnarray*}
\item[\rm (b)] Le sous-anneau $\rm ker\,\it\partial$ est alg\'ebriquement clos dans $A$; 
tout \'el\'ement de $A$ satisfaisant une relation de d\'ependance alg\'ebrique \`a coefficients
dans $\rm ker\,\it\partial$ appartient \`a $\rm ker\,\it\partial$.
\item[\rm (c)] $\rm ker\,\it\partial$ est une sous-alg\`ebre de dimension $\rm dim\,\it A - \rm 1$.
\item[\rm (d)] Si $\rm car\, \it\mathbf{k} = p\rm >0$ et si $A = \mathbf{k}[y]$
est l'alg\`ebre des polyn\^omes \`a une variable alors il existe des scalaires
 $c_{1},\ldots, c_{r}\in
\mathbf{k}^{\star}$ et une suite strictement croissante d'entiers $0<s_{1}<\ldots<s_{r}$
tels que l'on ait l'\'egalit\'e
\begin{eqnarray*}
e^{x\partial}(y) = y + \sum_{i = 1}^{r}c_{i}\cdot x^{p^{s_{i}}}. 
\end{eqnarray*}
\item[\rm (e)] Si $A^{\star}$ repr\'esente l'ensemble des \'el\'ements
inversibles (pour la multiplication) de $A$ alors nous avons
$A^{\star}\subset \rm ker\,\it\partial$ ; de sorte que, 
$A^{\star} = \left(\rm ker\,\it\partial\right)^{\star}$.
\item[\rm (f)] Un id\'eal principal $(f) = fA$ est $\partial$-stable 
si et seulement si $f\in\rm ker\,\it\partial$.
\end{enumerate}
\end{proposition}
Pour les assertions (a), (b), (c)
nous renvoyons le lecteur \`a [CM, $2.1$, $2.2$] et \`a
l'exemple $3.5$ dans [Cr].
Notons que les assertions $\rm (a), (b)$ sont obtenues en utilisant
la fonction degr\'e par rapport \`a la variable $x$,
\begin{eqnarray*}
A-\{0\}\rightarrow \mathbb{N},\,\,\,f\mapsto \rm deg_{\it x}\,\it e^{x\partial}(f).
\end{eqnarray*}
En particulier, nous remarquons que $\rm (b)$
implique que l'anneau $\rm ker\,\it\partial$ est normal si $A$ est normal. L'assertion $\rm (e)$ est
une cons\'equence imm\'ediate de $\rm (a)$. L'assertion (d) est ais\'ee et laiss\'ee au 
lecteur\footnote{L'esquisse de d\'emonstration suivante de l'assertion (d) est sugg\'er\'ee par un des membres du jury 
que nous remer\c cions. L'ensemble $\rm Hom (\mathbb{G}_{a}, \rm Aut (\it \mathbb{A}^{\rm 1}_{\it \mathbf{k}}))$
des op\'erations de $\mathbb{G}_{a}$ dans $\mathbb{A}^{1}_{\mathbf{k}}$ s'identifie naturellement \`a l'ensemble
$\rm End(\it\mathbb{G}_{a})$ des endomorphismes du groupe alg\'ebrique $\mathbb{G}_{a}$. Une v\'erification directe
montre que $\rm End(\it\mathbb{G}_{a})$ est l'ensemble des $p$-polyn\^omes sur $\mathbb{G}_{a}$.}  
En utilisant les arguments de d\'emonstration
de [FZ $2$, $1.2\,\it (b)$] nous obtenons l'assertion $\rm (f)$:

\begin{proof}
(f) Par la propri\'et\'e $5.3.1\,\rm (iii)$, nous pouvons consid\'erer $d\in\mathbb{Z}_{>0}$ 
tel que $f':= \partial^{(d)}(f)\neq 0$ appartient \`a $\rm ker\,\it\partial$. 
Si l'id\'eal $(f)$ est $\partial$-stable alors
$f'\in\rm ker\,\it\partial\cap (f)$, de sorte que $f' = af$, pour un certain \'el\'ement $a\in A$.
Par l'assertion $5.3.4\,\rm (a)$, nous obtenons $f\in\rm ker\,\it\partial$.
R\'eciproquement, soit $a'\in A$. En utilisant $5.3.1\,\rm (ii)$, pour chaque $i\in\mathbb{N}$,
nous avons $\partial^{(i)}(a'f) = \partial^{(i)}(a')f$ et donc l'id\'eal $(f)$ est 
$\partial$-stable. 
\end{proof}

Dans le lemme suivant, nous \'etudions les extensions
des LFIHD sur un localis\'e 
$T^{-1}A$ donn\'e par une partie multiplicative $T\subset A$. 
Pour l'\'enonc\'e ci-dessous, nous nous sommes inspir\'es
des identit\'es remarquables bien connues des d\'erivations 
de Hasse-Teichm\"uller (voir [JKS, \S$2$]). 
\begin{lemme}
Soit $T$ un sous-ensemble de $A$ contenant $1$, stable par la multiplication de $A$
et tel que $0\not\in T$. Soit $\partial$ un syst\`eme it\'er\'e de d\'erivations 
d'ordre sup\'erieur sur l'alg\`ebre $A$. Pour tous $i,j\in\mathbb{Z}_{>0}$
tels que $j\leq i$, nous posons
\begin{eqnarray*}
E(i,j) = \left\{(s_{1},\ldots, s_{j})\in\mathbb{Z}^{j}_{>0},\,\, 
\sum_{l = 1}^{j}s_{l} = i\right\}. 
\end{eqnarray*}
Alors $\partial$ s'\'etend en un unique syst\`eme it\'er\'e de d\'erivations
d'ordre sup\'erieur $\bar{\partial} = \{\bar{\partial}^{(i)}\}_{i\in\mathbb{N}}$
sur l'alg\`ebre $T^{-1}A$. Pour tout $f\in T$ et pour tout
$i\in\mathbb{Z}_{>0}$, on a l'\'egalit\'e
\begin{eqnarray*}
\bar{\partial}^{(i)}\left(\frac{1}{f}\right)
= \sum_{j = 1}^{i}\frac{(-1)^{j}}{f^{j+1}}
\sum_{(s_{1},\ldots,s_{j})\in E(i,j)}
\partial^{(s_{1})}(f)\ldots\partial^{(s_{j})}(f)\,. 
\end{eqnarray*}
De plus, si $\partial$ est une $\rm LFIHD$ sur $A$ et si $T\subset\rm ker\,\it\partial$
alors $\bar{\partial}$ est une $\rm LFIHD$ sur $T^{-1}A$.
\end{lemme}
\begin{proof}
L'existence et l'unicit\'e de $\bar{\partial}$ sont donn\'ees
dans [Ma, $3.7$, $5.8$], [Voj, Section $3$]. En proc\'edant par
r\'ecurrence, le calcul de $\bar{\partial}^{(i)}(\frac{1}{f})$ 
est une cons\'equence de la propri\'et\'e $5.3.1\,\rm (ii)$. 
Le reste de la d\'emonstration est facile. 
\end{proof}
Comme cons\'equence du lemme pr\'ec\'edent, nous obtenons un r\'esultat
sur les rev\^etements cycliques d'une vari\'et\'e affine munie
d'une op\'eration du groupe additif (voir aussi [FZ $2$, Lemma $1.8$]).
\begin{corollaire}
Posons $K = \rm Frac\, A$. Consid\'erons une $\rm LFIHD$ $\partial$ sur $A$ et $f\in \rm ker\,\it\partial$ 
un \'el\'ement non nul. Soient $d\in\mathbb{Z}_{>0}$ et $u$ un \'el\'ement alg\'ebrique sur $K$
satisfaisant $u^{d}-f = 0$.
Si $B$ est la fermeture int\'egrale de $A[u]$ dans son corps
des fractions alors
$\partial$ s'\'etend en une unique $\rm LFIHD$ $\partial'$ sur l'alg\`ebre $B$
telle que $u\in\rm ker\,\it\partial'$.
\end{corollaire}
\begin{proof}
Par le lemme $5.3.5$, on peut \'etendre la LFIHD $\partial$ sur $A$ 
en un syst\`eme it\'er\'e de d\'erivations d'ordre sup\'erieur
sur le corps $K$, et sur l'alg\`ebre des polyn\^omes \`a une variable 
$K[t]$ en posant $\bar{\partial}^{(i)}(t) = 0$, pour tout
entier $i\geq 1$. Consid\'erons le morphisme de $K$-alg\`ebres $\phi:K[t]\rightarrow K[u]$, 
$t\mapsto u$. Soit $P\in K[t]$ le polyn\^ome unitaire engendrant l'id\'eal $\rm ker\,\it\phi$. 
Nous pouvons \'ecrire $t^{d}-f = FP$, pour un polyn\^ome $F\in K[t]$. 
Remarquons que $F$ est \'egalement unitaire.
Puisque $A$ est un anneau int\'egralement clos, nous obtenons $F,P\in A[t]$. 
De plus, pour tout $i\in\mathbb{Z}_{>0}$, $\bar{\partial}^{(i)}(FP) = \bar{\partial}^{(i)}(t^{d}-f) = 0$. Notons que $A[t]$
est $\bar{\partial}$-stable et que la restriction de $\bar{\partial}$ \`a  $A[t]$ est une
LFIHD. Par cons\'equent par la proposition $5.3.4.\, \rm (a)$, 
$P\in A[t]\cap \rm\ker\,\it \bar{\partial}$, d\'efinissant un
syst\`eme it\'er\'e de d\'erivations d'ordre sup\'erieur $\partial'$ sur $A[u]$. 
Ensuite, on remarque que l'op\'eration de $\mathbb{G}_{a}$ dans $\rm Spec\,\it A[u]$
correspondante \`a $\partial'$ se prolonge sur $B$ par la propri\'et\'e universelle de
la normalisation. Le reste de la d\'emonstration suit ais\'ement.
\end{proof} 

Dans la suite,
nous consid\'erons un tore alg\'ebrique d\'eploy\'e $\mathbb{T}\simeq \mathbb{G}^{n}_{m}$ sur le corps
 $\mathbf{k}$ et nous notons $M = \rm Hom(\it \mathbb{T},\mathbb{G}_{m})$ son r\'eseau
des caract\`eres. Comme d'habitude pour un vecteur $m\in M$, nous notons par $\chi^{m}$ 
le mon\^ome de Laurent associ\'e.
Nous supposons aussi que $A$ est $M$-gradu\'ee et nous 
\'ecrivons\footnote{Notons que la condition ''$A_{m}\subset K_{\rm 0}$, pour tout $m\in\sigma^{\vee}_{M}$'' est une hypoth\`ese suppl\'ementaire
par rapport aux autres conditions \'enonc\'ees.}
\begin{eqnarray*}  
A=\bigoplus_{m\in\sigma^{\vee}_{M}}A_{m}\chi^{m}\subset K_{0}[M]\,\,\,
\rm avec\,\,\,\it K_{\rm 0} = (\rm Frac\,\it A)^{\it \mathbb{T}}\,\,\,
\rm et\,\,\,\it A_{m}\subset K_{\rm 0}, 
\end{eqnarray*}
pour tout $m\in\sigma^{\vee}_{M}$.
Introduisons la notion de syst\`emes it\'er\'es de d\'erivations d'ordre sup\'erieur 
homog\`enes.
\begin{definition}
Soit $\partial$ un syst\`eme it\'er\'e de d\'erivations d'ordre sup\'erieur.
Consid\'erons un vecteur $e\in M$. La suite $\partial$
est dite \em homog\`ene de degr\'e \rm $e$ si pour chaque $i\in\mathbb{N}$ et
pour tout $m\in M$, nous avons la relation
\begin{eqnarray*}
\partial^{(i)}(A_{m}\chi^{m})\subset A_{m+ie}\chi^{m+ie}. 
\end{eqnarray*}
Le vecteur $e$ est appel\'e le \em degr\'e \rm de $\partial$ 
et on le note parfois par $\rm deg\,\it \partial$.

Dans le cas o\`u $\mathbf{k}$ est un corps de caract\'eristique $p>0$,
nous avons une d\'efinition plus g\'en\'erale.
\'Etant donn\'e $r\in\mathbb{N}$, nous disons que $\partial$ est \em rationnellement homog\`ene \rm 
de degr\'e $e/p^{r}$ (ou de bidegr\'e $(e,r)$) s'il satisfait les assertions suivantes.
\begin{enumerate}
\item[\rm (i)]
Si $i\in\mathbb{N}$ alors pour chaque $m$, l'application $\partial^{(ip^{r})}$ 
envoie $A_{m}\chi^{m}$ dans $A_{m+ie}\chi^{m+ie}$.
\item[\rm (ii)]
$\partial^{(j)} = 0$ lorsque $p^{r}$ ne divise pas $j$.
\end{enumerate}
\end{definition}
Dans [Li, Section $1.2$] on montre qu'une d\'erivation usuelle sur une alg\`ebre
multigradu\'ee qui envoie une pi\`ece gradu\'ee dans une autre est homog\`ene.
Cependant, cela n'est pas vrai en g\'en\'eral pour les syst\`emes de d\'erivations
d'ordre sup\'erieur.
Notons \'egalement que le noyau d'une LFIHD homog\`ene $\partial$ sur $A$
est une sous-alg\`ebre $M$-gradu\'ee de $A$. 
Dans la suite, nous introduisons quelques notions
dans le but d'avoir une interpr\'etation 
g\'eom\'etrique\footnote{On remarquera que les paragraphes $5.3.8$, $5.3.9$
se g\'en\'eralisent aux contextes des sch\'emas en groupes et des 
alg\`ebres de Hopf, lorsque $\mathbf{k}$ est un corps arbitraire.} 
des LFIHD homog\`enes et rationnellement homog\`enes.
\begin{notation}
Supposons que $\mathbf{k}$ est alg\'ebriquement clos. Alors on rappelle que le 
groupe des $\mathbf{k}$-points de $\mathbb{T}$ est isomorphe au groupe abstrait
$(\mathbf{k}^{\star})^{n}$. Soit $e\in M$ un vecteur de r\'eseau. D\'esignons par $G_{e}$ 
le groupe dont l'ensemble sous-jacent est $\mathbb{T}\times \mathbb{G}_{a}$ et
la loi de composition interne est d\'efinie par
\begin{eqnarray*}
(t_{1},\alpha_{1})\cdot (t_{2},\alpha_{2})
 = (t_{1}\cdot t_{2}, \chi^{-e}(t_{2})\cdot \alpha_{1} + \alpha_{2}), 
\end{eqnarray*}
o\`u $t_{i}\in\mathbb{T}$ et $\alpha_{i}\in\mathbb{G}_{a}$. 
En fait, tout produit semi-direct $\mathbb{T}\ltimes\mathbb{G}_{a}$ 
donn\'e par un caract\`ere alg\'ebrique
$\chi^{-e}:\mathbb{T}\rightarrow \rm Aut\,\it\mathbb{G}_{a}\simeq \mathbb{G}_{m}$ 
est isomorphe \`a un groupe de la forme $G_{e}$.
\end{notation}
La proposition suivante est presque analogue \`a [FZ $2$, Lemma $2.2$].
Par commodit\'e nous incluons une courte d\'emonstration.
\begin{proposition}
Supposons que $\mathbf{k}$ est un corps alg\'ebriquement clos.
\begin{enumerate}
 \item[\rm (i)] Si $A$ est $M$-gradu\'ee et si $\partial$ est une $\rm LFIHD$ homog\`ene
sur $A$ et de degr\'e $e$
alors l'op\'eration de $\mathbb{G}_{a}$ correspondante est normalis\'ee par l'op\'eration de $\mathbb{T}$.
Cela signifie que le groupe $G_{e}$ op\`ere dans la vari\'et\'e $X = \rm Spec\, A$. 
Le morphisme associ\'e est donn\'e par
\begin{eqnarray*}
\phi(t,\alpha) = t\cdot e^{\alpha\partial}(f), 
\end{eqnarray*}
o\`u $(t,\alpha)\in G_{e}$ et $f\in A$.
\item[\rm (ii)]
R\'eciproquement, si $G_{e}$ op\`ere dans $X = \rm Spec\,A$ alors 
les op\'erations des sous-groupes $\mathbb{T}$ et $\mathbb{G}_{a}$
donnent respectivement une $M$-graduation sur $A$ et une $\rm LFIHD$ homog\`ene
de degr\'e $e$.
\item[\rm (iii)]
Supposons que $\rm car\it \,\mathbf{k} = p\rm >0$. Soit 
$F_{r}:\mathbb{G}_{a}\rightarrow\mathbb{G}_{a}$, $t\mapsto t^{p^{r}}$ 
le morphisme de Frobenius g\'eom\'etrique it\'er\'e. Se donner une $\rm LFIHD$ rationnellement
homog\`ene de degr\'e $e/p^{r}$ sur $A$ est \'equivalent \`a se donner une op\'eration $\mathbb{G}_{a}\times X\rightarrow X$ \'egale \`a $\psi\circ (F_{p^{r}},\rm id_{\it X})$ o\`u $\psi$ est une
op\'eration de $\mathbb{G}_{a}$ normalis\'ee
par $\mathbb{T}$ comme dans l'assertion $\rm (i)$ ci-dessus.   
\end{enumerate}
\end{proposition}
\begin{proof}
$\rm (i)$
\'Etant donn\'es $(t,\alpha)\in G_{e}$
et $f\in A$, par homog\'en\'eit\'e de $\partial$, nous avons pour tout $i\in\mathbb{N}$,
\begin{eqnarray}
t\cdot\partial^{(i)}(f) = \chi^{ie}(t)\,\partial^{(i)}(t\cdot f). 
\end{eqnarray}
Cela donne
\begin{eqnarray*}
t\cdot e^{\alpha\partial}(f) = \sum_{i = 0}^{\infty}\chi^{ie}(t)\alpha^{i}
\,\partial^{(i)}(t\cdot f) = e^{\chi^{e}(t)\alpha\partial}(t\cdot f). 
\end{eqnarray*}
D'o\`u pour tous $(t_{1},\alpha_{1}), (t_{2},\alpha_{2})\in G_{e}$,
\begin{eqnarray*}
\phi((t_{1},\alpha_{1})\cdot (t_{2},\alpha_{2}))(f) = e^{\chi^{e}(t_{1})\alpha_{1}
\partial}\circ e^{\chi^{e}(t_{1}t_{2})\alpha_{2}\partial}(t_{1}t_{2}\cdot f) =  
\phi(t_{1},\alpha_{1})(\phi(t_{2},\alpha_{2})(f)).
\end{eqnarray*}
On conclut que $\phi$ d\'efinit une op\'eration de $G_{e}$ dans la vari\'et\'e $X = \rm Spec\,\it A$.

$\rm (ii)$ L'op\'eration du sous-groupe $\mathbb{G}_{a}\subset G_{e}$ donne
une LFIHD $\partial$ sur l'alg\`ebre $A$.
Pour $\alpha\in \mathbb{G}_{a}$ et $f\in A$, nous avons 
$\phi(1,\alpha)(f) = e^{\alpha\partial}(f)$.
Donc pour tout $t\in\mathbb{T}$, il vient 
\begin{eqnarray*}
t\cdot e^{\alpha\partial}(f) = 
\phi((1,\chi^{e}(t)\alpha)\cdot (t,0)) (f) = 
e^{\chi^{e}(t)\alpha\partial}(t\cdot f).
\end{eqnarray*}
En identifiant les coefficients nous obtenons $(5.1)$.
Ainsi, la LFHID $\partial$ est homog\`ene pour la $M$-graduation
donn\'ee par l'op\'eration du sous-groupe $\mathbb{T}\subset G_{e}$.
L'assertion $\rm (iii)$ suit imm\'ediatement de $\rm (i),(ii)$.
\end{proof}
Pour un corps arbitraire $\mathbf{k}$, nous consid\'erons la d\'efinition
naturelle suivante.
\begin{definition}
Supposons que le tore $\mathbb{T}$ op\`ere dans $X = \rm Spec\,\it A$. 
Une op\'eration de $\mathbb{G}_{a}$ dans $X$ est dite \em normalis\'ee \rm (resp.
\em normalis\'ee \`a Frobenius pr\`es \rm) par l'op\'eration de $\mathbb{T}$ 
si la LFIHD correspondante $\partial$ est homog\`ene (resp. rationnellement homog\`ene). 
\end{definition}

Pour classifier les op\'erations normalis\'ees de $\mathbb{G}_{a}$,
il est commode de les s\'eparer en deux types (voir [FZ $2$, $3.11$] 
et [Li, $1.11$] pour des cas sp\'ecifiques). 
\begin{definition}
Une LFIHD homog\`ene
$\partial$ est dite de type \em vertical \rm (ou aussi de type fibre) 
si $\bar{\partial}^{(i)}(K_{0}) = \{0\}$, pour tout $i\in\mathbb{Z}_{>0}$.
Dans le cas contraire, on dit que $\partial$ est de \em type horizontal. \rm 
Nous utilisons des qualifications analogues pour les op\'erations normalis\'ees
du groupe additif $\mathbb{G}_{a}$. Une $\mathbb{T}$-vari\'et\'e affine d\'eploy\'ee
est dite \em vertical \rm (resp. \em horizontal\rm) si elle est munie 
d'une op\'eration normalis\'ee de $\mathbb{G}_{a}$ de type vertical (resp. horizontal). 
\end{definition}
Bien s\^ur, une LFIHD homog\`ene de type horizontal n'est pas triviale.
Dans le cas vertical, on peut \'etendre une LFIHD homog\`ene
sur $A$ en une LFIHD de l'alg\`ebre de mono\"ide $K_{0}[\sigma^{\vee}_{M}]$
par endomorphismes $K_{0}$-lin\'eaires.
\begin{lemme}
Soit $\partial$ une $\rm LFIHD$ homog\`ene de type vertical sur l'alg\`ebre $M$-gradu\'ee $A$. 
Alors $\partial$ s'\'etend en un unique syst\`eme it\'er\'e de $K_{0}$-d\'erivations d'ordre sup\'erieur 
localement fini sur l'alg\`ebre de mono\"ide $K_{0}[\sigma^{\vee}_{M}]$. 
\end{lemme}
\begin{proof}
Par le lemme $5.3.5$, la LFIHD $\partial$ s'\'etend en un syst\`eme it\'er\'e de d\'erivations
d'ordre sup\'erieur $\partial'$ sur $K_{0}[M]$. Puisque $\partial$
est de type vertical, la condition $5.3.1\rm (iii)$ implique que 
chaque $\partial'^{(i)}$ est une application $K_{0}$-lin\'eaire. Par cons\'equent,
si $S\subset M$ est le mono\"ide des poids de l'alg\`ebre $M$-gradu\'ee $A$
alors $B:= K_{0}[S] = A\otimes_{\mathbf{k}}K_{0}$ est $\partial'$-stable.
Montrons que $\partial'|_{B}$ est une LFIHD sur $B$. Soit $f\chi^{m}\in B$ 
un \'el\'ement homog\`ene avec $f\in K_{0}^{\star}$. 
\'Ecrivons $f\chi^{m} = f'h\chi^{m}$, pour $f'\in K_{0}$
et $h\in A_{m}$. Il existe $r\in\mathbb{Z}_{>0}$
tel que pour tout $i\geq r$,
\begin{eqnarray*}
\partial'^{(i)}(f\chi^{m}) = f'\partial^{(i)}(h\chi^{m}) = 0. 
\end{eqnarray*}
Puisque chaque \'el\'ement de $B$ est une somme d'\'el\'ements homog\`enes, on
conclut que $\partial'|_{B}$ est localement finie sur $B$. Ainsi, $\partial'|_{B}$ 
s'\'etend en une LFIHD $\bar{\partial}$ sur la fermeture int\'egrale 
$\bar{B} = K_{0}[\sigma^{\vee}_{M}]$ (voir [Se] et [Ma, $5.8$]).  
\end{proof}
Dans les deux prochaines assertions, nous donnons quelques r\'esultats \'el\'ementaires
sur les LFIHD de l'alg\`ebre des polyn\^omes \`a une variable. Ces r\'esultats seront utiles
pour \'etudier les op\'erations normalis\'ees du groupe additif de type
horizontal de la section $5.6$.
\begin{lemme}
Supposons que la caract\'eristique du corps $\mathbf{k}$ est un nombre premier $p$.
Soit $\partial$ une $\rm LFIHD$ sur l'alg\`ebre des polyn\^omes $k[t]$ \`a une variable $t$. 
\'Ecrivons
\begin{eqnarray*}
e^{x\partial}(t) = t + \sum_{i = 1}^{r}\lambda_{i}x^{p^{s_{i}}}, 
\end{eqnarray*}
o\`u $\lambda_{i}\in\mathbf{k}^{\star}$ et $s_{1}<\ldots <s_{r}$ sont des entiers naturels
$\rm ($voir $5.3.4 \rm (d)\rm )$. Nous consid\'erons la valuation naturelle
\begin{eqnarray*}
\rm  ord \, :\it\mathbf{k}[t]-\{\rm 0\it\}\rightarrow\mathbb{N},\,\,\,
\sum_{i}a_{i}t^{i}\mapsto \min\{i\,|\, a_{i}\neq \rm 0\it\}  
\end{eqnarray*}
et nous fixons un entier $i\in\mathbb{N}$. Si $l\in\mathbb{N}$ v\'erifie $l\geq ip^{s_{1}}$ alors
\begin{eqnarray*}
\partial^{(ip^{s_{1}})}(t^{l}) = \lambda_{1}^{i}\binom{l}{i}t^{l - i}, 
\end{eqnarray*}
et par cons\'equent $\rm ord\,\it\partial^{(ip^{s_{\rm 1\it}})}(t^{l}) = l - i$ d\`es que $\binom{l}{i}\neq 0$.
\end{lemme}
\begin{proof}
Tout d'abord, 
\begin{eqnarray*}
e^{x\partial}(t^{l}) = e^{x\partial}(t)^{l} = \left(t + \sum_{i = 1}^{r}\lambda_{i}x^{p^{s_{i}}}\right)^{l}
= \sum_{i_{0} + \ldots + i_{r} = l,\,i_{0},\ldots,i_{r}\geq 0}
\binom{l}{i_{0}\ldots i_{r}}t^{i_{0}}\prod_{\alpha = 1}^{r}(\lambda_{\alpha}x^{p^{s_{\alpha}}})^{i_{\alpha}}. 
\end{eqnarray*}
En consid\'erant le terme de degr\'e $ip^{s_{1}}$ en $x$ de la somme ci-dessus, nous sommes amen\'es \`a \'etudier
les conditions suivantes:
\begin{eqnarray}
ip^{s_{1}} = i_{1}p^{s_{1}} + \ldots + i_{r}p^{s_{r}}\,\,\, 
\rm
et
\,\,\,\it
i_{\rm 0\it} + i_{\rm 1\it} + \ldots + i_{r} = l 
\end{eqnarray}
o\`u $(i_{0},i_{1},\ldots, i_{r})\in\mathbb{N}^{r}$. 
Notons qu'une liste
$(i_{0},i_{1},\ldots, i_{r})$ satisfaisant $(5.2)$ existe puisque $l\geq ip^{s_{1}}$; en effet, nous pouvons prendre 
\begin{eqnarray*}
(i_{0},i_{1},\ldots, i_{r}) = (l-i,i,0,\ldots, 0).
\end{eqnarray*}
Montrons que c'est l'unique choix pour $i_{0}\in\mathbb{N}$ minimal.
Soit $(\gamma_{0},\gamma_{1},\ldots, \gamma_{r})\in\mathbb{N}^{r}$ un 
$r+1$-uplet satisfaisant la propri\'et\'e $(5.2)$ avec $\gamma_{0}\in\mathbb{N}$ minimal. Alors nous avons
\begin{eqnarray*}
l - i = l - \sum_{\alpha = 1}^{r}\gamma_{\alpha}p^{s_{\alpha} - s_{1}}\leq 
l - \sum_{\alpha = 1} \gamma_{\alpha} = \gamma_{0}.  
\end{eqnarray*}
D'o\`u par minimalit\'e, $\gamma_{0} = l - i$, de sorte que $i = \sum_{\alpha = 1}^{r}\gamma_{\alpha}$.
Ainsi, 
\begin{eqnarray*}
\left( \sum_{\gamma_{\alpha}}^{r}\gamma_{\alpha}\right) p^{s_{1}} = \sum_{\alpha = 1}^{r}\gamma_{\alpha}p^{s_{\alpha}}.
\end{eqnarray*}
Nous obtenons $(\gamma_{0},\gamma_{1},\ldots, \gamma_{r}) = (l-i,i,0,\ldots, 0)$.
Cela implique en particulier que $\partial^{(ip^{s_{1}})} = \lambda_{1}^{i}\binom{l}{i}t^{l - i}$,
comme demand\'e. 
\end{proof}
Dans le prochain corollaire, nous utilisons les notations du lemme $5.3.13$.
\begin{corollaire}
Supposons que la caract\'eristique du corps $\mathbf{k}$ est un nombre premier $p$.
\'Etendons la fonction $\rm ord$ et la $\rm LFIHD$ $\partial$ en un syst\`eme it\'er\'e de d\'erivations
d'ordre sup\'erieur sur l'alg\`ebre $\mathbf{k}[t,t^{-1}]$ des polyn\^omes de Laurent. Fixons $i\in\mathbb{N}$. 
Pour tout $l\geq ip^{s_{1}}$, on a
\begin{eqnarray*}
\partial^{(ip^{s_{1}})}(t^{-l}) = (-\lambda_{1})^{i}\binom{l}{i}t^{-l - i}. 
\end{eqnarray*}
Par cons\'equent, dans cette situation, $\rm ord\,\it\partial^{(ip^{s_{\rm 1\it}})}(t^{l}) = -l - i$ d\`es que $\binom{l}{i}\neq 0$.
\end{corollaire}
\begin{proof}
Consid\'erons l'assertion suivante $\mathscr{A}_{i}$: 
Pour tout $l\in\mathbb{N}$ tel que $l\geq ip^{s_{1}}$, on a 
$\partial^{(ip^{s_{1}})}(t^{-l}) = (-\lambda_{1})^{i}\binom{l}{i}t^{-l - i}$.
Montrons par r\'ecurrence que pour tout $i\in\mathbb{N}$,
 $\mathscr{A}_{i}$ est vraie. \'Evidemment, l'assertion $\mathscr{A}_{0}$ est vraie. 
En fixant $i\in\mathbb{Z}_{>0}$,
nous pouvons supposer que pour tout entier naturel $j<i$, $\mathscr{A}_{j}$ est vraie.
Alors pour tout $l\geq ip^{s_{1}}$, 
\begin{eqnarray*}
0 = \partial^{(ip^{s_{1}})}(t^{l}\cdot t^{-l}) = \sum_{i_{1} + i_{2} = i, i_{1},i_{2}\in\mathbb{N}}\partial^{(i_{1}p^{s_{1}})}(t^{l})
\cdot\partial^{(i_{2}p^{s_{1}})}(t^{-l}), 
\end{eqnarray*}
puisque pour tout $s\in\mathbb{N}-\mathbb{Z}p^{s_{1}}$, l'application $\partial^{(s)}$ est identiquement nulle.
En utilisant le lemme $5.3.13$, nous avons
\begin{eqnarray*}
0 = \sum_{i_{1} + i_{2} = i, i_{1},i_{2}\in\mathbb{N}}\lambda_{1}^{i_{1}}\binom{l}{i_{1}}t^{l-i_{1}}
\partial^{(i_{2}p^{s_{1}})}(t^{-l}). 
\end{eqnarray*}
Donc par l'hypoth\`ese de r\'ecurrence, il s'ensuit que
\begin{eqnarray*}
\partial^{(ip^{s_{1}})} = -\lambda_{1}^{i}\left(\sum_{i_{1}+i_{2} = i, i_{1}\geq 1, i_{2}\geq 0}\binom{l}{i_{1}}
\binom{l}{i_{2}}(-1)^{i_{2}}\right)t^{-l-i}. 
\end{eqnarray*}
En utilisant la r\`egle de produit I dans [JKS, \S$2$] pour $z = (-1)\cdot 1$, on obtient
\begin{eqnarray*}
\partial^{(ip^{s_{1}})}(t^{-l}) = -\lambda_{1}^{i}\left( \binom{l}{i}(-1)^{i -1}\right)t^{-l-i}
 = (-\lambda_{1})^{i}\binom{l}{i}t^{-l-i}. 
\end{eqnarray*}
Cela donne l'assertion $\mathscr{A}_{i}$ et termine la d\'emonstration du corollaire.
\end{proof}

\section{Op\'erations du groupe additif dans les vari\'et\'es toriques affines}
Soit $\mathbf{k}$ un corps.
Dans cette section, nous donnons une description combinatoire 
des vari\'et\'es toriques affines munies d'une op\'eration normalis\'ee du 
groupe additif \`a Frobenius pr\`es.
Pour un c\^one poly\'edral $\sigma\subset N_{\mathbb{Q}}$,
nous d\'esignons par $\sigma(1)$, l'ensemble de ses ar\^etes.
Comme d'habitude, une ar\^ete de $\sigma$ et son vecteur primitif correspondant
sont \'ecrits par une m\^eme lettre $\rho$.
Le point suivant est une d\'efinition classique, voir 
par exemple [De, Section $4.5$], 
[Li, $2.3$], [AL, $1.5$]. 
\begin{definition}
Soit $\sigma\subset N_{\mathbb{Q}}$ un c\^one 
poly\'edral saillant. Un vecteur $e\in M$ est une \em racine de Demazure \rm (ou simplement
\em une racine) \rm si les conditions suivantes sont vraies.
\begin{enumerate}
 \item[\rm (i)]
Il existe $\rho\in \sigma(1)$ tel que $\langle e,\rho\rangle =-1$.
\item[\rm (ii)] Pour tout $\rho'\in \sigma(1)-\{\rho\}$, nous avons 
$\langle e,\rho'\rangle\geq 0$. 
\end{enumerate}
Le vecteur $\rho$ satisfaisant $\langle e,\rho\rangle =-1$ est appel\'e
le \em rayon distingu\'e \rm de la racine $e\in M$. 
Nous d\'esignons par $\rm Rt\,\it\sigma$ l'ensemble des racines de Demazure du
c\^one $\sigma$.
D'apr\`es [Li, $2.5$], tout \'el\'ement de $\sigma(1)$ est le rayon distingu\'e d'une
racine de $\rm Rt\,\it\sigma$. 
\end{definition}
Mentionnons quelques d\'eveloppements autour de la th\'eorie des
racines de Demazure. Le lecteur peut consulter [De, Ni, Ba, AHHL] 
pour l'\'etude des automorphismes des $\mathbb{T}$-vari\'et\'es compl\`etes. 
Voir [Li $3$, Ko] pour les racines des groupes de Cremona affines
et le probl\`eme de surjectivit\'e.

Puisque le sous-ensemble $\mathbf{k}[\mathbb{T}]^{\star}$ engendre l'alg\`ebre 
$\mathbf{k}[\mathbb{T}]$, la proposition $5.3.4\,\rm (e)$ implique que
$\mathbf{k}[\mathbb{T}]$ ne poss\`ede pas de LFIHD non triviale. Donc
nous consid\'erons uniquement des vari\'et\'es toriques affines de la forme
$X_{\sigma} = \rm Spec\,\it\mathbf{k}[\sigma^{\vee}_{M}]$ donn\'ees par un 
c\^one poly\'edral saillant non nul $\sigma\subset N_{\mathbb{Q}}$. 
\begin{exemple}
Soit $e\in\rm Rt\,\it \sigma$ une racine. Consid\'erons la d\'erivation homog\`ene
sur l'alg\`ebre de mono\"ide $\mathbf{k}[\sigma^{\vee}_{M}]$ d\'efinie par
\begin{eqnarray*}
\partial_{e}(\chi^{m}) = \langle m,\rho\rangle \chi^{m+e}, 
\end{eqnarray*}
o\`u $\rho$ est le rayon distingu\'e de $e$. Alors  
$\partial_{e}$ est localement nilpotente et 
donne une op\'eration de $\mathbb{G}_{a}$ dans la vari\'et\'e $X_{\sigma}$ comme suit. 
La LFIHD homog\`ene est donn\'ee par les applications\footnote{Pour tous $r_{1},r_{2}\in\mathbb{N}$,
nous posons $\binom{r_{1}}{r_{2}} = 0$ lorsque $r_{1}<r_{2}$.
} 
\begin{eqnarray*}
 \partial_{e}^{(i)}(\chi^{m}) = \binom{\langle m,\rho\rangle}{i}\chi^{m+ie},\,\, i = 0,1,2,3,\ldots,
\end{eqnarray*}
o\`u $m\in\sigma^{\vee}_{M}$.
Son noyau est $\mathbf{k}[\rho^{\star}\cap M]$.
Le sous-ensemble $\rho^{\star}\subset \sigma^{\vee}$ est la face duale de $\rho$.
Supposons $\rm car \,\it \mathbf{k}\rm = \it p\rm  >0$.
En partant de $\partial_{e}$ et d'un entier $r\in\mathbb{N}$, 
nous pouvons \'egalement d\'efinir une LFIHD rationnellement homog\`ene $\partial_{e,r}$ de degr\'e $e/p^{r}$. 
Son application exponentielle est
\begin{eqnarray*}
e^{x\partial_{e,r}} = \sum_{i = 0}^{\infty}\partial_{e}^{(i)}\,x^{ip^{r}}\,. 
\end{eqnarray*}
On v\'erifie ais\'ement que $\rm ker\,\it \partial_{r,e} = \mathbf{k}[\rho^{\star}\cap M]$. 
De plus, pour tout $m\in\sigma^{\vee}_{M}$, nous obtenons 
\begin{eqnarray*}
\rm deg_{\it x}\,\it e^{x\partial_{e,r}}(\chi^{m}) = p^{r}\langle m,\rho\rangle. 
\end{eqnarray*}

\end{exemple}
Dans la suite, nous d\'esignons par $A$ l'alg\`ebre de mono\"ide 
$\mathbf{k}[\sigma^{\vee}_{M}]$, o\`u $\sigma\subset N_{\mathbb{Q}}$
est un c\^one poly\'edral saillant non nul.
Nous commen\c cons par d\'ecrire le noyau et les vecteurs degr\'e possibles 
d'une LFIHD homog\`ene sur $A$.
\begin{lemme}
Consid\'erons une $\rm LFIHD$ homog\`ene non triviale $\partial$ sur $A$. 
Alors les assertions suivantes sont vraies.
\begin{enumerate}
 \item[\rm (i)] Il existe $\rho\in\sigma(1)$ tel que $\rm 
\rm ker\,\it \partial = \mathbf{k}[\rho^{\star}\cap M]$, o\`u $\rho^{\star}\subset M_{\mathbb{Q}}$
est la face duale de $\rho$.
\item[\rm (ii)]
Le degr\'e $e\in M$ de la $\rm LFIHD$ $\partial$ est une racine de Demazure du c\^one $\sigma$ 
et $\rho$ est son rayon distingu\'e.
\end{enumerate}  
\end{lemme}
\begin{proof}
$\rm (i)$ Par le proposition $5.3.4\,\rm (a)$, nous avons $\rm ker\,\partial = \it
\mathbf{k}[W\cap\sigma^{\vee}_{M}]$, pour un hyperplan $W\subset M_{\mathbb{Q}}$.
Supposons que $W\cap\sigma^{\vee}$ n'est pas une face de $\sigma^{\vee}$. Alors $W$
divise $\sigma^{\vee}$ en deux parties. Nous pouvons trouver $m\in\sigma^{\vee}_{M}$ 
tel que pour tout $r\in\mathbb{N}$, 
$m+re\not\in W$. Puisque 
$\chi^{m}\not\in \rm ker\,\it\partial$, il existe $r_{0}\in\mathbb{Z}_{>0}$ 
satisfaisant $\partial^{(r_{0})}(\chi^{m})\neq 0$.
D'o\`u $\partial^{(r_{0})}(\chi^{m})$ est homog\`ene de degr\'e $m+r_{0}e$. 
Par l'argument pr\'ec\'edent, 
\begin{eqnarray*}
\partial^{(r'_{1})}\circ\partial^{(r_{0})}(\chi^{m})\neq 0, 
\end{eqnarray*}
pour $r'_{1}\in\mathbb{Z}_{>0}$. En utilisant $5.3.1\,\rm (iv)$, on a
$\partial^{(r_{0}+r_{1}')}(\chi^{m})\neq 0$. Posons $r_{1} = r_{0}+r'_{1}$. 
En proc\'edant par r\'ecurrence, nous pouvons construire une suite strictement croissante
d'entiers $\{r_{j}\}_{j\in\mathbb{N}}$ avec $r_{0}$ non nul
v\'erifiant $\partial^{(r_{j})}(\chi^{m})\neq 0$, 
pour tout $j\in\mathbb{N}$. Cela contredit le fait que $\partial$ est une LFIHD. Ainsi 
$W\cap \sigma^{\vee}$ est une face de $\sigma^{\vee}$. Puisque $\rm ker\,\partial$ est une
sous-alg\`ebre de dimension $\rm dim\,\it A - \rm 1$, nous arrivons au r\'esultat souhait\'e.

$\rm (ii)$ 
Si $e\in\sigma^{\vee}_{M}$ alors par le m\^eme argument que ci-dessus, on
obtient \`a nouveau une contradiction. Le reste de la d\'emonstration suit
de [Li, Lemma $2.4$].  
\end{proof}
Dans le lemme suivant, nous donnons quelques propri\'et\'es des LFIHD homog\`enes sur $A$.
Elles seront utiles pour notre r\'esultat principal de classification (voir $5.4.5$). 
\begin{lemme}
Soit $\partial$ une $\rm LFIHD$ homog\`ene non triviale sur $A$ de degr\'e $e$. 
Pour tout $m\in\sigma^{\vee}_{M}$ et tout $i\in\mathbb{N}$, 
nous posons 
\begin{eqnarray*}
\partial^{(i)}(\chi^{m}) = c_{i}(m)\chi^{m+ie},
\end{eqnarray*}
o\`u $c_{i}(m)\in\mathbf{k}$. \'Ecrivons $\rm ker\,\it\partial = 
\mathbf{k}[\rho^{\star}\cap M]$ pour une ar\^ete $\rho\in\sigma(1)$
$\rm ($voir $5.4.3\rm )$.
Alors la suite de fonctions $\{c_{i}\}_{i\in\mathbb{N}}$ v\'erifie 
les \'enonc\'es suivants.
\begin{enumerate}
 \item[\rm (i)] $c_{0}$ est l'application constante $m\mapsto 1$.
\item[\rm (ii)] Pour tous $m,m'\in\sigma^{\vee}_{M}$, nous avons
\begin{eqnarray}
c_{i}(m+m') = \sum_{j = 0}^{i}c_{i-j}(m)\cdot c_{j}(m'). 
\end{eqnarray}
\item[\rm (iii)] \'Etant donn\'e $m\in\sigma^{\vee}$, il existe
$r\in\mathbb{N}$ tel que pour tout $i\geq r$, $c_{i}(m) = 0$. 
\item[\rm (iv)] Pour tous $i,j\in\mathbb{N}$ et pour tout
$m\in\sigma^{\vee}_{M}$, on a 
\begin{eqnarray*}
\binom{i+j}{i}\,c_{i+j}(m) = c_{i}(m+je)\cdot c_{j}(m). 
\end{eqnarray*}
\item[\rm (v)] Pour tout $m'\in\rho^{\star}\cap M$ et pour tout 
$m\in\sigma^{\vee}_{M}$, $c_{i}(m+m') = c_{i}(m)$.
\end{enumerate}
\end{lemme}
\begin{proof}
Les assertions $\rm (i), (ii), (iii), (iv)$ suivent de la d\'efinition
d'une LFIHD. Montrons l'assertion $\rm (v)$. 
Puisque $\chi^{m'}\in \rm ker\,\it\partial$,
pour tout $j\in\mathbb{Z}_{>0}$, nous avons $c_{j}(m') = 0$. 
En appliquant l'\'egalit\'e $(5.3)$, nous obtenons $c_{i}(m+m') = c_{i}(m)$.
\end{proof}
Le prochain r\'esultat donne une classification des op\'erations normalis\'ees du groupe additif $\mathbb{G}_{a}$ 
dans $X_{\sigma}$. Voir le th\'eor\`eme $2.7$ dans [Li] pour le cas o\`u $\mathbf{k}$ est de caract\'eristique
z\'ero.
\begin{theorem}
Soit $\sigma\subset N_{\mathbb{Q}}$ un c\^one poly\'edral saillant non nul.
Toute op\'eration non triviale normalis\'ee de $\mathbb{G}_{a}$ dans $X_{\sigma}$ par l'op\'eration de 
$\mathbb{T}$ est donn\'ee par une $\rm LFIHD$ homog\`ene de la forme $\lambda\partial_{e}$,
o\`u $e\in \rm Rt\,\it \sigma$ et $\lambda\in\mathbf{k}^{\star}$ $\rm ($voir $5.4.2\rm )$.   
\end{theorem}
\begin{proof}
Posons $A = \mathbf{k}[\sigma^{\vee}_{M}]$ et soit $\partial$ 
une LFIHD homog\`ene non triviale de degr\'e $e$.
Par le lemme $5.4.3$, il existe un rayon $\rho\in\sigma(1)$ tel que
$\rm ker\,\it\partial = \mathbf{k}[\rho^{\star}\cap M]$. De plus,
$e\in \rm Rt\,\it \sigma$ est une racine et $\rho$ est son rayon distingu\'e.

Montrons d'abord qu'il existe un vecteur $m\in\sigma^{\vee}_{M}$ tel que
$\langle m, \rho\rangle = 1$. \'Etant donn\'e $m'\in\sigma^{\vee}_{M}$ non contenu
dans $\rho^{\star}$, nous avons $\langle m',\rho\rangle > 1$. Par [Li, Lemma $2.4$],
nous obtenons
\begin{eqnarray*}
m := m' + (\langle m',\rho\rangle-1)\cdot e\in\sigma^{\vee}_{M}, 
\end{eqnarray*}
et donc $\langle m,\rho\rangle  = 1$.

Pour la LFIHD $\partial$, consid\'erons les m\^emes notations que 
dans le lemme $5.4.4$.
Soit $B = \mathbf{k}[x]$ l'alg\`ebre des polyn\^omes \`a une variable $x$. 
En utilisant la base $(1,x,x^{2},\ldots)$, nous d\'efinissons une suite
d'op\'erateurs lin\'eaires $\bar{\partial}$
sur l'espace vectoriel $B$ de la fa\c con suivante.
Fixons un vecteur $m\in\sigma^{\vee}_{M}$ v\'erifiant 
$\langle m,\rho\rangle  = 1$.
Pour tous entiers naturels $i,r$, nous posons
\begin{eqnarray*}
\bar{\partial}^{(i)}(x^{r}) = c_{i}(rm)x^{r-i}. 
\end{eqnarray*}
Notons que $\bar{\partial}$ est bien d\'efinie. 
En effet, supposons que $i,r\in\mathbb{N}$ 
satisfont $i>r$. Nous avons
\begin{eqnarray*}
\partial^{(i)}(\chi^{rm}) = c_{i}(rm)\chi^{rm+ie}\in A\,\,\,\rm et\,\,\,\it
\langle rm+ie, \rho \rangle = r-i\rm < 0,  
\end{eqnarray*}
donnant $c_{i}(rm) = 0$. D'o\`u pour $i>r$, on obtient 
$\bar{\partial}^{(i)}(x^{r}) = 0$. 

Par les assertions $\rm (i), (ii), (iii), (iv), (v)$
du lemme $5.4.4$, la suite d'op\'erateurs 
$\bar{\partial}$ est une LFIHD sur l'alg\`ebre $B$. 
Par exemple, montrons que $\bar{\partial}$ 
v\'erifie $5.3.1\,\rm (iv)$.
En consid\'erant $i,j\in\mathbb{N}$, nous avons
\begin{eqnarray*}
\bar{\partial}^{(i)}\circ \bar{\partial}^{(j)}(x^{r}) = \bar{\partial}^{(i)}(c_{j}(rm)x^{r-j})
 = c_{i}((r-j)m)\cdot c_{j}(rm)x^{r-i-j}.  
\end{eqnarray*}
Puisque $e\in \rm Rt\,\it\sigma$ est une racine ayant $\rho$
comme rayon distingu\'e, il s'ensuit que
\begin{eqnarray*}
v:= rm + je - (r-j)m = j(m+e)\in\rho^{\star}\cap M. 
\end{eqnarray*}
L'assertion $(v)$ du lemme $5.4.4$ implique
\begin{eqnarray*}
c_{i}((r-j)m) = c_{i}((r-j)m + v) = c_{i}(rm + je).  
\end{eqnarray*}   
Par $5.4.4\,\rm (iv)$, on conclut que 
\begin{eqnarray*}
\bar{\partial}^{(i)}\circ\bar{\partial}^{(j)}(x^{r}) = \binom{i+j}{i}c_{i+j}(rm)x^{r-i-j} =
\binom{i+j}{i}\bar{\partial}^{(i+j)}(x^{r}).  
\end{eqnarray*}
Les conditions $\rm (i),(ii),(iii)$ de $5.3.1$ pour la suite $\bar{\partial}$
suivent d'un calcul ais\'e.

Puisque $\bar{\partial}$ est homog\`ene pour la graduation naturelle
de $B$, par la proposition $5.3.4\,\rm (d)$,
il existe $\lambda\in\mathbf{k}$ tel que
chaque $c_{i}$ v\'erifie 
\begin{eqnarray*}
c_{i}(rm) = \binom{r}{i} \,\lambda^{i},
\end{eqnarray*}
pour tout $r\in\mathbb{N}$. Nous utilisons ici la convention $\lambda^{0} = 1$ lorsque
$\lambda = 0$. Soit $w\in\sigma^{\vee}_{M}$. Les \'el\'ements 
\begin{eqnarray*}
w + \langle w,\rho\rangle e\rm\,\,\, et\,\,\,\it \langle w,\rho\rangle e + \langle w,\rho\rangle m  
\end{eqnarray*}
appartiennent \`a $\rho^{\star}\cap M$. Par le lemme $5.4.4\,\rm (v)$, cela implique que
\begin{eqnarray}
c_{i}(w) = c_{i}\left( w + \langle w,\rho\rangle e + \langle w,\rho\rangle m\right )=
c_{i}\left (\langle w,\rho\rangle m\right ) = \binom{\langle w,\rho\rangle}{i}\,\lambda^{i}.
\end{eqnarray}
Puisque $\partial$ n'est pas triviale, $\lambda\in\mathbf{k}^{\star}$. 
En vertu de
$(5.4)$ la suite $\partial$ est donn\'ee par la LFIHD $\lambda\partial_{e}$ (voir l'exemple $5.4.2$). 
\end{proof}
Comme cons\'equence imm\'ediate, nous obtenons une description de toutes les op\'erations normalis\'ees 
du groupe additif \`a Frobenius pr\`es dans les vari\'et\'es toriques affines.
Ce r\'esultat nous sera utile dans le but de classifier les $\mathbb{T}$-vari\'et\'es affines horizontales
de complexit\'e $1$.
\begin{corollaire}
Pour toute $\rm LFIHD$ rationnellement homog\`ene non triviale $\partial$ sur l'alg\`ebre $A$
et de degr\'e $e/p^{r}$,
il existe une racine $e\in Rt\,\it\sigma$ de rayon $\rho$, 
un entier $r\in\mathbb{N}$,
et un scalaire $\lambda\in\mathbf{k}^{\star}$ satisfaisant l'\'enonc\'e suivant.
L'application exponentielle de $\partial$ est donn\'ee par
\begin{eqnarray*}
e^{x\partial}(\chi^{m}) = \sum_{i = 0}^{\infty}\binom{\langle m,\rho\rangle}{i}\lambda^{i}\,\chi^{m+ie}x^{ip^{r}}, 
\end{eqnarray*}
o\`u $m\in\sigma^{\vee}_{M}$.
R\'eciproquement, toute $\rm LFIHD$ rationnellement homog\`ene sur $A$ provient de cette mani\`ere.
\end{corollaire}
Dans le prochain corollaire, nous g\'en\'eralisons dans notre contexte des r\'esultats de 
[Li, Section $2$]. Concernant l'assertion $\rm (i)$,
le lecteur peut consulter [Ku, Corollary $3.5$] pour le cas classique.
\begin{corollaire}
Soit $\sigma\subset N_{\mathbb{Q}}$ un c\^one poly\'edral saillant et posons $A = \mathbf{k}[\sigma^{\vee}_{M}]$.
\begin{enumerate}
\item[\rm (i)] Pour toute op\'eration normalis\'ee de $\mathbb{G}_{a}$ \`a Frobenius pr\`es
dans une vari\'et\'e torique affine, l'alg\`ebre des invariants correspondante est de type 
fini sur $\mathbf{k}$.  
\item[\rm (ii)]
Il existe un nombre fini de $\rm LFIHD$ homog\`enes sur $A$
ayant des noyaux deux \`a deux distincts.
\end{enumerate}
\end{corollaire}
\begin{proof}
L'assertion $\rm (i)$ est une cons\'equence directe de l'exemple $5.4.2$, du lemme $5.4.3$
et des arguments de d\'emonstration de [AH, $4.1$]. Pour l'assertion $\rm (ii)$,
cela suit du fait que $\sigma(1)$ est un ensemble fini.  
\end{proof}

\section{Op\'erations du groupe additif de type vertical}
Soit $\mathbf{k}$ un corps.
Dans cette section, nous classifions les op\'erations normalis\'ees du groupe additif
de type vertical dans une $\mathbb{T}$-vari\'et\'e affine $X = \rm Spec\,\it A$
de complexit\'e $1$ sur $\mathbf{k}$. Voir [Li $2$] pour
la complexit\'e arbitraire lorsque le corps de base est alg\'ebriquement 
clos de caract\'eristique z\'ero. Pour cela, nous consid\'erons 
$A = A[C,\mathfrak{D}]$, o\`u $C$ est une courbe r\'eguli\`ere sur 
$\mathbf{k}$ et $\mathfrak{D} = \sum_{z\in C}\Delta_{z}\cdot z$ 
est un diviseur $\sigma$-poly\'edral propre. Dans la suite, 
pour tout $m\in\sigma^{\vee}_{M}$, nous posons 
\begin{eqnarray*} 
A_{m} = H^{0}(C,\mathcal{O}_{C}(\lfloor\mathfrak{D}(m)\rfloor))\,\,\,\rm et
\,\,\,\it K_{\rm 0\it} = \mathbf{k}(C).  
\end{eqnarray*}
Le r\'esultat suivant donne quelques propri\'et\'es des 
LFIHD homog\`enes sur l'alg\`ebre $M$-gradu\'ee $A$. 
\begin{lemme}
Soit $\partial$ une $\rm LFIHD$ homog\`ene non triviale sur $A$ de degr\'e $e$.
Alors les assertions suivantes sont vraies.
\begin{enumerate}
 \item[\rm (i)] Si $\partial$ est de type vertical alors $e\not\in\sigma^{\vee}$
et $\rm ker\,\it\partial = \bigoplus_{m\in\tau_{M}}A_{m}\chi^{m}$, 
pour une face $\tau$ de codimension $1$ du c\^one $\sigma^{\vee}$. En particulier, 
l'alg\`ebre $\rm ker \,\it\partial$ est de type fini sur $\mathbf{k}$.
\item[\rm (ii)] Si $A$ n'est pas elliptique alors $\partial$ est de type vertical si et seulement
si $e\not\in\sigma^{\vee}$. 
\end{enumerate}
\end{lemme}
\begin{proof}
$\rm (i)$ Par le lemme $5.3.12$, nous pouvons \'etendre $\partial$ en une LFIHD homog\`ene sur l'alg\`ebre
de mono\"ide $K_{0}[\sigma^{\vee}_{M}]$. 
Par le lemme $5.4.3$, nous avons $e\in\rm Rt\,\it\sigma$ 
et donc $e\not\in\sigma^{\vee}$.
Soit $\bar{\partial}$ l'extension de $\partial$ sur $K_{0}[\sigma^{\vee}_{M}]$.
\`A nouveau par le lemme $5.4.3$, nous obtenons $\rm ker\,\it\bar{\partial} = K_{\rm 0\it}[\tau_{M}]$, pour
une face $\tau$ de codimension $1$ de $\sigma^{\vee}$.
Ainsi,
\begin{eqnarray*}
\rm ker\,\it\partial = A\cap\rm ker\,\it\bar{\partial} = \it\bigoplus_{m\in\tau_{M}}A_{m}\chi^{m}. 
\end{eqnarray*}
Par les arguments de d\'emonstration de [AH, $4.1$], l'alg\`ebre $\rm ker\,\it\partial$ est de type fini sur $\mathbf{k}$.

$\rm (ii)$ Supposons que $A$ n'est pas elliptique et \'etendons $\partial$ en un syst\`eme it\'er\'e de d\'erivations
d'ordre sup\'erieur $\bar{\partial}$ sur la $K_{0}$-alg\`ebre $K_{0}[M]$. Si $e\not\in\sigma^{\vee}$ alors pour tout 
$i\in\mathbb{N}$, nous avons
$\partial^{(i)}(A_{0}) = A_{ie} = \{0\}$. Puisque $K_{0} = \rm Frac\,\it A_{\rm 0\it}$, on conclut
que $\partial$ est de type vertical. 
\end{proof}
Comme remarqu\'e dans [Li, $3.2$], dans le cas elliptique, l'alg\`ebre $M$-gradu\'ee $A$ admet en g\'en\'eral des 
LFIHD $\partial$ de type horizontal
satisfaisant $\rm deg\,\it \partial\not\in\sigma^{\vee}$. 
Dans la suite, nous introduisons des donn\'ees combinatoires attach\'ees \`a l'alg\`ebre 
$A= A[C,\mathfrak{D}]$ dans le but de d\'ecrire toutes les op\'erations normalis\'ees du groupe
additif de type vertical dans $X = \rm Spec\,\it A$.
\begin{notation}
Soit $e\in\rm Rt\,\it\sigma$ une racine de Demazure de rayon distingu\'e $\rho$. 
Alors nous consid\'erons le diviseur de Weil rationnel
\begin{eqnarray*}
\mathfrak{D}(e) = \sum_{z\in C}\min_{v\in V(\Delta_{z})}\langle e, v\rangle\cdot z. 
\end{eqnarray*}
On rappelle que $V(\Delta_{z})$ est l'ensemble des sommets de $\Delta_{z}$. 
Nous d\'esignons par $\Phi_{e}$ le module 
$H^{0}(C,\mathcal{O}_{C}(\lfloor \mathfrak{D}(e)\rfloor))$
sur l'anneau $A_{0}$. De plus,
si $\varphi\in\Phi_{e}$ est une section alors pour tout vecteur
$m\in\sigma^{\vee}-\rho^{\star}$, 
\begin{eqnarray*}
\rm div\it\, \varphi \rm\geq \it -\mathfrak{D}(e)\rm\geq\it \mathfrak{D}(m) - \mathfrak{D}(m+e). 
\end{eqnarray*}
Cette derni\`ere in\'egalit\'e nous sera utile pour le r\'esultat suivant.
\end{notation}
En partant des donn\'ees combinatoires introduites ci-dessus, nous construisons
une LFIHD homog\`ene de type vertical de la fa\c con suivante.
\begin{lemme}
 Consid\'erons $e\in Rt\,\it\sigma$ une racine de Demazure de rayon distingu\'e
$\rho$ et $\varphi\in \Phi_{e}$ une section. Notons
$\bar{\partial} = \varphi\,\partial_{e}$, o\`u $\partial_{e}$ est la
$\rm LFIHD$ correspondante
\`a la racine $e$ op\'erant par endomorphismes $K_{0}$-lin\'eaires sur 
la $K_{0}$-alg\`ebre $K_{0}[\sigma^{\vee}_{M}]$ , 
voir l'exemple $5.4.2$. Alors pour tout $i\in\mathbb{N}$,
\begin{eqnarray*}
\bar{\partial}^{(i)}(A)\subset A. 
\end{eqnarray*}
Par cons\'equent, la suite
\begin{eqnarray*}
\partial = \partial_{e,\varphi} = 
\{\bar{\partial}^{(i)}_{|A}\}_{i\in\mathbb{N}}
\end{eqnarray*} 
d\'efinit une $\rm LFIHD$ homog\`ene de type vertical
sur l'alg\`ebre $M$-gradu\'ee $A$.
\end{lemme}
\begin{proof}
Fixons un entier $i\in\mathbb{Z}_{>0}$ et soit $f\in A_{m}$ une section non nulle.
Si $\bar{\partial}^{(i)}(f\chi^{m})\neq 0$
alors $m\in \sigma^{\vee}_{M}-\rho^{\star}_{M}$ et on a
\begin{eqnarray*}
\rm div\,\it \bar{\partial}^{(i)}(f\chi^{m})/\chi^{m+ie} + \lfloor \mathfrak{D}(m+ie)\rfloor
\rm = \it i\rm div\,\it \varphi + \rm div\,\it f + \lfloor \mathfrak{D}(m+ie)\rfloor
\end{eqnarray*}
\begin{eqnarray*}
\geq i(\mathfrak{D}(m/i) - \mathfrak{D}(m/i + e)) - \lfloor \mathfrak{D}(m)\rfloor
+ \lfloor \mathfrak{D}(m+ie)\rfloor 
\end{eqnarray*}
\begin{eqnarray*}
\geq \{\mathfrak{D}(m)\} - \{\mathfrak{D}(m+ie)\}. 
\end{eqnarray*}
Puisque que les coefficients du diviseur de Weil rationnel $\{\mathfrak{D}(m)\} - \{\mathfrak{D}(m+ie)\}$
appartiennent \`a $]-1,1[$, nous avons 
\begin{eqnarray*}
\rm div\,\it \bar{\partial}^{(i)}(f\chi^{m})/\chi^{m+ie} + \lfloor \mathfrak{D}(m + ie)\rfloor \rm \geq 0. 
\end{eqnarray*}
Cela montre que $A$ est $\bar{\partial}$-stable. Le reste de la d\'emonstration est facile et laiss\'e au lecteur.
\end{proof}
\begin{theorem} 
Pour l'alg\`ebre $A = A[C,\mathfrak{D}]$,
les $\rm LFIHD$ homog\`enes de type vertical sur $A$ sont de la forme
$\partial_{e, \varphi} = 
\varphi\partial_{e}|_{A}$, o\`u $e\in\rm Rt\,\sigma$
est une racine, $\varphi\in\Phi_{e}$ est une section non nulle et $\rho$
est le rayon distingu\'e de $e$ (voir le lemme $5.5.3$ pour des pr\'ecisions
sur les notations). R\'eciproquement, une racine $e\in \rm Rt\,\it\sigma$ et une section $\varphi\in\Phi_{e}$
donnent une $\rm LFIHD$ homog\`ene de type vertical sur $A$.
\end{theorem}
\begin{proof}
Soit $\partial$ une LFIHD homog\`ene sur $A$. Par le lemme $5.3.12$, 
$\partial$ s'\'etend en un syst\`eme it\'er\'e de $K_{0}$-d\'erivations d'ordre sup\'erieur
localement fini $\bar{\partial}$ sur l'alg\`ebre de mono\"ide
$K_{0}[\sigma^{\vee}_{M}]$. Par le th\'eor\`eme $5.4.5$,
$\bar{\partial}$ est donn\'ee par une racine $e\in \rm Rt\,\it\sigma$ comme dans l'exemple $5.4.2$, 
pour $\varphi\in K_{0}^{\star}$. 

Montrons que $\varphi\in\Phi_{e}$. Soit $\rho$ le rayon distingu\'e de $e$. 
Pour $z\in C$, consid\'erons un vecteur $v_{z}\in V(\Delta_{z})$,
de sorte que  
\begin{eqnarray*}
\mathfrak{D}(e) = \sum_{z\in C}\langle v_{z}, e\rangle\cdot z.
\end{eqnarray*}
Pour tout $z\in C$, on pose
\begin{eqnarray*}
\omega_{z} = \{m\in\sigma^{\vee}\,|\, h_{z}(m) = \langle m,v_{z}\rangle\}.
\end{eqnarray*}
Le sous-ensemble $\omega_{z}$ est un c\^one de dimension $n = \rm rang\,\it M$ (voir [AH, Section $1$]). 
Soit $m_{z}\in \sigma^{\vee}_{M} - \rho^{\star}_{M}$ tel que $m_{z}, m_{z} + e\in\omega_{z}$, $\rm deg\,\it \mathfrak{D}(m_{z})\geq g$ et 
$\langle m_{z},\rho\rangle\not\in p\mathbb{Z}$, o\`u $p$ est la caract\'eristique du corps $\mathbf{k}$
et $g$ est le genre de la courbe $C$. Choisir un tel $m_{z}$ est possible puisque $\omega_{z}$ est d'int\'erieur
non vide, le diviseur poly\'edral $\mathfrak{D}$ est propre, et le vecteur $\rho$ est primitif. 
D'apr\`es le th\'eor\`eme de Riemann-Roch, nous avons $A_{m_{z}}\neq \{0\}$.
De plus, l'inclusion 
\begin{eqnarray*}
\partial^{(1)}(A_{m_{z}}\chi^{m_{z}})\subset A_{m_{z}+e}\chi^{m_{z}+e} 
\end{eqnarray*}
implique $\varphi A_{m_{z}}\subset A_{m_{z}+e}$. Par cons\'equent, pour tout $z\in C$,
\begin{eqnarray*}
\rm div\,\it\varphi  \rm \geq\it \mathfrak{D}(m_{z}) -\mathfrak{D}(m_{z} +e). 
\end{eqnarray*}
Le coefficient du diviseur $\mathfrak{D}(m_{z}) -\mathfrak{D}(m_{z} +e)$ au point
$z\in C$ est $-\langle v_{z}, e\rangle$. Ainsi, $\rm div\,\it \varphi\rm\geq \it -\mathfrak{D}(e)$ 
et $\varphi\in\Phi_{e}$, comme demand\'e. Le reste de la d\'emonstration est donn\'e par le lemme $5.5.3$.   
\end{proof}
Par analogie avec le cas torique trait\'e dans la section pr\'ec\'edente (voir le corollaire $5.4.7$), 
toute famille de LFIHD homog\`ene de type vertical sur $A$ ayant des noyaux deux \`a deux distincts
forme un ensemble fini. Le prochain r\'esultat donne un crit\`ere combinatoire pour que $A$
admette une LFIHD homog\`ene non triviale de type vertical. 
\begin{corollaire}
Posons $A = A[C,\mathfrak{D}]$ et fixons un rayon $\rho\subset\sigma$. Alors 
l'alg\`ebre $M$-gradu\'ee $A$ admet une $\rm LFIHD$ homog\`ene de type vertical avec
$e = \rm deg\,\it \partial$, pour une racine $e\in\rm Rt\,\it\sigma$ de rayon distingu\'e $\rho$,
si et seulement si, l'une des assertions suivantes est vraie.
\begin{enumerate}
\item[\rm (i)] $C$ est affine. 
\item[\rm (ii)] $C$ est projective et $\rho\cap \rm deg\,\it\mathfrak{D} = \emptyset$.
\end{enumerate}
\end{corollaire}
\begin{proof}
Si $C$ est une courbe affine alors tout diviseur de Weil sur $C$ admet une section globale non nulle et donc
pour tout $e\in\rm Rt\,\it\sigma$, nous avons 
$\rm dim\,\Phi_{\it e}\rm >0$. Dans ce cas, on conclut par le th\'eor\`eme $5.5.4$.

Supposons que $C$ est projective.  
Fixons une racine $e\in\rm Rt\,\it\sigma$ avec rayon distingu\'e $\rho$. 
Notons que pour tout $m\in\rho^{\star}_{M}$,
on a $e + m\in\rm  Rt\,\it\sigma$. De plus, 
\begin{eqnarray*}
\mathfrak{D}(e + m)\geq  \mathfrak{D}(m) + \mathfrak{D}(e)\,\,\,\rm
et\,\,\, donc\,\,\, 
deg\,\it \mathfrak{D}(m + e)\rm\geq deg\,\it\mathfrak{D}(m) + \rm deg\,\it \mathfrak{D}(e).
\end{eqnarray*}
D'o\`u par le th\'eor\`eme de Riemann-Roch et par la propret\'e de $\mathfrak{D}$, 
si $\rho\cap \rm deg\,\it\mathfrak{D} = \emptyset$ alors il existe $m\in\rho^{\star}_{M}$
tel que $\rm dim\,\Phi_{\it e + m}\rm >0$.

R\'eciproquement, supposons que $\rho\cap \rm deg\,\it\mathfrak{D}\neq \emptyset$. 
Puisque $\langle e, \rho\rangle = - 1$, il existe un vecteur 
$v\in\rm deg\,\it\mathfrak{D}$ tel que $\langle e, v\rangle < 0$ et par cons\'equent, 
$\rm deg\,\it \mathfrak{D}(e) \rm < 0$. Sous ces derni\`eres conditions, il vient
$\rm dim\,\Phi_{\it e}\rm = 0$. \`A nouveau on conclut par le th\'eor\`eme $5.5.4$
pour le cas o\`u la courbe $C$ est projective.
\end{proof}

\section{Op\'erations du groupe additif de type horizontal}
Soit $\mathbf{k}$ un corps. 
Le but de cette section est de classifier toutes les
 $\mathbb{T}$-vari\'et\'es affines horizontales de complexit\'e $1$
sur un corps parfait en terme de diviseurs poly\'edraux
(voir la terminologie de $5.3.11$). Le lecteur peut consulter 
[Li, Section $3.2$] pour le cas classique o\`u le corps de base
est alg\'ebriquement clos de caract\'eristique z\'ero.
Donnons quelques notations que nous utiliserons par la suite.
Soit $C$ une courbe r\'eguli\`ere sur $\mathbf{k}$. 
Fixons un c\^one poly\'edral saillant $\sigma\subset N_{\mathbb{Q}}$ 
et posons $A = A[C,\mathfrak{D}]$, o\`u $\mathfrak{D}$ est un
diviseur $\sigma$-poly\'edral propre sur $C$. Nous consid\'erons \'egalement
\begin{eqnarray*}
A_{m} = H^{0}(C,\mathcal{O}_{C}(\lfloor\mathfrak{D}(m)\rfloor)),\,\,\,\rm de\,\,sorte\,\, que\,\,\,
\it A = \bigoplus_{m\in\sigma^{\vee}_{M}}A_{m}\chi^{m}.
\end{eqnarray*}
Pour simplifier, nous notons $h_{z}$ la fonction de support du coefficient $\Delta_{z}$ de $\mathfrak{D}$
au point $z\in C$. Rappelons qu'un \em quasi-\'eventail \rm de $\sigma^{\vee}$ est un ensemble $\Lambda$ de c\^ones poly\'edraux
de $M_{\mathbb{Q}}$ satisfaisant les propri\'et\'es suivantes.
\begin{enumerate}
\item[\rm (i)] La r\'eunion $\bigcup_{\delta\in \Lambda}\delta$ est \'egale \`a $\sigma^{\vee}$.
\item[\rm (ii)] Si $\delta\in\Lambda$ et si $\delta'$ est une face de $\delta$ alors $\delta'\in\Lambda$.
\item[\rm (iii)] Si $\delta,\delta'\in \Lambda$ alors $\delta\cap \delta'$ est une face commune de $\delta$
et de $\delta'$. 
\end{enumerate}
Un quasi-\'eventail ne comportant que des c\^ones poly\'edraux saillants est appel\'e un \em \'eventail. \rm
On d\'esigne par $\Lambda (\mathfrak{D})$ le quasi-\'eventail formant la subdivision la moins fine de $\sigma^{\vee}$ o\`u $\mathfrak{D}$
est lin\'eaire sur chacun de ses c\^ones (voir [AH, Section 1], [Li, 1.3]). Pour un ouvert non vide $C_{0}\subset C$, nous notons
\begin{eqnarray*}
\mathfrak{D}|_{C_{0}} = \sum_{z\in C_{0}}\Delta_{z}\cdot z, 
\end{eqnarray*}
la \em restriction \rm de $\mathfrak{D}$ \`a $C_{0}$.

{\em Dans la suite, de nombreux r\'esultats de cette section demandent
l'hypoth\`ese de perfection sur le corps de base $\mathbf{k}$. Lorsque que cette hypoth\`ese
n'est pas mentionn\'ee cela veux dire que l'on suppose $\mathbf{k}$ arbitraire.}

D'apr\`es un r\'esultat de Rosenlicht [Ro], dans le cas o\`u $\mathbf{k}$ est alg\'ebriquement
clos, le lemme suivant montre en particulier que les $\mathbb{T}$-vari\'et\'es affines horizontales
de complexit\'e $1$ ont une orbite ouverte sous l'op\'eration naturelle du groupe alg\'ebrique
$\mathbb{T}\ltimes\mathbb{G}_{a}$. 
\begin{lemme}
Supposons que $A$ admet une $\rm LFIHD$ homog\`ene $\partial$
de type horizontal. 
Alors pour l'op\'eration correspondante de $\mathbb{G}_{a}$
dans la vari\'et\'e $X = \rm Spec\,\it A$, nous avons l'\'egalit\'e
\begin{eqnarray*}
\mathbf{k}(X)^{\mathbb{G}_{a}}
\cap\mathbf{k}(X)^{\mathbb{T}} = \mathbf{k}.  
\end{eqnarray*}
\end{lemme}
\begin{proof}
Posons $\Omega = \mathbf{k}(X)^{\mathbb{G}_{a}}
\cap\mathbf{k}(X)^{\mathbb{T}}$. 
Supposons que l'extension $\mathbf{k}(X)^{\mathbb{T}}/\Omega$
est alg\'ebrique et consid\'erons une fonction rationnelle invariante non nulle 
$F\in \mathbf{k}(X)^{\mathbb{T}}$. En remarquant que  $\mathbf{k}(X)^{\mathbb{G}_{a}}$ 
est le corps des fractions de l'anneau $\rm ker\,\it\partial$, nous pouvons trouver 
$a\in \rm ker\,\it \partial$ tel que $aF$ est entier sur $\rm ker\,\it\partial$.
Puisque l'alg\`ebre $A$ est normale, $aF\in A$. Par la proposition $5.3.4\rm (b)$, nous avons 
$aF\in\rm ker\,\it\partial$.
D'o\`u $\Omega =\mathbf{k}(X)^{\mathbb{T}}$, contredisant le fait que $\partial$
est de type horizontal. Notons
que $\mathbf{k}(X)^{\mathbb{T}}$ a un degr\'e de transcendance \'egal \`a $1$ sur $\mathbf{k}$.
Ainsi, en utilisant le fait que $\mathbf{k}$ est alg\'ebriquement clos dans $\Omega$, 
nous obtenons $\Omega = \mathbf{k}$. 
\end{proof}
Ensuite, nous montrons que l'existence d'une LFIHD sur l'alg\`ebre $A = A[C,\mathfrak{D}]$
impose de fortes contraintes sur la courbe $C$. Nous renvoyons le lecteur \`a [FZ $2$, $3.5$], [Li, $3.16$] 
pour le cas classique. Notons que dans le cas o\`u le corps de base est de caract\'eristique
z\'ero, l'assertion $\rm (i)$ ci-apr\`es peut \^etre obtenue \`a partir des r\'esultats
de [Ku].  
\begin{lemme}
Supposons que $A$ admet une $\rm LFIHD$ $\partial$ homog\`ene de type horizontal. Consid\'erons $\omega$ $\rm ($resp. $L \rm )$ 
le c\^one $\rm ($resp. sous-r\'eseau $\rm )$ engendr\'e par l'ensemble des poids de $\rm ker\,\it \partial$ 
et posons $\omega_{L} = \omega\cap L$.
Alors les assertions suivantes sont vraies.
\begin{enumerate}
 \item[\rm (i)] Le noyau de $\partial$ est une alg\`ebre de mono\"ide, i.e., 
\begin{eqnarray*}
 \rm ker\,\it \partial = \bigoplus_{m\in\omega_{L}}\mathbf{k}\varphi_{m}\chi^{m},
\end{eqnarray*}
o\`u $\varphi_{m}\in \mathbf{k}(C)^{\star}$.
\item[\rm (ii)] Nous avons $C\simeq \mathbb{P}^{1}_{\mathbf{k}}$, dans le cas o\`u $A$
est elliptique.
\item[\rm (iii)] Si $\mathbf{k}$ est un corps parfait alors $C\simeq \mathbb{A}^{1}_{\mathbf{k}}$ dans 
le cas o\`u $A$ n'est pas elliptique.
\end{enumerate}
\end{lemme}
\begin{proof}
$\rm (i)$ Soient $a,a'\in\rm ker\,\it\partial$ des \'el\'ements homog\`enes
de m\^eme degr\'e. Par le lemme $5.6.1$, on a 
\begin{eqnarray*}
a/a'\in \mathbf{k}(X)^{\mathbb{G}_{a}}\cap\mathbf{k}(X)^{\mathbb{T}} = \mathbf{k}^{\star}.
\end{eqnarray*}
Ainsi $\rm ker\it\,\partial$
est une alg\`ebre de mono\"ide avec ensemble des poids $\omega_{L}$ (voir $5.3.4\rm (b)$). 

$\rm (ii)$ Posons $K = \rm Frac\,\it A$ et consid\'erons le sous-corps 
$E = K^{\mathbb{G}_{a}}$.
Par [CM, Lemma $2.2$], il existe une variable $x$ sur le corps 
$E$ telle que $E(x) = K$. Par l'assertion $\rm (i)$ dans $5.6.2$, l'extension $E/\mathbf{k}$ est transcendante pure 
et donc $K/\mathbf{k}$ l'est \'egalement. Puisque $\mathbf{k}(C)\subset K$, 
la courbe projective r\'eguli\`ere $C$ est unirationnelle. D'apr\`es le th\'eor\`eme de Lur\"oth, 
il s'ensuit que 
$C\simeq \mathbb{P}^{1}_{\mathbf{k}}$. 

$\rm (iii)$ Supposons que $A$ n'est pas elliptique. 
Notons que par notre convention, $\mathbf{k}$ est alg\'ebriquement
clos dans $\rm Frac\,\it A$. 
Soit $\bar{\mathbf{k}}$ une cl\^oture alg\'ebrique de $\mathbf{k}$, 
de sorte que l'extension de corps $\bar{\mathbf{k}}/\mathbf{k}$ est s\'eparable.
Soit $B$ la normalisation de l'alg\`ebre int\`egre $A\otimes_{\mathbf{k}}\bar{\mathbf{k}}$. Alors 
$B$ est une alg\`ebre $M$-gradu\'ee de type fini sur $\bar{\mathbf{k}}$.
L'hypoth\`ese de perfection sur $\mathbf{k}$ implique que $C$ est lisse. Donc la pi\`ece gradu\'ee $B_{0}$ de $B$ est 
\'egale \`a $\mathbf{k}[C]\otimes\bar{\mathbf{k}}$ (voir [Ha, III.10.1.$(b)$]). 
De plus, $\partial$ s'\'etend en une LFIHD homog\`ene de type horizontal
sur l'alg\`ebre $B$. En adaptant dans notre contexte les arguments de d\'emonstration de [Li, $3.16$], nous avons $B \simeq 
\bar{\mathbf{k}}[t]$, pour une variable $t$ sur $\bar{\mathbf{k}}$. Par s\'eparabilit\'e de $\bar{\mathbf{k}}/\mathbf{k}$, 
cela donne $A_{0}\simeq \mathbf{k}[t]$ (voir par exemple [Ru, As]).
\end{proof}
\begin{rappel}
D'apr\`es les r\'esultats pr\'ec\'edents, nous consid\'erons seulement les cas o\`u $C = \mathbb{A}^{1}_{\mathbf{k}}$,
lorsque $A$ n'est pas elliptique, et
$C = \mathbb{P}^{1}_{\mathbf{k}}$, lorsque $A$ est elliptique. Supposons que $A$ a une LFIHD $\partial$ homog\`ene 
de type horizontal et consid\'erons
\begin{eqnarray*}
\rm ker\,\it \partial = 
\bigoplus_{m\in\omega_{L}}\varphi_{m}\chi^{m}
\end{eqnarray*}
le noyau de $\partial$ comme dans $5.6.2$. Pour simplifier,
nous pouvons supposer que $\mathbf{k}(C) = \mathbf{k}(t)$ pour un param\`etre local $t$;
lorsque $C$ est affine, nous prenons $t$ tel que $\mathbf{k}[C] = \mathbf{k}[t]$.
\end{rappel}
\begin{lemme}
Gardons les m\^emes notations que dans $5.6.3$. Les assertions suivantes sont
vraies.
\begin{enumerate}
 \item[\rm (i)]
 Si $C = \mathbb{A}^{1}_{\mathbf{k}}$ alors pour tout $m\in\omega_{L}$, on a 
$\rm div\,\it \varphi_{m} + \mathfrak{D}(m) \rm = 0$.
\item[\rm (ii)]
Supposons que $C = \mathbb{P}^{1}_{\mathbf{k}}$. Alors il existe  
$z_{\infty}\in C$ tel que pour tout $m\in\omega_{L}$, le diviseur 
rationnel effectif
\begin{eqnarray*}
\rm div\,\it \varphi_{m} + \mathfrak{D}(m) \rm 
\end{eqnarray*}
a au plus $z_{\infty}$ dans son support\footnote{En particulier, si $m\in\omega_{L}$ est dans 
l'int\'erieur relatif de $\sigma^{\vee}$ et si $\mathfrak{D}(m)$ est entier alors par propret\'e
de $\mathfrak{D}$, le singleton $\{z_{\infty}\}$ est le support de  $\rm div\,\it \varphi_{m} + \mathfrak{D}(m)$.}. 
\item[\rm (iii)]
Le c\^one $\omega$ est un c\^one maximal du quasi-\'eventail $\Lambda\left(\mathfrak{D}\right)$  
dans le cas non elliptique, et de
$\Lambda(\mathfrak{D}_{|\mathbb{P}^{1}_{\mathbf{k}}-\{z_{\infty}\}})$
dans le cas elliptique.
\item[\rm (iv)] 
Le rang du r\'eseau $L$ est \'egal \`a celui de $M$. Le r\'eseau $M$ est engendr\'e par $e:=\rm deg\,\it \partial$
et $L$. De plus, si $d$ est le plus petit entier strictement positif tel que $de\in L$ alors nous pouvons
\'ecrire de fa\c con unique tout vecteur $m\in M$ comme $m = l+re$, pour $l\in L$ et $r\in\mathbb{Z}$ tel que 
$0\leq r< d$. 
\item[\rm (v)]
Si $\mathbf{k}$ est un corps parfait alors le point $z_{\infty}$ de l'assertion $5.6.4\,\rm (ii)$ est rationnel; 
en d'autres termes, le corps r\'esiduel du point $z_{\infty}$
est $\mathbf{k}$.
\end{enumerate}
\end{lemme}
\begin{proof}
$\rm (i)$
\'Etant donn\'e un vecteur $m\in\sigma^{\vee}_{M}$, nous posons
\begin{eqnarray*}
A_{m} = f_{m}\cdot \mathbf{k}[t]\,,  
\end{eqnarray*}
o\`u $f_{m}\in\mathbf{k}(t)^{\star}$. Supposons que $m\in\omega_{L}$. Alors nous 
avons $\varphi_{m} = Ff_{m}$, pour un \'el\'ement non nul $F\in\mathbf{k}[t]$.
Par la proposition $5.3.4 \rm (a)$, le polyn\^ome $F$ est constant et
\begin{eqnarray*}
\rm div\,\it \varphi_{m} + \lfloor \mathfrak{D}(m)\rfloor \rm = 0.
\end{eqnarray*}
Par cons\'equent, pour tout $r\in\mathbb{N}$,
\begin{eqnarray*}
r\lfloor \mathfrak{D}(m)\rfloor = -r\rm div\,\it \varphi_{m} = -\rm div\,\it \varphi_{rm}
= \lfloor \mathfrak{D}(rm)\rfloor.
\end{eqnarray*}
Cela montre que $\mathfrak{D}(m)$ est entier sous la condition $m\in\omega_{L}$. 
L'assertion $\rm (i)$ suit ais\'ement.

$\rm (ii)$ Supposons qu'il existe $m\in\omega_{L}$
tel que 
\begin{eqnarray*}
\rm div\,\it \varphi_{m} + \mathfrak{D}(m) \rm\geq  \it [z_{\infty}] + [\,z_{\rm 0}\,],
\end{eqnarray*}
o\`u $z_{0}$, $z_{\infty}$ sont des points distincts de $C$. Pour le param\`etre local $t$, 
d\'esignons par $\infty$ le point \`a l'infini de
$C = \mathbb{P}^{1}_{\mathbf{k}}$. Consid\'erons $p_{0}(t),p_{\infty}(t)\in\mathbf{k}(t)$
deux fonctions rationnelles v\'erifiant les conditions suivantes. Si le point $z_{0}$ (resp. $z_{\infty}$)
appartient \`a $\mathbb{A}^{1}_{\mathbf{k}} = \rm Spec\,\it \mathbf{k}[t]$
alors $p_{0}(t)$ (resp. $p_{\infty}(t)$) est le polyn\^ome unitaire g\'en\'erateur de l'id\'eal
associ\'e \`a $z_{0}$ (resp. $z_{\infty}$) dans $\mathbf{k}[t]$. Sinon $z_{0} = \infty$ (resp. $z_{\infty} = \infty$)
et nous posons $p_{0}(t) = 1/t$ (resp. $p_{\infty}(t) = 1/t$). 

Consid\'erons la fonction rationnelle
$f := p_{0}(t)/p_{\infty}(t)$. Alors $f$ et $f^{-1}$ ne sont pas constantes et 
$f\varphi_{m}, f^{-1}\varphi_{m}$ appartiennent \`a $A_{m}$. Par
la proposition $5.3.4\rm (a)$, on a  
\begin{eqnarray*}
f\varphi_{m}\chi^{m}\cdot f^{-1}\varphi_{m}\chi^{m} = \varphi_{2m}\chi^{2m}\in\rm ker\,\it\partial\Rightarrow
f\varphi_{m}\chi^{m},f^{-\rm 1\it}\varphi_{m}\chi^{m}\in\rm ker\,\partial, 
\end{eqnarray*}
donnant une contradiction avec $5.6.2\,\rm (i)$. On conclut
que $\rm div\,\it \varphi_{m} + \mathfrak{D}(m) \rm$ est support\'e par 
au plus un point.

$\rm (iii)$
Par les assertions $\rm (i)$ et $\rm (ii)$ du lemme $5.6.4$, l'application $m\mapsto\mathfrak{D}(m)$
dans le cas non elliptique, et l'application $m\mapsto\mathfrak{D}_{|\mathbb{P}^{1}_{\mathbf{k}}-\{z_{\infty}\}}(m)$
dans le cas elliptique, sont lin\'eaires sur le c\^one $\omega$. Cela implique qu'il existe un c\^one
maximal $\omega_{0}$ appartenant \`a $\Lambda(\mathfrak{D})$ dans le cas non elliptique, et appartenant 
\`a $\Lambda(\mathfrak{D}_{|\mathbb{P}^{1}_{\mathbf{k}}-\{z_{\infty}\}})$
dans le cas elliptique, tel que $\omega\subset\omega_{0}$.

Montrons l'inclusion r\'eciproque. Soit $m\in\omega_{0}$.
En rempla\c cant $m$ par un multiple entier, nous pouvons supposer que $m\in L$
et que le diviseur de Weil $\mathfrak{D}(m)$ est entier.
D'apr\`es les assertions $5.6.2\,\rm (i)$ et $5.3.4\, \rm (c)$, le
c\^one $\omega$ est d'int\'erieur non vide dans $M_{\mathbb{Q}}$. D'o\`u
$m+m'\in\omega_{L}$, pour un $m'\in\omega_{L}$. Consid\'erons
une section non nulle $f_{m}\in A_{m}$ telle que
\begin{eqnarray*}
\rm div\,\it f_{m} + \mathfrak{D}(m)\rm = 0 
\end{eqnarray*}
dans le cas non elliptique, et telle que 
\begin{eqnarray*}
\rm div_{|\it\mathbb{P}^{\rm 1\it}_{\mathbf{k}}-\{z_{\infty}\}}\,\it f_{m} + 
\mathfrak{D}_{|\mathbb{P}^{\rm 1\it}_{\mathbf{k}}-\{z_{\infty}\}}(m)\rm = 0 
\end{eqnarray*}
dans le cas elliptique. La lin\'earit\'e de $\mathfrak{D}$ ou de 
$\mathfrak{D}_{|\mathbb{P}^{1}_{\mathbf{k}} - \{z_{\infty}\}}$ sur le c\^one $\omega_{0}$,
selon les cas elliptique et non elliptique, donne
\begin{eqnarray*}
f_{m}\chi^{m}\cdot \varphi_{m'}\chi^{m'} = \lambda\varphi_{m+m'}\chi^{m+m'}, 
\end{eqnarray*}
pour un $\lambda\in\mathbf{k}^{\star}$. Par cons\'equent, $f_{m}\chi^{m}\in\rm ker\,\it\partial$
et \`a nouveau par $5.3.4\,\rm (a)$, nous avons $m\in \omega$. Cela montre l'\'egalit\'e $\omega_{0} = \omega$. 

$\rm (iv)$ Puisque la partie $\sigma^{\vee}_{M}$ engendre le r\'eseau
$M$ et que $\partial$ est une LFIHD homog\`ene sur $A$, pour tout $m\in M$,
nous avons $m + se\in L$, pour $s\in \mathbb{Z}$. En rempla\c cant
$r := -s$ par le reste de la division euclidienne de $r$
par $d$, si n\'ecessaire, nous obtenons $m = l + re$, o\`u $l\in L$
et $0\leq r < d$. La minimalit\'e de $d$ implique que
cette derni\`ere d\'ecomposition est unique.
  
$\rm (v)$ Supposons que le corps $\mathbf{k}$ est parfait
et fixons $\bar{\mathbf{k}}$ une cl\^oture alg\'ebrique de $\mathbf{k}$.
Consid\'erons l'alg\`ebre $B = A\otimes_{\mathbf{k}}\bar{\mathbf{k}}$. Si nous
posons 
$\mathfrak{D} = \sum_{z\in C}\Delta_{z}\cdot z$ alors par la proposition $4.5.10$, 
le diviseur poly\'edral 
\begin{eqnarray*}
\mathfrak{D}_{\bar{\mathbf{k}}} =  \sum_{z\in C}\Delta_{z}\cdot S^{\star}\,z
\end{eqnarray*}
sur $\mathbb{P}^{1}_{\bar{\mathbf{k}}}$ v\'erifie 
\begin{eqnarray*}
B = \bigoplus_{m\in\sigma^{\vee}_{M}}B_{m}\chi^{m},\,\,\,\rm avec\,\,\,\it 
B_{m} =  H^{\rm 0\it }(\mathbb{P}^{\rm 1\it}_{\bar{\mathbf{k}}},
\mathcal{O}_{\mathbb{P}^{\rm 1\it}_{\bar{\mathbf{k}}}}(\lfloor 
\mathfrak{D}_{\bar{\mathbf{k}}}(m)\rfloor)).
\end{eqnarray*}
Nous pouvons \'etendre $\partial$ en une LFIHD homog\`ene
 $\partial_{\bar{\mathbf{k}}}$ de type 
horizontal sur $B$. Pour tout $m\in\omega_{L}$, 
nous avons $\varphi_{m}\chi^{m}\in\rm ker\,\it \partial_{\bar{\mathbf{k}}}$
et il existe $\lambda_{m}\in\mathbb{Q}_{\geq 0}$ tel que
\begin{eqnarray*}
\rm div\,\it \varphi_{m} + \mathfrak{D}(m)
 = \lambda_{m}\cdot z_{\infty}.
\end{eqnarray*}
En appliquant $S^{\star}$ \`a l'\'egalit\'e d'avant, on a
\begin{eqnarray*} 
\rm div_{\bar{\mathbf{k}}}\,\it \varphi_{m} + \mathfrak{D}_{\bar{\mathbf{k}}}(m)
 = \lambda_{m}\cdot S^{\star}\,z_{\infty}.
\end{eqnarray*}
Supposons que $z_{\infty}$ n'est pas un point rationnel et que $\lambda_{m}>0$ pour un
vecteur $m\in\omega_{L}$. En rempla\c cant $m$ par un multiple 
entier, nous pouvons supposer que 
$\lambda_{m}>1$. Utilisons les m\^emes notations que dans 
la d\'emonstration de $5.6.4\,\rm (ii)$.
Puisque l'extension de corps $\bar{\mathbf{k}}/\mathbf{k}$ est s\'eparable
et que
\begin{eqnarray*}
\rm deg\,\it p(t) = \rm deg(\it z_{\infty})\rm > 1,  
\end{eqnarray*}
le polyn\^ome $p_{z_{\infty}}(t)$ a au moins deux racines distinctes, disons 
$z_{1},z_{2}\in\bar{\mathbf{k}}$. Notons que les points $z_{1}, z_{2}$ appartiennent
au support de $S^{\star}\, z_{\infty}$.
En consid\'erant la fonction rationnelle
\begin{eqnarray*}
f = (t-z_{1})/(t-z_{2}),
\end{eqnarray*}
nous avons la contradiction
\begin{eqnarray*} 
f\varphi_{m}\chi^{m}\cdot f^{-1}\varphi_{m}\chi^{m} = \varphi_{2m}\chi^{2m}\in 
\rm ker\,\it \partial_{\bar{\mathbf{k}}}\Rightarrow f\varphi_{m}\chi^{m}, f^{-\rm 1\it}\varphi_{m}\chi^{m}\in 
\rm ker\,\it \partial_{\bar{\mathbf{k}}}. 
\end{eqnarray*}
Cela montre l'assertion $\rm (v)$.
\end{proof}
Dans la suite, nous consid\'erons les m\^emes notations que dans $5.6.3$. 
Sans perte de g\'en\'eralit\'e, nous pouvons supposer que $z_{\infty}$ 
est le point rationnel $\infty$ selon le param\`etre local $t$.  
\begin{lemme}
Les assertions suivantes sont vraies.
\begin{enumerate}
\item[\rm (i)] Si $C = \mathbb{P}^{1}_{\mathbf{k}}$ alors la normalisation de la sous-alg\`ebre 
$A[t]\subset\mathbf{k}(t)[M]$ est
\begin{eqnarray*}
A'= A[\mathbb{A}^{1}_{\mathbf{k}},\mathfrak{D}_{|\mathbb{A}^{1}_{\mathbf{k}}}],
\end{eqnarray*}
o\`u $\mathbb{A}^{1}_{\mathbf{k}} = \rm Spec\,\it\mathbf{k}[t]$.
\end{enumerate}
\begin{enumerate}
\item[\rm (ii)] Si le degr\'e de $\partial$ 
appartient \`a $\omega$ et si la fonction \'evaluation du diviseur poly\'edral
$\mathfrak{D}_{|\mathbb{A}^{1}_{\mathbf{k}}}$ est lin\'eaire 
alors $\partial$ s'\'etend en une $\rm LFIHD$ homog\`ene $\partial'$ sur $A'$ de type horizontal.
De plus, $\rm ker\,\it\partial = \rm ker\,\it \partial'$.
\item[\rm (iii)] Soit $d$ le plus petit entier strictement positif tel que pour tout $m\in\omega_{M}$
le diviseur $\mathfrak{D}(d\cdot m)$ est entier. Alors nous avons $d\cdot M\subset L$.
\end{enumerate}
\end{lemme}
\begin{proof}
$\rm (i)$ Cela suit de l'assertion $\rm (iii)$ du th\'eor\`eme $4.4.4$ (voir aussi [La $2$, $2.5$])
en prenant un ensemble de g\'en\'erateurs homog\`enes de $A[t]$.

$\rm (ii)$
En posant 
\begin{eqnarray*}
A' = \bigoplus_{m\in\sigma^{\vee}_{M}}A'_{m}\chi^{m},
\,\,\, \rm avec\,\,\,\it
A'_{m} = H^{\rm 0\it}(\mathbb{A}^{\rm 1\it}_{\mathbf{k}},
\mathcal{O}_{\mathbb{A}^{\rm 1\it}_{\mathbf{k}}}
(\lfloor\mathfrak{D}_{|\mathbb{A}^{\rm 1\it}_{\mathbf{k}}}(m)\rfloor)),
\end{eqnarray*}
pour tout $m\in\sigma^{\vee}_{M}$, nous pouvons \'ecrire $A'_{m} = \varphi_{m}\cdot\mathbf{k}[t]$,
o\`u $\varphi_{m}$ est une fonction rationnelle non nulle satisfaisant 
\begin{eqnarray*}
\rm div\,\it\varphi_{m} +\lfloor\mathfrak{D}_{|\mathbb{A}^{\rm 1\it}_{\mathbf{k}}}(m)\rfloor
\rm =0.
\end{eqnarray*}
Si $m\in\omega_{L}$ alors, \`a la multiplication d'un scalaire non nul pr\`es, $\varphi_{m}$ est le m\^eme que dans l'\'enonc\'e
 $5.6.4\,\rm (ii)$.

Par le lemme $5.3.5$, nous pouvons \'etendre $\partial$ en un syst\`eme it\'er\'e de d\'erivations d'ordre
sup\'erieur $\partial'$ sur l'alg\`ebre de mono\"ide $\mathbf{k}(t)[M]$. 
D\'esignons par $\partial'^{(i)}$ le $i$-\`eme terme de $\partial'$.
Consid\'erons $f\in A'_{m}$ pour $m\in\sigma^{\vee}_{M}$ et fixons un entier $i\in\mathbb{Z}_{>0}$.
Montrons que 
\begin{eqnarray*}
\partial'^{(i)}(f\chi^{m})\in A'.
\end{eqnarray*}
En utilisant l'assertion $\rm (ii)$ du lemme $5.6.4$ et la propret\'e de 
$\mathfrak{D}$, nous pouvons trouver un vecteur $m'\in\omega_{L}$
v\'erifiant les conditions suivantes. 
Les vecteurs $m,m'$ appartiennent \`a un m\^eme c\^one maximal de $\Lambda(\mathfrak{D})$ et
le coefficient de $\rm div\,\it\varphi_{m'} + \mathfrak{D}(m')$ au point $\infty$
est entier, strictement positif, et strictement plus grand que $-\rm div\,\it f - \lfloor\mathfrak{D}(m)\rfloor$. 
Par cons\'equent,
\begin{eqnarray*}
\rm div\,\it f\varphi_{m'} + \lfloor\mathfrak{D}(m'+m)\rfloor\rm  = \it
\rm div\,\it f + \lfloor\mathfrak{D}(m)\rfloor +  
\rm div\,\it\varphi_{m'} + \mathfrak{D}(m')\rm \geq 0.
\end{eqnarray*}
En particulier, $\varphi_{m'}f$ appartient \`a $A_{m + m'}$. D'o\`u il s'ensuit que
\begin{eqnarray*}
\varphi_{m'}\chi^{m'}\partial'^{(i)}(f\chi^{m}) = \partial^{(i)}(\varphi_{m'}f\chi^{m'+m})
\in A_{m' + m + ie}\chi^{m' + m + ie}. 
\end{eqnarray*}
Par hypoth\`ese, nous avons $e\in\omega = \sigma^{\vee}$, de sorte que $m + ie\in\sigma^{\vee}_{M}$.
Puisque $\mathfrak{D}_{|\mathbb{A}^{1}_{\mathbf{k}}}$ est lin\'eaire et que $\mathfrak{D}(m')$
est entier, nous obtenons les \'egalit\'es de diviseurs de Weil rationnels sur $\mathbb{A}^{1}_{\mathbf{k}}$:
\begin{eqnarray*}
-\rm div\,\it \varphi_{m'+m+ie} = \lfloor \mathfrak{D}_{|\mathbb{A}^{\rm 1\it}_{\mathbf{k}}}
(m' + m + ie)\rfloor
 = \lfloor\mathfrak{D}_{|\mathbb{A}^{\rm 1\it}_{\mathbf{k}}}(m')\rfloor + \lfloor 
\mathfrak{D}_{|\mathbb{A}^{\rm 1\it}_{\mathbf{k}}}(m+ie)\rfloor.  
\end{eqnarray*}
Donc  
\begin{eqnarray*}
\varphi_{m'+m+ie} = \lambda \varphi_{m'}\cdot\varphi_{m+ie},
\end{eqnarray*}
pour $\lambda\in\mathbf{k}^{\star}$. Par cons\'equent, cela implique
\begin{eqnarray*}
\varphi_{m'}\chi^{m'}\partial'^{(i)}(f\chi^{m})\in A_{m'+m+ie}\chi^{m'+m+ie}
\subset \varphi_{m'}\cdot\varphi_{m+ie}\cdot\mathbf{k}[t]\,\chi^{m'+m+ie}. 
\end{eqnarray*}
En simplifiant par $\varphi_{m'}\chi^{m'}$, cela donne 
\begin{eqnarray*}
\partial'^{(i)}(f\chi^{m})\in\varphi_{m+ie}\cdot\mathbf{k}[t]\,\chi^{m+ie} = 
A'_{m+ie}\chi^{m+ie}\subset A'. 
\end{eqnarray*}
On conclut que la sous-alg\`ebre $A'$ est $\partial'$-stable.

Ensuite, nous montrons que $\partial'$ est une LFIHD homog\`ene sur $A'$.
Par hypoth\`ese sur le vecteur $m'$, nous avons $t\varphi_{m'}\chi^{m'}\in A$.
Ainsi, il existe $s\in\mathbb{Z}_{>0}$ tel que pour tout
$i\geq s$,
\begin{eqnarray*}
\varphi_{m'}\chi^{m'}\partial'^{(i)}(t) = \partial^{(i)}(t\varphi_{m'}\chi^{m'}) = 0. 
\end{eqnarray*}
D'o\`u $\partial'$ op\`ere de fa\c con localement finie sur la variable $t$ et
donc \'egalement sur $A[t]$. Consid\'erons $f\in A'_{m}$ et choisissons $s'\in\mathbb{Z}_{>0}$
tel que le faisceau 
\begin{eqnarray*}
\mathcal{O}_{\mathbb{P}^{1}_{\mathbf{k}}}(\lfloor \mathfrak{D}(m+s'm')\rfloor) 
\end{eqnarray*}
est globalement engendr\'e. Ainsi,
\begin{eqnarray*}
\varphi_{s'm'}f\chi^{m+s'm'}\in A'_{m+s'm'} = \mathbf{k}[t]\otimes_{\mathbf{k}}A_{m+s'm'}\subset A[t]. 
\end{eqnarray*}
Puisque $\varphi_{s'm'}\chi^{s'm'}$ est dans le noyau de $\partial$, on conclut que $\partial'$ 
agit de fa\c con localement finie sur $f\chi^{m}$.
Cela montre que $\partial'$ est une LFIHD. Le fait que  $\partial'$ est de type horizontal
est facile et la d\'emonstration est laiss\'ee au lecteur.

Il reste \`a montrer que $\rm ker\,\it \partial = \rm ker\,\it \partial'$.
Par l'assertion $5.6.2\,\rm (i)$, le noyau 
$\rm ker\,\it\partial$ est une alg\`ebre associ\'ee \`a un mono\"ide $\omega_{L'}$, 
o\`u $L'$ est un sous-r\'eseau de rang maximal.
Puisque $\rm ker\,\it\partial\subset \rm ker\,\it \partial'$, nous avons $L\subset L'$
et $L'/L$ est un groupe ab\'elien fini. Posons 
\begin{eqnarray*}
\rm ker\,\it\partial =\bigoplus_{m\in\omega_{L}}\mathbf{k}\varphi_{m}\chi^{m}
\,\,\,\rm et\,\,\,\it \rm ker\,\it\partial' =\bigoplus_{m\in\omega_{L'}}
\mathbf{k}\varphi'_{m}\chi^{m}
\end{eqnarray*}
comme dans le lemme $5.6.2$. En consid\'erant 
$m\in L'$, on a $rm\in L$ pour $r\in\mathbb{Z}_{>0}$ et en 
utilisant les assertions (i),(ii) du lemme $5.6.4$, 
nous pouvons \'ecrire 
\begin{eqnarray*}
\lambda\varphi_{rm} = \varphi'_{rm} = (\varphi'_{m})^{r}, 
\end{eqnarray*}
o\`u $\lambda\in\mathbf{k}^{\star}$. Donc $\varphi'_{m}\chi^{m}$
est entier sur $\rm ker\,\it\partial$. Par normalit\'e de $A$
et puisque $\rm ker\,\it\partial$ est alg\'ebriquement clos dans $A$, on a 
$\varphi'_{m}\chi^{m}\in\rm ker\,\partial$. D'o\`u $L' = L$ et donc 
$\rm ker\,\it\partial =\rm ker\,\it\partial'$.

$\rm (ii)$ En rempla\c cant chaque terme  
\begin{eqnarray*}
\partial^{(i)}\,\,\,\,\rm par\,\,\,\,\it \varphi_{im}\chi^{im}\partial^{(i)}, 
\end{eqnarray*}
pour $m\in\omega_{L}$ et $\varphi_{m}$ comme dans $5.6.3$, nous pouvons supposer que $e\in \omega$. 
En particulier, l'alg\`ebre
\begin{eqnarray*}
A_{\omega} = \bigoplus_{m\in\omega_{M}}A_{m}\chi^{m}
\end{eqnarray*}
est $\partial$-stable. D'apr\`es les assertions pr\'ec\'edentes $\rm (i), (ii)$ 
(nous l'appliquons au couple $(A_{\omega},\partial_{|A_{\omega}})$),
consid\'erons seulement le cas o\`u $C = \mathbb{A}^{1}_{\mathbf{k}}$.
Soit $m\in\omega_{M}$. Puisque pour tout $m'\in\omega_{M}$, on a
\begin{eqnarray*}
A_{dm+m'} = A_{dm}\cdot A_{m'} = \varphi_{dm} A_{m'}, 
\end{eqnarray*}
l'id\'eal principal $(\varphi_{dm}\chi^{dm})$ dans l'anneau $A_{\omega}$
est $\partial_{|A_{\omega}}$-stable.
Par la proposition $5.3.4\,\rm (f)$, on a $\varphi_{dm}\chi^{dm}\in\rm ker\,\it\partial$
et donc $dm\in\omega_{L}$. Ainsi, nous obtenons $d\cdot\omega_{M}\subset\omega_{L}$ et l'assertion
 $\rm (iii)$ suit.
\end{proof}
Dans le r\'esultat suivant, nous donnons une caract\'erisation g\'eom\'etrique
des $\mathbb{G}_{m}$-surfaces affines horizontales non hyperboliques. Voir [FZ $3$, 2.13]
et [FZ $2$, $3.3$, $3.16$] pour le cas o\`u le corps de base est $\mathbb{C}$.  
\begin{corollaire}
Supposons que $\mathbf{k}$ est un corps parfait. Consid\'erons le cas o\`u $N = \mathbb{Z}$ et 
$\sigma = \mathbb{Q}_{\geq 0} \rm ;$ de sorte que $\mathfrak{D}$
est uniquement d\'etermin\'e par le diviseur de Weil rationnel $\mathfrak{D}(1)$.
Si l'alg\`ebre gradu\'ee $A$ admet une $\rm LFIHD$ homog\`ene de type horizontal 
alors les assertions suivantes sont vraies.
\begin{enumerate}
\item[\rm (i)] Si $C = \mathbb{A}^{1}_{\mathbf{k}}$ alors la partie fractionnaire $\{\mathfrak{D}(1)\}$
a au plus un point dans son support.
\item[\rm (ii)] Si $C = \mathbb{P}^{1}_{\mathbf{k}}$ alors $\{\mathfrak{D}(1)\}$ a au plus deux points 
dans son support.
\end{enumerate}
Dans chacun des cas pr\'esent\'es ci-dessus, le support de $\{\mathfrak{D}(1)\}$ ne contient que des points rationnels.
En particulier, toute $\mathbb{G}_{m}$-surface affine horizontale non hyperbolique sur $\mathbf{k}$ est torique.  
\end{corollaire}
\begin{proof}
$\rm (i)$
Tout d'abord, nous supposons que le corps de base $\mathbf{k}$ est alg\'ebriquement clos.
Soit $d$ le plus petit entier strictement positif tel que $\mathfrak{D}(d)$ est un diviseur entier.
En prenant un g\'en\'erateur $f\in \mathbf{k}(t)$ du $A_{0}$ module $A_{d}$, i.e. $A_{d} = f\cdot A_{0}$, 
nous consid\'erons $B$ la fermeture int\'egrale de l'alg\`ebre $A[\sqrt[d]{f}\chi]$ 
dans son corps des fractions.
En ajoutant un diviseur principal \`a $\mathfrak{D}(1)$, si n\'ecessaire, 
nous pouvons supposer que $f\in \mathbf{k}[t]$ est un polyn\^ome. 
Par l'assertion $\rm (iii)$ du lemme $5.6.5$, nous avons $f\chi^{d}\in\rm ker\,\it\partial$.
D'o\`u en utilisant le corollaire $5.3.6$, on obtient l'existence d'une LFIHD $\partial'$ sur $B$
comme extension de $\partial$ et satisfaisant $\sqrt[d]{f}\chi\in\rm ker\,\it\partial'$.
\'Ecrivons $B = A[C',\mathfrak{D}']$ pour un diviseur poly\'edral propre $\mathfrak{D}'$ 
sur une courbe affine r\'eguli\`ere $C' = \rm Spec\,\it B_{\rm 0}$. 
En fait, $B_{0}$ est la normalisation de $\mathbf{k}[t,\sqrt[d]{f}]$
et est aussi une alg\`ebre de polyn\^ome d'une variable sur $\mathbf{k}$, voir l'assertion $\rm (iii)$ du lemme 
$5.6.2$. Le fait que $B_{0}^{\star} = \mathbf{k}^{\star}$ et que $B_{0}$ est un anneau factoriel 
implique que $f = (t-z)^{r}$, pour $z\in \mathbf{k}$
et pour $r\in\mathbb{Z}_{>0}$.
Puisque $\rm div\,\it f + d\cdot\mathfrak{D}(\rm 1\it) = \rm 0$, on conclut que le support de $\{\mathfrak{D}(1)\}$
contient au plus $z$. 

Supposons que $\mathbf{k}$ n'est pas alg\'ebriquement clos et que le support de $\{\mathfrak{D}(1)\}$
contient au moins deux points. En calculant le diviseur poly\'edral de
$A\otimes_{\mathbf{k}}\bar{\mathbf{k}}$, on obtient une contradiction avec l'\'etape pr\'ec\'edente.
Cela montre l'assertion $\rm (i)$.

$\rm (ii)$ En multipliant $\partial$ par un \'el\'ement homog\`ene de son noyau, nous pouvons supposer  
que le degr\'e de $\partial$ est positif. En utilisant l'assertion $5.6.5\,\rm (ii)$, 
$\partial$ s'\'etend en une LFIHD homog\`ene $\partial'$ de type horizontal sur la normalisation $A'$ 
de l'alg\`ebre $A[t]$. Notons que l'alg\`ebre gradu\'ee $A'$ est donn\'ee par le diviseur poly\'edral 
$\mathfrak{D}_{|\mathbb{A}^{1}_{\mathbf{k}}}$. En appliquant l'assertion $\rm (i)$ ci-dessus, 
pour l'alg\`ebre gradu\'ee parabolique $A'$, 
la partie fractionnaire $\{\mathfrak{D}_{|\mathbb{A}^{1}_{\mathbf{k}}}(1)\}$
a au plus un point dans son support. Donc le support de $\{\mathfrak{D}(1)\}$ a au plus deux points. 
Cela donne l'assertion $\rm (ii)$.

Par un argument analogue, on d\'eduit que dans chaque cas le support de $\{\mathfrak{D}(1)\}$
ne contient que des points rationnels (voir la proposition $4.5.10$). 

Montrons la derni\`ere affirmation. Supposons que $A$ n'est pas elliptique. Puisque $\{\mathfrak{D}(1)\}$
a au plus un point dans son support (et cet \'eventuel point est rationnel), sans perte de g\'en\'eralit\'e,
nous pouvons supposer que
\begin{eqnarray*}
\mathfrak{D}(1) = -\frac{e}{d}\cdot 0, \,\,\,\rm avec\,\,\, 0\leq \it e\rm <\it d\rm\,\,\,
et\,\,\, p.g.c.d. \it(e,d)\rm = 1. 
\end{eqnarray*}
Un calcul direct montre que 
\begin{eqnarray*}
A = \bigoplus_{b\geq 0,\,ad-be\geq 0}\mathbf{k}\,t^{a}\chi^{b}, 
\end{eqnarray*}
voir par exemple [FZ, $2.4$, $3.8$] et [Li, $3.21$]. L'alg\`ebre $A$ admet une  
$\mathbb{Z}^{2}$-graduation faisant de $X = \rm Spec\,\it A$ une surface torique. 
Supposons que $A$ est elliptique. En utilisant le fait que tout diviseur entier sur $\mathbb{P}^{1}$ de degr\'e $0$ 
est principal, nous pouvons nous r\'eduire au cas o\`u $\mathfrak{D}$ est support\'e par les points $0$ et $\infty$.
On conclut par un argument analogue \`a [Li, $3.21$].
\end{proof}
Comme cons\'equence du corollaire $5.6.6$, nous obtenons le point suivant.
\begin{corollaire}
Avec les hypoth\`eses de $5.6.3$, consid\'erons $A_{\omega} = \bigoplus_{m\in\omega_{M}}A_{m}\chi^{m}$
et posons $\tau  = \omega^{\vee}$.
Alors $A_{\omega}\simeq A[C,\mathfrak{D}_{\omega}]$, comme alg\`ebre $M$-gradu\'ee, o\`u 
$\mathfrak{D}_{\omega}$ est un diviseur $\tau$-poly\'edral propre sur la courbe $C$ v\'erifiant
les assertions suivantes. 
\begin{enumerate}
\item[\rm (i)] Si $A$ n'est pas elliptique alors $\mathfrak{D}_{\omega} = (v+\tau)\cdot 0$,
pour $v\in N_{\mathbb{Q}}$.
\item[\rm (ii)] Si $A$ est elliptique alors $\mathfrak{D}_{\omega} = (v+\tau)\cdot 0 +
\Delta'_{\infty}\cdot\infty$, pour $v\in N_{\mathbb{Q}}$ et pour
$\Delta'_{\infty}\in \rm Pol_{\it \tau}\it(N_{\mathbb{Q}})$ satisfaisant 
$v + \Delta'_{\infty}\subsetneq \tau$.  
\end{enumerate}
\end{corollaire}
\begin{proof}
$\rm (i)$ Nous allons suivre les id\'ees de d\'emonstration de [Li, $3.23$]. Notons
que le degr\'e $e$ de $\partial$ appartient \`a $\omega$ (voir 5.5.1 $\rm (ii)$). Pour 
$l\in\omega_{L}$, d\'esignons par $\partial_{l}$ la
LFIHD homog\`ene o\`u le $i$-\`eme terme est l'application lin\'eaire $\varphi_{l}^{i}\partial^{(i)}$.
La sous-alg\`ebre 
\begin{eqnarray*}
A_{(l+e)} = \bigoplus_{r\geq 0}A_{r(l+e)}\chi^{r(l+e)}
\end{eqnarray*}
est $\partial_{l}$-stable. Puisque la LFIHD homog\`ene $\partial_{l}|_{A_{(l+e)}}$
est de type horizontal, le corollaire $5.6.6$ s'applique et $\{\mathfrak{D}(l+e)\}$ est
support\'e par au plus un point. Par l'assertion $\rm (i)$ du lemme $5.6.4$,
pour tous $l,l'\in\omega_{L}$, 
\begin{eqnarray*}
-\rm div\,\it\varphi_{l'} + \mathfrak{D}(l+e) = \mathfrak{D}(l+l'+e) = \mathfrak{D}(l'+e) 
-\rm div\,\it\varphi_{l} 
\end{eqnarray*}
et donc $\{\mathfrak{D}(l+e)\} = \{\mathfrak{D}(l'+e)\}$.
Ainsi, la r\'eunion de tous les supports $\{\mathfrak{D}(l+e)\}$, o\`u $l$ parcourt
$\omega_{L}$, a au plus un point. Par lin\'earit\'e de $\mathfrak{D}$ sur $\omega$ et
l'assertion $\rm (iv)$ du lemme $5.6.4$, apr\`es avoir ajout\'e un diviseur 
poly\'edral principal \`a $\mathfrak{D}$, le diviseur poly\'edral associ\'e \`a $A_{\omega}$ 
a au plus un point dans son support. Cet \'eventuel point est rationnel. Cela donne 
l'assertion $\rm (i)$.

$\rm (ii)$ 
Tout d'abord, nous pouvons supposer que $e\in\omega$. Soit $A'_{\omega}$
la normalisation de $A_{\omega}[t]$. Par le lemme $5.6.5$, un \'el\'ement de
degr\'e $m\in\omega_{M}$ de $A'_{\omega}$ est la multiplication d'une section
globale de $\lfloor\mathfrak{D}_{|\mathbb{A}^{1}_{\mathbf{k}}}(m)\rfloor$ et d'un caract\`ere $\chi^{m}$, 
et r\'eciproquement. De plus, $\partial$ s'\'etend en une LFIHD homog\`ene de type horizontal
sur $A'_{\omega}$. Par les arguments ci-dessus, la
r\'eunion de tous les supports de $\{\mathfrak{D}_{|\mathbb{A}^{1}_{\mathbf{k}}}(m)\}$,
o\`u $m$ parcourt $\omega_{M}$, a au plus un point. \`A nouveau, cet \'eventuel point est rationnel.
Cela nous permet de conclure.
\end{proof}
Le th\'eor\`eme suivant donne une premi\`ere classification des op\'erations normalis\'ees du groupe additif
de type horizontal pour une classe de vari\'et\'es toriques affines
munies d'une op\'eration de tore alg\'ebrique de complexit\'e $1$.  
\begin{theorem}
Soit $\mathfrak{D}$ un diviseur $\sigma$-poly\'edral propre sur une courbe r\'eguli\`ere $C$. 
Supposons que $\mathfrak{D}$ v\'erifie une des conditions suivantes.
\begin{enumerate}
 \item[\rm (i)] Si $C$ est affine alors 
$C = \mathbb{A}^{1}_{\mathbf{k}} = \rm Spec\,\it\mathbf{k}[t]$ et 
$\mathfrak{D} = (v + \sigma)\cdot 0$, pour $v\in N_{\mathbb{Q}}$.
\item[\rm (ii)] Si $C$ est projective alors $C = \mathbb{P}^{1}_{\mathbf{k}}$
et $\mathfrak{D} = (v + \sigma)\cdot 0 + \Delta_{\infty}\cdot\infty$, pour
$v\in N_{\mathbb{Q}}$ et pour $\Delta_{\infty}\in \rm Pol_{\it\sigma}(\it N_{\mathbb{Q}})$.
\end{enumerate}
Soit $d$ le plus petit entier strictement positif tel que $dv\in N$. Pour tout $m\in M$, posons $h(m) = \langle m, v\rangle$. Soit $\hat{\sigma}$ le c\^one de $N_{\mathbb{Q}}\times \mathbb{Q}$ engendr\'e par $(v, 1)$ et $(\sigma, 0)$ si $C = \mathbb{A}^{1}_{\mathsf{k}}$ et par $(v, 1)$, $(\sigma, 0)$, et $(\Delta_{\infty}, -1)$ si $C = \mathbb{P}^{1}_{\mathbf{k}}$. 
Alors il existe une $\rm LFIHD$ homog\`ene $\partial$ de type horizontal sur $A = A[C,\mathfrak{D}]$ avec
$\rm deg\,\it\partial = e$ si et seulement si un des \'enonc\'es est vrai.
\begin{enumerate}
 \item[\rm (a)] Si $\rm car\,\it\mathbf{k} = p \rm > 0$ 
alors il existe une suite
$0\leq s_{1} < s_{2} < \ldots < s_{r}$ de nombres entiers telle que pour $i = 1,\ldots, r$ nous avons $(p^{s_{i}}e,-1/d - h(p^{s_{i}}e))\in {\rm Rt}\,\hat{\sigma}$.
\item[\rm (b)] Si $\rm car\,\it \mathbf{k} \rm  = 0$ alors $(e,-1/d - h(e))\in {\rm Rt}\,\hat{\sigma}$. 
\end{enumerate}
Sous ces derni\`eres conditions, la $\rm LFIHD$ $\partial$ est de la forme suivante. Posons $\zeta = \sqrt[d]{t}$.
Fixons une $\rm LFIHD$ $\partial_{\zeta}$ sur l'alg\`ebre des polyn\^omes $\mathbf{k}[\zeta]$
\`a une variable $\zeta$ dont l'application exponentielle est donn\'ee par 
\begin{eqnarray}
e^{x\partial_{\zeta}}(\zeta) = \zeta + \sum_{i = 1}^{r}\lambda_{i}x^{p^{s_{i}}}, 
\end{eqnarray}
o\`u $\lambda_{1},\ldots \lambda_{r}\in\mathbf{k}^{\star}$ $\rm ($resp. avec $\partial^{(1)}_{\zeta} = 
\lambda \frac{\rm d}{\rm d\it \zeta}$,
o\`u $\lambda\in\mathbf{k}^{\star}\rm )$ lorsque $\rm car\it\,\mathbf{k}\rm >0$ $\rm ($resp. $\rm car\it\,
\mathbf{k}\rm = 0)$.
Alors le $i$-\`eme terme de $\partial$ est donn\'e par l'\'egalit\'e 
\begin{eqnarray}
\partial^{(i)}(t^{l}\chi^{m}) = \zeta^{-dh(m+ie)}\partial^{(i)}_{\zeta}(\zeta^{dh(m)}t^{l})\chi^{m+ie}, 
\end{eqnarray}
o\`u $t^{l}\chi^{m}\in A$.   
\end{theorem}
\begin{proof} 

Supposons que $\mathfrak{D}$ satisfait la condition $\rm (i)$
et fixons une LFIHD homog\`ene $\partial$ de type horizontal sur l'alg\`ebre $A$ et
de degr\'e $e$. Soit $B$ la normalisation de la sous-alg\`ebre 
\begin{eqnarray*}
A\left[\zeta^{-dh(e)}\chi^{e}\right]\subset\mathbf{k}(\zeta)[M].
\end{eqnarray*}
En consid\'erant la droite affine $C' = \rm Spec\,\it\mathbf{k}[\zeta]$ et 
le diviseur poly\'edral $\mathfrak{D}' = (dv + \sigma)\cdot 0$ sur $C'$,
l'alg\`ebre $A[C',\mathfrak{D}']$ est pr\'ecis\'ement $B$ 
(voir le th\'eor\`eme $4.4.4$ ou [La $2$, $2.5$]). D'apr\`es l'assertion $\rm (ii)$ 
du lemme $5.5.1$, nous avons $e\in\sigma^{\vee}$. Donc $A\left[\zeta^{-dh(e)}\chi^{e}\right]$
est une extension cyclique de l'anneau $A$. 
Puisque $\varphi_{de}\chi^{de}\in\rm ker\,\it\partial$, par le corollaire $5.3.6$, 
$\partial$ s'\'etend en une LFIHD $\partial'$ sur l'alg\`ebre $B$. 
En utilisant de plus que $dv\in N$, nous obtenons un isomorphisme 
d'alg\`ebres $M$-gradu\'ees
\begin{eqnarray*}
\varphi : B\rightarrow E,\,\,\,
\zeta^{l}\chi^{m}\mapsto \zeta^{dh(m)+l} \chi^{m},
\end{eqnarray*} 
o\`u $E = \mathbf{k}[\sigma^{\vee}_{M}][\zeta]$.
Consid\'erons $\varphi_{\star}\partial'$ la LFIHD homog\`ene de type horizontal sur $E$ 
d\'efinie par l'\'egalit\'e
\begin{eqnarray*}
\varphi_{\star}\partial'^{(i)} = \varphi\circ\partial'^{(i)}\circ\varphi^{-1},
\end{eqnarray*}
o\`u $i\in\mathbb{N}$. L'assertion $\rm (iii)$ du lemme $5.6.5$ implique que
$\rm ker\,\it\varphi_{\star}\partial' = \mathbf{k}[\sigma^{\vee}_{M}]$.
En posant $e = \rm deg\,\it\partial$, nous d\'eduisons que $\varphi_{\star}\partial' = \chi^{e}\cdot\partial_{\zeta}$,
pour une LFIHD non triviale $\partial_{\zeta}$ sur $\mathbf{k}[\zeta]$. 
Un calcul facile montre que $\partial = 
\varphi_{\star}^{-1}(\varphi_{\star}\partial')$ satisfait l'\'egalit\'e $(5.6)$. 
  
Supposons que $\rm car\,\it\mathbf{k} = p\rm >0$ et montrons que la condition $\rm (a)$
est vraie pour le vecteur $e = \rm deg\,\it\partial$. 
Par la proposition $5.3.4\,\rm (d)$,
l'application exponentielle de $\partial_{\zeta}$ est donn\'ee comme dans $(5.5)$
pour des entiers $0\leq s_{1}<\ldots < s_{r}$.
Si $p$ ne divise pas $d$ alors consid\'erons $l\in\mathbb{N}-p\mathbb{Z}$ tel que
$dl\geq p^{s_{1}}$. Notons que $t^{l}\in A$. Par le lemme $5.3.13$ et $(5.6)$, 
nous obtenons l'\'egalit\'e
\begin{eqnarray*}
\partial^{(p^{s_{1}})}(t^{l}) = \lambda_{1} dlt^{-1/d-h(p^{s_{1}}e)+l}\chi^{p^{s_{1}}e}. 
\end{eqnarray*}
Puisque $\partial^{(p^{s_{1}})}(t^{l})\in A-\{0\}$, il s'ensuit que $-1/d-h(p^{s_{1}}e)\in\mathbb{Z}$. 

Supposons maintenant que $p$ divise $d$. Par minimalit\'e de $d$, il existe $m\in\sigma^{\vee}_{M}$ tel que
 $dh(m)$ n'est pas divisible par $p$. En prenant $l\in\mathbb{N}$ tel que $dl\geq \max(p^{s_{1}},-dh(m))$,
nous avons $t^{l}\chi^{m}\in A-\{0\}$ et donc le lemme $5.3.13$ implique
\begin{eqnarray*}
\partial^{(p^{s_{1}})}(t^{l}\chi^{m}) = \lambda_{1} dh(m)t^{-1/d-h(p^{s_{1}}e)+l}\chi^{m+p^{s_{1}}e}\in A-\{0\}. 
\end{eqnarray*}
D'o\`u dans tous les cas $e_{1} : = (p^{s_{1}}e, -1/d-h(p^{s_{1}}e))\in \hat{M}$, 
o\`u $\hat{M} = M\times\mathbb{Z}$. 

Remarquons que 
\begin{eqnarray*}
A[C,\mathfrak{D}] = \bigoplus_{(m,l)\in\hat{\sigma}^{\vee}_{\hat{M}}}\mathbf{k}\,\chi^{(m,l)} = 
\mathbf{k}[\hat{\sigma}^{\vee}_{\hat{M}}], 
\end{eqnarray*}
o\`u $\chi^{(m,l)} = t^{l}\chi^{m}$ et $\hat{\sigma}$ est le c\^one engendr\'e
par $(v,1)$ et $(\sigma,0)$. Un calcul facile montre que $e_{1} = (p^{s_{1}}e, -1/d-h(p^{s_{1}}e))\in
\rm Rt\,\it\hat{\sigma}$, pour le rayon distingu\'e $\rho = (dv,d)$. Donc par le corollaire $5.4.6$, l'alg\`ebre 
$\hat{M}$-gradu\'ee $A$ admet des LFIHD rationnellement homog\`enes de degr\'e  $e_{1}/p^{s_{1}}$ 
provenant de la racine $e_{1}$. Une d'entre elles est donn\'ee par l'\'egalit\'e
\begin{eqnarray*}
e^{x\partial_{1}}(t^{l}\chi^{m}) = \sum_{i = 0}^{\infty}\binom{d(l+h(m))}{i}\lambda_{1}^{i}t^{l-i(1/d+h(p^{s_{1}}e))}\chi^{m+ip^{s_{1}}e}x^{ip^{s_{1}}}, 
\end{eqnarray*}
o\`u $\lambda_{1}\in\mathbf{k}^{\star}$ est le m\^eme que dans $(5.5)$. 
De plus, par le corollaire $5.3.6$, nous \'etendons $\partial'_{1}$
en une LFIHD homog\`ene, not\'ee $\partial_{1}'$, sur l'alg\`ebre $M$-gradu\'ee $B$. 
Supposons que $r\geq 2$. On peut voir $e^{x\partial'}$ et $e^{x\partial_{1}'}$
comme des automorphismes de l'alg\`ebre $B[x]$ en posant $e^{x\partial'}(x) = e^{x\partial_{1}}(x) = x$.
D'o\`u en utilisant cette convention, il vient 
\begin{eqnarray*}
e^{x\partial'}\circ (e^{x\partial'_{1}})^{-1} = e^{x\varphi_{\star}^{-1}(\chi^{e}\partial_{\zeta,1})}, 
\end{eqnarray*}
o\`u $\partial_{\zeta,1}$ est la LFIHD sur $\mathbf{k}[\zeta]$ d\'efinie par
\begin{eqnarray*}
 e^{x\partial_{\zeta,1}}(\zeta) = \zeta + \sum_{i = 2}^{r}\lambda_{i}x^{p^{s_{i}}}.
\end{eqnarray*}
Par cons\'equent, l'application $e^{x\partial'}\circ (e^{x\partial'_{1}})^{-1}$ 
donne une LFIHD homog\`ene $\partial''_{1}$ sur $A$. 

En fait, en rempla\c cant $\partial_{\zeta}$ par $\partial_{\zeta,1}$,
la LFIHD $\partial''_{1}$ satisfait l'\'egalit\'e $(5.6)$. 
\`A nouveau, nous d\'eduisons que 
$e_{2}:= (p^{s_{2}}e, -1/d-h(p^{s_{2}}e))\in \hat{M}$ est une racine du c\^one
$\hat{\sigma}$. On conclut que l'assertion $\rm (a)$ est vraie en raisonnant par r\'ecurrence.

Si $\rm car\,\it\mathbf{k}\rm  = 0$ alors la d\'erivation localement nilpotente
 $\partial^{(1)}_{\zeta}$ sur l'alg\`ebre des polyn\^omes $\mathbf{k}[\zeta]$ est 
$\lambda \frac{\rm d}{\rm d\it \zeta}$ pour un scalaire $\lambda\in\mathbf{k}^{\star}$.
En utilisant $(5.6)$, nous avons
\begin{eqnarray*}
\partial^{(1)}(t) = \lambda dt^{-1/d-h(e) + 1}\chi^{e}\in A-\{0\} 
\end{eqnarray*}
et donc l'assertion $\rm (b)$ est v\'erifi\'ee.

Montrons que l'assertion $\rm (a)$ ou $\rm (b)$ est vraie dans la situation o\`u $\mathfrak{D}$ v\'erifie $\rm (ii)$ et
$A$ admet une LFIHD homog\`ene de type horizontal avec $e = \rm deg\,\it\partial$. Supposons que la 
caract\'eristique du corps $\mathbf{k}$ est arbitraire.
Soit $A'$ la normalisation de $A[t]$ dans son corps des fractions $\rm Frac\,\it A$.
En utilisant l'assertion $\rm (iii)$ du lemme $5.6.5$, on a 
\begin{eqnarray*}
d\cdot M = h^{-1}(\mathbb{Z}) \subset L,
\end{eqnarray*}
o\`u $L$ est le sous-r\'eseau de $M$ engendr\'e par les poids de
$\rm ker\,\it \partial$. Donc en rempla\c cant $\partial$ par $\varphi_{m}\cdot\partial$,
pour $m\in\sigma^{\vee}_{d\cdot M}$, nous pouvons supposer que $e\in\sigma^{\vee}$
(voir les notations de $5.6.3$). Plus pr\'ecis\'ement, remplacer $e$ par $e + m$, pour 
$m\in\sigma^{\vee}_{d\cdot M}$ ne change pas les assertions $\rm (a)$ et $\rm (b)$.
Avec ces nouvelles hypoth\`eses, \`a nouveau par le lemme $5.6.5$, on peut \'etendre $\partial$
en une LFIHD homog\`ene $\bar{\partial}$ de type horizontal sur l'alg\`ebre $A'$.
Par l'argument pr\'ec\'edent (le cas o\`u $C = \mathbb{A}^{1}_{\mathbf{k}}$)
appliqu\'e au couple $(A',\bar{\partial})$, nous obtenons $\rm (a)$ si $\rm car\,\it\mathbf{k}\rm >0$ 
et $\rm (b)$ sinon.

Il reste \`a montrer que si un vecteur entier $e$ satisfait les conditions $\rm (a)$ et $\rm (b)$
alors nous pouvons construire une LFIHD homog\`ene de type horizontal sur $A = A[C,\mathfrak{D}]$
et de degr\'e $e$ comme dans $(5.6)$. 
Supposons que $\rm car\,\it\mathbf{k}\rm >0$
et posons $e_{i} = (p^{s_{i}}e,-1/d-h(p^{s_{i}}e))$. Par l'assertion $\rm (a)$, nous avons
$e_{i}\in\rm Rt\,\it \hat{\sigma}$ et nous pouvons consid\'erer les LFIHD rationnellement
homog\`enes $\partial_{e_{1},s_{1}},\ldots, \partial_{e_{r},s_{r}}$ 
sur l'alg\`ebre de mono\"ide $\mathbf{k}[\hat{\sigma}^{\vee}_{\hat{M}}]$ 
(voir l'exemple $5.4.2$). En utilisant l'isomorphisme $\varphi$ et en consid\'erant
chaque $e^{x\partial_{e_{i},s_{i}}}$ comme un automorphisme de l'anneau $A[x]$,
un calcul montre que la composition
\begin{eqnarray*}
e^{x\partial_{e_{1},s_{1}}}\circ e^{x\partial_{e_{2},s_{2}}}\circ
\ldots\circ e^{x\partial_{e_{r},s_{r}}}
\end{eqnarray*}
d\'efinit une LFIHD comme dans $(5.6)$. Pour le cas de 
$\rm car\,\it \mathbf{k}\rm =0$, nous utilisant un argument analogue (voir \'egalement [Li, $3.20$, $3.21$]).
Cela termine la d\'emonstration du th\'eor\`eme $5.6.8$. 
\end{proof}
Dans l'assertion suivante, nous \'etudions comment change la pr\'esentation d'Altmann-Hausen
apr\`es un rev\^etement cyclique. 
\begin{lemme}
Posons $A = A[C,\mathfrak{D}]$, 
o\`u $C$ est une courbe r\'eguli\`ere sur $\mathbf{k}$ avec corps des fonctions rationnelles
$K_{0}$ et $\mathfrak{D} = \sum_{z\in C}\Delta_{z}\cdot z$
est un diviseur $\sigma$-poly\'edral propre. Consid\'erons la normalisation
$A'$ de l'extension cyclique $A[s\chi^{e}]$, o\`u $e\in \sigma^{\vee}_{M}$, $s^{d}\in A_{de}$ et $d\in\mathbb{Z}_{>0}$. 
Supposons que $\mathbf{k}$ est alg\'ebriquement clos dans $A'$.
Alors $A' = A[C',\mathfrak{D}']$
o\`u $C'$ et $\mathfrak{D'}$ sont d\'efinis par les points suivants.
\begin{enumerate}
\item[\rm (i)] Le corps $K_{0}[s]$ est un corps de fonctions alg\'ebriques d'une variable sur $\mathbf{k}$.
Si $A$ est elliptique alors $A'$ l'est aussi et $C'$ est la courbe
projective r\'eguli\`ere sur $\mathbf{k}$ associ\'ee au corps de fonctions alg\'ebriques $K_{0}[s]$.
\item[\rm (ii)] Si $A$ n'est pas elliptique alors $A'$ l'est aussi et $C' = \rm Spec\,\it A'_{\rm 0}$,
o\`u $A'_{0}$ est la normalisation de $A_{0}$ dans $K_{0}[s]$. 
\item[\rm (iii)] 
Dans tous les cas, $\mathfrak{D}' = \sum_{z\in C}\Delta_{z}\cdot \pi^{\star}(z)$,
o\`u $\pi:C'\rightarrow C$ est la projection naturelle.
\end{enumerate}
\end{lemme}
\begin{proof}
Notons $K = \rm Frac\,\it A$. Par d\'efinition, $A'$ est la cl\^oture int\'egrale de $A$ dans 
$K[s\chi^{e}] = K[s]$. L'alg\`ebre $A'$ est $M$-gradu\'ee, normale, et de type fini
sur $\mathbf{k}$ (voir le th\'eor\`eme $2$ de [Bou, V3.2]). On a $K[s]^{\mathbb{T}} = K_{0}(s)$
de sorte que $X' = \rm Spec\,\it A'$ est une $\mathbb{T}$-vari\'et\'e de complexit\'e un.
Par ailleurs, l'inclusion $A[s\chi^{e}]\subset K_{0}(s)[\sigma^{\vee}_{M}]$
implique $A'\subset K_{0}(s)[\sigma^{\vee}_{M}]$ et $\sigma^{\vee}$
est le c\^one des poids de $A'$. Par hypoth\`ese, $\mathbf{k}$ est alg\'ebriquement
clos dans
\begin{eqnarray*}
\rm Frac\,\it A' = K[s] = \rm Frac\,\it K_{\rm 0\it}[s][M]
\end{eqnarray*}  
et donc par les th\'eor\`emes 4.4.4 (ii) et 4.5.6 (i), il existe une courbe r\'eguli\`ere $C'$
sur $\mathbf{k}$ et un unique diviseur $\sigma$-poly\'edral propre $\mathfrak{D}'$ sur $C'$
tels que $A' = A[C',\mathfrak{D}']$, o\`u $C'$ v\'erifie (i) ou (ii).

Il reste \`a montrer (iii). Notons $\pi :C'\rightarrow C$
la projection\footnote{On remarque que le morphisme $\pi$ est surjectif puisque $\pi$ est dominant et 
fini.} provenant de l'inclusion $\mathbf{k}(C)\subset \mathbf{k}(C')$. Ecrivons
pour $z\in C$,
\begin{eqnarray*}
\pi^{\star}(z) = \sum_{z'\in C'}a_{z'}\cdot z'\,\,\,\rm et\,\,\,\rm 
\mathfrak{D}' = \sum_{\it z'\in C'}\rm \Delta'_{\it z'}\it\cdot z'. 
\end{eqnarray*}
Soient $f_{1}\chi^{m_{1}},\ldots, f_{r}\chi^{m_{r}}\in A$ des \'el\'ements homog\`enes de 
degr\'es respectifs $m_{1},\ldots, m_{r}$ tels que 
\begin{eqnarray*}
A = \mathbf{k}[C][f_{1}\chi^{m_{1}},\ldots, f_{r}\chi^{m_{r}}]. 
\end{eqnarray*}
Alors $A'$ est la normalisation de l'alg\`ebre
\begin{eqnarray*}
\mathbf{k}[C'][\pi^{\star}(f_{1})\chi^{m_{1}},\ldots, \pi^{\star}(f_{r})\chi^{m_{r}}, s\chi^{e}]. 
\end{eqnarray*}
Par les th\'eor\`emes 4.4.4 (iii) et 4.5.6 (iii), pour tout $z'\in C'$ dans le support 
de $\pi^{\star}(z)$, on obtient 
\begin{eqnarray*}
\Delta'_{z'} = \{v\in N_{\mathbb{Q}}\,|\,\langle m_{i}, v\rangle \geq -\rm ord_{\it z',C'}(\it \pi^{\star}(f_{i})),
i = \rm 1,\ldots,\it r, \langle de, v\rangle \geq  -\rm ord_{\it z',C'}(\it \pi^{\star}(s^{d}))\}.
\end{eqnarray*}
Puisque pour tout $f\in\mathbf{k}(C)^{\star}$, 
\begin{eqnarray*}
\rm div_{\it C'}\it (\pi^{\star}(f)) = 
\sum_{z\in C}\rm ord_{\it z, C}\it\, f\cdot \pi^{\star}(z),
\end{eqnarray*}
 pour tout $z'\in C'$ dans le support
de $\pi^{\star}(z)$, on a 
\begin{eqnarray*}
\rm ord_{\it z',C'}\it (\pi^{\star}(f_{i})) = a_{z'}\cdot \rm ord_{\it z,C}\it(f_{\it i})
\,\,\,\rm et\,\,\, ord_{\it z',C'}\it (\pi^{\star}(s^{d})) = a_{z'}\cdot \rm ord_{\it z,C}\it(s^{d}),
\end{eqnarray*}
de sorte que $\Delta'_{z'} = a_{z'}\cdot \Delta_{z}$. Cela montre 
l'\'egalit\'e $\mathfrak{D}' = \sum_{z\in C}\Delta_{z}\cdot\pi^{\star}(z)$ et termine la preuve.
\end{proof}

La remarque suivante sera utilis\'ee dans la d\'emonstration du lemme $5.6.11$.
\begin{remarque}
Supposons que $\mathbf{k}$ est un corps parfait et soit $r\in\mathbb{Z}_{>0}$. Alors
l'application de Frobenius it\'er\'e $F:\mathbf{k}\rightarrow \mathbf{k}, \lambda\mapsto \lambda^{p^{r}}$ 
est un automorphisme de corps. Soit $t$ une variable sur $\mathbf{k}$ et notons $x = t^{p^{r}}$. \'Etudions
la ramification de l'extension alg\'ebrique $\mathbf{k}(t)/\mathbf{k}(x)$. Pour cela,
choisissons un polyn\^ome irr\'eductible $P(x)\in \mathbf{k}[x]$. Alors, on a
\begin{eqnarray*}
P(x) = P(t^{p^{r}}) = (F^{\star}(P)(t))^{p^{r}}, 
\end{eqnarray*}
o\`u $F^{\star}(P) = \sum F^{-1}(a_{i})T^{i}$, l'\'el\'ement $a_{i}\in\mathbf{k}$ correspond au $i$-\`eme 
coefficient de $P$. On note que $F^{\star}(P)(t)$ est irr\'eductible
dans $\mathbf{k}[t]$. Soient $C, C'$ les courbes projectives r\'eguli\`eres sur $\mathbf{k}$
associ\'ees respectivement \`a $\mathbf{k}(t)$, $\mathbf{k}(x)$. On a $C\simeq \mathbb{P}^{1}_{\mathbf{k}}\simeq C'$.
L'inclusion $\mathbf{k}(x)\subset \mathbf{k}(t)$ induit un morphisme fini purement ins\'eparable
$\pi :C'\rightarrow C$. Par le calcul pr\'ec\'edent, pour tout $z\in C$, 
on a l'\'egalit\'e\footnote{La v\'erification de cette \'egalit\'e pour la place \`a l'infini est laiss\'ee au lecteur.} 
de diviseurs de Weil, $\pi^{\star}z = p^{r}\cdot z'$,
o\`u $z'\in C$ est dans la fibre sch\'ematique de $z\in C$.   
\end{remarque}

\begin{lemme}
Supposons que $\mathbf{k}$ est parfait.
Soit $\mathfrak{D}$ un diviseur $\sigma$-poly\'edral propre sur $C = \mathbb{A}^{1}_{\mathbf{k}}$
ou sur $C = \mathbb{P}^{1}_{\mathbf{k}}$. 
Supposons qu'il existe un c\^one maximal $\omega$ du quasi-\'eventail 
$\Lambda(\mathfrak{D})$ ou de $\Lambda(\mathfrak{D}_{|\mathbb{A}^{1}_{\mathbf{k}}})$, respectivement 
selon les cas non elliptique et elliptique,
tel que pour tout $z\in C$ diff\'erent de $0,\infty$, $h_{z}|_{\omega} = 0$. Soit
$\partial$ la $\rm LFIHD$ de degr\'e $e$ sur l'alg\`ebre $A[C,\mathfrak{D}_{\omega}]$ donn\'ee par 
la formule $(5.6)$ $\rm ($voir les notations de $5.6.7\rm )$. Consid\'erons $p$ l'exposant caract\'eristique de $\mathbf{k}$. 
Alors $\partial$ s'\'etend en une $\rm LFIHD$ sur $A = A[C,\mathfrak{D}]$ si et seulement 
si pour tout $m\in\sigma^{\vee}_{M}$ tel que $m + p^{s_{1}}e\in \sigma^{\vee}_{M}$,
les assertions suivantes sont vraies. 
\begin{enumerate}
\item[\rm (i)] Si $h_{z}(m+p^{s_{1}}e)\neq 0$ alors $\lfloor p^{k}h_{z}(m+p^{s_{1}}e)\rfloor - 
\lfloor p^{k}h_{z}(m)\rfloor\geq 1$, $\forall z\in C, z\neq 0,\infty$.
\item[\rm (ii)] Si $h_{0}(m+p^{s_{1}}e)\neq h(m + p^{s_{1}}e)$ alors 
$\lfloor dh_{0}(m+p^{s_{1}}e)\rfloor - \lfloor dh_{0}(m)\rfloor\geq 1 - dh(p^{s_{1}}e)$.
\item[\rm (iii)] Si $C = \mathbb{P}^{1}_{\mathbf{k}}$ alors $\lfloor dh_{\infty}(m+p^{s_{1}}e)\rfloor  - \lfloor dh_{\infty}(m)\rfloor
\geq -1 - dh(p^{s_{1}}e)$. 
\end{enumerate}
Ici $h$ est l'extension lin\'eaire de $h_{0}|_{\omega}$, $d$ est le plus petit
entier strictement positif tel que $dh$ est enti\`ere et $d = lp^{k}$, avec $\rm p.g.c.d.\it (l,p) \rm = 1$. 
\end{lemme}
\begin{proof}
Nous allons adapter \`a notre contexte les arguments de d\'emonstration de [Li, $3.26$]. 
En consid\'erant $m\in\sigma^{\vee}_{M}$, nous pouvons \'ecrire $h(m) = \langle m, v\rangle$, pour un $v\in N_{\mathbb{Q}}$.
Puisque chaque $h_{z}$ est concave, $h_{z}(m)\leq 0$, $\forall z\in C-\{0,\infty\}$, et \'evidemment $h_{0}(m)\leq h(m)$.
En posant 
\begin{eqnarray*}
A_{M} = \bigoplus_{m\in M}\varphi_{m}\mathbf{k}[t]\,\chi^{m},
\end{eqnarray*}
o\`u $\varphi_{m} = t^{-\lfloor h(m)\rfloor}$ et en localisant par des \'el\'ements homog\`enes
du noyau $\rm ker\,\it\partial$, par le lemme $5.3.5$, $\partial$ s'\'etend en une LFIHD homog\`ene 
sur l'alg\`ebre $A_{M}$; \`a nouveau, nous d\'esignons par $\partial$ son extension.
Ainsi $\partial$ s'\'etend en une LFIHD sur $A$ si et seulement si l'extension $\partial$ sur $A_{M}$ 
stabilise $A$. De plus, nous pouvons supposer que $\mathbf{k} = \bar{\mathbf{k}}$
est alg\'ebriquement clos puisque l'extension $\partial_{\bar{\mathbf{k}}}$ de $\partial$ 
sur $A_{M}\otimes_{\mathbf{k}} \bar{\mathbf{k}}$ laisse stable $A\otimes_{\mathbf{k}} \bar{\mathbf{k}}$ si et seulement $\partial$ laisse stable $A$.  

Nous commen\c cons par montrer que $\partial$ stabilise $A$ dans des cas particuliers.

\em Cas $h = 0$. \rm Dans ce cas, nous avons $d = 1$, $L = M$ et
par le th\'eor\`eme $5.6.8$, $\partial =\chi^{e}\partial_{t}$, pour une LFIHD $\partial_{t}$
sur $\mathbf{k}[t]$. Nous supposons maintenant $\rm car\,\it\mathbf{k} = p\rm >0$.
Pour le cas de la caract\'eristique z\'ero le lecteur peut consulter l'argument de [Li, $3.26$]. 

La LFIHD $\partial_{t}$ est donc d\'etermin\'ee par une suite 
$0\leq s_{1}<\ldots < s_{r}$ de nombres entiers (voir $5.3.4 \rm (d)$). De plus,
par le paragraphe d'avant, pour tout $z\in\mathbb{A}^{1}_{\mathbf{k}}$, 
$h_{z}\leq 0$ et $h_{\infty}\geq 0$ pour le cas elliptique.
En fixant $m\in\sigma^{\vee}_{M}$ tel que $m + p^{s_{1}}e\in\sigma^{\vee}_{M}$,
les conditions de notre lemme deviennent
\begin{enumerate}
\item[ (i')] Si $h_{z}(m+p^{s_{1}}e)\neq 0$ alors $\lfloor h_{z}(m + p^{s_{1}}e)\rfloor
 - \lfloor h_{z}(m)\rfloor \geq 1$, $\forall z\in\mathbb{A}^{1}_{\mathbf{k}}$.
\item[(iii')] Si $C = \mathbb{P}^{1}_{\mathbf{k}}$ alors
$\lfloor h_{\infty}(m+p^{s_{1}}e)\rfloor  - \lfloor h_{\infty}(m)\rfloor
\geq -1$.
\end{enumerate}
Sous les hypoth\`eses pr\'ec\'edentes,    
\begin{eqnarray*}
A_{m} = H^{0}(C,\mathcal{O}_{C}(\lfloor \mathfrak{D}(m)\rfloor))\subset \mathbf{k}[t]
\end{eqnarray*}
et $\partial$ stabilise $A$ si et seulement si 
\begin{eqnarray*}
f(t)\in A_{m}\Rightarrow \partial^{(i)}_{t}(f(t))\in A_{m + ie}, \forall m\in\sigma^{\vee}_{M}, \forall
i\in\mathbb{N},
\end{eqnarray*} 
ou de mani\`ere \'equivalente, 
\begin{eqnarray*}
\rm div\,\it f + \lfloor \mathfrak{D}(m)\rfloor \rm \geq 0\it\Rightarrow 
\rm div\,\it\partial_{t}^{(i)}(f) + \lfloor\mathfrak{D}(m+ie)\rfloor 
\rm\geq 0, \it \forall m\in\sigma^{\vee}_{M}, \forall i\in\mathbb{N}. 
\end{eqnarray*}
ou encore $\it \forall m\in\sigma^{\vee}_{M}, \forall i\in\mathbb{N}, \forall z\in C,$
\begin{eqnarray}
\rm ord_{\it z}\it\, f + \lfloor h_{z}(m)\rfloor \rm \geq 0\Rightarrow
\rm ord_{\it z}\it\,\partial_{t}^{(i)}(f) + \lfloor h_{z}(m + ie)\rfloor
\rm \geq 0. 
\end{eqnarray}
Montrons que $\rm (i')\it\Rightarrow \rm (5.7)$.
Si $h_{z}(m+p^{s_{1}}e)\neq 0$ avec $m\in\sigma^{\vee}_{M}$ tel que 
$m + p^{s_{1}}e\in\sigma^{\vee}_{M}$ alors on voit facilement que $h_{z}(m)\neq 0$,
de sorte que $f\in (t-z)\mathbf{k}[t]$.

Soit $i\in\mathbb{N}$. Si $\partial^{(i)}_{t}(f) = 0$ alors $\partial^{(i)}_{t}(f)\in A_{m+ie}$.
Sinon, $\partial^{(i)}_{t}(f)\neq 0$ et donc $m+ie\in\sigma^{\vee}$. Par cons\'equent,
il existe $i_{1},\ldots, i_{r}\in\mathbb{N}$ tels que 
\begin{eqnarray*}
i_{1}p^{s_{1}} + \ldots + i_{r}p^{s_{r}} = i,\,\,\, \it i_{\rm 1\it } + \ldots + i_{r}\rm 
\leq ord_{\it z}\,\it f, 
\end{eqnarray*}
\begin{eqnarray*}
\rm ord_{\it z}\,\it \partial_{t}^{(i)}(f) = \rm ord_{\it z}\it (f) - (i_{\rm 1\it} + \ldots + i_{r})\rm \geq 0
\end{eqnarray*}
(voir l'argument de la d\'emonstration du lemme $5.3.13$). En posant $i = lp^{s_{1}}$ pour
$l\in\mathbb{N}$, nous avons $l\geq i_{1}+\ldots + i_{r}$. D'o\`u il s'ensuit que
\begin{eqnarray*}
\rm ord_{\it z}\it\partial^{(i)}(f) + \lfloor h_{z}(m + ie)\rfloor\rm \geq 
ord_{\it z}\it (f) + \lfloor h_{z}(m)\rfloor  + (\lfloor h_{z}(m + lp^{s_{\rm 1\it}}e)\rfloor  - \lfloor h_{z}(m)\rfloor -l). 
\end{eqnarray*}  
Par convexit\'e de $\sigma^{\vee}$ pour $1\leq j\leq l$, on a $m + jp^{s_{1}}e\in \sigma^{\vee}$.
Si $h_{z}(m+ie) = 0$ alors $(5.7)$ est v\'erifi\'ee. Sinon $h_{z}(m+ie)\neq 0$ et
en utilisant les in\'egalit\'es pr\'ec\'edentes, l'assertion $\rm (i')$, et puisque 
$\rm ord_{\it z}\,\it f  + \lfloor h_{z}(m)\rfloor\rm \geq 0$,
on obtient 
\begin{eqnarray*}
\rm ord_{\it z}\it\partial^{(i)}(f) + \lfloor h_{z}(m + ie)\rfloor \rm \geq 
ord_{\it z}\it (f) + \lfloor h_{z}(m)\rfloor + 
\end{eqnarray*}
\begin{eqnarray*}
\sum_{j = 1}^{r}(\lfloor h_{z}(m + (l-j)p^{s_{\rm 1\it}}e + p^{s_{1}}e)\rfloor  - \lfloor h_{z}(m + (l-j)p^{s_{\rm 1\it}}e )\rfloor  -1)\geq 0
\end{eqnarray*}
Cela donne $\rm (i')\Rightarrow (5.7)$.

Montrons maintenant la r\'eciproque dans le cas o\`u $C$ est affine. Nous supposons que $\Lambda(\mathfrak{D})$ 
a moins deux \'el\'ements maximaux. Soit $\omega_{0}\in \Lambda(\mathfrak{D})$ 
un \'el\'ement maximal distinct de $\omega$.
Alors il existe un vecteur entier $m\in \rm relint\,\it \omega_{\rm 0}$ tel que $h_{z}(m)\in\mathbb{Z}$
et $\partial^{(lp^{s_{1}})}(\varphi_{m})\neq 0$, pour $l\in\mathbb{N}$. Notons qu'ici le noyau est $\rm ker\,\it \partial
 = \bigoplus_{m\in\omega_{M}}\mathbf{k}\varphi_{m}\chi^{m}$. En prenant $s>0$ suffisamment grand, 
par le lemme $5.3.13$ et par le th\'eor\`eme
de Lucas\footnote{On suppose que $p$ est un nombre premier. Soient $a,b$ des nombres entiers naturels
de developpements $p$-adiques $a =\sum_{i = 0}^{r}a_{i}p^{i}$ et $b = \sum_{i = 0}^{r}b_{i}p^{i}$.
Notons que $a_{r}$ ou $b_{r}$ peuvent \^etre \'eventuellement nuls. Alors le th\'eor\`eme de Lucas
affirme que l'on a la congruence $\binom{a}{b}\equiv \prod_{i = 1}^{r}\binom{a_{i}}{b_{i}}\,(\rm mod\,\it p\rm )$.
En particulier, $\binom{a}{b}\equiv 0\,(\rm mod\,\it p\rm )$ si et seulement si
il existe $i\in\{1,\ldots, r\}$ tel que $b_{i}>a_{i}$.}, nous avons 
\begin{eqnarray*}
\partial^{(lp^{s_{1}})}(\varphi_{m}^{p^{s}+1})\neq 0. 
\end{eqnarray*}
D'o\`u nous pouvons supposer que $-h_{z}(m)\geq lp^{s_{1}}$.
Ainsi, \`a nouveau par $5.3.13$, 
\begin{eqnarray*}
\rm ord_{\it z}\partial_{\it t}^{(\it lp^{s_{\rm 1\it}})}(\varphi_{\it m}) \it = -h_{z}(m) - l.
\end{eqnarray*}
Par $(5.7)$, il vient 
\begin{eqnarray}
\lfloor h_{z}(m + lp^{s_{\rm 1\it}}e)\rfloor  - h_{z}(m) - l\rm \geq 0.
\end{eqnarray}
En notant $\bar{h}_{z}$ l'extension lin\'eaire de $h_{z}|_{\omega_{0}}$, nous obtenons
\begin{eqnarray}
\lfloor h_{z}(m + lp^{s_{\rm 1\it}}e)\rfloor = \lfloor h_{z}(m) + l\bar{h}_{z}(p^{s_{\rm 1\it}}e)\rfloor
=  h_{z}(m) + \lfloor l\bar{h}_{z}(p^{s_{\rm 1\it}}e)\rfloor
\end{eqnarray}
D'o\`u $(5.8)$ et $(5.9)$ donnent 
\begin{eqnarray*}
l\bar{h}_{z}(p^{s_{1}}e)\geq \lfloor l\bar{h}_{z}(p^{s_{1}}e)\rfloor \geq l
\end{eqnarray*}
et donc $\bar{h}_{z}(p^{s_{1}}e)\geq 1$. Soit $m\in\sigma^{\vee}_{M}$ arbitraire. Alors,
\begin{eqnarray*}
\lfloor h_{z}(m + p^{s_{\rm 1\it}}e)\rfloor
\geq  \lfloor h_{z}(m)\rfloor  + \lfloor \bar{h}_{z}(p^{s_{\rm 1\it}}e)\rfloor\geq 
\lfloor h_{z}(m)\rfloor + 1. 
\end{eqnarray*}
On conclut que $\rm (i')$ est vraie lorsque que $C$ est affine et lorsque $\Lambda(\mathfrak{D})$
a au moins deux \'el\'ements maximaux. Si $\omega$ est l'unique \'el\'ement maximal de $\Lambda(\mathfrak{D})$
alors pour tout $z\in C$, $h_{z}$ est identiquement nulle et $\rm (i')$ est trivialement v\'erifi\'ee. 
Par cons\'equent, 
$\rm (i')$ est \'equivalente \`a $(5.7)$ dans le cas o\`u $C$ est affine. 

Supposons que $C$ est projective. Alors pour tout $z\in C$ et pour tout $m\in\sigma^{\vee}_{M}$ tel
que $A_{m}\neq 0$, on peut trouver $\varphi_{m,z}\in A_{m}$ satisfaisant $\rm ord_{\it z}\it (\varphi_{m,z}) + \lfloor
h_{z}(m)\rfloor \rm = 0$.
En rempla\c cant $\varphi_{m}$ par $\varphi_{m,z}$ dans l'argument ci-dessus et en utilisant
le lemme $5.3.14$ pour $z = \infty$, nous obtenons l'\'equivalence entre $(5.7)$ et $\rm (i'), (iii')$.  

\em Cas $h$ enti\`ere. \rm Nous supposons que la caract\'eristique de $\mathbf{k}$
est arbitraire. Dans ce nouveau cas, nous avons $d = 1$.
En rempla\c cant le diviseur poly\'edral $\mathfrak{D}$ par $\mathfrak{D}'  = \mathfrak{D} + (-v+\sigma)\cdot 0$
si $C$ est affine, et $\mathfrak{D}$ par $\mathfrak{D}' = \mathfrak{D} + (-v+\sigma)\cdot 0 + 
(v+\sigma)\cdot\infty$ si $C$ est projective, nous revenons \`a l'\'etape pr\'ec\'edente. 
En utilisant l'isomorphisme $A\simeq A[C',\mathfrak{D}']$, l'alg\`ebre
$A$ est $\partial$-stable si et seulement si les assertions $\rm (i'), (iii')$ sont vraies
pour le diviseur poly\'edral $\mathfrak{D}'$. 
Un calcul facile montre que cela est \'equivalent \`a ce que $\mathfrak{D}$ satisfait $\rm (i),(ii),(iii)$
pour $d = 1$ et $k = 0$.

\em Cas g\'en\'eral. \rm Nous pouvons supposer que la fonction $h$
n'est pas enti\`ere, i.e., $d>1$. Nous consid\'erons la normalisation 
de $B$ de 
\begin{eqnarray*}
A[\zeta^{-dh(w)}\chi^{w}]\subset\mathbf{k}(\zeta)[M], 
\end{eqnarray*}
o\`u $\zeta = \sqrt[d]{t}$, et o\`u $w\in\rm relint\,\it\omega_{M}$
satisfait $\rm p.g.c.d.\it (dh(w),d) \rm = 1$, de sorte que $B$ est une extension cyclique de $A$.
Notons que $B$ est naturellement $M$-gradu\'ee.
Nous obtenons que 
\begin{eqnarray*}
K'_{0} = \left\{\frac{a}{b}\,|\,a,b\in B_{m},\, m\in M, b\neq 0\right\} = \mathbf{k}(\zeta). 
\end{eqnarray*}
En particulier, $\mathbf{k}$ est alg\'ebriquement clos dans $B$. Ainsi, nous pouvons \'ecrire
$B = A[C',\mathfrak{D}']$ avec $C'\simeq \mathbb{P}^{1}_{\mathbf{k}}$ si $A$ est elliptique 
et $C'\simeq\mathbb{A}^{1}_{\mathbf{k}}$ dans le cas contraire. Notons $d = p^{k}l$
avec $k,l\in\mathbb{N}$ et $\rm p.g.c.d.\it (l,p) \rm = 1$. Si $\pi : C'\rightarrow C$ est le morphisme induit 
par l'inclusion 
\begin{eqnarray*}
K_{0} = \mathbf{k}(t)\subset \mathbf{k}(\zeta) = K'_{0}, 
\end{eqnarray*}
alors par la remarque $5.6.10$ (qui utilise l'hypoth\`ese de perfection sur le corps $\mathbf{k}$),
l'exercice $3.8$ dans [St, Section $3.12$], et le lemme $5.6.9$, on obtient
\begin{eqnarray*}
\mathfrak{D}' = d\cdot \Delta_{0}\cdot 0 + \sum_{z'\in C'-\{0\}}p^{k}\cdot \Delta_{z'}\cdot z'\,\,\,
\rm si\,\,\,\it C = \mathbb{A}^{\rm 1\it}_{\mathbf{k}}, 
\end{eqnarray*}
et 
\begin{eqnarray*}
\mathfrak{D}' = d\cdot \Delta_{0}\cdot 0 + d\cdot\Delta_{\infty}\cdot \infty + \sum_{z'\in C'-\{0,\infty\}}p^{k}\cdot \Delta_{z'}\cdot z'\,\,\,
\rm si\,\,\,\it C = \mathbb{P}^{\rm 1\it}_{\mathbf{k}}. 
\end{eqnarray*}
Donc $h'_{0} = dh_{0}$, $h'_{\infty} = dh_{\infty}$,
et $h'_{z'} = p^{k}h_{z}$, avec $\pi(z') = z$ et $h'_{z'}$ est la fonction
de support de $\mathfrak{D}'$ au point $z'\in C'$. De plus, $h'_{0}|_{\omega}$
est enti\`ere et $B$ satisfait les conditions du cas pr\'ec\'edent. Notons
\begin{eqnarray*}
B'_{M} = \bigoplus_{m\in M}\varphi'_{m}\,\mathbf{k}[\zeta]\,\chi^{m},\,\,\,\rm avec\,\,\,
\it \varphi'_{m} = \zeta^{-dh(m)}. 
\end{eqnarray*}
Puisque $A_{M}\subset B_{M}$ est cyclique, par le lemme $5.3.6$, la LFIHD $\partial$ sur $A_{M}$
s'\'etend en une LFIHD sur $B_{M}$, not\'ee encore $\partial$. De plus, $A$ est $\partial$-stable
si et seulement si $B$ est $\partial$-stable (voir les arguments de d\'emonstration
de [Li, $3.26$]). Par le cas pr\'ec\'edent, $B$ est $\partial$-stable si et seulement si,
pour tout $m\in\sigma^{\vee}_{M}$ tel que $m + p^{s_{1}}e\in\sigma^{\vee}_{M}$, les conditions
suivantes sont vraies.
\begin{enumerate}
 \item[(i'')] Si $h'_{z'}(m + p^{s_{1}}e)\neq 0$ alors $\lfloor h'_{z'}(m + p^{s_{1}}e)\rfloor 
- \lfloor h'_{z'}(m)\rfloor \geq 1$, $\forall z'\in C'-\{0,\infty\}$.
 \item[(ii'')] Si $h'_{0}(m + p^{s_{1}}e)\neq h'(m+p^{s_{1}}e)$, alors 
$\lfloor h'_{0}(m+p^{s_{1}}e)\rfloor - \lfloor h'_{0}(m)\rfloor \geq 1 + h'(p^{s_{1}}e)$. 
\item[(iii'')] Si $C = \mathbb{P}^{1}_{\mathbf{k}}$ alors 
$\lfloor h'_{\infty}(m + p^{s_{1}}e)\rfloor - \lfloor h'_{\infty}(m)\rfloor\geq -1-h'(p^{s_{1}}e)$. 
\end{enumerate}
En rempla\c cant, dans $\rm (i'')-(iii'')$, $h'$ par $dh$, $h'_{0}$ par $dh_{0}$, $h'_{\infty}$ par
$dh_{\infty}$ et $h'_{z'}$ par $p^{k}h_{z}$, cela montre que $A$ est $\partial$-stable
si et seulement si $\rm (i)-(iii)$ est vraie. D'o\`u le r\'esultat.     
\end{proof}
Le th\'eor\`eme suivant donne une classification des $\mathbb{T}$-vari\'et\'es affines horizontales 
de complexit\'e $1$ sur un corps parfait. Ce th\'eor\`eme est une cons\'equence imm\'ediate 
des r\'esultats pr\'ec\'edents de cette section.
\begin{theorem}
Supposons que le corps de base $\mathbf{k}$ est parfait.
Consid\'erons $p$ l'exposant caract\'eristique de $\mathbf{k}$.
Soit $\mathfrak{D}$ un diviseur $\sigma$-poly\'edral propre
sur une courbe r\'eguli\`ere $C$ et posons $A = A[C,\mathfrak{D}]$. Soient $\omega\subset M_{\mathbb{Q}}$
un c\^one et $e\in M$. Alors il existe une $\rm LFIHD$ 
homog\`ene de type horizontal sur $A$ avec $\rm deg\,\it\partial = e$ et avec $\omega$
comme c\^one des poids de $\rm ker\,\it\partial$ si et seulement si les conditions
conditions $\rm (i)-(iv)$ sont satisfaites.
\begin{enumerate}
\item[\rm (i)] $C = \mathbb{A}^{1}_{\mathbf{k}}$ ou $C = \mathbb{P}^{1}_{\mathbf{k}}$.
\item[\rm (i')] Si $C = \mathbb{A}^{1}_{\mathbf{k}}$ alors $\omega$ est un c\^one maximal 
dans le quasi-\'eventail $\Lambda(\mathfrak{D})$, et il existe un point rationnel $z_{0}\in C$
tel que $h_{z|\omega}$ est entier $\forall z\in C, z\neq z_{0}$. 
\item[\rm (i'')] Si $C = \mathbb{P}^{1}_{\mathbf{k}}$ alors il existe un point rationnel
$z_{\infty}$ tel que $\rm (i')$ est vraie pour $C_{0} :=\mathbb{P}^{1}_{\mathbf{k}}-\{z_{\infty}\}$. 
\end{enumerate} 
Sans perte de g\'en\'eralit\'e, on peut supposer que $z_{0} = 0$, $z_{\infty} = \infty$, 
et $h_{z|\omega} = 0$ $\forall z\in C, z\neq 0,z\neq \infty$. Soit $h$
l'extension lin\'eaire de $h_{0|\omega}$, $d>0$ le plus petit entier tel que $dh$
est enti\`ere et $d = lp^{k}$, avec $\rm p.g.c.d.\it (l,p) \rm = 1$.
\begin{enumerate}
\item[\rm (ii)] Il existe $s_{1}\in\mathbb{N}$ tel que $-1/d-h(p^{s_{1}}e)\in\mathbb{Z}$.
\end{enumerate}
Pour tout $m\in\sigma^{\vee}_{M}$ tel que $m+p^{s_{1}}e\in\sigma^{\vee}_{M}$,
les assertions suivantes sont vraies.
\begin{enumerate}
\item[\rm (iii)] Si $h_{z}(m+p^{s_{1}}e)\neq 0$ alors $\lfloor p^{k}h_{z}(m+p^{s_{1}}e)\rfloor - 
\lfloor p^{k}h_{z}(m)\rfloor\geq 1$, $\forall z\in C, z\neq 0,\infty$.
\item[\rm (iv)] Si $h_{0}(m+p^{s_{1}}e)\neq h(m + p^{s_{1}}e)$ alors 
$\lfloor dh_{0}(m+p^{s_{1}}e)\rfloor - \lfloor dh_{0}(m)\rfloor\geq 1 - dh(p^{s_{1}}e)$.
\item[\rm (v)] Si $C = \mathbb{P}^{1}_{\mathbf{k}}$ alors $\lfloor dh_{\infty}(m+p^{s_{1}}e)\rfloor  - \lfloor dh_{\infty}(m)\rfloor
\geq -1 - dh(p^{s_{1}}e)$. 
\end{enumerate}
Plus pr\'ecis\'ement, toute $\rm LFIHD$ homog\`ene $\partial$ de type horizontal sur $A$
avec $e, \omega$ satisfaisant $\rm (i)-(iv)$ est donn\'ee par la formule $(5.6)$
du th\'eor\`eme $5.6.8\rm ;$ si $\rm car\,\it\mathbf{k}\rm >0$ alors $\partial$
est d\'ecrite par une suite d'entiers $0\leq s_{1}<s_{2}<\ldots <s_{r},$
o\`u chaque $(p^{s_{i}}e,-1/d-h(p^{s_{i}}e))$ est une racine de Demazure $\rm ($voir $5.6.8\rm )$. De plus,  
\begin{eqnarray*}
\rm ker\,\it\partial = \bigoplus_{m\in\omega_{L}}\mathbf{k}\varphi_{m}\chi^{m},
\end{eqnarray*}
o\`u $L = h^{-1}(\mathbb{Z})$ et $\varphi_{m}\in A_{m}$ satisfait la relation
\begin{eqnarray*}
\rm div\,\it\varphi_{m} + \mathfrak{D}(m) \rm = 0\,\,\,\rm  si\,\,\,\it 
C = \mathbb{A}^{\rm 1\it}_{\mathbf{k}}\,\,\,\rm ou\,\,\,
(\rm div\,\it\varphi_{m})_{|C_{\rm 0\it}} + \mathfrak{D}(m)_{|C_{\rm 0\it}}\rm = 0
\,\,\,\rm si \it\,\,\,C = \mathbb{P}^{\rm 1\it}_{\mathbf{k}}.
\end{eqnarray*}  
\end{theorem}
Dans la suite, nous r\'e\'ecrivons les r\'esultats du th\'eor\`eme $5.6.12$ dans le langage
des diviseurs poly\'edraux colori\'es (voir [AL, Section $1$]). 
La lettre $C$ d\'esigne indiff\'eremment les courbes 
$\mathbb{A}^{1}_{\mathbf{k}}$ ou $\mathbb{P}^{1}_{\mathbf{k}}$. 
\begin{definition}
Un \em diviseur $\sigma$-poly\'edral colori\'e \rm sur $C$ est une collection
$\widetilde{\mathfrak{D}} = \{\mathfrak{D},\, v_{z}\,|\,z\in C\}$ si $C = \mathbb{A}^{1}_{\mathbf{k}}$
et $\widetilde{\mathfrak{D}} = \{\mathfrak{D},\, z_{\infty},\, v_{z}\,|\,z\in C-\{z_{\infty}\}\}$ 
si $C = \mathbb{P}^{1}_{\mathbf{k}}$, satisfaisant 
les conditions suivantes. 
\begin{enumerate}
 \item[(1)] $\mathfrak{D} = \sum_{z\in C}\Delta_{z}\cdot z$ est un diviseur $\sigma$-poly\'edral
propre sur $C$, $z_{\infty}\in C$ est un point rationnel et $v_{z}$ est un sommet de $\Delta_{z}$.
Posons $C' = C$ si $C = \mathbb{A}^{1}_{\mathbf{k}}$ et $C' = C-\{z_{\infty}\}$ si 
$C = \mathbb{P}^{1}_{\mathbf{k}}$. 
\item[(2)] $v_{\rm deg} :=  \sum_{z\in C'}[\kappa_{z}:\mathbf{k}]\cdot v_{z}$ est un sommet
de $\rm deg\,\it \mathfrak{D}_{|C'}$.
\item[(3)] $v_{z}\in N$ sauf pour au plus un point rationnel $z_{0}\in C$.
\end{enumerate}
Nous disons que $\widetilde{\mathfrak{D}}$ est un \em coloriage \rm de $\mathfrak{D}$
. Le point $z_{0}$ est appel\'e le \em point base \rm (ou le point marqu\'e), $z_{\infty}$
est le \em point \`a l'infini \rm si $C = \mathbb{P}^{1}_{\mathbf{k}}$ et $v_{z}$ est
une \em couleur \rm du poly\`edre $\Delta_{z}$.
\end{definition}
\begin{rappel}
Soit $\widetilde{\mathfrak{D}}$ un diviseur $\sigma$-poly\'edral colori\'e. Alors 
il est possible de construire le c\^one $\omega$ de l'\'enonc\'e $5.6.12$
\`a partir de $\widetilde{\mathfrak{D}}$. En effet, le dual $\omega^{\vee}$ de $\omega$
peut \^etre d\'efini comme le c\^one poly\'edral de $N_{\mathbb{Q}}$ engendr\'e par
$\rm deg\,\it\mathfrak{D}_{|C'} - v_{\rm deg}$. Nous d\'esignons par $\hat{M}$ et $\hat{N}$
les r\'eseaux respectifs $M\times \mathbb{Z}$ et $N\times \mathbb{Z}$. Notons
$\widetilde{\omega}^{\vee}\subset \hat{N}_{\mathbb{Q}}$ le c\^one engendr\'e par
$(\omega^{\vee},0)$ et $(v_{z_{0}}, 1)$ si $C = \mathbb{A}^{1}_{\mathbf{k}}$ et
par $(\omega^{\vee},0)$, $(v_{z_{0}}, 1)$ et $(\Delta_{z_{\infty}} + v_{\rm deg}- \it v_{z_{\rm 0}} + \it \omega^{\vee},\rm -1\it)$
si $C = \mathbb{P}^{1}_{\mathbf{k}}$. 
D\'esignons aussi par $d$ le plus petit entier strictement positif tel que $d v_{z_{0}}\in N$. 
Le c\^one $\omega^{\vee}$ est dit \em associ\'e \rm au diviseur $\sigma$-poly\'edral colori\'e 
$\widetilde{\mathfrak{D}}$.  
\end{rappel}
Nous introduisons maintenant la notion d'assemblage coh\'erent. Voir [AL, $1.9$] pour 
le cas classique.
\begin{definition}
Un quadruplet $(\widetilde{\mathfrak{D}},e,s,\lambda)$, o\`u $\widetilde{\mathfrak{D}}$
est un diviseur $\sigma$-poly\'edral colori\'e, $e\in M$, $s$ est une suite 
$0\leq s_{1}<s_{2}<\ldots<s_{r}$ de nombres entiers et $\lambda$
est une suite $(\lambda_{1},\ldots, \lambda_{r})$ de $\mathbf{k}^{\star}$, est appel\'e un 
\em assemblage coh\'erent \rm si les conditions suivantes sont satisfaites.
\begin{enumerate}
 \item[(A)] Notons $p$ l'exposant caract\'eristique de $\mathbf{k}$. 
Pour tout $i = 1,\ldots,r$, il existe $u_{i}\in \mathbb{Z}$
tel que $\widetilde{e}_{i} = (p^{s_{i}}e,u_{i})\in \hat{M}$
est une racine de Demazure du c\^one $\widetilde{\omega}^{\vee}$ avec rayon 
distingu\'e $\widetilde{\rho} = (d\cdot v_{z_{0}}, d)$. En particulier
$u_{i} = -1/d-\langle p^{s_{i}}e, v_{z_{0}}\rangle$. 
\item[(B)] $p^{k}\langle p^{s_{1}}e, v\rangle \geq 1 + p^{k}\langle p^{s_{1}}e, v_{z}\rangle$, 
pour tout $z\in C'-\{z_{0}\}$ et
pour tout sommet $v$ non colori\'e du poly\`edre $\Delta_{z}$, o\`u $d = lp^{k}$
avec $\rm p.g.c.d.(\it l,p\rm )  = 1$.
\item[($C_{0}$)] $d\langle p^{s_{1}}e, v\rangle\geq 1+d\langle p^{s_{1}}e, v_{z_{0}}\rangle$, pour tout
sommet $v\neq v_{z_{0}}$ du poly\`edre $\Delta_{z_{0}}$.
\item[($C_{\infty}$)] Si $C = \mathbb{P}^{1}_{\mathbf{k}}$ alors $d\langle p^{s_{1}}e, v\rangle \geq -1-d
\langle p^{s_{1}}e, v_{\rm deg}\rangle$, pour tout sommet $v\in\Delta_{z_{\infty}}$. 
\end{enumerate}
\end{definition}
Le th\'eor\`eme suivant donne une classification des LFIHD homog\`enes de type horizontal
sur $A[C,\mathfrak{D}]$ en termes d'assemblages coh\'erents. Ce r\'esultat est une cons\'equence 
imm\'ediate du th\'eor\`eme $5.6.12$.
\begin{theorem}
Soit $\mathfrak{D}$ un diviseur $\sigma$-poly\'edral propre sur $C = \mathbb{A}^{1}_{\mathbf{k}}$
ou $C = \mathbb{P}^{1}_{\mathbf{k}}$. Alors l'ensemble des $\rm LFIHD$ homog\`enes de type horizontal
sur l'alg\`ebre $A[C,\mathfrak{D}]$ est en bijection avec l'ensemble 
des assemblages coh\'erents $(\widetilde{\mathfrak{D}},e,s,\lambda)$ o\`u $\widetilde{\mathfrak{D}}$
est un coloriage de $\mathfrak{D}$. La $\rm LFIHD$ associ\'ee \`a $(\widetilde{\mathfrak{D}},e,s,\lambda)$
a degr\'e $e$ et v\'erifie la formule $(5.6)$ du th\'eor\`eme $5.6.8$. De plus, 
\begin{eqnarray*}
\rm ker\,\it\partial = \bigoplus_{m\in \omega_{L}}\mathbf{k}\varphi_{m}\chi^{m}, 
\end{eqnarray*}
o\`u $L = \{m\in M\,|\, \langle m, v_{z_{0}}\rangle \in \mathbb{Z}\}$, $\omega^{\vee}$ est
le c\^one associ\'e \`a $\widetilde{\mathfrak{D}}$ et $\varphi_{m}$ satisfait les conditions
du th\'eor\`eme $5.6.12$.  
\end{theorem}

\begin{exemple}
Supposons que $\mathbf{k}$ est un corps parfait. On consid\`ere le diviseur 
$\sigma$-poly\'edral propre $\mathfrak{D} = \Delta_{0}\cdot 0 + \Delta_{1}\cdot 1 + \Delta_{\infty}\cdot \infty$
sur la droite projective $\mathbb{P}^{1}_{\mathbf{k}}$ dont les coefficients non triviaux sont donn\'es
par 
\begin{eqnarray*}
\Delta_{0} = \left(\frac{1}{2}, 0\right) + \sigma,\,\,\, \Delta_{1} = 
\left[\left(0,0\right) ,\left(-\frac{1}{2},\frac{1}{2}\right)\right] + \sigma,\,\,\,
\Delta_{\infty} = \left(\frac{1}{2}, 0\right) + \sigma,  
\end{eqnarray*}
o\`u $\sigma = \mathbb{Q}_{\geq 0}^{2}$. Consid\'erons les 
fractions rationnelles
\begin{eqnarray*}
t_{1} = \frac{t-1}{t}\chi^{(2,0)},\,\,\, t_{2} = \chi^{(0,1)},\,\,\, t_{3} = \chi^{(1,1)},
\,\,\, t_{4} = \frac{(t-1)^{2}}{t}\chi^{(2,0)}, \,\,\, t_{5} = \frac{(t-1)^{2}}{t}\chi^{(3,0)}. 
\end{eqnarray*}
En calculant un ensemble de g\'en\'erateurs homog\`enes, on v\'erifie que 
$A = A[\mathbb{P}^{1}_{\mathbf{k}},\mathfrak{D}] = \mathbf{k}[t_{1},t_{2},t_{3},t_{4}, t_{5}]$.
Les \'el\'ements $t_{1},t_{2},t_{3},t_{4}, t_{5}$ satisfont les relations  
\begin{eqnarray*}
t_{2}t_{5} - t_{3}t_{4} = t_{3}t_{5} - t_{1}^{2}t_{2} - t_{1}t_{2}t_{4}= t_{5}^{2} - t_{1}^{2}t_{4} - t_{1}t_{4}^{2} =  0.
\end{eqnarray*}
On a un isomorphisme d'alg\`ebres
\begin{eqnarray*}
A\simeq \frac{\mathbf{k}[x_{1},x_{2},x_{3},x_{4}, x_{5}]}{(x_{2}x_{5} - x_{3}x_{4}, x_{3}x_{5} - x_{1}^{2}x_{2} - x_{1}x_{2}x_{4}, x_{5}^{2} - x_{1}^{2}x_{4} - x_{1}x_{4}^{2})}, 
\end{eqnarray*}
envoyant $t_{i}$ sur la classe de $x_{i}$,
o\`u $x_{1},x_{2},x_{3},x_{4}, x_{5}$ sont des variables ind\'ependantes sur $\mathbf{k}$.
Donnons un exemple de LFIHD homog\`ene de type horizontal sur
$A = A[\mathbb{P}^{1}_{\mathbf{k}},\mathfrak{D}]$.
Tout d'abord, on remarque que
\begin{eqnarray*}
\widetilde{\mathfrak{D}} = \left\{\mathfrak{D}, z_{\infty} = \infty,\,\,\, v_{0} = \left(\frac{1}{2}, 0\right),
\,\,\, v_{1} = (0,0)\right\}
\end{eqnarray*}
est une coloration de $\mathfrak{D}$. En effet, en posant $C' = \mathbb{P}^{1}-\{\infty\}$, le vecteur
$v_{\rm deg} = \it \left(\rm \frac{1}{2},0\right)$ est un sommet de 
\begin{eqnarray*}
\rm deg\,\it \mathfrak{D}_{|C'} = \rm \left[\left(\frac{1}{2},0\right),\left(0,\frac{1}{2}\right)\right]. 
\end{eqnarray*}
Le c\^one $\omega^{\vee}\subset \mathbb{Q}^{2}$ engendr\'e par $\rm deg\,\it\mathfrak{D}_{|C'} - v_{\rm deg}$ est
$\mathbb{Q}_{\geq 0}(-1,1) + \mathbb{Q}_{\geq}(1,0)$, de sorte que $\omega = \mathbb{Q}_{\geq 0}(1,1) + 
\mathbb{Q}_{\geq 0}(0,1)$ est un \'el\'ement maximal du quasi-\'eventail $\Lambda(\mathfrak{D}_{|C'})$.
De plus,
\begin{eqnarray*}
L = \{m\in\mathbb{Z}^{2}\,|\,\langle m, v_{0}\rangle \in\mathbb{Z}\}
 = \{(2m_{1},m_{2})\,|\,m_{1},m_{2}\in\mathbb{Z}\}.  
\end{eqnarray*}
Donc pour tout assemblage coh\'erent associ\'e \`a la coloration $\widetilde{\mathfrak{D}}$,
le noyau de la LFIHD correspondante est \'egal \`a 
\begin{eqnarray*}
\bigoplus_{m_{2} - 2m_{1}\geq 0, m_{1}\geq 0}\mathbf{k}\,t^{-m_{1}}\chi^{(2m_{1},m_{2})}. 
\end{eqnarray*} 
On observe aussi que 
\begin{eqnarray*}
\widetilde{\omega}^{\vee} = \mathbb{Q}_{\geq 0}(-1,1,0) + \mathbb{Q}_{\geq 0}(1,0,0) + \mathbb{Q}_{\geq 0} (-1,0,2)
+\mathbb{Q}_{\geq 0}(-1,1,1) + \mathbb{Q}_{\geq 0}(1,0,1). 
\end{eqnarray*}
Supposons que la caract\'eristique du corps $\mathbf{k}$ est $3$.
Posons $e = (1,2)$. Alors $\widetilde{e} = (3e,1)$ est une racine de $\widetilde{\omega}^{\vee}$ de rayon
distingu\'e $\widetilde{\rho} = (dv_{0},d) = (-1,0,2)$. Consid\'erons l'assemblage
coh\'erent $(\widetilde{\mathfrak{D}}, e=(1,2), s = 1, \lambda = 1)$ et soit $\partial$ la LFIHD associ\'ee.
Rappelons que si $t^{r}\chi^{m}\in A$, $m = (m_{1},m_{2})$ et si $\widetilde{e} = (p^{s}e,u)$ alors 
\begin{eqnarray*}
e^{x\partial}(t^{r}\chi^{m}) = \sum_{i = 0}^{\infty}\binom{d(\langle m, v_{0}\rangle + r)}{i}
\chi^{m+ip^{s}e}t^{r+iu}x^{ip^{s}}. 
\end{eqnarray*} 
Ici, on a $d = 2$, $p = 3$, $s = 1$, $u = 1$, et donc dans ce cas la LFIHD $\partial$ est d\'ecrite par la formule
\begin{eqnarray*}
e^{x\partial}(t^{r}\chi^{m})=  \sum_{i =  0}^{\infty}\binom{m_{1} + 2r}{i}\chi^{(m_{1} + 3i,m_{2} + 6i)}t^{r+i}x^{3i}.
\end{eqnarray*}

\end{exemple}

\newpage
\strut 

\end{document}